\documentclass[10pt]{amsart}
\usepackage{amsmath}
\usepackage{cases}
\usepackage{mathrsfs}
\usepackage[utf8]{inputenc} 
\usepackage{CJKutf8}
\usepackage{fullpage}
\usepackage{bbm}
\usepackage{amssymb}
\usepackage{amscd}
\usepackage{amsfonts,latexsym,amsthm,amsxtra,mathdots,latexsym,mathabx}
\usepackage[all,cmtip]{xy}
\RequirePackage{amsmath} \RequirePackage{amssymb}
\usepackage{color}
\usepackage{colordvi}
\usepackage[T1]{fontenc}
\usepackage{multicol}
\usepackage{hyperref}
\usepackage{mathtools}
\usepackage{comment}
\usepackage[margin=1.1in]{geometry}
\usepackage{xcolor}
\usepackage{parskip}
\allowdisplaybreaks[4]
\hypersetup{
	colorlinks,
	linkcolor={red!50!black},
	citecolor={blue!50!black},
	urlcolor={blue!80!black}
}

\newcommand{\Mod}[1]{\ (\mathrm{mod}\ #1)}

\newcommand{\sumstar}{\sideset{}{^*}\sum}

\newcommand{\sumb}{\sideset{}{^\flat}\sum}

\newtheorem{thm}{Theorem}[section] \newtheorem{cor}[thm]{Corollary}
\newtheorem{lem}[thm]{Lemma}  \newtheorem{prop}[thm]{Proposition}

\newtheorem {rem}[thm]{Remark}

\newtheorem{lemma}[thm]{Lemma}

\newcommand{\sumtwo}{\operatorname*{\sum\sum}}
\newcommand{\sumtwodee}{\operatorname*{\sideset{}{^d}\sum\sideset{}{^d}\sum}}
\newcommand{\sumtwostar}{\operatorname*{\sideset{}{^*}\sum\sideset{}{^*}\sum}}
\newcommand{\sumthree}{\operatorname*{\sum\sum\sum}}
\newcommand{\sumfour}{\operatorname*{\sum\sum\sum\sum}}
\newcommand{\sumfive}{\operatorname*{\sum\sum\sum\sum\sum}}
\newcommand{\sumsix}{\operatorname*{\sum\sum\sum\sum\sum\sum}}
\newcommand{\sumseven}{\operatorname*{\sum\sum\sum\sum\sum\sum\sum}}
\newcommand{\sumeight}{\operatorname*{\sum\sum\sum\sum\sum\sum\sum\sum}}

\newcommand{\sumten}{\operatorname*{\sum\sum\sum\sum\sum\sum\sum\sum\sum\sum}}
\newcommand{\sumtwlve}{\operatorname*{\sum\sum\sum\sum\sum\sum\sum\sum\sum\sum\sum\sum}}
\newcommand{\sumonethree}{\operatorname*{\sum\sum\sum\sum\sum\sum\sum\sum\sum\sum\sum\sum\sum}}
\newcommand{\sumonefour}{\operatorname*{\sum\sum\sum\sum\sum\sum\sum\sum\sum\sum\sum\sum\sum\sum}}

\newcommand{\sumonesix}{\operatorname*{\sum\sum\sum\sum\sum\sum\sum\sum\sum\sum\sum\sum\sum\sum\sum\sum}}

\newcommand{\sumd}{\sideset{}{^d}\sum}

\newcommand{\tRe}{\textup{Re }}

\newcommand{\bfrac}[2]{\left(\frac{#1}{#2}\right)}
\newcommand{\al}{\boldsymbol\alpha}
\newcommand{\be}{\boldsymbol \beta}

\newcommand{\chiq}{\chi \Mod{q}}

\newcommand{\cb}{\overline{\chi}}
\newcommand{\V}{V}
\newcommand{\W}{\mathcal W}
\newcommand{\D}{\mathcal D (\Psi_1,\Psi_2, Q_1,Q_2; \boldsymbol \alpha, \boldsymbol \beta)}

\newcommand{\Hc}{\mathcal H}
\newcommand{\B}{\mathcal B}

\newcommand{\R}{{\mathrm{Re}}}

\newcommand{\h}{\frac{1}{2}}

\newcommand{\Wt}{\widetilde{\mathcal{W}}}

\newcommand{\Z}{\mathcal{Z}}

\newcommand{\A}{\mathcal{A}}

\newcommand{\Q}{\mathcal{Q}}

\newcommand{\calS}{\mathcal{S}}
\newcommand{\tQ}{\widetilde{\mathcal{Q}}}

\newcommand{\m}{\mathfrak{m}}
\newcommand{\n}{\mathfrak{n}}
\newcommand{\ex}{\mathrm{e}}

\allowdisplaybreaks


\begin{document}
	\title{The second moment of twisted modular $L$-functions and Dirichlet $L$-functions at the central point}
   \author{Zhengye Chen}
\address{School of Mathematics, Shandong University, Jinan 250100, China}
\email{\href{mailto:202211798@mail.sdu.edu.cn}{202211798@mail.sdu.edu.cn}}

    \author{Yongxiao Lin}
\address{Data Science Institute, Shandong University, Jinan 250100, China}
\address{State Key Laboratory of Cryptography and Digital Economy Security, Shandong University, Jinan 250100, China}
\email{\href{mailto:yongxiao.lin@sdu.edu.cn}{yongxiao.lin@sdu.edu.cn}}
	
	\begin{abstract}
		We prove an asymptotic formula with a power-saving error term for
		the second moment
		$$
		\sum_{q \in \mathcal{Q}}  \; \; \sumb_{\chi \bmod q} \left| L\big( \tfrac{1}{2} , \chi \big) L\big(\tfrac{1}{2},f\otimes \chi)\right|^2
		$$
		of the twisted ${\rm GL}(2)$ {$L$}-function and the Dirichlet {$L$}-function at the central point under the assumption of Selberg's eigenvalue conjecture. Here $f$ is a fixed Hecke holomorphic cusp form for $\mathrm{SL}(2,\mathbb{Z})$ and the sum over $\chi$ runs over all primitive even Dirichlet characters modulo $q$, and $$\mathcal{Q}=\left\{q=q_1 q_2\asymp Q: q_1 \leq Q_1, q_2 \leq Q_2,\ Q_2\asymp Q^{\delta_1},(q,6)=1,(q_1,q_2)=1 \right\}$$ with $0<\delta_1<0.0004$.
	\end{abstract}
	\maketitle
	
	\section{Introduction }
	Moments of $L$-functions have been extensively studied due to their fundamental importance in number theory. The study of moments of $L$-functions originated with the pioneering work \cite{HL} of Hardy and Littlewood, who established an asymptotic formula for the second moment of the Riemann zeta function on the critical line
	\[
	\int_0^T \left|\zeta\!\left(\tfrac12+it\right)\right|^2\,dt.
	\]
	More generally, let
	\[
	I_k(T):=\int_0^T \left|\zeta\!\left(\tfrac12+it\right)\right|^{2k}\,dt.
	\]
Asymptotic formulas are known only for the cases $k=1$ and $k=2$. The latter was first established by Ingham \cite{Ingham}.
	
	 For higher moments, no asymptotic formula has yet been established.
It is conjectured in \cite{CFKRS} that
\[
I_k(T)=TP_{k^2}(\log T)+O(T^{\frac12+\varepsilon}),
\]
where $P_{k^2}(x)$ is a polynomial of degree $k^2$. Determining the coefficients of $P_{k^2}(x)$ is itself a highly nontrivial problem. A major breakthrough was made by Keating and Snaith \cite{KS}, who introduced a random matrix model and, based on this model, conjectured the leading order term of the asymptotic formula. More precisely, they predicted that there exists an explicit constant $c_k$ such that
\[
I_k(T)\sim c_k\,T(\log T)^{k^2}.
\]
 Subsequently, Conrey, Farmer, Keating, Rubinstein, and Snaith \cite{CFKRS} conjectured the complete asymptotic formula by giving an explicit expression for the polynomial $P_{k^2}(x)$. Their conjecture is consistent with all currently known asymptotic formulas for low moments and has been further supported by a substantial body of subsequent works, which can be seen in \cite{CFKRS,CGh,CGo,CK,DGH,HB}.

The conjecture in \cite{CFKRS} also applies to moments of primitive Dirichlet $L$-functions. Since both families are of unitary symmetry type, the conjectural main terms are given by polynomials of degree $k^2$ with analogous structures, while the arithmetic factors in $c_k$ reflect the differences between the two families.

 However, the rich arithmetic structure encoded in Dirichlet characters makes the evaluation of moments of Dirichlet $L$-functions considerably more challenging. For the moments
\[
\sideset{}{^*}\sum_{\chi \bmod q}\left|L\!\left(\tfrac12,\chi\right)\right|^{2k},
\]
where $\sum^*$ denotes the sum over primitive Dirichlet characters modulo $q$, asymptotic formulas are currently known only for $k=1$ and $k=2$. For a long time, obtaining an asymptotic formula for the fourth moment remained a major open problem. A major breakthrough was made by Young \cite{Young}, who established an asymptotic formula with a power saving error term for the fourth moment over prime moduli by means of a delicate spectral analysis of the off-diagonal terms. More precisely, for a prime modulus $p$, he obtained
\[
\frac{1}{\phi^\star(p)}\
\sumstar_{\chi \bmod p}
\left|L\!\left(\tfrac12,\chi\right)\right|^4
=
P_4(\log p)
+
O\!\left(p^{-\frac{1}{512}+\varepsilon}\right),
\]
where $\phi^\star(p)$ is the number of primitive Dirichlet characters modulo $p$. Subsequently, Wu \cite{Wu1} extended Young's result from prime moduli to arbitrary moduli, proving that
\[
\frac{1}{\phi^\star(q)}\
\sumstar_{\chi \bmod q}
\left|L\!\left(\tfrac12,\chi\right)\right|^4
=
\prod_{p \mid q}
\frac{(1-p^{-1})^3}{1+p^{-1}}
P_4(\log q)
+
O\!\left(q^{-\frac{1}{14}+\frac{3}{7}\theta+\varepsilon}\right),
\]
where $\theta$ denotes the best currently known approximation towards the Ramanujan--Petersson conjecture for Maass cusp forms, and one may take
$
\theta=\frac{7}{64}
$
by the work \cite{KimS} of Kim and Sarnak.

	For $k\geq 3$, it appears to be beyond the reach of currently available methods and obtaining an asymptotic formula remains a fundamental open problem.
	
	If one considers a longer average over the modulus $q$, Conrey, Iwaniec, and Soundararajan invented the \emph{asymptotic large sieve} technique in \cite{CISa}, and using that method they \cite{CIS} established an asymptotic formula for
	\begin{equation}\label{CIS-sixmoment}
	\sum_{q\le Q}\ \sumb_{\chi \bmod q}
	\int_{\mathbb{R}}
	\left|L\!\left(\tfrac12+it,\chi\right)\right|^6
	\phi(t)\,dt,
	\end{equation}
	where $\sumb$ denotes the sum over primitive even Dirichlet characters modulo $q$, and $\phi$ is a smooth weight function supported on an interval of length $\asymp 1$. Along this direction, Chandee and Li \cite{CL} obtained an asymptotic formula for the eighth moment under the assumption of GRH and later Chandee, Li, Matom\"aki and Radziwi{\l\l} \cite{CLMR2} improved the result by removing the GRH. More recently, Conrey, Kwan, Lin, and Turnage-Butterbaugh proved, among other things, an asymptotic formula \cite[Corollary 1.5]{CKLTB} for the second moment of the twisted $\rm{GL(3)}$ $L$-function, which can be regarded as a cuspidal analogue of \eqref{CIS-sixmoment}.
	
	It is worth noting that introducing an additional average in the $t$-aspect makes the problem considerably more treatable. Indeed, after applying the approximate functional equation, one obtains
	\begin{equation}\label{sixmoment}
		\left|L\!\left(\tfrac12,\chi\right)\right|^6
		\approx
		\sumtwo_{\substack{m,n\\ mn\ll q^3}}
		\frac{d_3(m)d_3(n)\chi(m)\overline{\chi(n)}}{\sqrt{mn}},
	\end{equation}
	while
	\[
	\int_{\mathbb{R}}
	\left|L\!\left(\tfrac12+it,\chi\right)\right|^6\phi(t)\,dt
	\approx
	\sumtwo_{\substack{m,n\\ (\max(m,n))^2\ll q^3}}
	\frac{d_3(m)d_3(n)\chi(m)\overline{\chi(n)}}{\sqrt{mn}}.
	\]
	The additional integration over $t$ effectively replaces the condition
	$mn\ll q^3$ by the stronger condition
	$(\max(m,n))^2\ll q^3$, thereby eliminating the contribution of
	highly unbalanced sums, namely those terms where
	$\max(m,n)$ is much larger than $\min(m,n)$. 
	
	Without the additional average in the $t$-aspect, the problem becomes substantially more difficult because of the presence of the unbalanced range. A major breakthrough was made by Chandee, Li, Matom\"aki and Radziwi{\l\l} \cite{CLMR2}, who removed this additional $t$-averaging by using the Kuznetsov trace formula to handle the unbalanced range. As a result, they established an asymptotic formula for
	\[
	\sum_{q\le Q}\ \sumb_{\chi \bmod q}
	\left|L\!\left(\tfrac12,\chi\right)\right|^6.
	\]
	
	The present work is motivated by \cite{CLMR2}. Let
	\[
	E(z)=\left.\frac{\partial}{\partial s}E(z,s)\right|_{s=\frac12}
	\]
	denote the derivative of the Eisenstein series at the central point. Then
	\[
	\left|L\!\left(\tfrac12,\chi\right)\right|^6
	=
	\left|L\!\left(\tfrac12,E\otimes\chi\right)\right|^2
	\left|L\!\left(\tfrac12,\chi\right)\right|^2.
	\]
	This naturally leads to the following question: if one replaces the Eisenstein series $E$ by a fixed holomorphic Hecke  eigenform $f$, can one still establish an asymptotic formula? This problem is closely related to the work \cite{BM} of Blomer and Mili\'cevi\'c, who established an asymptotic formula for
	\[
	\sideset{}{^*}\sum_{\chi \bmod q}
	\left|L\!\left(\tfrac12,f\otimes\chi\right)\right|^2
	\]
	under the assumption that $q$ admits a suitable factorization. See also the work \cite{KMS} by Kowalski, Michel, and Sawin that resolves the problem in the case of prime moduli. More recently, Mili\'cevi\'c, Qin, and Wu \cite{MQW} and Pascadi \cite{Pas} independently extended this asymptotic formula to arbitrary moduli $q$ (see also \cite{BP} for further improvements).
	
	In our case, due to certain technical difficulties, we have to assume that $q$ also admits a suitable factorization. Removing this assumption would be an interesting problem for future research.
	
	We may gain some insight into the difficulty of the problem by considering two separate aspects. The first is the evaluation of the main term, and the second is the estimation of the error term.
	
For the main term, it is easy to see that the leading order term is of size $Q^2\log^2 Q$. To determine the lower-order terms, one may follow the framework of \cite{CIS}, although the calculations become substantially more delicate and involved. In our case we assume that $q$ admits a suitable factorization, and this makes the summation over $q$ considerably more complicated.

	On the other hand, the difficulty of controlling the error term can be seen by comparing our setting with that of \eqref{sixmoment}. Applying the approximate functional equation heuristically, we obtain
	\[
	\left|L\!\left(\tfrac12,f\otimes\chi\right)\right|^2
	\left|L\!\left(\tfrac12,\chi\right)\right|^2
	\approx
	\sum_{\substack{m,n\\ mn\ll q^3}}
	\frac{(1*\lambda_f)(m)\,(1*\lambda_f)(n)\chi(m)\overline{\chi(n)}}{\sqrt{mn}}.
	\]
	Since $d_3=1*1*1$, each variable $m$ in \eqref{sixmoment} can be decomposed into three smooth variables, which provides greater flexibility to balance the variables. In contrast, the coefficient $(1*\lambda_f)(m)$ admits only a two-fold convolution structure. Consequently, the corresponding unbalanced range is substantially larger than in the sixth moment problem case, making the analysis significantly more challenging.
	
	This heuristic is reflected in the available methods. In fact, the approach in \cite{CLMR2}, while successful for the sixth moment case, does not completely cover the unbalanced range arising in the present setting. This necessitates the introduction of new ideas to handle the remaining range.
	
	Our main innovation is a generalization of the $p$-adic stationary phase method \cite{BM} of Blomer and Mili\'cevi\'c. By combining this with the ideas of \cite{CLMR2} in two complementary regimes, we are able to treat the entire unbalanced range. We refer the reader to Section \ref{sketch-of-proof} for a detailed account of the strategy and the role of each ingredient.
	
	Let $0<\delta_1<0.0004$. We write $\mathcal{Q}=\left\{q=q_1 q_2\asymp Q: q_1 \leq Q_1, q_2 \leq Q_2,\ Q_2\asymp Q^{\delta_1},(q,6)=1, (q_1,q_2)=1 \right\}$. Then we obtain the following result
	\begin{cor}Assume Selberg's eigenvalue conjecture.  As $Q\to \infty$, we have
		$$
		\sum_{q \in \mathcal{Q}}  \; \; \sumb_{\chi \bmod q} \left| L\big( \tfrac{1}{2} , \chi \big) L\big(\tfrac{1}{2},f\otimes \chi)\right|^2\sim 2c_1L(1,f)L(1,\mathrm{sym}^2f)\sum_{q \in \mathcal{Q}}  h_q\phi^\flat(q)\left(\log q\right)^2,
		$$
		where $$
		c_1=\prod_p\left(1-\frac1p\right)^2\left(1+\frac{\lambda_f(p)+2}{p}+\frac{1}{p^2}\right)
		$$
		and
		$$
		h_q=\prod_{p\mid q}\frac{\left(1-\frac{\lambda_f(p)}{p}+\frac{1}{p^2}\right)\left(1-\frac{\lambda_f(p^2)}{p}+\frac{\lambda_f(p^2)}{p^2}-\frac{1}{p^3}\right)}{1+\frac{\lambda_f(p)+2}{p}+\frac{1}{p^2}},
		$$
		and $\phi^\flat(q)$ counts the number of primitive even characters with modulus $q$.
	\end{cor}

	We remark that the primitive odd characters case can be treated by analogous method. Next, we state our main theorem. The corollary is an immediate consequence of it.
	\begin{thm}\label{MTFinal}
		Assume Selberg's eigenvalue conjecture. Let $f$ be a fixed Hecke holomorphic cusp form for $\mathrm{SL}(2,\mathbb{Z})$ of weight $\kappa$ with normalized Hecke eigenvalues $\lambda_f(n)$. Let $0<\delta_1< 0.0004$, $Q \geq 3$, and $Q\ll Q_1Q_2 \ll Q$ with $Q_2\asymp Q^{\delta_1}$. Then for any smooth function $\Psi_1$ and $\Psi_2$ supported on $[1,2]$ satisfying $\Psi_1^{(j)}(x)\ll Q^{j\varepsilon}$ and $\Psi_2^{(j)}(x)\ll Q^{j\varepsilon}$, we have
		\begin{align*}
			\sumtwo_{\substack{q_1,q_2\\ (q_1q_2,6)=1\\(q_1,q_2)=1}} \Psi_1\left(\frac{q_1}{Q_1}\right) \Psi_2\left(\frac{q_2}{Q_2}\right)&\ \sumb_{\chi(\bmod q_1q_2)}\left| L\big( \tfrac{1}{2} , \chi \big) L\big(\tfrac{1}{2},f\otimes \chi)\right|^2\\ 
			&=	\sumtwo_{\substack{q_1,q_2\\ (q_1q_2,6)=1\\(q_1,q_2)=1}} \Psi_1\left(\frac{q_1}{Q_1}\right) \Psi_2\left(\frac{q_2}{Q_2}\right) \phi^\flat(q_1q_2) g(q_1q_2)+O\left(Q^{2-\frac15\delta_1 }\right),
		\end{align*}
		where
		$$g(q)=2c_1h_q\cdot L(1,f)L(1,\mathrm{sym}^2f)\Big(\left(\log q +c_2+i_q\right)\cdot\left(\log q +c_3+j_q\right)+c_4+k_q\Big),$$
        and the parameters
		$c_1,c_2,c_3,c_4,h_q,i_q,j_q,k_q$
		are defined explicitly in Appendix~\ref{Appendixa}.
	\end{thm}
	The  theorem is derived from Theorem \ref{Main theorem}, and the proof will be given in the appendix.

	\begin{rem}
		\par (a) We remark that, when expanded in powers of  $\log q$, the coefficients of the resulting polynomial do not simplify to a convenient closed form. For this reason, we do not expand them explicitly in the statement of the theorem, instead retain them in their product form.

		\par (b) Since $h_q,i_q,j_q,k_q= (\log\log q)^{O(1)}$, the main term is closed to a polynomial in $\log q$ whose coefficients vary only by $(\log\log q)^{O(1)}$.
	\end{rem}
	\subsection{ Sketch of Proof}\label{sketch-of-proof}
	By the approximate functional equation for $L$-functions, the problem of evaluating the mixed moment is roughly equivalent to estimating
	\begin{equation}\label{sketch1}
		\sum_{q_1 \sim Q_1}\sum_{q_2\sim Q_2} \sum_{\chi \bmod q_1q_2} \sum_{n\sim N} \sum_{m \sim M} \frac{1*\lambda_f(n) \chi(n) 1*\lambda_f(m) \bar{\chi}(m)}{\sqrt{n m}}
	\end{equation}
	for all ranges of $M$ and $N$ with $MN\ll Q^{3+\varepsilon}$, $Q\asymp Q_1Q_2$ and $Q_2\asymp Q^{\delta_1}$ with $\delta_1$ small. Roughly speaking, when $M$ and $N$ are small, one can subtract various main terms, and the main term is of size $Q^2(\log Q)^2$. For the remaining range, we need to control the bound within $Q^{2-\varepsilon}$ to give a power-saving error term. The trivial bound of \eqref{sketch1} is $Q^{2.5+\varepsilon}$. Hence our goal is to save $Q^{1/2+\varepsilon}$ in \eqref{sketch1}. Without loss of generality, we assume that $M\geq N$. Moreover, in this sketch we often omit the $\varepsilon$ for the sake of clarity; for instance, we may write $Q^2$ instead of $Q^{2+\varepsilon}$. We shall consider the hardest situation where $MN=Q^3$.\par 
    
	By using orthogonality of characters we can think of the sum as essentially
	\begin{equation}\label{sketch2}
		Q \sum_{q_1 \sim  Q_1}\sum_{q_2\sim Q_2} \; \sumtwo_{{\substack{n \sim N, \; m \sim M \\ n \equiv m \bmod{q_1q_2}}}} \frac{1*\lambda_f(n)1*\lambda_f(m)}{\sqrt{n m}}.
	\end{equation}
	For the diagonal term $m=n$, we select it into one of the main terms, and for the off-diagonal term we apply the divisor switching trick following Conrey--Iwaniec--Soundararajan, that is to say, writing $n-m=hq_1q_2$, the sum over $q_1$ and $n\equiv m \bmod q_1q_2$ can be transformed into a sum over $h$ and $n\equiv m\bmod hq_2$, and the new variable $h$ is of size $M/Q$. We rewrite \eqref{sketch2} as
	$$
	Q \sum_{h \ll  M/Q}\sum_{q_2\sim Q_2} \; \sumtwo_{{\substack{n \sim N, \; m \sim M \\ n \equiv m \bmod{hq_2}}}} \frac{1*\lambda_f(n)1*\lambda_f(m)}{\sqrt{n m}}.
	$$
	Then we remove the congruence condition by orthogonality of characters, introducing a sum over all characters $\psi \bmod hq_2$ as roughly
	$$
	\frac{QQ_1}{M} \sum_{h \ll M/Q} \sum_{q_2\sim Q_2}\sum_{\psi \bmod hq_2} \sum_{n\sim N} \sum_{m \sim M} \frac{1*\lambda_f(n) \psi(n) 1*\lambda_f(m) \bar{\psi}(m)}{\sqrt{n m}}.
	$$
	The contribution of the principal character $\psi_0 \bmod hq_2$ gives another main term. For those non-principal characters we use a standard contour shift argument, getting roughly
	$$
	\frac{QQ_1}{M}\sum_{l\ll M/Q_1}\sum_{\psi\bmod l}\left|L\left(\frac 12,\chi\right)L\left(\frac 12,f\otimes \chi\right)\right|^2.
	$$
	Applying the large sieve inequality, one expects a bound of size $MQ_2Q^{\varepsilon}$. However, due to a technical obstruction 
        we only obtain an upper bound of size $MQ_2^2Q^{\varepsilon}$, which is sufficient provided that $M \ll Q^{2-\delta}Q_2^{-2}$. The bound can be found in Section \ref{ssec:errorEg} and \eqref{fitsterrorsketch}. Notice that the saving comes from shortening the length of the sum $q$ from of size $Q$ to $\frac{M}{Q_1}$. 
	
	Next we assume that $M>Q^{2-\delta}Q_2^{-2}$. In this range we open up $1*\lambda_f(m)$ and $1*\lambda_f(n)$ in \eqref{sketch1}, and by smooth partition of unity and orthogonality of characters, we aim to estimate a sum roughly of the form 
	\begin{equation}\label{sketch3}
		\frac{Q}{\sqrt{EGN_1N_2}} \sum_{q_1 \sim Q_1}\sum_{q_2\sim Q_2} \; \sumfour_{\substack{n_1 \sim N_1, n_2\sim N_2 \\ e \sim E, g \sim G\\ eg\equiv n_1n_2\bmod q_1q_2}} \lambda_f(e)\lambda_f(n_1).
	\end{equation}
	We first apply Voronoi on $e$, whose detail can be seen in Lemma \ref{lem:Voro}, getting roughly
	\begin{equation}\label{sketch4}
		\frac{\sqrt{E}}{Q\sqrt{GN_1N_2}} \sum_{q_1 \sim Q_1}\sum_{q_2\sim Q_2} \; \sumfour_{\substack{n_1 \sim N_1, n_2\sim N_2 \\ e \ll \frac{Q^2}{E}, g \sim G}} \lambda_f(e)\lambda_f(n_1)S(e,\overline{g}n_1n_2;q_1q_2),
	\end{equation}
	where $S(a,b;q)$ is the standard Kloosterman sum. Next, we use Kuznetsov's trace formula, following the strategy used in \cite[Section 9]{CLMR2}. There they adjust the approach used in \cite[Theorem 10]{DI} and assume that $\theta=\sup_{q\geq 1}\sqrt{\max(0,\frac14-\lambda_1(q))}$, where $\lambda_1(q)=\frac14+(i\kappa_q)^2$ is the Laplacian eigenvalue for the Maass cusp form of level $q$, getting a bound roughly of the form  
	\begin{align}\label{Kuzinsketch}
		&	\sum_{g\sim G}\sum_{e\sim E}\sum_{n\sim N}\sum_{q_1\sim Q_1}\sum_{q_2\sim Q_2}S(e,n\bar{g};q_1q_2) \\ &\hskip1in\ll QGQ_2\sqrt{EGN}\left(1+\sqrt{\frac{E}{GQ_2}}\right)\left(1+\sqrt{\frac{N}{GQ_2}}\right)\left(1+\frac{X^2}{\left(1+\frac{GQ_2}{E}\right)^2\left(1+\frac{GQ_2}{N}\right)}\right)^{\theta},\nonumber
	\end{align}
	when $X:=\frac{Q\sqrt{G}}{\sqrt{EN}}\gg 1$. We may briefly outline their approach. Firstly, we rewrite the classical Kloosterman sums in terms of Kloosterman sums associated with cusps. We then apply the Kuznetsov trace formula, which decomposes the expression into three parts: the holomorphic spectrum, the Maass spectrum, and the continuous spectrum, and the Maass spectrum dominates. Then we apply bounds for the Bessel transform and in this step we take advantage of $\theta$'s upper bound(one may take $\theta\leq\frac{7}{64}$ due to Kim and Sarnak \cite{KimS}), and finally we use the spectral large sieve to get the final bound.
	\par 
	Now we plug the bound \eqref{Kuzinsketch} in \eqref{sketch4}, getting that \eqref{sketch1} is roughly bounded by
	\begin{align}\label{Kuzbound}
		Q \sqrt{Q_2}\left(G\sqrt{Q_2}+\sqrt{NG}\right)\left(1+\frac{Q^2}{Q_2(N+GQ_2)}\right)^{\theta}.
	\end{align}
	In \cite[Equation (7)]{CLMR2}, they got roughly the same bound \eqref{Kuzbound} via Kuznetsov (since $Q_2\asymp Q^{\delta_1}$ and $\delta_1$ is sufficiently small). However, in the ternary divisor function $d_3(n)$ setting, they have the freedom to assume that $E\geq F\geq G$ by symmetry, and $G$ will be somewhat small. Then the bound \eqref{Kuzbound}, even with $\theta=\frac{7}{64}$, is sufficient for their purposes to give a power-saving error term. In contrast, in our setting we do not have the feature to use symmetry to assume that $G$ is small, and when $G$ is large, it is difficult to deal with. \par Up to this point, we have followed the argument of \cite{CLMR2}. We now proceed with a different approach adapted to our setting. \par 
	We return to our case.
	We assume Selberg's eigenvalue conjecture \cite{Selberg65}, namely, that there are no exceptional eigenvalues, so that $\theta=0$. We remark that if we do not assume the conjecture, then the bound getting from the spectral analysis is unfortunately not enough for our use. Hence the bound \eqref{Kuzbound} is enough if and only if $G\ll Q$ since $N\ll Q$, and when $G\gg Q$, we apply Poisson on $g$ in \eqref{sketch3}, and we can get a bound $\ll Q^{-A}$, which is enough. Hence we only need to consider the case where $G\asymp Q$.\par
	For $x,y,z\in \mathbb{Z}$, we denote by
	$$\mathcal {KS}(x,y,z; q) := \sumtwostar_{a, b \bmod q} \ex\bfrac{ax+by+\overline{ab}z}{q}
	$$
	the hyper-Kloosterman sum and 
	\begin{equation}\label{KL3}\text{Kl}_3(u;q):=\frac{1}{q}\mathcal {KS}(u,1,1; q).
	\end{equation}
	
	Next we apply Poisson on $g$ in \eqref{sketch4}, whose detail can be seen in Lemma \ref{lem:3timespoisson}, getting roughly
	\begin{align}\label{Kl3sketch}
		\frac{\sqrt{EG}}{Q\sqrt{N_1N_2}} \sum_{q_1 \sim Q_1}\sum_{q_2\sim Q_2} \; \sumfour_{\substack{n_1 \sim N_1, n_2\sim N_2 \\ e \ll \frac{Q^2}{E}, g \ll \frac{Q}{G}}} \lambda_f(e)\lambda_f(n_1)\text{Kl}_3(egn_1n_2;q_1q_2).
	\end{align}
	
	Then we apply Iwaniec's idea in \cite{I96} (see also \cite{Luo} by Luo) and use the Cauchy--Schwarz inequality to transform the Type~II sum into a Type~I sum. At this stage, we substantially require that $q$ admits a suitable factorization $q=q_1q_2$; without such a factorization, the off-diagonal contribution in \eqref{offdiag} would be beyond our reach. We group $egn_1n_2$ into a longer variable $l$ and apply the Cauchy--Schwarz inequality and Poisson summation over $l$ thereafter. We find that \eqref{Kl3sketch} is bounded by
	\begin{align}\label{Sketch8}
		Q^{1/2}Q_2^{1/2}\left(\sumthree_{\substack{q_1\sim Q_1 \\q_1^\prime \sim Q_1\\q_2\sim Q_2}}\frac{L}{q_2[q_1,q_1^\prime]}\sum_{|h|<\frac{q_2[q_1,q_1^\prime]}{L}}\sum_{b\bmod q_2[q_1,q_1^\prime]}\text{Kl}_3(b;q_2q_1)\overline{\text{Kl}_3(b;q_2q_1^\prime )} \,\ex\bfrac{hb}{q_2[q_1,q_1^\prime]} \right)^{1/2},
	\end{align}
	where $L=\frac{Q^3N}{EG}$.\\
	For the diagonal term $h=0$, we use a trivial bound. For the off-diagonal terms, we apply a new $p$-adic stationary phase method to obtain a square-root cancellation estimate for the resulting algebraic exponential sum. (See Sections \ref{Squarerootcancellation} and \ref{padiccomputation} for details). \\
	We may explain more in this argument. By the Chinese Remainder Theorem, we are reduced to estimating 
	\begin{align}
		\sum_{b\bmod p^\alpha}\text{Kl}_3(A^3\overline{\nu}b;p^\alpha)\overline{\text{Kl}_3(B^3\overline{\nu}b;p^\beta)} \,\ex \bfrac{Cb}{p^\alpha},
	\end{align}
    where $\alpha\geq\beta\geq  1$.
	We obtain the following Theorem (the proof can be found in Section \ref{padiccomputation}):
	
	\begin{thm}\label{cruciallemma}
		For any prime $p>3$, $\alpha\geq\beta\geq  1$, $A,B\in (\mathbb{Z}/p\mathbb{Z})^{\times 3}$, $C\in \mathbb{Z}$ and $\nu\in (\mathbb{Z}/p\mathbb{Z})^{\times}$, the following bound holds:
		\begin{equation}
			\sum_{b\bmod p^\alpha}\mathrm{Kl}_3(A\overline{\nu}b;p^\alpha)\overline{\mathrm{Kl}_3(B\overline{\nu}b;p^\beta)} \,\ex \bfrac{Cb}{p^\alpha} \ll p^{\alpha/2}\left(A-B,C,p^\alpha\right)^{1/2}
		\end{equation}
		and
		\begin{equation}
			\sum_{b\bmod p^\alpha}\mathrm{Kl}_3(A\overline{\nu}b;p^\alpha)\ex \bfrac{Cb}{p^\alpha}\ll p^{\alpha/2}\left(A,C,p^\alpha\right)^{1/2},
		\end{equation}
		where the implied constant does not depend on $A$, $B$, $C$, $p$, $\alpha$, $\beta$ and $\nu$. \par 
		When $\alpha>\beta $, we have an improved bound
		$$
		\sum_{b\bmod p^\alpha}\mathrm{Kl}_3(A\overline{\nu}b;p^\alpha)\overline{\mathrm{Kl}_3(B\overline{\nu}b;p^\beta)} \,\ex \bfrac{Cb}{p^\alpha}\ll p^{\alpha/2}.
		$$
	\end{thm}
	By the argument of Blomer and Mili\'cevi\'c in \cite[Proposition 23]{BM}, for a prime $p>2$, $p\nmid d$, and $s\ge 1$, it is proved that
	\begin{align}
		\sum_{m \bmod p^s} S\left(k_1, m d, p^s\right) S\left(k_2, m d, p^s\right) e\left(-\frac{h m}{p^s}\right)\ll p^{\frac32s}\left(k_1-k_2,h,p^s\right)^{\frac12}.
	\end{align}
	Hence our result generalizes theirs. Our method is to rewrite the hyper-Kloosterman sum $\text{Kl}_3(\cdot;p^\alpha)$ in terms of $p$-adic cubic roots when $\alpha\geq 2$ (see Lemma \ref{Kloo} for instance); then $\text{Kl}_3(\cdot;p^\alpha)$ naturally extends from residue classes modulo $p^\alpha$ to the $p$-adic integers $\mathbb Z_p$. In this new $p$-adic parametrization, the hyper-Kloosterman sum reduces to a much simpler exponential sum. Then a stationary phase argument can be applied to proceed. Moreover, the prime modulus case was treated by Fouvry, Kowalski and Michel in \cite{FKM}, and their result suffices for our purposes.\par
	
	The key idea within both Blomer--Mili\'cevi\'c's and our argument is that the completed sum 
	$$\sum_{b\bmod p^\alpha}f(b)$$
	with $f(\cdot)$ a $p^\alpha$-periodic function satisfying a suitable expansion property under modulo $p$ shifts, can be evaluated by a stationary phase argument for $\alpha\geq 2$. The sum can then be transformed into a sum over the solutions of a congruence equation modulo $p$. Then we can use Hensel's lemma to lift the solution modulo $p$ into a solution in $\mathbb{Z}_p$, and then the completed sum can have a very nice simplification by the $p$-adic parametrization.

	We remark that, in the argument rewriting the hyper-Kloosterman sum in terms of $p$-adic cubic roots, we use Hensel's lemma to lift solutions of
	\[
	u^2\equiv a \bmod p
	\quad \text{and} \quad
	u^3\equiv a \bmod p
	\]
	to solutions of
	\[
	u^2=a
	\quad \text{and} \quad
	u^3=a
	\]
	in $\mathbb{Z}_p$. It is in this step of applying Hensel's lemma that we need to assume that $p>3$, and that is why we assume $(q_1q_2,6)=1$ at the beginning. 
	
	Denote $$(a,b)^{\#}=\prod_{\substack{p\mid (a,b)\\ v_p(a)=v_p(b)}}p^{v_p(a) }.$$\\
	We obtain the following theorem, whose proof can be found in Section \ref{padiccomputation}.

	\begin{thm}\label{Squarerootcancel}
		For $(k_1k_1^\prime k_2,6)=1$, $(\nu,k_1k_1^\prime k_2)=1$ and $h\in \mathbb{Z}$, we have
		$$\sum_{b\bmod k_2[k_1,k_1^\prime]}\mathrm{Kl}_3(b\overline{\nu};k_2k_1)\overline{\mathrm{Kl}_3(b\overline{\nu};k_2k_1^\prime )} \,\ex\bfrac{hb}{k_2[k_1,k_1^\prime]}\ll (k_2[k_1,k_1^\prime])^{1/2}(h,(k_1k_2,k_1^\prime k_2)^{\#})^{1/2}.$$
	\end{thm}
	This roughly yields square-root cancelation. Upon plugging in such a bound, \eqref{Sketch8} is bounded by 
	\begin{align}\label{offdiag}
		\frac{Q^2\sqrt{NQQ_2}}{\sqrt{EG}}+Q^2Q_2^{-\frac 14}\ll 	\frac{Q^4\sqrt{Q_2}}{EG}+Q^2Q_2^{-\frac 14}.
	\end{align}
	This bound suffices, provided that $EG \gg Q^2\sqrt{Q_2}$. We therefore assume that $EG \ll Q^2\sqrt{Q_2}$. Since $EG \asymp M \gg Q^{2-\delta}Q_2^{-2}$ and $Q_2 \asymp Q^{\delta_1}$, it follows that it suffices to consider the case $EG \asymp Q^2$. In particular, taking $G \asymp Q$, we are reduced to the case where $E \asymp G \asymp N \asymp Q$. 
	\par 
	Next in \eqref{Kl3sketch} we apply Poisson on $n_2$ (See Section \ref{Secondkuz} for details), reducing the hyper-Kloosterman sum to a standard Kloosterman sum. We remark that although this statement may appear to be simple, carrying it out in practice is quite involved and requires a substantial and careful analysis. Hence we obtain a sum roughly of the form
	\begin{equation}\label{sketch10}
		\frac{\sqrt{EGN_2}}{Q^2\sqrt{N_1}} \sum_{q_1 \sim Q_1}\sum_{q_2\sim Q_2} \; \sumfour_{\substack{n_1 \sim N_1, n_2\ll \frac{Q}{N_2} \\ e \ll \frac{Q^2}{E}, g \ll \frac{Q}{G}}} \lambda_f(e)\lambda_f(n_1)S(e,\overline{n_2}n_1g;q_1q_2),
	\end{equation}
	for which we have the feature to treat the standard Kloosterman sum via spectral analysis. We apply the bound from Kunetsov trace formula and spectral large sieve inequality in \eqref{Kuzinsketch}, getting that \eqref{sketch10} is bounded by (assuming $\theta=0$)
	\begin{align}
		\frac{Q^2Q_2}{N_2}+\frac{Q^{2.5}\sqrt{Q_2}}{\sqrt{EN_2}}+\frac{Q^2\sqrt{Q_2N_1}}{\sqrt{N_2G}}+\frac{Q^{2.5}\sqrt{N_1}}{\sqrt{EG}}.
	\end{align}
	This bound is enough as long as $N_2\gg Q^\varepsilon$. Hence we are reduced to the case $N_2\asymp 1$ and $E \asymp G \asymp N_1 \asymp Q$.
	\par 
	At this stage, we recall that if we bound the sum in \eqref{sketch3} trivially, we get an upper bound $O(Q^{5/2+\varepsilon})$. Hence we need to save a little more than $O(Q^{1/2})$. Our next observation is that one can eventually save $O(E)$ which is of size $\asymp Q$ from the sum over $q_1$. To achieve this, we first apply Poisson on $g$ in \eqref{sketch3}, to transform the original sum into a sum roughly of the form
	\begin{align*}
		\frac{\sqrt{G}}{\sqrt{EN_1N_2}} \sum_{q_1 \sim Q_1}\sum_{q_2\sim Q_2} \; \sumfour_{\substack{n_1 \sim N_1, n_2\sim N_2 \\ e \sim E, g \ll \frac{Q}{G}}} \lambda_f(e)\lambda_f(n_1)e\left(\frac{\overline{e}gn_1n_2}{q_1q_2}\right).
	\end{align*}
	Next, we apply additive reciprocity formula, getting 
	\begin{align*}
		\frac{\sqrt{G}}{\sqrt{EN_1N_2}} \sum_{q_1 \sim Q_1}\sum_{q_2\sim Q_2} \; \sumfour_{\substack{n_1 \sim N_1, n_2\sim N_2 \\ e \sim E, g \ll \frac{Q}{G}}} \lambda_f(e)\lambda_f(n_1)e\left(-\frac{\overline{q_1q_2}gn_1n_2}{e}\right)e\left(\frac{gn_1n_2}{eq_1q_2}\right),
	\end{align*}
	where since $\frac{gn_1n_2}{eq_1q_2}\ll \frac{1}{Q}$, the exponential $e(-\overline{q_1q_2}gn_1n_2/e)$ is flat. Next we apply Voronoi on $n_1$, followed by Poisson on $q_1$, getting roughly 
	\begin{align}\label{sketchfinal}
		\frac{\sqrt{G}}{\sqrt{EN_1N_2}}\cdot\frac{N_1}{E}\cdot Q_1\sum_{q_2\sim Q_2}\sumfour_{\substack{n_1 \ll \frac{E^2}{N_1}, n_2\sim N_2 \\ e \sim E, g \ll \frac{Q}{G}}} \lambda_f(e)\lambda_f(n_1)\sum_{\substack{q_1\ll \frac{E}{Q_1}\\ q_2n_1\equiv gn_2q_1\bmod e}}1.
	\end{align}
	Since now the length of $q_2n_1>E$, we can bound the sum trivially to save $O(E)$, getting that \eqref{sketchfinal} is bounded by 
	$$
	\frac{QE^2\sqrt{N_2}Q_2}{\sqrt{EGN_1}}\ll Q^{3/2},
	$$
	which completes the proof, and the detail can be seen in Section \ref{Additiverecip}. Notice that the Voronoi summation over $n_1$ is used so that we obtain a congruence condition modulo $e$ over the $q_1$ sum in \eqref{sketchfinal} which eventually saves us $O(E)$, otherwise we would have obtained a Kloosterman sum $S(q_1\overline{q_2}gn_1n_2;e)$ which only saves $O(E^{1/2})$, falling short of what we need to save.\par 
    
	In Section \ref{sectionlargesieve}, we also establish a new large sieve inequality whose application in the proof of Theorem \ref{MTFinal} was skipped in the sketch above. 
	\begin{thm}\label{Largesieve}
		Let $V_i$ be a smooth function such that for $\operatorname{Re}(s)>0$ and $x\gg 1$,
		\begin{align*}
			\widetilde{V}_i(s)\ll |s|^{-A}\ \text{and}\ V_i(x)\ll |x|^{-A}
		\end{align*}
		for any $A\geq 1$.
		If $\lambda\ll Q^B$ for some $B>0$, then we have
		$$
		\sum_{q\leq Q}\ \ \sumstar_{b\bmod q}\left|\sum_{m}\sum_{(n,\lambda q)=1}\lambda_f(m)e\left(\frac{m\overline{n}b}{q}\right)V_1\left(\frac{mM}{q^2}\right)V_2\left(\frac{n}{N}\right) \right|^2\ll \frac{Q^{4+\varepsilon}N}{M}+Q^{1+\varepsilon}N^2\min(Q,\frac{Q^2}{M})
		$$
		and
		$$
		\sum_{q\leq Q}\ \ \sumstar_{b\bmod q}\left|\sum_{m}\sum_{(n,\lambda q)=1}\lambda_f(m)e\left(\frac{m\overline{n}b}{q}\right)V_1\left(\frac{mM}{q^2}\right)V_2\left(\frac{nN}{q}\right) \right|^2\ll \frac{Q^{5+\varepsilon}}{MN}+\frac{Q^{3+\varepsilon}}{N^2}\min(Q,\frac{Q^2}{M})\ll \frac{Q^{5+\varepsilon}}{MN}.
		$$
	\end{thm}
	\section{A Shifted moment}
    Theorem \ref{MTFinal} will be derived from Theorem \ref{Main theorem}, and to state the latter, we need to introduce some notation.
    
	 Let $\chi(\bmod q)$ be a primitive even Dirichlet character, and let (for $\operatorname{Re} s>1$),
	$$L(s, \chi)=\sum_{n=1}^{\infty} \frac{\chi(n)}{n^s}=\prod_p\left(1-\frac{\chi(p)}{p^s}\right)^{-1}$$
	be the Dirichlet $L$-function associated to it. The completed $L$-function $\Lambda(s, \chi)$ is defined by
	$$
	\Lambda\left(\frac{1}{2}+s, \chi\right):=\left(\frac{q}{\pi}\right)^{s / 2} \Gamma\left(\frac{1}{4}+\frac{s}{2}\right) L\left(\frac{1}{2}+s, \chi\right)
	$$
	and it satisfies the functional equation
	$$
	\Lambda\left(\frac{1}{2}+s, \chi\right)=\epsilon_\chi \Lambda\left(\frac{1}{2}-s, \bar{\chi}\right),
	$$
	where $\left|\epsilon_\chi\right|=1$.\par 
	
	Let $f$ be a holomorphic modular form of weight $\kappa$ for the full modular group and suppose that $f$ is an eigenfunction of all the Hecke operators. We write the Fourier expansion of $f$ as
	\begin{equation}\label{Nomolize}
		f(z)=\sum_{n=1}^{\infty} \lambda_f(n) n^{(\kappa-1) / 2} e(n z),
	\end{equation}
	with $\lambda_f(1)=1$, and $f$ is normalized so that Deligne's bound asserts that $\left|\lambda_f(n)\right| \leq d(n)$ for all $n$, where $d(n)$ denotes the number of divisors of $n$. The $L$-function associated to $f$ is given by
	$$
	L(s, f)=\sum_{n=1}^{\infty} \frac{\lambda_f(n)}{n^s}=\prod_p\left(1-\frac{\lambda_f(p)}{p^s}+\frac{1}{p^{2 s}}\right)^{-1}=\prod_p\left(1-\frac{\alpha_p}{p^s}\right)^{-1}\left(1-\frac{\beta_p}{p^s}\right)^{-1}=\prod_p L_p(s,f),
	$$
	which converges absolutely for $\operatorname{Re}(s)>1$, extends analytically to the entire complex plane and satisfies $\alpha_p+\beta_p=\lambda_f(p) $\text { and }$ \alpha_p \beta_p=1$.
	
	The symmetric square $L$-function attached to $f$ is defined by 
	$$
	L\left(s,\operatorname{sym}^2 f\right):=\prod_p \prod_{m=0}^2\left(1-\alpha_p^{2-m} \beta_p^m p^{-s}\right)^{-1}=\prod_p L_p(s,\operatorname{sym}^2 f)
	$$
	for $\operatorname{Re}(s)>1$.\par
	Let $f\otimes \chi $ be the twist of $f$ by the character $\chi$ defined above, and $L(s,f\otimes \chi )$ denotes the twisted $L$-function
	$$
	L(s, f\otimes\chi )=\sum_{n=1}^{\infty} \frac{\lambda_f(n)\chi(n)}{n^s}.
	$$
	We set 
	$$
	\Lambda(\frac{1}{2}+s, f\otimes\chi ):=(\frac{q}{2\pi })^{s} \Gamma\left(s+\frac{\kappa}{2}\right) L(\frac{1}{2}+s, f\otimes\chi), 
	$$
	and then the twisted $L$-function satisfies the functional equation
	$$
	\Lambda(\frac{1}{2}+s, f\otimes\chi )= \epsilon_\chi ^2\Lambda(\frac{1}{2}-s, f\otimes\overline{\chi}) .
	$$
	
	We will mostly follow the notation in \cite{CLMR2}. Let $\al = (\alpha_1, \alpha_2)$ and $\be = (\beta_1, \beta_2)$.  For convenience, we also write $\alpha_{2 + j} = \beta_j$ for $j = 1, 2$. Moreover let $S_{4}$ be the permutation group on four elements. For $\pi \in S_4$, define
	$$\pi(\al, \be) = (\pi(\al), \pi(\be)) = (\alpha_{\pi(1)}, ..., \alpha_{\pi(4)}),$$ where we take $\pi(\al)$ as the first two coordinates of $\pi(\al, \be)$ and $\pi(\be)$ as the last two coordinates of $\pi(\al, \be)$. 
	
	Now let
	$$ \Lambda(s, \chi; \al, \be) :=  \Lambda(s+\alpha_1, \chi) \Lambda(s-\beta_1, \overline{\chi})\Lambda(s+\alpha_2, f\otimes\chi) \Lambda(s-\beta_2, f\otimes\overline{\chi})$$
	and 
	$$ \Lambda(\chi, \al, \be) := \Lambda\left( \frac 12, \chi; \al, \be \right).$$
	
	Further, we let
	\begin{equation}
		\label{eq:Gdef}
		G(s, \al, \be) := \left(\frac12\right)^{\alpha_2-\beta_2}\Gamma\left(\frac{s}{2} + \frac{\alpha_1}{2}\right)\Gamma\left(\frac{s}{2}-\frac{\beta_1}{2}\right)\Gamma\left(s+\frac{\kappa-1}{2}+\alpha_2\right)\Gamma\left(s+\frac{\kappa-1}{2}-\beta_2\right),
	\end{equation}
	so that
	\begin{align*}
		\Lambda\left(\tfrac{1}{2}, \chi ; \al, \be \right) = \bfrac{q}{\pi}^{\delta(\al, \be)} G\left(\frac 12, \al, \be\right)  L\left( \frac{1}{2} + \alpha_1, \chi\right)&L\left(\frac{1}{2} - \beta_1,  \cb \right)\\ 
		&\cdot L\left(\frac{1}{2} + \alpha_2,  f\otimes\chi  \right)L\left(\frac{1}{2} - \beta_2,  f\otimes\overline{\chi } \right),
	\end{align*}
	where
	\begin{align*}
		\delta(\al, \be) := \frac{\alpha_1-\beta_1}{2}+\alpha_2-\beta_2  .
	\end{align*}
	For $\operatorname{Re}(s)\gg 1$, we may write
	\begin{equation}\label{eqn:Dirichletseries}
		L\left( s + \alpha_1, \chi\right)L\left(s - \beta_1,  \cb \right)L\left(s + \alpha_2,  f\otimes\chi  \right) L\left(s - \beta_2,  f\otimes\cb \right) = \sumtwo_{m, n \geq 1} \frac{\sigma (m; \al) \sigma (n; -\be)}{m^s n^s} \chi(m) \cb(n),
	\end{equation}
	where 
	\begin{align*}
		\sigma(m; \al)  := \sum_{m=ab} a^{-\alpha_1}b^{-\alpha_2}\lambda_f(b)
	\end{align*}and similarly for $\sigma(n; -\be)$.
	  Let
	\begin{equation}
		\label{eqn:sumoverm} \sum_{\substack{m=1\\ (m, q) = 1}}^\infty \frac{\sigma (m; \al) \sigma (m; -\be)}{m^{2s}} = \prod_{p\nmid q} \B_p(s; \al, \be),
	\end{equation}
	where 
	$$ \B_p(s; \al, \be) := \sum_{r = 0}^{\infty} \frac{\sigma(p^r; \al)\sigma(p^r; -\be)}{p^{2rs}}.$$
	
	Further, for $\zeta_p(x) = (1-p^{-x})^{-1}$, we let
	\begin{align*}
		\Z_p(s; \al, \be) =  \zeta_p(2s + \alpha_1 - \beta_1)&L_p(2s + \alpha_1 - \beta_2,f)\\ &\cdot L_p(2s + \alpha_2 - \beta_1,f)\zeta_p(2s + \alpha_2 - \beta_2)L_p(2s + \alpha_2 - \beta_2,\operatorname{sym}^2 f)
	\end{align*}
	and
	$$	Z(s; \al, \be) = \prod_{p} \Z_p(2s + \alpha_i - \beta_j).$$
	
	The sum $\B_p$ plays a role similar to that of $\Z_p$. More precisely, the Euler product defined by
	\begin{equation}\label{def:A}
		\A(s; \al, \be) := \prod_p \B_p(s; \al, \be) \Z_p(s; \al, \be)^{-1},
	\end{equation} will be absolutely convergent in a wider region. In particular, $\A(s; 0, 0)$ converges for $\operatorname{Re}(s)> 1/4$.
	
	Now, letting
	\begin{equation}
		\label{eq:Bqdef}
		\B_q(s; \al, \be) := \prod_{p|q} \B_p(s; \al, \be),
	\end{equation}
	we define
	\begin{equation} \label{def:Qalbe}
		\Q(q; \al, \be) = \bfrac{q}{\pi}^{\delta(\al, \be)} G\left( \frac 12; \al, \be\right) \frac{\A\left( \frac 12; \al, \be\right)\Z\left( \frac 12; \al, \be\right)}{\B_q\left( \frac 12; \al, \be\right)} ,
	\end{equation} which corresponds to the diagonal contribution $m=n$. Let $\pi\in S_4$ be the permutation that swaps the first and third components $\pi(\al, \be) = (\pi(\al), \pi(\be)) :=(\alpha_3,\alpha_2,\alpha_1,\alpha_4) $.
	\begin{equation*} 
		\tQ(q; \al, \be) =  \Q(q; \al, \be)+ \Q(q; \pi(\al), \pi(\be))+ \Q(q; \pi(\be), \pi(\al))+ \Q(q; \be, \al).
	\end{equation*}
	
	We now take $q=q_1q_2$ and we can state our result.

	\begin{thm}\label{Main theorem}
		Assume Selberg's eigenvalue conjecture. Let $0<\delta_1< 0.0004$, $Q \geq 3$, and $Q\ll Q_1Q_2 \ll Q$ with $Q_2\asymp Q^{\delta_1}$. Further let $\boldsymbol{\alpha}, \boldsymbol{\beta}$ be 2-tuples satisfying $\alpha_i, \beta_i \ll \frac{1}{\log Q}$ and $\alpha_i \neq \beta_j$ for all $1 \leq i, j \leq 2$. Then for any smooth function $\Psi_1$ and $\Psi_2$ supported on $[1,2]$ satisfying $\Psi_1^{(j)}(x)\ll Q^{j\varepsilon}$ and $\Psi_2^{(j)}(x)\ll Q^{j\varepsilon}$, we have
		\begin{align*}
			\sumtwo_{\substack{q_1,q_2\\ (q_1q_2,6)=1\\(q_1,q_2)=1}} \Psi_1\left(\frac{q_1}{Q_1}\right) \Psi_2\left(\frac{q_2}{Q_2}\right)&\ \sumb_{\chi(\bmod q_1q_2)} \Lambda(\chi ; \boldsymbol{\alpha}, \boldsymbol{\beta})\\ 
			&=	\sumtwo_{\substack{q_1,q_2\\ (q_1q_2,6)=1\\(q_1,q_2)=1}} \Psi_1\left(\frac{q_1}{Q_1}\right) \Psi_2\left(\frac{q_2}{Q_2}\right) \phi^\flat(q_1q_2) \widetilde{\mathcal{Q}}(q_1q_2 ; \boldsymbol{\alpha}, \boldsymbol{\beta})+O\left(Q^{2-\frac15\delta_1 }\right) .
		\end{align*}
	\end{thm}
	\section*{Notation and assumptions}
	Throughout the paper we shall assume the set-up of Theorem~\ref{Main theorem}. In particular $Q\ge 3$, $\al, \be$ are $2$-tuples satisfying $\alpha_i, \beta_i \ll \frac{1}{\log Q}$ with $\alpha_i \neq \beta_j$ for all $1 \leq i,j \leq 2$ and $\Psi_1$ and $\Psi_2$ are smooth functions supported on $[1, 2]$ with $\Psi_1^{(j)}(x)\ll Q^{j\varepsilon}$ and $\Psi_2^{(j)}(x)\ll Q^{j\varepsilon}$. We will also denote by 
	$$
	\sumb_{\chi \bmod q}
	$$
	a sum over primitive even characters modulo $q$, and by 
	$$
	\sumd_{M, N}
	$$
	a sum over $M$ and $N$ running over positive powers of two. Finally given a smooth function $v$, we will denote by 
	\begin{equation} \label{def:MellinV}
		\widetilde{v}(s) := \int_{0}^{\infty} v(x) x^{s - 1} dx
	\end{equation}
	the Mellin transform of $v$. We denote by 
	\begin{equation}\label{Fourier}
	\widehat{V}(x) := \int_{-\infty}^{\infty} V(\xi) \ex(- x \xi) d \xi
\end{equation}
	the Fourier transform of $V$, where $\ex(x) = e^{2\pi i x}.$
	
	Throughout the paper, $\varepsilon$ denotes a small positive real number and need not to be equal from line to line. Moreover, $\delta_0$ and $\delta'$ are fixed positive constants to be chosen later. 
	
	\section{Preliminary setup}
	\subsection{Some lemmas}
	Here we state some standard results.  Let 
	\begin{align*}
		H(s; \al, \be) := \left(s^2 - \bfrac{\alpha_1-\beta_1}{2}^2\right)\left(s^2 - \bfrac{\alpha_2-\beta_2}{2}^2\right),
	\end{align*}
	and for $\xi, \eta, \mu >0$,
	\begin{equation} \label{eqn:Walbe}
		W_{\al, \be}(\xi, \eta; \mu) := \bfrac{\mu}{\pi}^{\delta(\al, \be)}\frac{1}{2\pi i} \int_{(1)} G \left(\frac 12 + s; \al, \be\right) H(s; \al, \be) \left( \frac{4\pi^3\xi\eta }{\mu^3} \right)^{-s} \frac{ds}{s}.
	\end{equation}
	We also let
	\begin{align*}
		\Lambda_0(\chi; \al, \be) = \sumtwo_{m, n \geq 1} \frac{\sigma(m; \al) \sigma(n; -\be) \chi(m) \cb(n)}{\sqrt{mn}} W_{\al, \be}\left(m, n; q\right).
	\end{align*}
	
	The following lemma gives the approximate functional equation for $\Lambda(\chi; \al, \be)$.
	\begin{lem}\label{lem:approxfunc}
		With notation as above,
		\begin{align*}
			H(0; \al, \be) \Lambda(\chi; \al, \be) = \Lambda_0(\chi; \al, \be) + \Lambda_0(\chi; \be, \al).
		\end{align*}
	\end{lem}
	\begin{proof}
		See~\cite[Lemma 3.1]{CLMR2} for example. The proof follows by the same strategy.
	\end{proof}
	
	\begin{lemma} \label{lem:weightW} Let $W_{\al, \be}$ be defined in \eqref{eqn:Walbe}. For any non-negative integers $\ell_1, \ell_2, \ell_3$ and $\xi, \eta, \mu$,  
		\begin{align*} \frac{d^{\ell_1}}{d\xi}\frac{d^{\ell_2}}{d\eta} \frac{d^{\ell_3}}{d\mu}W_{\al, \be}(\xi, \eta; \mu) \ll_{\ell_1, \ell_2, \ell_3} \frac{1}{\xi^{\ell_1} \eta^{\ell_2} \mu^{\ell_3}}\left( \frac{\mu}{\pi}\right)^{\R \, \delta(\al, \be)}\exp \left(-c_0 \left( \frac{\xi \eta}{\mu^3}\right)^{1/3} \right)
		\end{align*}
		for some constant $c_0>0$. 
	\end{lemma}
	\begin{proof} The proof also follows the same strategy of \cite[Lemma 3.2]{CLMR2}.
	\end{proof}
	
	We also need the following standard orthogonality relation for primitive even characters (see e.g.~\cite[Lemma 2]{CIS}).
	\begin{lemma} \label{lem:orthogonal} Let $q \in \mathbb{N}$. If $m, n$ are integers with $(mn, q) = 1$ then
		$$ \sumb_{\chiq} \chi(m) \cb(n) = \frac{1}{2} \sum_{\substack{q = dr \\ r | (m \pm n)}} \mu(d) \phi(r).$$
	\end{lemma}
	We will also need the Voronoi summation formula for the holomorphic cusp form. We will collect the necessary information from \cite[Appendix A]{KMV}.\par 
	For $h(x) \in \mathcal{C}_{\mathcal{C}}(0, \infty)$, we set
	\begin{equation}\label{Besseltrans}
		\Phi_h(x)=2 \pi i^\kappa \int_0^{\infty} h(y) J_{\kappa-1}(4 \pi \sqrt{x y}) \mathrm{d} y,
	\end{equation}
	where $J_{\kappa-1}$ is the usual $J$-Bessel function of order $\kappa-1$. We have the following Voronoi summation formula (see \cite[Theorem A.4]{KMV}).
	\begin{lemma}\label{lem:voronoi}
		Let $q \in \mathbb{N}$ and $a \in \mathbb{Z}$ with $(a,q)=1$. For $X>0$, we have
		\[
		\sum_{n=1}^{\infty} \lambda_f(n)\, e\!\left(\frac{a n}{q}\right)
		h\!\left(\frac{n}{X}\right)
		=
		\frac{X}{q}
		\sum_{n=1}^{\infty} \lambda_f(n)\,
		e\!\left(-\frac{\bar a n}{q}\right)
		\Phi_h\!\left(\frac{nX}{q^{2}}\right),
		\]
		where $\bar a$ denotes the multiplicative inverse of $a$ modulo~$q$.
	\end{lemma}
	The next lemma is Wilton's bound \cite{Wilton33}.
	\begin{lemma}\label{Wilton}
		For any real number $\alpha$ and $\varepsilon>0$, we have
		$$
		\begin{aligned}
			\sum_{n\leq x}\lambda_{f}(n)e(\alpha n)\ll x^{1/2+\varepsilon}.
		\end{aligned}
		$$
	\end{lemma}
	Write, for sequences $\mathbf{a} = (a_m)_{m \geq 1}$ and $\mathbf{b} = (b_{n, r, s})_{n,r,s \geq 1}$,
	\[
	\Vert \mathbf{a} \Vert_2  = \sqrt{\sum_m |a_m|^2} \quad \text{and} \quad \Vert \mathbf{b} \Vert_2 = \sqrt{\sum_{n, r, s} |b_{n, r, s}|^2}.
	\]
	To deal with the averages of Kloosterman sums we shall use the following refinement of~\cite[Theorem 10]{DI}, which can be seen in \cite[Lemma 9.1]{CLMR2}.
	
	\begin{lemma}
		\label{le:Klo1}
		Let $C, M, N, R, S \geq 1/2$ and let $g \colon \mathbb{R}^5 \to \mathbb{R}$ be a smooth function with compact support on $[C, 2C] \times (0, \infty)^4$ such that, for any $\varepsilon > 0$
		\begin{align*}
			\left|\frac{\partial^{\nu_1 + \nu_2 + \nu_3 + \nu_4 + \nu_5}}{\partial c^{\nu_1} \partial m^{\nu_2} \partial n^{\nu_3} \partial r^{\nu_4} \partial s^{\nu_5}} g(c, m, n, r, s)\right| \ll_{\nu_j} (CMNRS)^\varepsilon c^{-\nu_1} m^{-\nu_2} n^{-\nu_3} r^{-\nu_4} s^{-\nu_5},
		\end{align*}
		for every $\nu_j \geq 0$,  $1 \leq j \leq 5$. Assume that
		\begin{equation}
			\label{def:X}
			X := \frac{CS\sqrt{R}}{4\pi \sqrt{MN}} \gg 1.
		\end{equation}
		Let $\mathbf{a} = (a_m)_{m \geq 1}$ and $\mathbf{b} = (b_{n,r,s})_{n,r,s \geq 1}$ denote two sequences. 
		Let
		\[
		\mathcal{L}^{\pm}(C, M, N, R, S) = \sumfour_{\substack{r \sim R, s \sim S, m \sim M, n \sim N \\ (r, s) = 1}}  a_{m} b_{n, r, s} \sum_{\substack{c \\ (c, r) = 1}} g(c, m, n, r, s) S(\pm n, m\overline{r}, sc).
		\]
		Then, for any $\varepsilon > 0$,
		\[
		\mathcal{L}^{\pm}(C, M, N, R, S) \ll (CMNRS)^\varepsilon L(C, M, N, R, S) \Vert \mathbf{a} \Vert_2 \Vert \mathbf{b} \Vert_2,
		\]
		where
		\[
		L(C, M, N, R, S) = CS\sqrt{R}\cdot \sqrt{RS} \left(1+\sqrt{\frac{M}{RS}}\right) \left(1+\sqrt{\frac{N}{RS}}\right)\left(1+\frac{X^2}{\left(1+\frac{RS}{M}\right)^2 \left(1+\frac{RS}{N}\right)}\right)^{\theta},
		\]
    and $\theta=\sup_{q\geq1}\sqrt{\max(0,\frac14-\lambda_1(q))}$, where $\lambda_1(q)=\frac14+(i\kappa_q)^2$ is the Laplacian eigenvalue for the Maass cusp form of level $q$. 
	\end{lemma}
    Assuming Selberg's eigenvalue conjecture, one can take $\theta=0$.
	
	\subsection{Dissection}
	We now turn to the first steps in the proof of Theorem~\ref{Main theorem}. We may follow the strategy of \cite{CLMR2}. We start by applying Lemma~\ref{lem:approxfunc} to
	\[
	\sumtwo_{\substack{q_1,q_2\\(q_1q_2,6)=1\\ (q_1,q_2)=1}} \Psi_1\bfrac{q_1}{Q_1}  \Psi_2\bfrac{q_2}{Q_2}\sumb_{\chi (\bmod q_1q_2)} \Lambda(\chi; \al, \be).
	\]
	Due to the symmetry of $\al, \be$, it is sufficient to consider the contribution from $\Lambda_0(\chi; \al, \be)$. 
	Hence we would like to evaluate
	\begin{align*}
		\sumtwo_{\substack{q_1,q_2\\(q_1q_2,6)=1\\ (q_1,q_2)=1}} \Psi_1\bfrac{q_1}{Q_1}  \Psi_2\bfrac{q_2}{Q_2}\sumb_{\chi (\bmod q_1q_2)}  \Lambda_0(\chi; \al, \be)  .
	\end{align*}
	We now extract diagonal terms and introduce smooth partitions of unity.  Let
	\begin{align}
		\label{def:D}
		\D &:= \sumtwo_{\substack{q_1,q_2\\(q_1q_2,6)=1\\ (q_1,q_2)=1}} \Psi_1\bfrac{q_1}{Q_1}  \Psi_2\bfrac{q_2}{Q_2}  \phi^\flat(q_1q_2) \\
		&\hskip1in\cdot\sum_{(m, q_1q_2) = 1} \frac{\sigma(m; \al) \sigma(m; -\be) }{m} W_{\al, \be}\left(m, m ; q_1q_2\right),\nonumber
	\end{align}
	and we write
	\begin{align}
		\label{eq:extractDiagonal}
		\sumtwo_{\substack{q_1,q_2\\(q_1q_2,6)=1\\ (q_1,q_2)=1}} &\Psi_1\bfrac{q_1}{Q_1}  \Psi_2\bfrac{q_2}{Q_2} \sumb_{\chi \bmod q_1q_2}  \Lambda_0(\chi; \al, \be) \\
		&= \D + \sumtwo_{\substack{q_1,q_2\\(q_1q_2,6)=1\\ (q_1,q_2)=1}} \Psi_1\bfrac{q_1}{Q_1}  \Psi_2\bfrac{q_2}{Q_2} \sumb_{\chi \bmod q_1q_2}\sumd_{M, N} S(M, N), \nonumber
	\end{align}
	where $\sumd_{M, N}$ denotes a sum over powers of $2$ and where
	\begin{equation} \label{def:SMN1}
		S(M, N) := \sumtwo_{\substack{m, n\\ m\neq n}} \frac{\sigma(m; \al) \sigma(n; -\be) \chi(m) \cb(n)}{\sqrt{mn}} W_{\al, \be}\left(m, n ; q_1q_2\right) V\bfrac{m}{M} V\bfrac{n}{N},
	\end{equation}
	with $V$ a smooth function supported on $[1/2, 5/2]$ satisfying
	$$\sumd_M V\bfrac{m}{M} = 1
	$$for all $m\ge 1$. Note that we can always remove and add back terms with $mn\gg Q^{3+\varepsilon}$ with a negligible error by using the rapid decay of $W_{\al, \be}\left(m, n; q\right)$ (see Lemma \ref{lem:weightW}).
	
	Let $\widetilde V(s)$ be the Mellin transform of $V(s)$, defined as in \eqref{def:MellinV}.
	Since $V(x)$ is smooth and compactly supported away from zero, the Mellin transform $\widetilde{V}$ is entire and decays rapidly along the vertical axis. 
	
	We now split our analysis into two main cases. 
	\subsubsection{Balanced sums}  The first case is the balanced sums where $M$ and $N$ are not too far apart, more precisely the case $\max(M, N) \leq Q^{2 - \delta_0},$ where $\delta_0$ is a fixed real positive number to be chosen later. Let
	\begin{equation}
		\label{eq:BSdef} \mathcal{BS}(\Psi, Q_1,Q_2; \al, \be)  := \sumtwo_{\substack{q_1,q_2\\(q_1q_2,6)=1\\ (q_1,q_2)=1}} \Psi_1\bfrac{q_1}{Q_1}  \Psi_2\bfrac{q_2}{Q_2} \sumb_{\chi \bmod q_1q_2}\sumd_{\substack{M, N \\ \max(M, N) \le Q^{2-\delta_0}}} S(M, N),
	\end{equation}
	with $S(M, N)$ being defined as in \eqref{def:SMN1}.  We will prove the following proposition.
	
	\begin{prop}\label{prop:offdiagonal} Let $\varepsilon, \delta_0 > 0$. Then
		\begin{align*}
			\mathcal{BS}(\Psi_1,\Psi_2, Q_1,Q_2; \al, \be) 
			&= H(0; \al, \be) \sumtwo_{\substack{q_1,q_2\\(q_1q_2,6)=1\\ (q_1,q_2)=1}} \Psi_1\bfrac{q_1}{Q_1}  \Psi_2\bfrac{q_2}{Q_2}\phi^{\flat}(q_1q_2)\mathcal Q(q_1q_2; \pi(\al), \pi(\be)) \\
			& \hskip 0.5in + 	O\left( Q^{2 - \frac{1}2\delta_0 +\delta_1+ \varepsilon} + Q^{\frac32+\frac12\delta_0+\delta_1+\varepsilon}+Q^{\frac{7}{4} + \frac{3}{4}\delta_0 +\frac32\delta_1+ \varepsilon}+Q^{2-\delta_1+\varepsilon} \right).
		\end{align*}
	\end{prop}
	
	The diagonal terms $\D$ also contributes to the main term, and we will prove the following proposition.
	
	\begin{prop}\label{prop:diagonal} Let $\varepsilon > 0$ and let $\D$ be as in \eqref{def:D}. Then
		$$\D = H(0; \al, \be) \sumtwo_{\substack{q_1,q_2\\(q_1q_2,6)=1\\ (q_1,q_2)=1}} \Psi_1\bfrac{q_1}{Q_1}  \Psi_2\bfrac{q_2}{Q_2}\phi^{\flat}(q_1q_2) \Q(q_1q_2; \al, \be) + O(Q^{5/4 + \varepsilon}).
		$$
	\end{prop}
	\begin{proof}
		The proof is completely the same with \cite[Section 5]{CLMR2}, with slight modifications. We shall omit it.
	\end{proof}
	
	\subsubsection{Unbalanced sums}
	In the second case one of $M$ and $N$ is much larger than the other. Without loss of generality, we can concentrate on the case $M > N$. We define
	\begin{equation}
		\label{eq:USdef}
		\mathcal{US}(\Psi_1,\Psi_2, Q_1,Q_2; \al, \be)  := \sumtwo_{\substack{q_1,q_2\\(q_1q_2,6)=1\\ (q_1,q_2)=1}} \Psi_1\bfrac{q_1}{Q_1}  \Psi_2\bfrac{q_2}{Q_2}\sumb_{\chi\bmod q_1q_2} \sumd_{\substack{M, N \\ M \ge Q^{2-\delta_0} \\ M > N}} S(M, N),
	\end{equation}
	with $S(M, N)$ being as in \eqref{def:SMN1}, and we will show the following.
	
	\begin{prop}\label{prop:unbalanced} Let $\varepsilon_1>0$, $ 0<\varepsilon\leq  0.001$ and $0<\delta_0\leq  0.001$, we have that
		$$\mathcal{US}(\Psi_1,\Psi_2, Q_1,Q_2; \al, \be)  \ll  Q^{2-\varepsilon}+Q^{1.5+0.01+\frac32\delta_1}+Q^{2+\varepsilon_1-\frac14\delta_1}+Q^{1.834+0.045+8\delta_1}.
		$$
	\end{prop}
	The proof of Proposition \ref{prop:offdiagonal} and Proposition \ref{prop:unbalanced} will be given in the next section.
	
	\section{Proof of Proposition \ref{prop:offdiagonal}, Proposition \ref{prop:unbalanced} and Theorem \ref{Main theorem}}
	We can now quickly deduce Theorem~\ref{Main theorem} assuming Propositions~\ref{prop:diagonal},~\ref{prop:offdiagonal}, and~\ref{prop:unbalanced}. Recall that we would like to evaluate 
	$$\sumtwo_{\substack{q_1,q_2\\(q_1q_2,6)=1\\ (q_1,q_2)=1}} \Psi_1\bfrac{q_1}{Q_1}  \Psi_2\bfrac{q_2}{Q_2}\sumb_{\chi\bmod q_1q_2} \Lambda(\chi; \al, \be). $$
	From Lemma \ref{lem:approxfunc}, we have
	\begin{equation*}
		H(0; \al, \be) \Lambda(\chi; \al, \be) = \Lambda_0(\chi; \al, \be) + \Lambda_0(\chi; \be, \al),
	\end{equation*}
	and by~\eqref{eq:extractDiagonal},~\eqref{eq:BSdef}, and~\eqref{eq:USdef} we have
	\begin{align*}
		&H(0; \al, \be) \sumtwo_{\substack{q_1,q_2\\(q_1q_2,6)=1\\ (q_1,q_2)=1}} \Psi_1\bfrac{q_1}{Q_1}  \Psi_2\bfrac{q_2}{Q_2}\sumb_{\chi\bmod q_1q_2} \Lambda_0(\chi; \al, \be) \\
		&= \mathcal D(\Psi_1,\Psi_2, Q_1,Q_2 ; \al, \be) + \mathcal {BS}(\Psi_1,\Psi_2, Q_1,Q_2; \al, \be) + 2\mathcal {US}(\Psi_1,\Psi_2, Q_1Q_2 ; \al, \be) +  O(1/Q).
	\end{align*}
	We shall see from Propositions~\ref{prop:diagonal}--\ref{prop:unbalanced} that (assuming $0<\delta_0<0.001$ and $\delta_0=\frac52\delta_1$)
	\begin{align}
		\label{eq:H0pf2.1}
		&H(0; \al, \be)\sumtwo_{\substack{q_1,q_2\\(q_1q_2,6)=1\\ (q_1,q_2)=1}} \Psi_1\bfrac{q_1}{Q_1}  \Psi_2\bfrac{q_2}{Q_2}\sumb_{\chi\bmod q_1q_2} \Lambda(\chi; \al, \be) \\
		&= H(0; \al, \be)\sumtwo_{\substack{q_1,q_2\\(q_1q_2,6)=1\\ (q_1,q_2)=1}} \Psi_1\bfrac{q_1}{Q_1}  \Psi_2\bfrac{q_2}{Q_2}\sumb_{\chi\bmod q_1q_2}\phi^{\flat}(q_1q_2) \tQ(q; \al, \be) \\
		&+ O\left( Q^{2-\frac14\delta_1+\varepsilon}  + Q^{2-0.001}  \right).\nonumber
	\end{align}
	Finally we choose $\varepsilon=\frac{1}{20}\delta_1$.\\
	Similarly to~\cite[End of Section 11]{CIS}, we can remove the factor $H(0; \al, \be)$ and conclude the proof of Theorem~\ref{Main theorem}.
	\subsection{Proof of Proposition \ref{prop:offdiagonal}}
	
	From \eqref{def:MBEB}, \eqref{eqn:BGrelatedtoMBGandEBG}, and Lemma \ref{lem:calcEBG}, we derive that
	\begin{align*} \mathcal {BD}(M, N) + \mathcal {BG}(M, N)  
		&=  \mathcal{MBG}_1(M, N) + O \left( \frac{Q^{2 + \varepsilon}}{D_0}   + Q^{2 - \delta_0 + \varepsilon} Q_2^2D_0 + Q^{\frac32+\varepsilon}Q_2^2D_0+Q^{\frac74+\varepsilon}Q_2^{\frac32}D_0^{\frac32}\right).
	\end{align*}
	Then from \eqref{eqn:BSinitial}, Lemma \ref{lem:sumMBG1} and Proposition \ref{prop:maincont9terms}, we obtain that
	\begin{align*} &	\mathcal{BS}(\Psi_1,\Psi_2, Q_1,Q_2; \al, \be) =H(0; \al, \be) \sumtwo_{\substack{q_1,q_2\\(q_1q_2,6)=1\\ (q_1,q_2)=1}} \Psi_1\bfrac{q_1}{Q_1}  \Psi_2\bfrac{q_2}{Q_2}\phi^{\flat}(q_1q_2)\mathcal Q(q_1q_2; \pi(\al), \pi(\be)) \\
		& + O\left(  \frac{Q^{2 + \varepsilon}}{D_0}   + Q^{2 - \delta_0 + \varepsilon} Q_2^2D_0 + Q^{\frac32+\varepsilon}Q_2^2D_0+Q^{\frac74+\varepsilon}Q_2^{\frac32}D_0^{\frac32}+Q^{\frac 74 +\frac 34\delta_0+\varepsilon}+\frac{Q^{2+\varepsilon}}{Q_1}+\frac{Q^{2+\varepsilon}}{Q_2}\right).
	\end{align*}
	
	We recall that $Q_2\asymp Q^{\delta_1}$. To balance the error terms $\frac{Q^{2 + \varepsilon}}{D_0}$ and $Q^{2 - \delta_0 + \varepsilon}Q_2^2 D_0$ we choose $D_0 = Q^{\frac{1}{2}\delta_0-\delta_1}$. Then the error terms is 
	\begin{align*}
		O\left( Q^{2 - \frac{1}2\delta_0 +\delta_1+ \varepsilon} + Q^{\frac32+\frac12\delta_0+\delta_1+\varepsilon}+Q^{\frac{7}{4} + \frac{3}{4}\delta_0 +\frac32\delta_1+ \varepsilon}+Q^{2-\delta_1+\varepsilon} \right),
	\end{align*}
	so the claim follows.
	
	\subsection{Proof of Proposition \ref{prop:unbalanced}}

	From \eqref{gPoisson}, we see that if $G>Q^{1+\varepsilon}$, then $G>Q^{\frac12\varepsilon}d_1d_2r_1r_2\geq \nu_1r_1r_2Q^{\frac12\varepsilon}$, and the bound is $\ll Q^{-A}$. Hence, we can assume $G\leq Q^{1+\varepsilon}$. Since $EG>Q^{2-\delta_0}$ and $EGN\leq Q^{3+\varepsilon}$, we can bound $N\leq Q^{1+\delta_0+\varepsilon}$.
	
	From \eqref{S1kuzDlargest}, \eqref{S1kuzsmall} and \eqref{S1kuzmiddle}, we obtain that
	\begin{align*}
		\mathcal{US}(\Psi_1,\Psi_2, Q_1,Q_2; \al, \be)&\ll \frac{Q^{\frac32+\frac{5}{2}\delta+\varepsilon}}{E}+Q^{1+\frac12\delta_0+\varepsilon}Q_2D^2G+Q^{1+\frac12\delta_0+\varepsilon}\sqrt{GN}+\frac{Q^{2+\varepsilon}}{D}+Q^{1+\varepsilon}\sqrt{N}\\
		&\hskip2.5in+\frac{Q^{\frac32+\varepsilon}\sqrt{G}}{\sqrt{D}}+Q^{1-\frac12\delta+\varepsilon}\sqrt{GN},\nonumber
	\end{align*}
	When $G\leq Q^{1-3\delta_0-\frac32\delta_1-6\varepsilon}$, we choose $D=Q^{2\varepsilon}$ and $\delta=0.1$. And we use $N\leq Q^{1+\delta_0+\varepsilon}$, getting
	\begin{align}
		\mathcal{US}(\Psi, Q_1,Q_2; \al, \be)&\ll Q^{\frac74+\varepsilon}+Q^{2-\frac52\delta_0-\frac12\delta_1-\varepsilon}+Q^{2-\frac12\delta_0-\frac34\delta_1-\frac32\varepsilon}+Q^{2-\varepsilon}+Q^{2-\frac{1}{40}}.
	\end{align}
	Hence we can assume $Q^{1-3\delta_0-\frac32\delta_1-6\varepsilon}<G\leq Q^{1+\varepsilon}$.\par 
	From \eqref{n2=0}, \eqref{S2kuzDsmall} and \eqref{S2kuzDlarge}, we obtain that 
	\begin{align*}
		\mathcal{US}(\Psi_1,\Psi_2, Q_1,Q_2; \al, \be)&\ll \frac{Q^{\frac52+\varepsilon}}{G}+\frac{Q^{2+\varepsilon}Q_2^{\frac32}D}{N_2}+\frac{Q^{\frac52+\varepsilon}\sqrt{Q_2D}}{\sqrt{EN_2}}+\frac{Q^{2+\varepsilon}\sqrt{Q_2N_1}}{\sqrt{N_2G}}+\frac{Q^{\frac52+\varepsilon}\sqrt{N_1}}{\sqrt{EG}}\\
		&\hskip3in+\frac{Q^{2+\varepsilon}}{D}+\frac{Q^{\frac32+\varepsilon}\sqrt{N_1}}{\sqrt{G}}.\nonumber
	\end{align*}
	When $N_2\geq Q^{2\delta_0+\frac32\delta_1+8\varepsilon}$, we choose $D=Q^{2\varepsilon}$. Then we use $N_1N_2\asymp N\leq Q^{1+\delta_0+\varepsilon}$, $EG>Q^{2-\delta_0}$ and the bound $Q^{1-3\delta_0-\frac32\delta_1-6\varepsilon}<G\leq Q^{1+\varepsilon}$ for $G$, getting
	\begin{align}
		\mathcal{US}(\Psi_1,\Psi_2, Q_1,Q_2; \al, \be)&\ll Q^{\frac32+3\delta_0+\frac32\delta_1+7\varepsilon}+Q^{2-2\delta_0-5\varepsilon}+Q^{2-\frac12\delta_0-\frac14\delta_1-\frac32\varepsilon}+Q^{2-\frac14\delta_1-\frac72\varepsilon}\\
		&\hskip3in+Q^{2-\frac34\delta_1-\frac52\varepsilon}+Q^{2-\varepsilon}.\nonumber
	\end{align}
	Hence the bound is enough, and we shall assume $N_2<Q^{2\delta_0+\frac32\delta_1+8\varepsilon}$.\par 
	From \eqref{kl3g=0} and \eqref{BoundfromSquarerootcancel}, we obtain that 
	\begin{align*}
		\mathcal{US}(\Psi_1,\Psi_2, Q_1,Q_2; \al, \be)&\ll\frac{Q^{\frac52+\varepsilon}}{E}+\frac{Q^{2+\varepsilon}\sqrt{NQQ_2}}{\sqrt{EG}}+Q^{2+\varepsilon_1}Q_2^{-\frac14}.
	\end{align*}
	Since $Q^{2-\delta_0}<EG\leq EQ^{1+\varepsilon}$, we have $E>Q^{1-\delta_0-\varepsilon}$. When $EG\geq Q^{2+\delta_0+\delta_1+5\varepsilon}$, we use $N\leq Q^{1+\delta_0+\varepsilon}$, getting
	\begin{align}
		\mathcal{US}(\Psi_1,\Psi_2, Q_1,Q_2; \al, \be)&\ll Q^{\frac32+\delta_0+2\varepsilon}+Q^{2-\varepsilon}+Q^{2+\varepsilon_1-\frac14\delta_1}.
	\end{align}
	Hence now we can assume $Q^{2-\delta_0}<EG< Q^{2+\delta_0+\delta_1+5\varepsilon}$.
	
	From  \eqref{ErrorinAdditiverecip} and \eqref{lowboundn}, we obtain that 
	$$
	\mathcal{US}(\Psi_1,\Psi_2, Q_1,Q_2; \al, \be)\ll\frac{	Q^{\frac52+\varepsilon}}{G}+\frac{Q^{1+\varepsilon}\sqrt{EGN}\sqrt{N}}{G\sqrt{N_2}}.
	$$
	When $N\leq Q^{1-\frac72\delta_0-2\delta_1-\frac{21}{2}\varepsilon}$, we use $EG< Q^{2+\delta_0+\delta_1+5\varepsilon}$ and the bound $Q^{1-3\delta_0-\frac32\delta_1-6\varepsilon}<G\leq Q^{1+\varepsilon}$ for $G$, getting
	\begin{align}
		\mathcal{US}(\Psi_1,\Psi_2, Q_1,Q_2; \al, \be)&\ll Q^{\frac32+3\delta_0+\frac32\delta_1+7\varepsilon}+Q^{2-\varepsilon}.
	\end{align}
	Hence we assume $Q^{1-\frac72\delta_0-2\delta_1-\frac{21}{2}\varepsilon}<N<Q^{1+\delta_0+\varepsilon}.$
	
	Finally, from \eqref{ErrorinAdditiverecip},  \eqref{largeAdditive} and \eqref{smallAdditive}, we obtain that 
	\begin{align}
		\mathcal{US}(\Psi_1,\Psi_2, Q_1,Q_2; \al, \be)&\ll\frac{	Q^{\frac52+\varepsilon}}{G}+\frac{Q^{1-\frac12\delta+\varepsilon}\sqrt{EGN}\sqrt{N}}{G\sqrt{N_2}}+\frac{Q^{1+\delta+\varepsilon}E^2N_2Q_2}{\sqrt{EGN}}.
	\end{align}
	Now we use $EGN\leq Q^{3+\varepsilon}$ and all the above bound of $G,N,N_2$ and $EG$, getting  
	\begin{align}
		\mathcal{US}(\Psi_1,\Psi_2, Q_1,Q_2; \al, \be)&\ll Q^{\frac32+3\delta_0+\frac32\delta_1+7\varepsilon}+Q^{2+\frac72\delta_0+\frac32\delta_1+8\varepsilon-\frac12\delta}+Q^{\frac32+\frac{45}{4}\delta_0+8\delta_1+\delta+\frac{135}{4}\varepsilon}.
	\end{align}
	We choose $\delta=\tfrac{1}{3}$, and adjust $\varepsilon$ and $\delta_0$ to balance $\delta_1$. Determining the best choice of $\varepsilon$, $\delta_0$, and $\delta_1$ is somewhat involved for hand calculations; for simplicity, we may assume that $0<\varepsilon\leq 0.001$ and $0<\delta_0\leq 0.001$. Consequently, combining all the bounds above, we obtain the following
	\[
	\mathcal{US}(\Psi_1,\Psi_2, Q_1,Q_2; \al, \be)\ll Q^{2-\varepsilon}+Q^{1.5+0.01+\frac32\delta_1}+Q^{2+\varepsilon_1-\frac14\delta_1}+Q^{1.834+0.045+8\delta_1}.
	\]

	\section{Balanced sums} \label{sec:balancedsum}
	
	\subsection{Initial reductions}
	We will follow \cite[Sections 6]{CLMR2} to calculate the balanced sum in Proposition~\ref{prop:offdiagonal}. 
	
	Since $(q_1,q_2)=1$, the sum over $\chi \bmod q_1q_2$ can be transformed into a sum over $\chi_1 \bmod q_1$ and a sum over $\chi_2\bmod q_2$.
	Then using orthogonality relation for characters given in Lemma \ref{lem:orthogonal}, we obtain that 
	
	\begin{align} \label{eqn:BSinitial}
		\begin{aligned}
			\mathcal{BS}(\Psi_1,\Psi_2, Q_1,Q_2; \al, \be)  &= \frac 12\sum_{\pm }\sumd_{\substack{M, N \\ \max(M, N) \le Q^{2-\delta_0}}} \sumtwo_{\substack{m, n \geq 1\\ m\neq n}} \frac{\sigma(m; \al) \sigma(n; -\be) }{\sqrt{mn}}  V\bfrac{m}{M} V\bfrac{n}{N} \\
			&  \cdot \sumfour_{\substack{d_1, r_1,d_2,r_2 \\ (d_1d_2r_1r_2, mn) = 1 \\ (d_1d_2r_1r_2,6)=1\\(d_1r_1,d_2r_2)=1\\ r_1r_2 | m \pm n}} \mu(d_1)\mu(d_2) \phi(r_1)\phi(r_2) \Psi_1 \left( \frac{d_1r_1}{Q_1}\right)  \Psi_2 \left( \frac{d_2r_2}{Q_2}\right) W_{\al, \be}\left(m, n ; dr\right) \\
			&=: \sumd_{\substack{M, N \\ \max(M, N) \le Q^{2-\delta_0}}}  \mathcal {BD}(M, N) + \sumd_{\substack{M, N \\ \max(M, N) \le Q^{2-\delta_0}}} \mathcal {BG}(M, N),
		\end{aligned}
	\end{align}
	where $D_0$ is a parameter to be chosen later, $\mathcal {BD}(M, N)$ is the contribution from terms with $d_1 > D_0$ and $\mathcal {BG}(M, N)$ is the contribution of terms with $d_1 \leq D_0.$
	
	We consider $\mathcal {BD} (M, N)$ first. 
	\begin{lem}  \label{lem:BD}
		Let $\delta_0 > 0$ and let $M, N$ be such that $\max(M, N) \leq Q^{2-\delta_0}$, and let $D_0 \geq 1/2$. Then
		\begin{align*}
			\mathcal{BD}(M, N) \ll \frac{Q^{2 + \varepsilon}}{D_0} ,
		\end{align*}	
		
	\end{lem}
	
	\begin{proof}
		The proof follows \cite[Lemma 6.1]{CLMR2}. The only different part is that we do not have a main term. We may explain why the main term vanishes.\par 
		By Lemma~\ref{lem:weightW} we can assume that, for any $\varepsilon > 0$, $MN \leq Q^{3+\varepsilon}$.	Now, we express the condition $r_1r_2|m \pm n$ in~\eqref{eqn:BSinitial} as $\frac{2}{\phi(r_1r_2)}\sum_{ \substack{\psi \bmod r_1r_2 \\ \psi(-1) = 1} } \psi(m) \overline{\psi(n)}$.  The contribution of the principal character is 
		
		\begin{align*}
			& \sumtwo_{\substack{q_1,q_2\\ (q_1q_2,6)=1\\(q_1,q_2)=1}} \left( \sum_{\substack{d_1r_1 = q_1 \\ d_1 > D_0}}\mu(d_1)\right)\left( \sum_{\substack{d_2r_2 = q_2 }}\mu(d_2)\right)  \Psi_1\left( \frac{q_1}{Q_1}\right) \Psi_2\left( \frac{q_2}{Q_2}\right)\\
			&\hskip 1in\cdot \sumtwo_{\substack{m, n \geq 1 \\ (mn, q_1q_2) = 1 \\ m \neq n}} \frac{\sigma(m, \al)\sigma(n, -\be)}{\sqrt{mn}} V\bfrac{m}{M} V\bfrac{n}{N}  W_{\al, \be}\left(m, n ; q_1q_2\right) .
		\end{align*}
		Hence the sum over $d_2,r_2$ vanishes when $q_2>1$.\par 
		Next let $\mathcal {EB}(M, N)$ be the contribution from the non-principal characters, so that
		\begin{align}
			\label{EBterm}
			\begin{aligned}
				\mathcal {EB}(M, N) &:= \sumfour_{\substack{d_1,d_2,r_1,r_2\\ d_1 > D_0\\ (d_1d_2r_1r_2,6)=1\\(d_1r_1,d_2r_2)=1}} \mu(d_1)\mu(d_2) \Psi_1\left( \frac{d_1r_1}{Q_1} \right)\Psi_2\left( \frac{d_2r_2}{Q_2} \right)  \sum_{ \substack{\psi \bmod r_1r_2 \\ \psi(-1) =1 \\ \psi \neq \psi_0}} \sumtwo_{\substack{m, n \geq 1 \\ m \neq n \\ (mn, d_1d_2r_1r_2) = 1}} \psi(m) \overline{\psi(n)}  \\
				& \hskip 1in \cdot \frac{\sigma(m; \al) \sigma(n, -\be)}{\sqrt{mn}}V\bfrac{m}{M} V\bfrac{n}{N}  W_{\al, \be}\left(m, n ; d_1d_2r_1r_2\right) ,
			\end{aligned}
		\end{align}
		where $\psi_0$ is the principal character.
		We will show that 
		$$\mathcal {EB}(M, N) \ll \frac{Q^{2 + \varepsilon}}{D_0}.$$ 
		Note that adding back the terms $m = n$ to~\eqref{EBterm} contributes $O\left( \frac{Q^{2 + \varepsilon}}{D_0}\right)$. By the definition of $W_{\al, \be}$ in \eqref{eqn:Walbe} and the Mellin inversion for $V$, the sum over $m, n$ in $\mathcal {EB}(M, N)$ (without the condition $m \neq n$) is
		\begin{align} \label{summninEB}
			\begin{aligned}
				\frac{1}{(2\pi i )^3}\int_{(1)} \int_{(\varepsilon)} \int_{(\varepsilon)} &G\left( \frac 12 + s; \al, \be \right) H(s; \al, \be) \left( \frac{d_1d_2r_1r_2}{4^{1/3}\pi}\right)^{3s + \delta(\al, \be)} \widetilde{V}(z) \widetilde{V}(w) M^{z} N^{w} \\
				& \cdot \sumtwo_{\substack{m, n \geq 1 \\ (mn, d_1d_2r_1r_2) = 1}} \psi(m) \overline{\psi(n)} \frac{\sigma(m; \al) \sigma(n, -\be)}{m^{1/2 + s + z} n^{1/2 + s + w}} \> dz \> dw\>\frac{ds}{s},
			\end{aligned}
		\end{align}
		where $\varepsilon > 0.$ The sum over $m, n$ can be expressed in terms of $L$-functions as 
		\begin{align*}
			&  \frac{ L(\tfrac 12 + \alpha_1 + s + z, \psi)L(\tfrac 12 + \alpha_2 + s + z, f\otimes \psi)}{ L_{d_1d_2r_1r_2}(\tfrac 12 + \alpha_1 + s + z, \psi)L_{d_1d_2r_1r_2}(\tfrac 12 + \alpha_2 + s + z, f\otimes \psi)}\\
			&\hskip 2in \cdot\frac{ L(\tfrac 12 - \beta_1 + s + w, \overline{\psi})L(\tfrac 12 - \beta_2 + s + z, f\otimes \overline{\psi})}{ L_{d_1d_2r_1r_2}(\tfrac 12 -\beta_1 + s + z, \overline{\psi})L_{d_1d_2r_1r_2}(\tfrac 12 - \beta_2 + s + z, f\otimes \overline{\psi})},
		\end{align*}
		where $L_{a}(s, \psi) = \prod_{p | a}\left( 1- \frac{\psi(p)}{p^s}\right)^{-1}$ and $L_{a}(s, f\otimes\psi) = \prod_{p | a}L_p(s,f\otimes\psi)$. Since $\psi$ is not the trivial character, the Dirichlet $L$-functions above are entire.  We thus move the integral over $s$ to $\R(s) = \varepsilon$ without crossing any poles of the integrand. We further note that the gamma factor $G$ is $\ll \exp(- |\textrm{Im}(s)|)$, $L_{d_1d_2r_1r_2}^{-1}(s, \psi),L^{-1}_{d_1d_2r_1r_2}(s, f\otimes\psi) \ll Q^{\varepsilon}$,  $\widetilde{V}(\sigma + it) \ll_{\sigma, A} \frac{1}{1 + |t|^A}$, and $M, N \ll Q^{2 - \delta_0}$. Hence, the triple integral in \eqref{summninEB} is bounded by
		\begin{align*} &Q^{O(\varepsilon)} \int_{(\varepsilon)} \int_{(\varepsilon)} \int_{(\varepsilon)} \exp(-|\textrm{Im}(s)|) \frac{1}{1+|z|^{10}} \frac{1}{1+|w|^{10}} 
			\Bigg(  \left| L\left( \frac 12 + \alpha_1+ s + z , \psi \right)\right|^4\\
			& +\left| L\left(\frac 12+\alpha_2+ s + z , f\otimes\psi \right)\right|^4 +\left| L\left(\frac 12+\beta_2+ s + w , f\otimes\overline{\psi} \right)\right|^4+ \left| L\left( \frac 12 - \beta_1 + s + w , \psi \right)\right|^4\Bigg) \> dz \> dw \> ds .
		\end{align*}
		We insert this into~\eqref{EBterm} and use the large sieve inequality (analogous to~\cite[Theorem 7.34]{IK}). Adjusting $\varepsilon$, we obtain that $\mathcal{EB}(M, N) \ll \frac{Q^{2 + \varepsilon}}{D_0}$. 
		
	\end{proof}
	
	We recall that 
	\begin{align*}
		\mathcal {BG}(M, N) &= \frac 12\sumtwo_{\substack{m, n \geq 1\\ m\neq n}} \frac{\sigma(m; \al) \sigma(n; -\be) }{\sqrt{mn}}  V\bfrac{m}{M} V\bfrac{n}{N} \\
		& \hskip 1in \cdot \sumfour_{\substack{d_1\leq D_0, r_1,d_2,r_2 \\ (d_1d_2r_1r_2, mn) = 1 \\ (d_1d_2r_1r_2,6)=1\\(d_1r_1,d_2r_2)=1\\ r_1r_2 | m \pm n}} \mu(d_1)\mu(d_2) \phi(r_1)\phi(r_2) \Psi_1 \left( \frac{d_1r_1}{Q_1}\right)  \Psi_2 \left( \frac{d_2r_2}{Q_2}\right) W_{\al, \be}\left(m, n ; dr\right) .
	\end{align*}
	
	Let $g = \gcd(m, n)$ and write $m = g\m$ and $n = g\n$. Arguing as in~\cite[Equations (29)--(31)]{CLMR2} or \cite[Section 6]{CIS}, we obtain
	\begin{lem} We define, for $x, y, u \geq 0$,
		\begin{equation}
			\label{eq:Walbedef}
			\mathcal W^{\pm}_{\al, \be, A} (x, y; u) := u|x \pm y| \Psi_1(Au |x\pm y|)  W_{\al, \be}\left(x, y ; u |x \pm y|\right).
		\end{equation}
		Then  $$ \mathcal {BG}(M, N) = \mathcal {BG}^{+}(M, N) + \mathcal {BG}^-(M, N),$$
		\text{where}
		\begin{align}
			\label{eq:BGpm1}
			\begin{aligned}
				\mathcal{BG}^{\pm}(M, N) &= \frac{ (Q_1Q_2)^{1 + \delta(\al, \be)}}{2}  \sumtwo_{\substack{m, n \geq 1\\ m\neq n}} \frac{\sigma(m; \al) \sigma(n; -\be) }{\sqrt{mn}}  V\bfrac{m}{M} V\bfrac{n}{N}   \\
				&\cdot \sumtwo_{ \substack{d_1 \leq D_0 ,d_2\\ (d_1d_2, 6mn) = 1 }}\frac{\mu(d_1)\mu(d_2)}{d_1d_2} \sum_{\substack{(r_2,6mn)=1\\(d_1,d_2r_2)=1}}\Psi_2\left(\frac{d_2r_2}{Q_2}\right)\frac{\phi(r_2)}{r_2}\sum_{\substack{ (a_1, 6gd_2r_2) = 1}}\sum_{b_1|6gd_2r_2} \frac{\mu(a_1)\mu(b_1)}{a_1} \\ &\hskip 0.8in\cdot\sum_{ \substack{(h,\m \n)=1\\\m \equiv \mp \n \Mod{a_1b_1r_2h}}}\mathcal W^{\pm}_{\al, \be, Q_2/d_2r_2} \left( \frac{g\m}{(Q_1Q_2)^{3/2}}, \frac{g\n}{(Q_1Q_2)^{3/2}}; \frac{(Q_1Q_2)^{1/2}d_1d_2}{gh}\right).
			\end{aligned}
		\end{align}
	\end{lem}
	\begin{proof}
		Since the proof is almost identical to that of \cite{CLMR2} or \cite{CIS}, we only sketch the argument and indicate the necessary modifications.\par 
		We start from the definition of $\mathcal {BG}(M, N)$. We only open up $\phi(r_1)$ and write $r_1=a_1l_1$. Let $\m =\mp \n +a_1l_1r_2h$. Our goal is to apply the divisor switching trick used in \cite{CIS} in order to transform the sum over $l_1$ into a sum over $h$, and hence it is necessary to remove all constraints on $l_1$.
		
		Note that the condition $(l_1, mn)=1$ is equivalent to $(l_1, g)=1$. Consequently, the coprimality condition on $l_1$ reduces to $(l_1, 6 g d_2 r_2)=1$. We remove this restriction by Möbius inversion, introducing $\mu(b_1)$. The remaining steps are completely analogous to those in \cite{CLMR2} or \cite{CIS}.
	\end{proof}
	\begin{rem} \label{rem:gsize}
		Since $g = \textrm{gcd}(m, n)$, $b_1 | gd_2r_2$, $m \asymp M$ and $n \asymp N$, we have that $b_1 \ll\min( M, N)Q_2 $ $g \ll \min( M, N).$  Moreover, the factor $\Psi_1 \left( \frac{d_1d_2|\m \pm \n|}{Q_1Q_2h}\cdot \frac{Q_2}{d_2r_2}\right)$ forces $h \ll \frac{ Q^{2 - \delta_0} D_0d_2}{Q} \ll   Q^{1 - \delta_0} D_0d_2. $
	\end{rem}
	It is convenient to group $q_2=d_2r_2$ for further computation, writing as 
	\begin{align}
		\label{eq:BGpm}
		\begin{aligned}
			\mathcal{BG}^{\pm}(M, N) &= \frac{ (Q_1Q_2)^{1 + \delta(\al, \be)}}{2}  \sumtwo_{\substack{m, n \geq 1\\ m\neq n}} \frac{\sigma(m; \al) \sigma(n; -\be) }{\sqrt{mn}}  V\bfrac{m}{M} V\bfrac{n}{N}   \\
			&\cdot \sum_{\substack{(q_2,6g\m\n)=1}}\frac{1}{q_2}\Psi_2\left(\frac{q_2}{Q_2}\right)\sum_{d_2r_2=q_2}\mu(d_2)\phi(r_2)
			\sum_{ \substack{d_1 \leq D_0\\ (d_1, 6g\m\n q_2) = 1 }}\frac{\mu(d_1)}{d_1} \sum_{\substack{ (a_1, 6gq_2) = 1}}\sum_{b_1|6gq_2} \frac{\mu(a_1)\mu(b_1)}{a_1} \\ 
			&\hskip 0.8in\cdot\sum_{ \substack{(h,\m \n)=1\\\m \equiv \mp \n \Mod{a_1b_1r_2h}}}\mathcal W^{\pm}_{\al, \be, Q_2/q_2} \left( \frac{g\m}{(Q_1Q_2)^{3/2}}, \frac{g\n}{(Q_1Q_2)^{3/2}}; \frac{(Q_1Q_2)^{1/2}d_1d_2}{gh}\right).
		\end{aligned}
	\end{align}
	
	Next we write the condition $\m \equiv \mp \n \Mod {abh}$ as a sum over characters $\psi \Mod{a_1b_1r_2h}$. Note that this is possible because $(\m \n, a_1b_1r_2 h) = 1$ since $(\m , \n) = 1$ and $\m \equiv \pm \n \Mod{abh}$. Then we split $\mathcal {BG}(M,N) $ into two terms.  One is the contribution of the principal character, which contributes to the main term, while the other is the contribution of the non-principal characters, which contributes to the error term. More precisely we write
	$$ \mathcal {BG}^{\pm}(M,N) = \mathcal {MBG}^{\pm}(M,N) + \mathcal {EBG}^{\pm}(M,N),$$
	where
	\begin{align} \label{def:MBG}
		\begin{aligned}
			&	\mathcal{MBG}^{\pm}(M, N) := \frac{ (Q_1Q_2)^{1 + \delta(\al, \be)}}{2}  \sumtwo_{\substack{m, n \geq 1\\ m\neq n}} \frac{\sigma(m; \al) \sigma(n; -\be) }{\sqrt{mn}}  V\bfrac{m}{M} V\bfrac{n}{N}   \\
			&\hskip0.4in\cdot \sum_{\substack{(q_2,6g\m\n)=1}}\frac{1}{q_2}\Psi_2\left(\frac{q_2}{Q_2}\right)\sum_{d_2r_2=q_2}\mu(d_2)\phi(r_2)
			\sum_{ \substack{d_1 \leq D_0\\ (d_1, 6g\m\n q_2) = 1 }}\frac{\mu(d_1)}{d_1} \sum_{\substack{ (a_1, 6gq_2) = 1\\(a_1,\m\n)=1}}\sum_{\substack{b_1|6gq_2\\(b_1,\m\n)=1}} \frac{\mu(a_1)\mu(b_1)}{a_1} \\ 
			&\hskip 1.3in\cdot\sum_{ \substack{(h,\m \n)=1}}\frac{1}{\phi(a_1b_1r_2h)}\mathcal W^{\pm}_{\al, \be, Q_2/q_2} \left( \frac{g\m}{(Q_1Q_2)^{3/2}}, \frac{g\n}{(Q_1Q_2)^{3/2}}; \frac{(Q_1Q_2)^{1/2}d_1d_2}{gh}\right).
		\end{aligned}
	\end{align}
	and 
	\begin{align} \label{def:EBG}
		\begin{aligned}
			&	\mathcal{EBG}^{\pm}(M, N) :=  \frac{ (Q_1Q_2)^{1 + \delta(\al, \be)}}{2}  \sumtwo_{\substack{m, n \geq 1\\ m\neq n}} \frac{\sigma(m; \al) \sigma(n; -\be) }{\sqrt{mn}}  V\bfrac{m}{M} V\bfrac{n}{N}   \\
			&\hskip 0.4in\cdot \sum_{\substack{(q_2,6g\m\n)=1}}\frac{1}{q_2}\Psi_2\left(\frac{q_2}{Q_2}\right)\sum_{d_2r_2=q_2}\mu(d_2)\phi(r_2)
			\sum_{ \substack{d_1 \leq D_0\\ (d_1, 6g\m\n q_2) = 1 }}\frac{\mu(d_1)}{d_1} \sum_{\substack{ (a_1, 6gq_2) = 1\\ (a_1,\m\n)=1}}\sum_{\substack{b_1|6gq_2\\ (b_1,\m\n)=1}} \frac{\mu(a_1)\mu(b_1)}{a_1} \\ 
			&\hskip 0.1in\cdot\sum_{ \substack{(h,\m \n)=1}}\frac{1}{\phi(a_1b_1r_2h)}\sum_{\substack{\psi \bmod a_1b_1r_2h\\ \psi \neq \psi_0}}\psi(\m)\overline{\psi}(\mp\n)\mathcal W^{\pm}_{\al, \be, Q_2/q_2} \left( \frac{g\m}{(Q_1Q_2)^{3/2}}, \frac{g\n}{(Q_1Q_2)^{3/2}}; \frac{(Q_1Q_2)^{1/2}d_1d_2}{gh}\right).
		\end{aligned}
	\end{align}
	Moreover we define
	\begin{align}
		\label{def:MBEB}
		\begin{aligned}
			\mathcal {MBG}(M, N) &:= \mathcal {MBG}^+(M, N) + \mathcal {MBG}^-(M, N)\\
			\mathcal {EBG}(M, N) &:= \mathcal {EBG}^+(M, N) + \mathcal {EBG}^- (M, N).
		\end{aligned} 
	\end{align}
	Thus
	\begin{align}\label{eqn:BGrelatedtoMBGandEBG}
		\mathcal {BG}(M, N) = \mathcal {MBG}(M, N) + \mathcal {EBG}(M, N).
	\end{align}
	To evaluate $\mathcal {MBG}(M, N)$ and $\mathcal {EBG}(M, N)$, we require information about the Mellin transforms of $$\mathcal W^{\pm}_{\al, \be, A} (x, y, u).$$  We will collect lemmas about three different types of Mellin transforms. The first type is in the $u$-variable, the second one is in the $x, y$-variables, and the third one is in all three variables. The proofs of these lemmas follow closely the proofs in \cite{CLMR, CLMR2, CIS}, 
	but using the bound for  $W_{\al, \be}(\xi, \eta; \mu)$ in Lemma \ref{lem:weightW} instead.
	
	\begin{lemma} \label{lem:Mellin1}
		Given positive real numbers $x, y$ and $1\leq A\leq 2$, let 
		$$ \widetilde{\mathcal W}^{\pm}_{1, A}(x, y; z) = \int_0^\infty \mathcal W^{\pm}_{\al, \be,A}(x, y; u) u^z \frac{du}{u}. $$
		Then the functions $\Wt^\pm_{1,A}(x, y; z)$
		are analytic for all $z \in \mathbb C$. We have the Mellin inversion formula
		\begin{align*}
			\W^\pm_{\al, \be,A}(x, y; u) = \frac{1}{2\pi i} \int_{(c)} \Wt^\pm_{1,A}(x, y; z) u^{-z} \> dz,
		\end{align*}
		where the integral is taken over the line $\tRe(z) = c$ for any real number $c$. The Mellin transforms $\Wt^\pm_{1,A}(x, y; z)$ satisfy, for any non-negative integer $\nu$,
		$$ |\Wt^\pm_{1,A}(x, y;z)| \ll_\nu |x \pm y|^{-\tRe z} \prod_{j = 1}^\nu |z + j|^{-1} \exp \left(-c_0 (xy)^{1/3}\right)$$
		for some absolute constant $c_0.$
	\end{lemma}
	\begin{proof}
		This is essentially the same as \cite[Lemma 4]{CIS} and \cite[Lemma 6.4]{CLMR2}.
	\end{proof}
	
	\begin{lemma} \label{lem:MellinXY}
		Given a positive real number $u$ and $1\leq A\leq 2$, we define
		$$ \Wt^\pm_{2,A} (s_1, s_2; u) = \int_0^\infty \int_0^\infty \W^\pm_{\al, \be, A}(x, y ; u) x^{s_1}y^{s_2} \frac{dx}{x} \frac{dy}{y}.$$
		Then the functions $\Wt^\pm_{2,A}(s_1, s_2 ; u)$ are analytic in the region $\tRe (s_1), \tRe(s_2) > 0$. We have the Mellin inversion formula
		$$ \W^\pm_{\al, \be,A}(x, y ; u) = \frac{1}{(2\pi i)^2} \int_{(c_1)}\int_{(c_2)} \Wt^\pm_{2,A} (s_1, s_2 ; u) x^{-s_1} y^{-s_2} \>d s_1 \> d s_2,$$
		where $c_1, c_2 > 0$. The Mellin transforms $\Wt^\pm_{2,A}(s_1,s_2 ; u)$ satisfy, for any $k \geq 1$ and $l \geq 0,$ and any $s_1, s_2$ with $0 < \tRe(s_1), \tRe(s_2) \leq 100$ 
		$$ |\Wt^\pm_{2,A}(s_1, s_2; u)| \ll \frac{1}{\tRe(s_1) \tRe(s_2)} \cdot \frac{(1 + u)^{k-1}}{\max(|s_1|, |s_2|)^k |s_1 + s_2|^l } .$$
	\end{lemma}
	\begin{proof}
		This is essentially the same as~\cite[Proof of Lemma 7.2]{CLMR}.
		\end{proof}

		The next lemma is similar to \cite[Lemma 6]{CIS}, \cite[Lemma 6.2]{CLMR} and \cite[Lemma 6.6]{CLMR2}. The proof follows closely the proof of \cite[Lemma 6.2]{CLMR} and we may omit the proof. 
		
		\begin{lemma} \label{lem:MellinXYU} For $1\leq A\leq 2$, we define
			$$ \Wt^\pm_{3,A} (s_1, s_2; z) = \int_0^\infty \int_0^\infty  \int_0^\infty \W^\pm_{\al,\be,A}(x, y ; u) u^z x^{s_1}y^{s_2} \frac{du}{u} \frac{dx}{x} \frac{dy}{y}$$
			and 
			$$ \Wt_{3,A} (s_1, s_2; z) = \Wt^+_{3,A} (s_1, s_2; z) + \Wt^-_{3,A}(s_1, s_2; z).$$
			Let $\omega = \frac{s_1 + s_2 - z}{2}$ and $\xi = \frac{s_1 - s_2 + z}{2}.$ For $\tRe(s_1), \tRe(s_2) > 0,$ and $|\tRe(s_1 -s_2)| < \tRe(z) < 1$ we have
			\begin{align*} 
				\Wt_{3,A}(s_1, s_2; z) =A^{-1+\frac{z}{2}-\frac{3}{2}s_1-\frac{3}{2}s_2-\delta(\al,\be)} \frac{\widetilde{\Psi}_1(1 + \delta(\al, \be) + 3\omega + z) H(\omega; \al, \be)}{2\omega 4^\omega\pi^{3\omega + \delta(\al, \be)}}  \Hc (\xi, z) G\left(\h + \omega; \al, \be \right), 
			\end{align*}
			where $\widetilde{\Psi}_1$ is the Mellin transform of $\Psi_1$, and 
			$$ \Hc (u, v) = \pi^{1/2} \frac{\Gamma\left(\tfrac{u}{2} \right)\Gamma\left(\tfrac{1-v}{2} \right)\Gamma\left(\tfrac{v-u}{2} \right)}{\Gamma\left(\tfrac{1-u}{2} \right)\Gamma\left(\tfrac{v}{2} \right)\Gamma\left(\tfrac{1-v + u}{2} \right)}.$$

			Let $x \neq y$ and $T \ge Q^{\varepsilon}$. For any $c_1, c_2 > 0$ with $ |c_1 - c_2| < c < 1$, one has the truncated Mellin inversion formulas
			\begin{align*}
				\W_A(x, y ; u) &=\frac{1}{(2\pi i)^3} \int_{(c)} \int_{c_1 - iT}^{c_1 + iT}\int_{c_2 - iT}^{c_2 + iT} \Wt_{3,A} (s_1, s_2 ; z) u^{-z} x^{-s_1} y^{-s_2} \>d s_2 \> d s_1 \> dz \\
				&\qquad \qquad \qquad \qquad +  O\left(\frac{ u^{-c} x^{-c_1} y^{-c_2}}{ T^{1 - c}\left|\log\left( \frac xy\right) \right| }  \right).
			\end{align*}
			Moreover, let $\Wt_{1,A}(x, y ; z) = \Wt_{1,A}^+(x, y ; z) + \Wt_{1,A}^-(x, y ; z)$. Then for $\tRe z = c$,
			\begin{align}
				\label{eqn:truncate2}
				\Wt_{1,A}(x, y ; z) &=\frac{1}{(2\pi i)^2} \int_{c_1 - iT}^{c_1 + iT}\int_{c_2 - iT}^{c_2 + iT} \Wt_{3,A} (s_1, s_2 ; z) x^{-s_1} y^{-s_2} \>d s_2 \> d s_1 \> \\
				\nonumber
				&\qquad \qquad \qquad \qquad +  O\left(\frac{  x^{-c_1} y^{-c_2}}{ T^{1 - c}\left|\log\left( \frac xy\right) \right| (1 + |z|)^{B} }  \right),
			\end{align}for any $B>0$.
			Finally, for $\tRe(s_1), \tRe(s_2) > 0,$ and $|\tRe(s_1 -s_2)| < \tRe(z) < 1$, the Mellin transform $\Wt_{3,A}(s_1, s_2; z)$ satisfies the bound
			\begin{equation} \label{eqn:boundWt3}
				|\Wt_{3,A}(s_1, s_2;z)| \ll (1 + |z|)^{-B} (1 + |\omega|)^{-B} (1 + |\xi|)^{\tRe(z) - 1},
			\end{equation}for any $B>0$.
			
		\end{lemma}
		\subsection{Evaluating the main terms} \label{ssec:evalMBG}
		In this section we will evaluate $\mathcal {MBG}^{\pm}(M,N)$ defined in \eqref{def:MBG}. We group $\ell_1=a_1b_1$, getting
		$$  \sum_{(a_1, 6g\m\n q_2) = 1} \frac{\mu(a_1)}{a_1} \sum_{\substack{b_1|6gq_2 \\ (b_1, \m\n) = 1}} \frac{\mu(b_1)}{\phi(a_1b_1r_2h)} = \sum_{(\ell_1, \m\n) = 1} \frac{\mu(\ell_1)(\ell_1,6 gq_2)}{\ell_1 \phi(\ell_1r_2 h)}.$$
		
		We may apply Mellin inversion to $\W_{\al,\be,A}$ to separate the variables, and then evaluate the sums over $h$ and $\ell_1$. To this end, we are led to consider
		\begin{equation}\label{eq:hl1-sum}
			\sum_{(h,\m\n)=1} \sum_{(\ell_1,\m\n)=1}
			\frac{\mu(\ell_1)\,(\ell_1,6g q_2)}{\ell_1\,\phi(\ell_1 r_2 h)\,h^s},
		\end{equation}
		which is precisely the type of sum treated in \cite[Lemma~7]{CIS}. However, due to the presence of the factor $r_2$, the inner sum is no longer jointly-multiplicative in $h$ and $\ell_1$, and hence the result of \cite{CIS} does not apply directly. We therefore re-evaluate this sum.
		\begin{lem}
			For $\tRe(s)>0$ and $a
			\in \mathbb{N}^+$, we have
			\begin{align*}
				\sum_{(h,\m\n)=1}\frac{1}{\phi(ah)h^s}=\frac{\zeta(1+s)}{a}\prod_{p\nmid\m\n}\left(1+\frac{1}{(p-1)p^{s+1}}\right)\prod_{p\mid \m\n}\left(1-\frac{1}{p^{s+1}}\right)\prod_{p\mid a}\frac{p^{s+2}}{p^{s+2}-p^{s+1}+1}.
			\end{align*} 
		\end{lem}
		\begin{proof}
			Write $h=h_0h_1$, where $(h_0,a)=1$ and $h_1\mid a^\infty$. Then the sum over $h_0$ and $h_1$ can be computed directly.
		\end{proof}
		Using this lemma, we have
		\begin{lem}\label{Sumhl}
			For $\tRe(s)>0$, $q_2=d_2r_2$ and $(q_2,6g\m\n)=1$, we have
			\begin{equation}
				\begin{aligned}
					\sum_{(h,\m\n)=1} \sum_{(\ell_1,\m\n)=1}
					&\frac{\mu(\ell_1)\,(\ell_1,6g q_2)}{\ell_1\,\phi(\ell_1 r_2 h)\,h^s}=\frac{\zeta(1+s)}{r_2}\prod_{p\mid r_2}\left(\left(1-\frac{1}{p}\right)\frac{p^{s+2}}{p^{s+2}-p^{s+1}+1}\right)\prod_{p\mid \m\n}\left(1-\frac{1}{p^{s+1}}\right)\\
					&\cdot \prod_{p\nmid\m\n}\left(1+\frac{1}{(p-1)p^{s+1}}\right)\prod_{p\nmid 6g\m\n q_2}\left(1-\frac{p^s}{p^{s+2}-p^{s+1}+1}\right)\prod_{\substack{p\nmid \m\n \\p\nmid r_2\\ p\mid 6gq_2}}\left(1-\frac{p^{s+1}}{p^{s+2}-p^{s+1}+1}\right)
					.
				\end{aligned}
			\end{equation}
		\end{lem}
		\begin{proof}
			This is a straightforward verification. We remark that when $r_2=1$, this formula is consistent with \cite[Lemma~7]{CIS}.
		\end{proof}
		Since $\prod_{p\mid 6gq_2}=\prod_{p\mid 6g}\prod_{\substack{p\nmid 6g\\p\mid q_2}}$, we define
		\begin{align}\label{K}
			&	\mathcal K(s; g, \m\n;r_2,q_2) := \phi(\m\n, s+1) \prod_{p\mid r_2}\left(\left(1-\frac{1}{p}\right)\frac{p^{s+2}}{p^{s+2}-p^{s+1}+1}\right)\prod_{p\nmid\m\n}\left(1+\frac{1}{(p-1)p^{s+1}}\right) \\
			&\hskip 0.5in \cdot \prod_{p\nmid 6g\m\n q_2}\left(1-\frac{p^s}{p^{s+2}-p^{s+1}+1}\right)\prod_{\substack{p\nmid \m\n \\p\mid 6g}}\left(1-\frac{p^{s+1}}{p^{s+2}-p^{s+1}+1}\right)\prod_{\substack{p\nmid r_2\\ p\mid q_2}}\left(1-\frac{p^{s+1}}{p^{s+2}-p^{s+1}+1}\right).\nonumber
		\end{align}
		and $\phi(\ell, s) := \prod_{p | \ell} \left( 1 - \frac{1}{p^s}\right).$  We now prove the following Lemma.
		
		\begin{lem} \label{lem:MBGMN} Let $\varepsilon > 0$. Let $\mathcal {MBG}^{\pm}(M, N)$ be as in \eqref{def:MBG} with $D_0 \geq 2$ to be chosen later. Then
			\begin{align*}
				\mathcal {MBG}^{\pm}(M, N) = \mathcal {MBG}_1^{\pm}(M, N) +  O \left( \frac{Q^{2 + \varepsilon}}{D_0} \right),
			\end{align*}
			where
			\begin{align}
				\label{def:MBG1}
				\begin{aligned}
					\mathcal {MBG}_1^{\pm}(M, N) := &\frac{(Q_1Q_2)^{1 + \delta(\al, \be)}}{2}
					\sum_{ \substack{m, n \geq 1 \\ m \neq n } } \frac{\sigma(m; \al) \sigma(n; -\be)}{\sqrt{mn}} V\left( \frac mM\right) V\left( \frac nN\right) \sum_{(q_2,6g\m\n)=1}\frac{1}{q_2}\Psi_2\left(\frac{q_2}{Q_2}\right)\\
					& \hskip 0.7in\cdot \frac{1}{2\pi i } \int_{(\varepsilon)}\sum_{d_2r_2=q_2}\frac{\mu(d_2)}{d_2^{z}}\frac{\phi(r_2)}{r_2} \widetilde{\mathcal W}_{1,Q_2/q_2}^\pm \left( \frac{g\m}{(Q_1Q_2)^{3/2}}, \frac{g\n}{(Q_1Q_2)^{3/2}} ; z \right) \\ 
					&\hskip1.3in\cdot \frac{\zeta(1-z) \mathcal K(-z; g, \m\n;r_2,q_2) }{\zeta(1 + z) \phi (6g\m\n q_2, 1 + z)}\left( \frac{(Q_1Q_2)^{1/2}}{g}\right)^{-z} \> dz.
				\end{aligned}
			\end{align}
		\end{lem}
		\begin{proof}
			We follow the arguments  in \cite[Equations (57)-(62)]{CIS}, using the  Mellin transform from Lemma \ref{lem:Mellin1}. This gives
			\begin{align*}
				&	\mathcal{MBG}^{\pm}(M, N) := \frac{ (Q_1Q_2)^{1 + \delta(\al, \be)}}{2}  \sumtwo_{\substack{m, n \geq 1\\ m\neq n}} \frac{\sigma(m; \al) \sigma(n; -\be) }{\sqrt{mn}}  V\bfrac{m}{M} V\bfrac{n}{N}   \\
				&\hskip0.4in\cdot \sum_{\substack{(q_2,6g\m\n)=1}}\frac{1}{q_2}\Psi_2\left(\frac{q_2}{Q_2}\right)\sum_{d_2r_2=q_2}\mu(d_2)\phi(r_2)
				\sum_{ \substack{d_1 \leq D_0\\ (d_1, 6g\m\n q_2) = 1 }}\frac{\mu(d_1)}{d_1} \sum_{\substack{ (a_1, 6gq_2) = 1\\(a_1,\m\n)=1}}\sum_{\substack{b_1|6gq_2\\(b_1,\m\n)=1}} \frac{\mu(a_1)\mu(b_1)}{a_1} \\ 
				&\hskip 0.5in\cdot\sum_{ \substack{(h,\m \n)=1}}\frac{1}{\phi(a_1b_1r_2h)}\frac{1}{2\pi i}\int_{(-\varepsilon)} \widetilde{\mathcal W}_{1,Q_2/q_2}^\pm \left( \frac{g\m}{(Q_1Q_2)^{3/2}}, \frac{g\n}{(Q_1Q_2)^{3/2}} ; z \right)\left(\frac{(Q_1Q_2)^{\frac{1}{2}}d_1d_2}{g}\right)^{-z}\>dz.
			\end{align*}
			We take $\tRe(z)=-\varepsilon$, since the sum over $h$ converges absolutely only for $\tRe(z)<0$. Now apply Lemma \ref{Sumhl} and move the line $\tRe(z)=-\varepsilon$ to $\tRe(z)=\varepsilon$. We encounter a pole at $z=0$ from $\zeta(1-z)$, whose residue (taking into account that the contour is
			oriented clockwise, and that the residue of $\zeta(1-z)$ at $z=0$ is $-1$) equals
			
			\begin{align*}
				&\frac{(Q_1Q_2)^{1 + \delta(\al, \be)}}{2}
				\sum_{ \substack{m, n \geq 1 \\ m \neq n } } \frac{\sigma(m; \al) \sigma(n; -\be)}{\sqrt{mn}} V\left( \frac mM\right) V\left( \frac nN\right) \sum_{(q_2,6g\m\n)=1}\frac{1}{q_2}\Psi_2\left(\frac{q_2}{Q_2}\right)\sum_{d_2r_2=q_2}\mu(d_2)\frac{\phi(r_2)}{r_2}\\
				& \hskip 0.15in\cdot 
				\sum_{\substack{d_1\leq D_0\\ (d_1,6g\m\n q_2)=1}}\frac{\mu(d_1)}{d_1} \widetilde{\mathcal W}_{1,Q_2/q_2}^\pm \left( \frac{g\m}{(Q_1Q_2)^{3/2}}, \frac{g\n}{(Q_1Q_2)^{3/2}} ; 0 \right)\prod_{p\mid r_2}\left(\left(1-\frac{1}{p}\right)\frac{p^2}{p^2-p+1}\right)	\prod_{p\mid \m\n}\left(1-\frac{1}{p}\right)\\
				&\hskip 0.4in \cdot
				\prod_{p\nmid \m\n}\left(1+\frac{1}{(p-1)p}\right)\prod_{p\nmid 6g\m\n q_2}\left(1-\frac{1}{p^2-p+1}\right)
				\prod_{\substack{p\mid 6g\\p\nmid \m\n}}\left(1-\frac{p}{p^2-p+1}\right)
				\prod_{\substack{p\mid d_2\\p\nmid r_2}}\left(1-\frac{p}{p^2-p+1}\right).
			\end{align*}
			We claim that the sum over $d_2$ and $r_2$ vanishes. To prove this, we collect together all terms depending on $d_2$ and $r_2$, obtaining
			\begin{align*}
				\sum_{d_2r_2=q_2}\mu(d_2)\frac{\phi(r_2)}{r_2}\prod_{p\mid r_2}\frac{p(p-1)}{p^2-p+1}\prod_{\substack{p\mid d_2\\p\nmid r_2}}\frac{(p-1)^2}{p^2-p+1}&=\sum_{d_2r_2=q_2}\mu(d_2)\prod_{p\mid r_2}\frac{(p-1)^2}{p^2-p+1}\prod_{\substack{p\mid d_2\\p\nmid r_2}}\frac{(p-1)^2}{p^2-p+1}\\
				&=\prod_{p\mid q_2}\frac{(p-1)^2}{p^2-p+1}\sum_{d_2r_2=q_2}\mu(d_2).
			\end{align*}
			Hence, when $q_2>1$, the residue vanishes. Therefore, $	\mathcal {MBG}^{\pm}(M, N) $ equals
			
			\begin{align*}
				\mathcal {MBG}^{\pm}(M, N) = &\frac{(Q_1Q_2)^{1 + \delta(\al, \be)}}{2} \sumtwo_{ \substack{m, n \geq 1 \\ m \neq n} } \frac{\sigma(m; \al) \sigma(n; -\be)}{\sqrt{mn}} V\left( \frac mM\right) V\left( \frac nN\right)\sum_{(q_2,6g\m\n)=1}\frac{1}{q_2}\Psi_2\left(\frac{q_2}{Q_2}\right)\\
				& \cdot\sum_{d_2r_2=q_2}\mu(d_2)\frac{\phi(r_2)}{r_2}  \frac{1}{2\pi i } \int_{(\varepsilon)} \widetilde{\mathcal W}_{1,Q_2/q_2}^\pm \left( \frac{g\m}{(Q_1Q_2)^{3/2}}, \frac{g\n}{(Q_1Q_2)^{3/2}} ; z \right) \zeta(1-z)\\
				&\hskip 1in\cdot  \mathcal K(-z; g, \m\n;r_2,q_2) \left( \frac{(Q_1Q_2)^{1/2}d_2}{g}\right)^{-z} \sum_{\substack{d \leq D_0 \\ (d, 6g\m\n q_2) = 1}} \frac{\mu(d_1)}{d_1^{1 + z}} \> dz,
			\end{align*}
			where as usual $m = g \m$, $n = g \n$ and $(\m, \n) = 1$.
			Next we deal with $\mathcal {MBG}^{\pm}(M, N)$ by following the argument in \cite[Equations (62) - (63)]{CIS}. To be more specific, we move the line of integration to $\tRe z = 1- \varepsilon$ and extend the sum over $d$ to all positive integers. Then we move the integration back to $\tRe z = \varepsilon$ 
			at a cost of $O (Q^{2 + \varepsilon} / D_0)$. We then obtain that 
			$$ \mathcal {MBG}^{\pm}(M, N) = \mathcal {MBG}_1^{\pm}(M, N) + O\left( \frac{Q^{2 + \varepsilon}}{D_0}\right).$$
			This concludes the proof of the lemma.

		\end{proof}
		Let 
		\begin{equation} \label{def:MBG1combinepm}
			\mathcal {MBG}_1(M, N) := \mathcal {MBG}_1^{+}(M, N) + \mathcal {MBG}_1^{-}(M, N).
		\end{equation}
		By \eqref{def:MBEB} and Lemma~\ref{lem:MBGMN},  we obtain that 
		$$\mathcal {MBG}(M, N) =  \mathcal {MBG}_1(M, N) + O \left( \frac{Q^{2 + \varepsilon}}{D_0} \right). $$
		
		Next we consider the main term contribution from $\mathcal {MBG}_1(M, N)$. The next lemma is a variant of \cite[Lemma~6.8]{CLMR2}, with the necessary modifications.

		\begin{lem} \label{lem:sumMBG1} Let $\varepsilon > 0$ and $\delta_0 > 0$ be fixed. Then
			$$\sumtwodee_{M, N \leq Q^{2 - \delta_0} } \mathcal {MBG}_1(M, N) = \sumd_M\sumd_N \mathcal {MBG}_1(M, N)  + O(Q^{2-\frac {1}{4} + \frac{3}{4} \delta_0 + \varepsilon}).$$	
		\end{lem}
		
		To prove Lemma \ref{lem:sumMBG1}, it is sufficient to show that $\mathcal {MBG}_1(M, N)$ is small
		when $M$ or $N \gg Q^{2 - \delta_0}$. Since we can assume $MN \ll Q^{3 + \varepsilon}$, without loss of generality, we assume that $ M \gg Q^{2 - \delta_0} $ and $ N \ll Q^{1 + \delta_0 + \varepsilon}$. Thus Lemma \ref{lem:sumMBG1} will immediately follow from the following lemma, which is a variant of \cite[Lemma~6.9]{CLMR2}.

		\begin{lem}  \label{lem:MBGunbalancedMN}  Let $\varepsilon > 0$. Let $\mathcal {MBG}_1(M, N)$ be as in \eqref{def:MBG1combinepm} with $\mathcal{MBG}_1^\pm(M, N)$ as in~\eqref{def:MBG1}. For any $M \gg Q^{2 - \delta_0}$ and $N \ll Q^{1 + \delta_0 + \varepsilon},$ we have
			$$ 
			\mathcal {MBG}_1(M, N) \ll Q^{\frac {7}{4} + \frac{3}{4} \delta_0 + \varepsilon}  .
			$$
		\end{lem}
		\begin{proof} 
			The proof follows that of \cite[Lemma~6.9]{CLMR2} very closely, with only minor differences in some computations. Moreover, these computations largely repeat those in Proposition~\ref{prop:maincont9terms}, so we omit the proof.

		\end{proof}
		
		Lemma~\ref{lem:sumMBG1} implies that we can extract the main contribution of $$ \sumtwodee_{M, N \leq Q^{2 - \delta_0} } \mathcal {MBG}_1(M, N)  $$ from the whole range of dyadic summation $M, N$. 
		From Equation  \eqref{def:MBG1},
		\begin{align*}
			\sumd_M\sumd_N \mathcal {MBG}_1(M, N) &= \frac{(Q_1Q_2)^{1 + \delta(\al, \be)}}{2} \sum_{ \substack{m, n \geq 1 \\ m \neq n } } \frac{\sigma(m; \al) \sigma(n; -\be)}{\sqrt{mn}} \sum_{(q_2,6g\m\n)=1}\frac{1}{q_2}\Psi_2\left(\frac{q_2}{Q_2}\right)\\
			& \cdot \sum_{d_2r_2=q_2}\frac{\mu(d_2)}{d_2^z}\frac{\phi(r_2)}{r_2}\frac{1}{2\pi i } \int_{(\varepsilon)} \widetilde{\mathcal W}_{1,Q_2/q_2}^\pm \left( \frac{g\m}{(Q_1Q_2)^{3/2}}, \frac{g\n}{(Q_1Q_2)^{3/2}} ; z \right)\\
			&\hskip1in\cdot  \frac{\zeta(1-z) \mathcal K(-z; g, \m\n;r_2,q_2) }{\zeta(1 + z) \phi (g\m\n, 1 + z)}\left( \frac{(Q_1Q_2)^{1/2}}{g}\right)^{-z} \> dz,
		\end{align*}
		which is the same expression as \cite[Equation (63)]{CIS} (although our definitions of $\widetilde{\mathcal{W}}_{1,Q_2/q_2}$ differ). Next we use Equation \eqref{eqn:truncate2} to express $\widetilde{\mathcal W}_{1,Q_2/q_2}^\pm$  as an integration over $s_1$ and $s_2$. Then we take advantage of the work in \cite[Section 10]{CIS}, which extracts from the above expression the main terms in $\tQ(q;\al, \be)$. We summarize the result in the proposition below. 
		
		\begin{prop} \label{prop:maincont9terms}
			Let $\mathcal {MBG}_1(M, N)$ be as in \eqref{def:MBG1combinepm} with $\mathcal{MBG}_1^\pm(M, N)$ as in~\eqref{def:MBG1}. Then 
			\begin{align*}
				\sumd_M\sumd_N \mathcal {MBG}_1(M, N) = H(0; \al, \be) \sumtwo_{\substack{q_1,q_2\\(q_1q_2,6)=1\\ (q_1,q_2)=1}} \Psi_1\bfrac{q_1}{Q_1}  \Psi_2\bfrac{q_2}{Q_2}&\phi^{\flat}(q_1q_2)Q(q_1q_2; \pi(\al), \pi(\be)))\\
				& + O\left(\frac{Q^{2+ \varepsilon}}{Q_2}+\frac{Q^{2+ \varepsilon}}{Q_1}\right).
			\end{align*}
		\end{prop}	
		\begin{proof}
			We will follow the strategy of \cite[Section 10]{CIS} and point out the necessary modification. In the definition of $\mathcal {MBG}_1(M, N)$ in \eqref{def:MBG1combinepm}, 
			we add up the $\widetilde{\mathcal W}_{1,Q_2/q_2}^+$ and $\widetilde{\mathcal W}_{1,Q_2/q_2}^-$ terms from $\mathcal {MBG}_1^\pm(M, N)$ in \eqref{def:MBG1}, getting
			\begin{align}
				\begin{aligned}
					\sumd_M\sumd_N \mathcal {MBG}_1(M, N) = &\frac{(Q_1Q_2)^{1 + \delta(\al, \be)}}{2}
					\sum_{ \substack{m, n \geq 1 } } \frac{\sigma(m; \al) \sigma(n; -\be)}{\sqrt{mn}}  \sum_{(q_2,6g\m\n)=1}\frac{1}{q_2}\Psi_2\left(\frac{q_2}{Q_2}\right)\\
					& \hskip 0.7in\cdot \frac{1}{2\pi i } \int_{(\varepsilon)}\sum_{d_2r_2=q_2}\frac{\mu(d_2)}{d_2^{z}}\frac{\phi(r_2)}{r_2}  \widetilde{\mathcal W}_{1,Q_2/q_2} \left( \frac{g\m}{(Q_1Q_2)^{3/2}}, \frac{g\n}{(Q_1Q_2)^{3/2}} ; z \right) \\ 
					&\hskip 0.8in\cdot \frac{\zeta(1-z) \mathcal K(-z; g, \m\n;r_2,q_2) }{\zeta(1 + z) \phi (6g\m\n q_2, 1 + z)}\left( \frac{(Q_1Q_2)^{1/2}}{g}\right)^{-z} \> dz +O\left(Q^{1 + \varepsilon}\right),
				\end{aligned}
			\end{align}
			where the $O$-term arises from adding back the contribution of the diagonal term $m=n$.

			Applying the Mellin transform in Equation \eqref{eqn:truncate2} by letting $T\to \infty $, we obtain that

			\begin{align}\label{MBG_1}
				\sumd_M\sumd_N	\mathcal {MBG}_1(M, N) = &\frac{(Q_1Q_2)^{1 + \delta(\al, \be)}}{2} \frac{1}{(2\pi i )^3} \sum_{(q_2,6)=1}\frac{1}{q_2}\Psi_2\left(\frac{q_2}{Q_2}\right) \int_{(\varepsilon)} \int_{(\frac 12 + \varepsilon )} \int_{(\frac 12 + \varepsilon )} \\ \nonumber
				&\cdot \sum_{d_2r_2=q_2}\frac{\mu(d_2)}{d_2^{z}}\frac{\phi(r_2)}{r_2} \widetilde{\mathcal W}_{3,Q_2/q_2} \left( s_1, s_2 ; z \right) \frac{\zeta(1-z) }{\zeta(1 + z) }(Q_1Q_2)^{\frac 32(s_1 + s_2)- \frac z2} \\ \nonumber
				&  \hskip 1.2in \cdot   \mathcal F(s_1, s_2; z;r_2,q_2)  \>ds_1 \> ds_2 \> dz \nonumber
			\end{align}
			where 
			$$ \mathcal F(s_1, s_2;z;r_2,q_2) := \sumtwo_{ \substack{m, n \geq 1 \\(mn,q_2)=1} } \frac{\sigma(m; \al) \sigma(n; -\be)}{m^{1/2 + s_1  } n^{1/2 + s_2 }}  \frac{g^z \mathcal K(-z; g, \m\n;r_2,q_2) }{\phi (6g\m\n q_2, 1 + z)}.$$
			By separating out the constant terms in $\frac{\mathcal K(-z; g, \m\n;r_2,q_2)}{\phi (6g\m\n q_2, 1 + z)}$, the inner sum becomes multiplicative. A straightforward calculation then yields
			\begin{align}\label{F1}
				&\mathcal F(s_1, s_2;z;r_2,q_2)=f\left(s_1,s_2;z;r_2,q_2\right)\prod_{p\nmid 6q_2}\Biggl( \left(1-\frac{p^{-z}}{p^{2-z}-p^{1-z}+1}\right)\left(1-\frac{1}{p^{1+z}}\right)\\ \nonumber
				&\hskip 1.5in+ \sumtwo_{\substack{a, b \geq 0 \\ \max (a, b) \geq 1\\a\neq b}} \frac{\sigma\left(p^a ; \boldsymbol{\alpha}\right) \sigma\left(p^b ;-\boldsymbol{\beta}\right)}{p^{a\left(\frac{1}{2}+s_1\right)} p^{b\left(\frac{1}{2}+s_2\right)}} p^{z \min (a, b)}\left(1-\frac{p}{p^{2-z}-p^{1-z}+1}\right)\\ \nonumber
				&\hskip 2in+\sum_{k=1}^{\infty} \frac{\sigma\left(p^k ; \boldsymbol{\alpha}\right) \sigma\left(p^k ;-\boldsymbol{\beta}\right)}{p^{k\left(1+s_1+s_2-z\right)}}\left(1-\frac{p^{1-z}}{p^{2-z}-p^{1-z}+1}\right)\Biggl),\nonumber
			\end{align}
			where 
			\begin{align*}
				&	f\left(s_1,s_2;z;r_2,q_2\right)=\prod_{p\mid r_2}\left(\left(1-\frac{1}{p}\right)\frac{p^{2-z}}{p^{2-z}-p^{1-z}+1}\right)\prod_{\substack{p\nmid r_2\\ p\mid q_2}}\left(1-\frac{p^{1-z}}{p^{2-z}-p^{1-z}+1}\right)\\
				&\hskip1in\cdot\prod_{p}\left(1+\frac{1}{(p-1)p^{1-z}}\right)\prod_{p}\left(1-\frac{1}{p^{1+z}}\right)^{-1}\prod_{p\mid 6}\Biggl( 1-\frac{p^{-z}}{p^{2-z}-p^{1-z}+1}\\
				&\hskip 1.5in+ \sumtwo_{\substack{a, b \geq 0 \\ \max (a, b) \geq 1\\a\neq b}} \frac{\sigma\left(p^a ; \boldsymbol{\alpha}\right) \sigma\left(p^b ;-\boldsymbol{\beta}\right)}{p^{a\left(\frac{1}{2}+s_1\right)} p^{b\left(\frac{1}{2}+s_2\right)}} p^{z \min (a, b)}\left(1-\frac{p}{p^{2-z}-p^{1-z}+1}\right)\\
				&\hskip 2in+\sum_{k=1}^{\infty} \frac{\sigma\left(p^k ; \boldsymbol{\alpha}\right) \sigma\left(p^k ;-\boldsymbol{\beta}\right)}{p^{k\left(1+s_1+s_2-z\right)}}\left(1-\frac{p^{1-z}}{p^{2-z}-p^{1-z}+1}\right)\Biggl).
			\end{align*}
			The behavior of $\mathcal F$ is dominated by the contributions from $(a,b)=(1,0),(0,1)$ and $k=1$ terms above. Thus we have

			\begin{align}\label{F2}
				&\mathcal F(s_1, s_2;z;r_2,q_2)=\frac{\zeta(2-z)\zeta\left(\frac{1}{2}+s_1+\alpha_1\right)L\left(\frac 12+s_1+\alpha_2,f\right)\zeta\left(\frac{1}{2}+s_2-\beta_1\right)L\left(\frac 12+s_2-\beta_2,f\right)}{\zeta\left(\frac{3}{2}+s_1+\alpha_1-z\right)L\left(\frac 32+s_1+\alpha_2-z,f\right)\zeta\left(\frac{3}{2}+s_2-\beta_1-z\right)L\left(\frac 32+s_2-\beta_2-z,f\right)}\\ \nonumber
				&\hskip 0.5in\cdot \zeta\left(1+s_1+s_2+\alpha_1-\beta_1-z\right)L\left(1+s_1+s_2+\alpha_1-\beta_2-z,f\right)L\left(1+s_1+s_2+\alpha_2-\beta_1-z,f\right)\\ \nonumber
				&\hskip 0.2in\cdot \zeta\left(1+s_1+s_2+\alpha_2-\beta_2-z\right)L\left(1+s_1+s_2+\alpha_2-\beta_2-z,\text{sym}^2f\right)\mathcal R(s_1, s_2; z;r_2,q_2),\nonumber
			\end{align}
			where $\mathcal R(s_1, s_2;z;r_2,q_2)$ is absolutely convergent in a wider range of $s_1, s_2$ and $z$, a subset of which is the region 
			
			$$ \R (z) <  \frac 54,  \ \ \ \ \frac 12 + \sum_{i = 1}^{2} \R (s_i ) > \R (z) + 2\max(|\alpha_i|, |\beta_j|), $$
			$$ \R(s_i ) > \max(|\alpha_i|, |\beta_j|) , \ \ \ \textrm{and}  \ \ \  1 + \R(s_i ) > \R(z) + \max(|\alpha_i|, |\beta_j|).   $$
			Keeping $z$ fixed, we move the lines of integration in $s_1$ and $s_2$ to $\operatorname{Re}\left(s_1\right)=2 \epsilon$ and $\operatorname{Re}\left(s_2\right)=2 \epsilon$, and encountour two poles at $s_1=\frac 12-\alpha_1$ and $s_2=\frac 12 +\beta_1$, and the error term can be easily bounded by $O(Q^{\frac 74+\varepsilon})$.\par 
			Now we work out the contribution of the residues. Our goal is to show that this term contributes
			\begin{equation}\label{MT}
				H(0; \al, \be) \sumtwo_{\substack{q_1,q_2\\(q_1q_2,6)=1\\ (q_1,q_2)=1}} \Psi_1\bfrac{q_1}{Q_1}  \Psi_2\bfrac{q_2}{Q_2}\phi^{\flat}(q_1q_2)\mathcal Q(q_1q_2; \pi(\al), \pi(\be)) ,
			\end{equation}
			where $\pi\in S_4$ is the permutation with $\pi\left(\al\right)=(\beta_1,\alpha_2)$ and $\pi\left(\be\right)=(\alpha_1,\beta_2)$. \par 
			The residue of $\mathcal F(s_1,s_2;z;r_2.q_2)$ at $s_1=\frac 12-\alpha_1$ and $s_2=\frac 12+\beta_1$ equals
			\begin{align*}
				L\left(1+\alpha_2-\alpha_1,f\right)&L\left(1+\beta_1-\beta_2,f\right)\zeta\left(2+\alpha_2-\alpha_1+\beta_1-\beta_2-z\right)\\
				&\cdot L\left(2+\alpha_2-\alpha_1+\beta_1-\beta_2-z,\text{sym}^2f\right)\mathcal R\left(\frac 12-\alpha_1,\frac 12+\beta_1;z;r_2,q_2\right).
			\end{align*}
			We then move the line of integration in $z$ to $\tRe z=\frac 54-\varepsilon$. In doing so, we encounter a simple pole at $z=1-\alpha_1+\beta_1$ (from the $\widetilde{\mathcal{W}}_{3,Q_2/q_2}\left(\frac{1}{2}-\alpha_1, \frac{1}{2}+\beta_1 ; z\right)$ term). Using Lemma \ref{lem:MellinXYU}, we get \eqref{MBG_1} equals
			\begin{align*}
				&\frac{Q_1^{2+\delta(\pi(\al),\pi(\be))}}{2\pi^{\delta(\al,\be)}}\sum_{(q_2,6)=1}q_2^{1+\delta(\pi(\al),\pi(\be))}\Psi_2\left(\frac{q_2}{Q_2}\right)\sum_{d_2r_2=q_2}\frac{\mu(d_2)}{d_2^{1-\alpha_1+\beta_1}}
				\cdot 	\frac{\phi(r_2)}{r_2} \widetilde{\Psi}_1\left(2+\delta(\pi(\al),\pi(\be))\right)\\
				& \hskip 0.5in\cdot H\left(0,\al,\be\right)\Hc\left(\frac 12-\alpha_1,1-\alpha_1+\beta_1\right)G\left(\frac 12;\al,\be\right)\frac{\zeta(\alpha_1-\beta_1)}{\zeta(2-\alpha_1+\beta_1)}L\left(1+\alpha_2-\alpha_1,f\right)\\
				&\cdot L\left(1+\beta_1-\beta_2,f\right)\zeta(1+\alpha_2-\beta_2)L\left(1+\alpha_2-\beta_2,\text{sym}^2f\right)\mathcal{R}\left(\frac{1}{2}-\alpha_1,\frac 12+\beta_1;1-\alpha_1+\beta_1;r_2,q_2\right).
			\end{align*}
			Using the functional equation connecting $\zeta\left(\alpha_1-\beta_1\right)$ and $\zeta\left(1-\alpha_1+\beta_1\right)$, the above simplifies to give
			\begin{align}\label{MT2}
				&\frac{Q_1^{2+\delta(\pi(\al),\pi(\be))}}{2\pi^{\delta(\al,\be)}}\sum_{(q_2,6)=1}q_2^{1+\delta(\pi(\al),\pi(\be))}\Psi_2\left(\frac{q_2}{Q_2}\right)\sum_{d_2r_2=q_2}\frac{\mu(d_2)}{d_2^{1-\alpha_1+\beta_1}}
				\cdot 	\frac{\phi(r_2)}{r_2} \widetilde{\Psi}_1\left(2+\delta(\pi(\al),\pi(\be))\right)\\
				& \hskip0.3in\cdot H\left(0,\al,\be\right)\mathcal Z\left(\frac 12;\pi(\al),\pi(\be)\right)G\left(\frac 12;\pi(\al),\pi(\be)\right)\frac{\mathcal R\left(\frac{1}{2}-\alpha_1,\frac 12+\beta_1;1-\alpha_1+\beta_1;r_2,q_2\right)}{\zeta(2-\alpha_1+\beta_1)}.\nonumber
			\end{align}
			It remains to match the quantity above with our desired object in \eqref{MT}. From this point on, the argument differs from that of \cite[Section~10]{CIS}, as we have an additional sum over $q_2$ and a more complicated summation structure. We therefore proceed by simplifying the expression step by step.\par 
			In \eqref{MT}, using that $\phi^b(q)=\frac{1}{2} \phi^*(q)+O(1)$, and the function $\phi^*$ is multiplicative with $\phi^*(p)=p-2$ and $\phi^*\left(p^k\right)=p^{k-2}(p-1)^2$ for $k \geq 2$, we may use a standard contour shift argument to evaluate the sum over $q_1$ above. Thus \eqref{MT} becomes
			\begin{align*}
				&\frac{Q_1^{2+\delta(\pi(\al),\pi(\be))}G\mathcal{AZ}\left(\frac 12;\pi(\al),\pi(\be)\right)}{2\pi^{\delta(\al,\be)}}\sum_{(q_2,6)=1}\phi^{*}(q_2)q_2^{\delta(\pi(\al),\pi(\be))}\Psi_2\left(\frac{q_2}{Q_2}\right)\prod_{p\mid q_2}\mathcal{B}_p^{-1}(\frac 12;\pi(\al),\pi(\be))\\
				&\cdot \widetilde{\Psi}_1\left(2+\delta(\pi(\al),\pi(\be))\right)H(0,\al,\be)\prod_{p\mid 6q_2}\left(1-\frac{1}{p}\right)\prod_{p\nmid 6q_2}\left(1+\mathcal{B}_p^{-1}(\frac 12;\pi(\al),\pi(\be))\left(\frac{1}{p}-\frac{1}{p^2}-\frac{1}{p^3}\right)\right)\\
				&\hskip 3in+O\left(\frac{Q^{2+\varepsilon}}{Q_1}+Q^{1+\varepsilon}\right).
			\end{align*}
			Ignoring the error terms and comparing with \eqref{MT2}, after cancelling the common factors, it suffices to show that 
			\begin{align}\label{MG3}
				\mathcal A\left(\frac 12;\pi(\al),\pi(\be)\right)&\sum_{(q_2,6)=1}\phi^{*}(q_2)q_2^{\delta(\pi(\al),\pi(\be))}\Psi_2\left(\frac{q_2}{Q_2}\right)\prod_{p\mid q_2}\mathcal{B}_p^{-1}\left(\frac 12;\pi(\al),\pi(\be)\right)\\
				&\hskip 0.5in\cdot \prod_{p\mid 6q_2}\left(1-\frac{1}{p}\right)\prod_{p\nmid 6q_2}\left(1+\mathcal{B}_p^{-1}\left(\frac 12;\pi(\al),\pi(\be)\right)\left(\frac{1}{p}-\frac{1}{p^2}-\frac{1}{p^3}\right)\right)\nonumber
			\end{align}
			contributes the same size of
			\begin{align}\label{MG4}
				\sum_{(q_2,6)=1}q_2^{1+\delta(\pi(\al),\pi(\be))}\Psi_2\left(\frac{q_2}{Q_2}\right)\sum_{d_2r_2=q_2}\frac{\mu(d_2)}{d_2^{1-\alpha_1+\beta_1}}
				\cdot 	\frac{\phi(r_2)}{r_2}\cdot \frac{\mathcal R\left(\frac 12-\alpha_1,\frac 12+\beta_1;1-\alpha_1+\beta_1;r_2,q_2\right)}{\zeta(2-\alpha_1+\beta_1)}.
			\end{align}
			We deal with \eqref{MG3} first. Again we use a standard contour shift argument to evaluate the sum over $q_2$, getting that \eqref{MG3} contributes
			
			\begin{align}\label{MG5}
				&	\mathcal A\left(\frac 12;\pi(\al),\pi(\be)\right)Q_2^{2+\delta(\pi(\al),\pi(\be))}\prod_{p\mid 6}\left(1+2\mathcal{B}_p^{-1}(\frac 12;\pi(\al),\pi(\be))\left(\frac{1}{p}-\frac{1}{p^2}-\frac{1}{p^3}\right)\right)^{-1}\\
				&\hskip 0.2in\cdot \widetilde{\Psi}_2\left(2+\delta(\pi(\al),\pi(\be))\right)\prod_{p}\left(\left(1-\frac 1p\right)^2\left(1+2\mathcal{B}_p^{-1}(\frac 12;\pi(\al),\pi(\be))\left(\frac{1}{p}-\frac{1}{p^2}-\frac{1}{p^3}\right)\right)\right)+O\left(Q_2^{1+\varepsilon}\right). \nonumber
			\end{align}
			Next we deal with \eqref{MG4}. We apply \eqref{F1} and \eqref{F2}, and once again evaluate the $q_2$-sum via a standard contour shift. After a straightforward but somewhat lengthy computation, we obtain \eqref{MG4} contributes
			\begin{align}\label{MG6}
				Q_2^{2+\delta(\pi(\al),\pi(\be))}\widetilde{\Psi}_2\left(2+\delta(\pi(\al),\pi(\be))\right)\mathcal{Z}^{-1}\left(\frac12;\pi(\al),\pi(\be)\right)\zeta(1-\alpha_1+\beta_1) \prod_{p\nmid 6}h(p)\prod_{p\mid 6}g(p)+O\left(Q_2^{1+\varepsilon}\right),
			\end{align}
			where 
			\begin{align*}
				h(p)=&	\left(1+\frac{1}{(p-1)p^{\alpha_1-\beta_1}}\right)\left(1-\frac{1}{p}\right)^3\Biggl[\left(1-\frac{p^{-1+\alpha_1-\beta_1}}{p^{1+\alpha_1-\beta_1}-p^{\alpha_1-\beta_1}+1}\right)\left(1-\frac{1}{p^{2-\alpha_1+\beta_1}}\right)\\ \nonumber
				&+\sumtwo_{\substack{a, b \geq 0 \\ \max (a, b) \geq 1\\a\neq b}} \frac{\sigma\left(p^a ; \boldsymbol{\alpha}\right) \sigma\left(p^b ;-\boldsymbol{\beta}\right)}{p^{a\left(1-\alpha_1\right)} p^{b\left(1+\beta_1\right)}} p^{(1-\alpha_1+\beta_1)\min (a, b)}\left(1-\frac{p}{p^{1+\alpha_1-\beta_1}-p^{\alpha_1-\beta_1}+1}\right)\\ \nonumber
				&+\sum_{k=1}^{\infty} \frac{\sigma\left(p^k ; \boldsymbol{\alpha}\right) \sigma\left(p^k ;-\boldsymbol{\beta}\right)}{p^{k}}\left(1-\frac{p^{\alpha_1-\beta_1}}{p^{1+\alpha_1-\beta_1}-p^{\alpha_1-\beta_1}+1}\right)  	\\ \nonumber
				&+\frac{1}{p}\left(1-\frac{1}{p^{1-\alpha_1+\beta_1}}+\frac{-p^{\alpha_1-\beta_1}+p^{-1+\alpha_1-\beta_1}+p^{-1+2(\alpha_1-\beta_1)}-1}{p^{1+\alpha_1-\beta_1}-p^{\alpha_1-\beta_1}+1}\right)\\
				&+\left(1-\frac{1}{p^{1-\alpha_1+\beta_1}}\right)\left(1-\frac 1p\right)\frac{p^{-1+\alpha_1-\beta_1}}{p^{1+\alpha_1-\beta_1}-p^{\alpha_1-\beta_1}+1}\Biggl]
			\end{align*}
			and 
			\begin{align*}
				g(p)=&	\left(1-\frac{1}{p}\right)\Biggl( 1-\frac{p^{\alpha_1-\beta_1}}{p^{1+\alpha_1-\beta_1}-p^{\alpha_1-\beta_1}+1}\\
				&+ \sumtwo_{\substack{a, b \geq 0 \\ \max (a, b) \geq 1\\a\neq b}} \frac{\sigma\left(p^a ; \boldsymbol{\alpha}\right) \sigma\left(p^b ;-\boldsymbol{\beta}\right)}{p^{a\left(1-\alpha_1\right)} p^{b\left(1+\beta_1\right)}} p^{(1-\alpha_1+\beta_1) \min (a, b)}\left(1-\frac{p}{p^{1+\alpha_1-\beta_1}-p^{\alpha_1-\beta_1}+1}\right)\\
				&\hskip 1in+\sum_{k=1}^{\infty} \frac{\sigma\left(p^k ; \boldsymbol{\alpha}\right) \sigma\left(p^k ;-\boldsymbol{\beta}\right)}{p^{k}}\left(1-\frac{p^{\alpha_1-\beta_1}}{p^{1+\alpha_1-\beta_1}-p^{\alpha_1-\beta_1}+1}\right).\Biggl)
			\end{align*}
			Comparing \eqref{MG5} and \eqref{MG6} and ignoring the error terms, to prove that they are equal it suffices to verify that the corresponding Euler products coincide, that is, that the factors over $\prod_p$ and the factors at primes dividing $6$ agree. 
			For more details, we check whether
			\begin{align}\label{MG7}
				\left(1+2\mathcal{B}_p^{-1}(\frac 12;\pi(\al),\pi(\be))\left(\frac{1}{p}-\frac{1}{p^2}-\frac{1}{p^3}\right)\right)^{-1}\stackrel{?}{=}g(p)h^{-1}(p)
			\end{align}
			and whether
			\begin{align}\label{MG8}
				\left(1-\frac 1p\right)^2\left(	\mathcal{B}_p(\frac 12;\pi(\al),\pi(\be))+2\left(\frac{1}{p}-\frac{1}{p^2}-\frac{1}{p^3}\right)\right)\stackrel{?}{=}\left(1-\frac{1}{p^{1-\alpha_1+\beta_1}}\right)^{-1}h(p).
			\end{align}
			With a little calculation, both \eqref{MG7} and \eqref{MG8} are equivalent to 
			\begin{align*}
				\mathcal{B}_p(\frac 12;\pi(\al),\pi(\be))&=\left(1-\frac 1p\right)\cdot\frac{1-p^{\beta_1-\alpha_1}}{1-p^{-1+\alpha_1-\beta_1}}\sumtwo_{\substack{a, b \geq 0 }} \frac{\sigma\left(p^a ; \boldsymbol{\alpha}\right) \sigma\left(p^b ;-\boldsymbol{\beta}\right)}{p^{a\left(1-\alpha_1\right)} p^{b\left(1+\beta_1\right)}} p^{(1-\alpha_1+\beta_1)\min (a, b)}\\
				&\hskip 2in+p^{\beta_1-\alpha_1}\sum_{k=0}^{\infty} \frac{\sigma\left(p^k ; \boldsymbol{\alpha}\right) \sigma\left(p^k ;-\boldsymbol{\beta}\right)}{p^{k}}.
			\end{align*}
			And this can be verified  upon using the following Parseval identities

			\begin{align*}
				\sum_{a, b \geq 0} & \frac{\sigma\left(p^a ; \boldsymbol{\alpha}\right) \sigma\left(p^b ;-\boldsymbol{\beta}\right)}{p^{a\left(1-\alpha_1\right)} p^{b\left(1+\beta_1\right)}} p^{\left(1-\alpha_1+\beta_1\right) \min (a, b)} \\
				= & \int_0^1\left(\sum_{a=0}^{\infty} \frac{\sigma\left(p^a ; \boldsymbol{\alpha}\right) e(a \theta)}{p^{a\left(1-\alpha_1-\beta_1\right) / 2}}\right)\left(\sum_{b=0}^{\infty} \frac{\sigma\left(p^b ;-\boldsymbol{\beta}\right) e(-b \theta)}{p^{b\left(1+\alpha_1+\beta_1\right) / 2}}\right) \\
				& \times\left(1+\sum_{k=1}^{\infty} \frac{e(k \theta)}{p^{k\left(1-\alpha_1+\beta_1\right) / 2}}+\sum_{\ell=1}^{\infty} \frac{e(-\ell \theta)}{p^{\ell\left(1-\alpha_1+\beta_1\right) / 2}}\right) d \theta,
			\end{align*}

			and
			
			\begin{align*}
				\sum_{k=0}^{\infty} \frac{\sigma\left(p^k ; \boldsymbol{\alpha}\right) \sigma\left(p^k ;-\boldsymbol{\beta}\right)}{p^k}=\int_0^1 \left(1-\frac{e(\theta)}{p^{\frac 12+\alpha_1}}\right)^{-1}\left(1-\frac{e(-\theta)}{p^{\frac 12-\beta_1}}\right)^{-1}
				\sum_{a=0}^{\infty}\frac{\lambda_f(p^a)e(a\theta)}{p^{(\frac 12+\alpha_2)a}}
				\sum_{b=0}^{\infty}\frac{\lambda_f(p^b)e(-b\theta)}{p^{(\frac 12-\beta_2)b}}d \theta.
			\end{align*}

		\end{proof}

		\subsection{Bounding the error terms} \label{ssec:errorEg}
		Now we consider the term $\mathcal {EBG}^\pm (M, N)$ defined in \eqref{def:EBG}. Below we will show that this contribution is negligible. The proof follows \cite[Section 6.3]{CLMR2}.
		
		\begin{lem} \label{lem:calcEBG} Let $\varepsilon, \delta_0 \in (0, 1/8)$ be fixed and let $D_0 \geq 1/2$. 
			We have, whenever $\max\{M, N\} \leq Q^{2-\delta_0}$, 
			$$ \mathcal {EBG}^\pm(M, N) \ll Q^{2 - \delta_0 + \varepsilon} Q_2^2D_0    + Q^{\frac 32+\varepsilon}Q_2^2D_0+Q^{\frac 74 + \varepsilon}Q_2^{\frac 32} D_0^{\frac 32}. $$
		\end{lem}
		
		\begin{proof}
			
			As in~\cite[Beginning of Section 8]{CIS}), we truncate the sum over $a_1$ at $a_1\leq 2Q_1$. Utilizing Remark~\ref{rem:gsize} we also make the truncations $b_1, g \leq Q^{2}$. 
			Using then the Mellin transform in Lemma \ref{lem:MellinXY} and the Mellin transform of $V$ in \eqref{def:MellinV}, we have that
			
			\begin{align*}
				&	\mathcal{EBG}^\pm(M, N) = \frac{ (Q_1Q_2)^{1 + \delta(\al, \be)}}{2} \sum_{(q_2,6)=1}\frac{1}{q_2}\Psi_2\left(\frac{q_2}{Q_2}\right)\sum_{d_2r_2=q_2}\mu(d_2)\phi(r_2)\sumthree_{\substack{a_1 \leq 2Q_1 \; b_1 \leq Q^2\\ h \ll  Q^{1-\delta_0}D_0 d_2}} \sum_{ \substack{\psi \Mod{a_1b_1r_2h} \\ \psi \neq \psi_0} } \\
				&\hskip1.5in\cdot  \sum_{\substack{g \leq Q^2 \\ b_1| 6gq_2, (a_1, 6gq_2) = 1}}\sum_{ \substack{d_1 \leq D_0 \\ (d_1, 6gq_2) = 1 }}   \frac{\mu(d_1) \mu(a_1) \mu(b_1) }{a_1d_1g \phi(a_1b_1r_2h)} \frac{1}{(2\pi i)^4} \int_{(\varepsilon)} \int_{(\varepsilon)}\int_{\left( \frac12 + \varepsilon \right)} \int_{\left( \frac12 + \varepsilon \right)} \\
				&\hskip 1.5in\cdot \widetilde{\mathcal W}_{2,Q_2/q_2}^\pm \left( s_1, s_2; \frac{(Q_1Q_2)^{1/2}d_1d_2}{gh} \right)\widetilde{V}(z_1) \widetilde{V}(z_2)\left( \frac{(Q_1Q_2)^{3/2}}{g}\right)^{s_1 + s_2}\frac{M^{z_1}N^{z_2}}{g^{z_1 + z_2}}\\
				& \hskip1.5in\cdot  \sumtwo_{\substack{\m, \n\\ \m\neq \n , (\m, \n) = 1 \\ (\m \n, d_1d_2) = 1}} \frac{\sigma(g\m; \al) \sigma(g\n; -\be) }{\m^{\frac 12 + s_1 + z_1} \n^{\frac 12 + s_2 + z_2}}   \psi(\m) \overline{\psi}(\mp \n) \>ds_1  \>ds_2 \>dz_1\>dz_2 +  O \left( Q^{\frac 32 + 3\varepsilon} \right).
			\end{align*}

			Next we express the sums over $\m, \n$ in terms of product of $L$-functions. Since $\psi$ is not a trivial character, the corresponding $L$-functions have no poles. As in~\cite[(53)--(54)]{CIS} we can move the line of integration over $s_i$ to $\R(s_i) = \varepsilon$. Then we change variables, letting $w_i = s_i + z_i$, and we see that the contribution to $\mathcal {EBG}^\pm(M, N)$ from the main term above is bounded by 
			\begin{align}\label{eqn:errorbeforedyadic}
				&\ll Q^{1 + \varepsilon}\sum_{(q_2,6)=1}\frac{1}{q_2}\Psi_2\left(\frac{q_2}{Q_2}\right)\sum_{d_2r_2=q_2}\phi(r_2) \sumthree_{\substack{a_1 \leq 2Q_1, \; b_1 \leq Q^2\\ h \ll Q^{1-\delta_0}D_0 d_2}} \sum_{ \substack{\psi \Mod{a_1b_1r_2h} \\ \psi \neq \psi_0} }\sum_{\substack{g \leq Q^2 \\ b_1| 6gq_2, (a_1, 6gq_2) = 1}}\sum_{ \substack{d_1 \leq D_0 \\ (d_1, 6gq_2) = 1 }} \\
				&   \cdot \frac{1}{a_1d_1g \phi(a_1b_1r_2h)} \int_{(\varepsilon)}\int_{(\varepsilon)} \int_{(2\varepsilon)}\int_{(2\varepsilon)}  \left| \widetilde{\mathcal W}_{2,Q_2/q_2}^\pm \left( s_1, s_2; \frac{(Q_1Q_2)^{1/2}d_1d_2}{gh} \right) \right| \left|\widetilde{V}(w_1 - s_1) \right| \left|\widetilde{V}(w_2 - s_2) \right|\nonumber \\
				& \cdot \left( 1 +  \left| L \left( \frac 12 +  w_1 + \alpha_1, \psi\right)L \left( \frac 12 + w_2 - \beta_1, \overline{\psi}\right)L\left(\frac 12+w_1+\alpha_2,f\otimes\psi \right)L\left(\frac 12+w_2-\beta_2,f\otimes\overline{\psi }\right)\right| \right)\nonumber \\
				&\hskip5.1in \>dw_1 \>dw_2 \> ds_1 \> ds_2 .\nonumber
			\end{align}
			
			We consider the sums over $g$ and $d$ and apply the bound for  $\widetilde{\mathcal W}_{2,Q_2/q_2}^{\pm}$ in Lemma \ref{lem:MellinXY} to derive that for any $h \geq 1,$ and $s_1, s_2$ with $\R(s_i) = \varepsilon$, and any fixed natural number $k$,
			\begin{align}
				\label{eq:W2etcbound}
				\begin{aligned}
					&\sum_{q_2}\frac{1}{q_2}\Psi_2\left(\frac{q_2}{Q_2}\right)\sum_{d_2r_2=q_2}\phi(r_2) \sum_{ b_1 \leq Q^2} \sum_{ \substack{\psi \Mod{a_1b_1r_2h} \\ \psi \neq \psi_0} }
					\sum_{\substack{g \leq Q^2 \\ b_1|6gq_2} } \frac 1g\sum_{d_1 \leq D_0} \frac{1}{d_1} \left| \widetilde{\mathcal W}_{2,Q_2/q_2}^\pm \left(s_1, s_2; \frac{(Q_1Q_2)^{1/2} d_1d_2}{gh} \right)\right| \\
					&\ll Q^\varepsilon\sum_{d_2\leq 2Q_2}\frac{1}{d_2}\sum_{r_2\sim Q_2/d_2}\sum_{x\mid 6d_2r_2}\sum_{b_1\leq Q^2/x}\sum_{ \substack{\psi \Mod{a_1b_1xr_2h} \\ \psi \neq \psi_0} }\sum_{\substack{g \leq Q^2 \\ b_1|g} } \frac 1g \sum_{d_1 \leq D_0} \frac{1}{d_1}  \frac{\left( 1 + \frac{Q^{1/2}D_0d_2}{gh} \right)^{k-1}}{\max \{|s_1|, |s_2| \}^{k} |s_1 + s_2|^3}  \\
					&\ll  \frac{Q^{\varepsilon}}{\max \{|s_1| + 1, |s_2| + 1\}^{k} (|s_1 + s_2| + 1)^3} \sum_{d_2\leq 2Q_2}\frac{1}{d_2}\sum_{y\leq 2Q_2}\sum_{x\mid 6d_2}\sum_{r_2\sim \frac{Q_2}{d_2y}} \sum_{b_1\leq\frac{Q^2}{xy}}\\
					&\hskip3.5in\cdot\sum_{ \substack{\psi \Mod{a_1b_1xy^2r_2h} \\ \psi \neq \psi_0} }\frac{1}{b_1}\left( 1 + \frac{Q^{1/2}D_0d_2}{b_1 h} \right)^{k-1}\\
					&\ll \frac{Q^{\varepsilon}}{\max \{|s_1| + 1, |s_2| + 1\}^{k} (|s_1 + s_2| + 1)^3}\sum_{l\leq 2Q_2}\frac{1}{l}\sum_{x\mid 6}\sum_{d_2\leq \frac{2Q_2}{l}}\frac{1}{d_2}\sum_{y\leq 2Q_2}\sum_{r_2\sim\frac{Q_2}{d_2ly}}\sum_{b_1\leq \frac{Q^2}{lxy}}\\
					&\hskip3.5in\cdot\sum_{ \substack{\psi \Mod{a_1b_1lxy^2r_2h} \\ \psi \neq \psi_0} }\frac{1}{b_1}\left( 1 + \frac{Q^{1/2}D_0d_2l}{b_1 h} \right)^{k-1},
				\end{aligned}
			\end{align}
			where, in the first inequality, we write $(b_1,6q_2)=x$; in the second, we write $(x,r_2)=y$, and in the last, we write $(x,d_2)=l$.
			
			First we note that for any $\ell \geq 0$, $\widetilde{V}(z) \ll \frac{1}{(1 + |z|)^\ell}$, so the contribution of $|w_i - s_i| \gg Q^{\varepsilon}$ to \eqref{eqn:errorbeforedyadic} is $\ll_A Q^{-A}$.
			
			Let us now return to~\eqref{eqn:errorbeforedyadic}. We divide the variables $d_2,a_1, b_1,l,y,r_2, h$ into dyadic blocks $d_2\sim D_2,a_1 \sim A, l\sim L,y\sim Y,b_1 \sim B, r_2\sim R,h \sim H$ (with $A \ll Q_1, D_2LYR\ll Q_2, Y\ll Q_2,B \ll \frac{Q^2}{LY}$ and $H \ll Q^{1-\delta_0} D_0D_2$), let $\ell = a_1b_1lxy^2r_2h$, and also subdivide $w_1$ and $w_2$ into blocks $w_1\sim W_1$ and $w_2\sim W_2$. Then we apply~\eqref{eq:W2etcbound} to  \eqref{eqn:errorbeforedyadic}. We derive that for any $k$, the contribution of each block is 
			\begin{align}
				\label{eq:dyadicblockcontr}
				\begin{aligned}
					&\ll \frac{Q^{1+ \varepsilon}}{A^2B^2L^2Y^2RH } \left( 1 + \frac{Q^{1/2}D_0D_2L}{BH}\right)^{k - 1} \frac{1}{\max(W_1,W_2)^k}  
					\sum_{ABLY^2RH \leq \ell < 128ABLY^2RH}  \\ 
					& \hskip 0.2in \cdot \sum_{ \substack{\psi \Mod{\ell} \\ \psi \neq \psi_0} }  \int_{W_1}^{2W_1} \int_{W_2}^{2W_2}\Bigg(  1 +   \left| L \left( \frac 12 + 2\varepsilon + iw_1 + \alpha_1, \psi \right)\right|^6+\left| L \left( \frac 12 + 2\varepsilon + iw_1 + \alpha_2, f\otimes\psi \right)\right|^3 \\
					&\hskip1in+ \left| L \left( \frac 12 + 2\varepsilon + iw_2- \beta_1, \overline{\psi } \right)\right|^6+\left| L \left( \frac 12 + 2\varepsilon + iw_2 - \beta_2, f\otimes\overline{\psi }\right)\right|^3 \Bigg)\>dw_1\>dw_2.
				\end{aligned}
			\end{align}
			As in \cite[Section~8]{CIS}, we apply the large sieve inequality. We use H\"older's inequality to bound the cubic moment by the product of the second and fourth moments, after which the large sieve becomes applicable. The precise bound required is a variant of~\cite[Proposition~3.2]{CLMR}, which follows completely similarly.  Consequently~\eqref{eq:dyadicblockcontr} is 
			\begin{align}
				\label{eq:dyadicblockcontr2}
				\begin{aligned}
					\ll  \frac{Q^{1+ \varepsilon}}{A^2B^2L^2Y^2RH }& \frac{ 1}{\max(W_1,W_2)^k} \left( 1 + \frac{Q^{1/2}D_0D_2L}{BH}\right)^{k - 1}\\
					& \cdot\left( W_1W_2(ABLY^2RH)^2 +  \min(W_1,W_2)(\max(W_1,W_2)ABLY^2RH)^{\frac{3}{2}}\right).
				\end{aligned}
			\end{align}
			
			When $\max(W_1,W_2) \leq  1 + \frac{Q^{1/2}D_0D_2L}{BH}$, we choose $k = 1$, and otherwise, we choose $k = 4$. In any case~\eqref{eq:dyadicblockcontr2} is
			\[
			\ll Q^{1+\varepsilon}Y^2RH+Q^{\frac{3}{2}+\varepsilon}Y^2RD_0D_2L+Q^{1+\varepsilon}Y\sqrt{RH}+Q^{\frac{7}{4}+\varepsilon}Y\sqrt{R}(D_0D_2L)^{\frac{3}{2}}.
			\]
			
			Recall that $H \ll {Q^{1 - \delta_0} D_0D_2}$ and $D_2LYR\ll Q_2$. Thus after dyadic summation $A, B, H, T$, we derive that the contribution to \eqref{eqn:errorbeforedyadic} from this case  is bounded by
			\begin{align}\label{fitsterrorsketch}
				\ll Q^{2 - \delta_0 + \varepsilon} Q_2^2D_0    + Q^{\frac 32+\varepsilon}Q_2^2D_0+Q^{\frac 74 + \varepsilon}Q_2^{\frac 32} D_0^{\frac 32},
			\end{align}
			so the claim follows by adjusting $\varepsilon$.
		\end{proof}
		
		\section{Unbalanced sums}\label{sec:unbalancedprelim}
		We now prepare to prove Proposition \ref{prop:unbalanced}. Recall that we are interested in bounding 
		$$\sumtwo_{\substack{q_1,q_2\\ (q_1q_2,6)=1\\(q_1,q_2)=1}} \Psi_1\bfrac{q_1}{Q_1}\Psi_2\bfrac{q_2}{Q_2}\sumb_{\chi \bmod q_1q_2}   S(M, N),
		$$when  $Q^{2-\delta_0}\le M \le Q^{3 + \varepsilon}$.  
		\subsection{Notational simplification}
		As in \cite[Section 7]{CLMR2}, we shall first follow their argument. To simplify notation, we set $\al = \be = (0, 0, 0)$.  The case where the shifts are nonzero may be proven similarly with no conceptual change. To be precise, we start by writing
		\begin{align*}
			S(M, N) =  \sumtwo_{m, n} \frac{1*\lambda_f(m) 1*\lambda_f(n) \chi(m) \cb(n)}{\sqrt{mn}} W_{0, 0}\left(m, n; q\right) V\bfrac{m}{M} V\bfrac{n}{N}.
		\end{align*}
		Recalling from \eqref{eqn:Walbe} that
		$$
		W_{0, 0}(m, n; q) = \frac{1}{2\pi i} \int_{\left(\frac{1}{\log Q}\right)} G \left(\frac 12 + s; 0, 0\right) H(s; 0, 0) \left( \frac{4\pi^3mn}{q^3} \right)^{-s} \frac{ds}{s},
		$$and the rapid decay of $G$ (which follows from the definition~\eqref{eq:Gdef} and Stirling's formula), we see that it suffices to bound
		$$\sumtwo_{m, n} \frac{1*\lambda_f(m) 1*\lambda_f(n)\chi(m) \cb(n)}{(mn)^{1/2+s}} V\bfrac{m}{M} V\bfrac{n}{N}
		$$for $|s| \ll q^\varepsilon$ and $\tRe s = \frac{1}{\log Q}$.  We will further allow ourselves to rewrite the above as 
		$$\sumtwo_{m, n} \frac{1*\lambda_f(m) 1*\lambda_f(n) \chi(m) \cb(n)}{(mn)^{1/2}} V\bfrac{m}{M} V\bfrac{n}{N}
		$$for slightly different functions $V$, where now
		\begin{equation}\label{eqn:Vnew}
			V^{(k)}{(x)} \ll q^\varepsilon,
		\end{equation}for all integer $k\ge 0$. We shall assume this throughout the rest of the paper. We shall remark that since $M\gg Q^{2-\delta_0}$ and $N\ll Q^{1 +\delta_0+\varepsilon}$, we must have $m\neq n$.
		
		We now open up $1*\lambda_f(m)$ and $1*\lambda_f(n)$ and write $m = eg$, and apply a smooth partition of unity to $e,g$ and $n$, to see that our sum is now at most $\log^3 Q$ many sums of the form
		\begin{align*}
			\sumthree_{e, g, n} \frac{\lambda_f(e)1*\lambda_f(n) \chi(eg) \cb(n)}{(egn)^{1/2}} V\bfrac{e}{E}  V\bfrac{g}{G} V\bfrac{eg}{M} V\bfrac{n}{N},
		\end{align*}where $EG \asymp M$.\par 
		We remark that unlike the case considered in \cite{CLMR2}, the above expression is not symmetric in $e$ and $g$, so we cannot assume $E \geq G$.
		Moreover, in \cite{CLMR2} the application of the Kuznetsov trace formula yields a sharp bound
		when the parameter $G$ is a relatively small power of $Q$. In our situation, however, $G$ can be very large, in which case the bound obtained from the Kuznetsov trace formula is no longer sufficiently sharp. Therefore we shall apply a new $p$-adic stationary phase method to overcome this difficulty, which we will discuss detailly in Section \ref{Squarerootcancellation}.

		Back to our situation, we may again neglect the factor $V\bfrac{eg}{M}$ in the same manner in which we removed $ W_{0,0}\left(m, n; q\right)$, and absorb a factor of $\frac{\sqrt{EGN}}{\sqrt{egn}}$ into the smooth functions $V$. Thus we will examine
		\begin{align*}
			S(E, G, N) := \frac{1}{\sqrt{EGN}}\sumthree_{e, g, n} \lambda_f(e)1*\lambda_f(n) \chi(eg) \cb(n) V\bfrac{e}{E}  V\bfrac{g}{G} V\bfrac{n}{N}
		\end{align*}
		for $EG \asymp M$, with a different $V(\cdot)$ satisfying \eqref{eqn:Vnew}.
		
		\subsection{Initial manipulations}
		By Lemma \ref{lem:orthogonal}, we have that
		\begin{align*}
			&\sumtwo_{\substack{q_1,q_2\\ (q_1q_2,6)=1\\(q_1,q_2)=1}} \Psi_1\bfrac{q_1}{Q_1}\Psi_2\bfrac{q_2}{Q_2}\sumb_{\chi \bmod q_1q_2}   S(M, N) \\
			&= \frac 12\sum_{\pm } \sumfour_{\substack{d_1,d_2,r_1,r_2\\(d_1d_2r_1r_2,6)=1\\(d_1r_1,d_2r_2)=1}} \Psi_1\bfrac{d_1r_1}{Q_1}\Psi_2\left(\frac{d_2r_2}{Q_2} \right)\mu(d_1)\mu(d_2) \phi(r_1)\phi(r_2)\\
			&\cdot  \frac{1}{\sqrt{EGN}} \sumthree_{\substack{e, g, n \\ e \equiv \pm \overline{g} n \bmod r_1r_2\\ (egn, d_1d_2r_1r_2) = 1}} \lambda_f(e)1*\lambda_f(n) V\bfrac{e}{E}  V\bfrac{g}{G} V\bfrac{n}{N}.
		\end{align*}
		
		The conditions $e \equiv \overline{g} n_1n_2 \bmod r_1r_2$ and $e \equiv -\overline{g} n_1n_2 \bmod r_1r_2$ are dealt with by similar methods, so we examine the case $e \equiv \overline{g} n_1n_2 \bmod r_1r_2$ only.  Thus, we will focus our attention on
		\begin{align}\label{Starterror}
			\calS &:= \sumfour_{\substack{d_1,d_2,r_1,r_2\\(d_1d_2r_1r_2,6)=1\\(d_1r_1,d_2r_2)=1}} \Psi_1\bfrac{d_1r_1}{Q_1}\Psi_2\left(\frac{d_2r_2}{Q_2} \right)\mu(d_1)\mu(d_2) \phi(r_1)\phi(r_2)\\
			&\cdot   \frac{1}{\sqrt{EGN}} \sumthree_{\substack{e, g, n \\ e \equiv  \overline{g} n \bmod r_1r_2\\ (egn, d_1d_2r_1r_2) = 1}} \lambda_f(e)1*\lambda_f(n) V\bfrac{e}{E}  V\bfrac{g}{G} V\bfrac{n}{N}.\nonumber
		\end{align}
		In \cite[Section~7]{CLMR2}, Poisson summation is applied separately to the variables $e $ and $f$. In our approach, we replace these two applications of Poisson summation by the Voronoi summation formula. After the Voronoi, we can get the standard Kloosterman sum, and we are ready to apply the Kuznetsov trace formula to get the bound. \par 
		We remark that the coprimality condition $(n,r_1)=1$ will eventually be removed in the application of the Kuznetsov trace formula. 
		If this condition were removed only at that stage, the resulting Kloosterman sums would produce an additional large contribution due to their more complicated arithmetic structure. 
		This phenomenon can already be observed in \cite[Equation~(73)]{CLMR2}, where an extra factor $\Omega$ appears.
		
		Although the quantity $\Omega$ turns out to be harmless for the final bounds in the setting of Chandee, Li, Matom{\"a}ki, and Radziwi{\l}\l{} in \cite{CLMR2}, following the same procedure in our situation would lead to a contribution that cannot be absorbed, and hence would cause a loss of control in the final estimates. 
		For this reason, it is crucial for our argument to remove the condition $(n,r_1)=1$ \emph{before} the appearance of Kloosterman sums, thereby avoiding this technical obstruction.
		
		To this end, we record the following slightly more complicated lemma.
		\par 
		\begin{lem}\label{lem:Voro}
			Let $r_1, r_2, g,\lambda_1 \in \mathbb{N}$ and $(gn,r_1r_2)=1$. Let $V(\cdot)\in \mathcal{C}_{\mathcal{C}}(0, \infty)$. Then
			\begin{align}\label{Voro=1}
				&	\delta_{(n,r_1)=1}	\sum_{\substack{e\equiv\overline{g} n \bmod r_1r_2 \\ (e, \lambda_1) = 1}} \lambda_f(e)V\bfrac{e}{E}= \\\nonumber
				&\frac{E}{r_1r_2} \sum_{\substack{\gamma\mid n\\\gamma\mid r_1\\ (\gamma,r_2)=1}}\mu(\gamma)
				\sum_{\substack{\nu_1\mid \lambda_1\\(\nu_1,r_1r_2)=1}}\frac{\mu(\nu_1)}{\nu_1}
				\sum_{\substack{\kappa\mid \frac{n}{\gamma}\\\kappa\mid \frac{r_1}{\gamma}\\ (\kappa,r_2)=1}}	\sum_{\substack{\nu_2\kappa\mid \nu_1\gamma \\(\nu_2,r_2)=1\\ (\nu_2,\frac{r_1}{\gamma\kappa})=1}}\frac{\mu(\nu_2\kappa)}{\nu_2}\lambda_f\left(\frac{\nu_1\gamma}{\nu_2\kappa}\right)\\ \nonumber
				&\cdot \sum_{k\mid \frac{r_1}{\gamma\kappa}r_2} \frac{1}{k}\sum_{e=1}^\infty \lambda_f(e)S(e,\overline{\nu_1\nu_2g}\frac{n}{\gamma\kappa};k)\Phi_V\left(\frac{eE}{\nu_1\nu_2\gamma\kappa k^2}\right),\nonumber
			\end{align}
			where $\Phi_V(\cdot)$ is given in \eqref{Besseltrans}.
		\end{lem}
		\begin{proof}
			We first remove $(e,\lambda_1)=1$ by the Möbius inversion, introducing $\mu(\nu_1)$. Since $(\overline{g}n,r_1r_2)=1$, we must have $(e,r_1r_2)=1$, and from $\nu_1\mid e$, we have $(\nu_1,r_1r_2)=1$. Hence we obtain
			\begin{align*}
				\sum_{\substack{e\equiv \overline{g}n \bmod r_1r_2 \\ (e, \lambda_1) = 1}} \lambda_f(e)V\bfrac{e}{E} =\sum_{\substack{\nu_1\mid \lambda_1\\(\nu_1,r_1r_2)=1}}\mu(\nu_1)\sum_{e\equiv \overline{\nu_1g}n\bmod r_1r_2}\lambda_f(\nu_1e)V\bfrac{\nu_1e}{E}.
			\end{align*}
			Next we remove $(n,r_1)=1$, introducing $\mu(\gamma)$, getting
			$$
			\sum_{\substack{\gamma\mid n\\ \gamma\mid r_1\\ (\gamma,r_2)=1}}\mu(\gamma)
			\sum_{\substack{\nu_1\mid \lambda_1\\(\nu_1,r_1r_2)=1}}\mu(\nu_1) 
			\sum_{e\equiv \overline{\nu_1g}\gamma\frac{n}{\gamma}\bmod \gamma \frac{r_1}{\gamma}r_2}\lambda_f(\nu_1e)V\bfrac{\nu_1e}{E},
			$$
			hence we must have $\gamma\mid e$. We have  the $e$-sum equals
			$$
			\sum_{e\equiv \overline{\nu_1g}\frac{n}{\gamma}\bmod \frac{r_1}{\gamma}r_2}\lambda_f(\nu_1\gamma e)V\bfrac{\nu_1\gamma e}{E}.
			$$
			\\
			Next we use the Hecke relations 
			$$
			\lambda_f(m n)=\sum_{d \mid(m, n)} \mu(d) \lambda_f\left(\frac{m}{d}\right) \lambda_f\left(\frac{n}{d}\right) ,
			$$
			getting that the $e$-sum becomes
			\begin{align*}
				\sum_{\substack{\nu_2\mid \nu_1\gamma \\(\nu_2,r_2)=1}}\mu(\nu_2)\lambda_f\left(\frac{\nu_1\gamma}{\nu_2}\right)	
				\sum_{\nu_2e\equiv \overline{\nu_1g}\frac{n}{\gamma}\bmod \frac{r_1}{\gamma}r_2}\lambda_f(e)V\bfrac{\nu_1\gamma\nu_2 e}{E}.
			\end{align*}
			Next we want to let $(\nu_2,\frac{r_1}{\gamma})=1$, and we let $(\nu_2,\frac{r_1}{\gamma})=\kappa$, and we write $\nu_2\kappa$ for $\nu_2$, getting
			\begin{align*}
				\sum_{\substack{\kappa\mid \frac{r_1}{\gamma}\\ (\kappa,r_2)=1}}	\sum_{\substack{\nu_2\kappa\mid \nu_1\gamma \\(\nu_2,r_2)=1\\ (\nu_2,\frac{r_1}{\gamma\kappa})=1}}\mu(\nu_2\kappa)\lambda_f\left(\frac{\nu_1\gamma}{\nu_2\kappa}\right)	
				\sum_{\kappa\nu_2e\equiv \overline{\nu_1g}\frac{n}{\gamma}\bmod \kappa\frac{r_1}{\gamma\kappa}r_2}\lambda_f(e)V\bfrac{\nu_1\gamma\nu_2\kappa e}{E}.
			\end{align*}
			Hence we must have $\kappa \mid \overline{\nu_1g}\frac{n}{\gamma}$, and we must have $\kappa\mid \frac{n}{\gamma}$ since $(\nu_1g,r_1)=1$ and $\kappa\mid \frac{r_1}{\gamma}$. We get \eqref{Voro=1} equals
			\begin{align*}
				\sum_{\substack{\gamma\mid n\\\gamma\mid r_1\\ (\gamma,r_2)=1}}\mu(\gamma)
				\sum_{\substack{\nu_1\mid \lambda_1\\(\nu_1,r_1r_2)=1}}\mu(\nu_1)
				\sum_{\substack{\kappa\mid \frac{n}{\gamma}\\\kappa\mid \frac{r_1}{\gamma}\\ (\kappa,r_2)=1}}	\sum_{\substack{\nu_2\kappa\mid \nu_1\gamma \\(\nu_2,r_2)=1\\ (\nu_2,\frac{r_1}{\gamma\kappa})=1}}\mu(\nu_2\kappa)\lambda_f\left(\frac{\nu_1\gamma}{\nu_2\kappa}\right)	
				\sum_{e\equiv \overline{\nu_1\nu_2g}\frac{n}{\gamma\kappa}\bmod  \frac{r_1}{\gamma\kappa}r_2}\lambda_f(e)V\bfrac{\nu_1\nu_2\gamma\kappa e}{E}.
			\end{align*}
			\\
			Finnally, the condition $e \equiv  \overline{\nu_1\nu_2g} \frac{n}{\gamma\kappa}\bmod \frac{r_1}{\gamma\kappa}r_2$ can be enforced by inserting the double sum 
			$$\frac{1}{\frac{r_1}{\gamma\kappa}r_2}\sum_{k\mid \frac{r_1}{\gamma\kappa}r_2}\ \ \sumstar_{b_1\bmod k}\ex\left(\frac{b_1(\overline{\nu_1\nu_2g}\frac{n}{\gamma\kappa}-e)}{k}\right),$$
			and applying Lemma \ref{lem:voronoi}.
		\end{proof}
		By Lemma~\ref{lem:Voro}, taking $\lambda_1=d_1d_2$, we obtain
		(note that the condition $(e,r_1r_2)=1$ can be absorbed into the congruence condition since $(\overline{g}n,r_1r_2)=1$)
		\begin{align*}
			\calS = \frac{\sqrt{E}}{\sqrt{GN}} &\sumfour_{\substack{d_1,d_2,r_1,r_2\\(d_1d_2r_1r_2,6)=1\\(d_1r_1,d_2r_2)=1}}\sum_{\substack{\nu_1\mid d_1d_2\\ (\nu_1,r_1r_2)=1}}\frac{\mu(\nu_1)}{\nu_1}\sum_{\substack{ (\gamma,6d_1d_2\nu_1r_2)=1}}\frac{\mu(\gamma)}{\gamma}\sum_{(\kappa,6d_1d_2\nu_1r_2)=1}\sum_{\substack{\nu_2\kappa\mid \nu_1\gamma\\ (\nu_2,r_1r_2)=1}}\frac{\mu(\nu_2\kappa)}{\nu_2\kappa}\lambda_f\left(\frac{\nu_1\gamma}{\nu_2\kappa}\right) \\ &\cdot \Psi_1\bfrac{d_1\gamma\kappa r_1}{Q_1}\Psi_2\left(\frac{d_2r_2}{Q_2} \right)\mu(d_1)\mu(d_2)\frac{\phi(r_1\gamma\kappa)}{r_1}\frac{\phi(r_2)}{r_2}
			\sum_{k_1\mid r_1}\sum_{k_2\mid r_2}\frac{1}{k_1k_2}
			\sumthree_{\substack{e, g, n \\ (g, d_1d_2r_1r_2\gamma\kappa) = 1\\ (n,d_1d_2r_2)=1}}\\ &\cdot\lambda_f(e)1*\lambda_f(\gamma\kappa n)S(e,\overline{\nu_1\nu_2g}n;k_1k_2) \Phi_V\left(\frac{eE}{\nu_1\nu_2\gamma\kappa k_1^2k_2^2}\right)  V\bfrac{g}{G} V\bfrac{\gamma\kappa n}{N}
		\end{align*}
		since $(r_1,r_2)=1$. Hence we obtain
		\begin{align}\label{Start}
			\calS &= \frac{\sqrt{E}}{\sqrt{GN}} \sumsix_{\substack{d_1,d_2,r_1,r_2,k_1,k_2\\(d_1d_2r_1r_2k_1k_2,6)=1\\(d_1r_1k_1,d_2r_2k_2)=1}}\sum_{\substack{\nu_1\mid d_1d_2\\ (\nu_1,r_1r_2k_1k_2)=1}}\frac{\mu(\nu_1)}{\nu_1}\sum_{\substack{ (\gamma,6d_1d_2\nu_1r_2k_2)=1}}\frac{\mu(\gamma)}{\gamma}\sum_{(\kappa,6d_1d_2\nu_1r_2k_2)=1}\sum_{\substack{\nu_2\kappa\mid \nu_1\gamma\\ (\nu_2,r_1r_2k_1k_2)=1}} \\  \nonumber
			&\hskip 1in\cdot \frac{\mu(\nu_2\kappa)}{\nu_2\kappa}\lambda_f\left(\frac{\nu_1\gamma}{\nu_2\kappa}\right)\Psi_1\bfrac{d_1\gamma\kappa r_1k_1}{Q_1}\Psi_2\left(\frac{d_2r_2k_2}{Q_2} \right)\mu(d_1)\mu(d_2)\frac{\phi(r_1\gamma\kappa k_1)}{r_1k_1^2}\frac{\phi(r_2k_2)}{r_2k_2^2}\\ \nonumber
			&\hskip 0.6in\cdot\sumthree_{\substack{e, g, n \\ (g, d_1d_2r_1r_2k_1k_2\gamma\kappa) = 1\\ (n,d_1d_2r_2k_2)=1}}\lambda_f(e)1*\lambda_f(\gamma \kappa n)S(e,\overline{\nu_1\nu_2g}n;k_1k_2) \Phi_V\left(\frac{eE}{\nu_1\nu_2\gamma\kappa k_1^2k_2^2}\right)  V\bfrac{g}{G} V\bfrac{\gamma\kappa n}{N}.\nonumber
		\end{align}
		We now split $\mathcal{S}$ into $\mathcal{S}(d_1d_2r_1r_2\gamma\kappa\leq D)$, $\mathcal{S}(Q^{1-\delta}>d_1d_2r_1r_2\gamma\kappa> D)$ and $\mathcal{S}(Q\geq d_1d_2r_1r_2\gamma\kappa\geq Q^{1-\delta})$. When $d_1d_2r_1r_2\gamma\kappa\geq Q^{1-\delta}$, we have $k_1k_2\leq Q^\delta$. Hence we first bound trivially to get 
		\begin{align}\label{S1kuzDlargest}
			\mathcal{S}(Q\geq d_1d_2r_1r_2\gamma\kappa\geq Q^{1-\delta})\ll Q^{\frac 52\delta+\varepsilon}\frac{\sqrt{GN}}{\sqrt{E}}\ll Q^{\frac 52\delta+\varepsilon}\frac{\sqrt{EGN}}{E}\ll \frac{Q^{\frac32+\frac 52\delta+\varepsilon}}{E}.
		\end{align}
		
		Then we consider $\mathcal{S}(d_1d_2r_1r_2\gamma\kappa\leq D)$, and $\mathcal{S}(d_1d_2r_1r_2\gamma\kappa> D)$ will be considered in Section \ref{sectionlargesieve} by large sieve.\par 
		Next we aim to apply the Kuznetsov trace formula to get the bound; this forces that the modulus $k_1k_2$ has smooth part, and we would like to choose the largest $k_1$ to be our smooth modulus. Hence we shall deal with $\phi(r_1\gamma\kappa k_1)$ first. One can observe that if $a \mid bc$, then there exists a unique factorization $a = a_1 a_2$ such that $a_1 \mid b$, $a_2 \mid c$, and $(a_2,\, b/a_1)=1$. Using this fact, we have 
		\begin{align*}
			\frac{\phi(r_1\gamma\kappa k_1)}{r_1\gamma\kappa k_1}&=\sum_{a\mid r_1\gamma\kappa k_1}\frac{\mu(a)}{a}=\sum_{a_2\mid r_1\gamma\kappa}\sum_{\substack{a_1\mid k_1\\ (a_2,k_1/a_1)=1}}\frac{\mu(a_1a_2)}{a_1a_2}\\
			&=\sum_{a_2\mid r_1\gamma\kappa}\sum_{\substack{a_1\mid k_1\\ (a_2,k_1/a_1)=1\\ (a_2,a_1)=1}}\frac{\mu(a_1)\mu(a_2)}{a_1a_2}=\sum_{a_2\mid r_1\gamma\kappa}\sum_{\substack{a_1\mid k_1\\ (a_2,k_1)=1}}\frac{\mu(a_1)\mu(a_2)}{a_1a_2}.
		\end{align*}
		Hence we get 
		\begin{align}\label{kl2start}
			&\mathcal{S}(d_1d_2r_1r_2\gamma\kappa\leq D)= \frac{\sqrt{E}}{\sqrt{GN}} \sumsix_{\substack{d_1,d_2,r_1,r_2,k_1,k_2\\(d_1d_2r_1r_2k_1k_2,6)=1\\(d_1r_1k_1,d_2r_2k_2)=1\\ d_1d_2r_1r_2\gamma\kappa\leq D}}
			\sum_{(a_1,6d_2r_2k_2)=1}\frac{\mu(a_1)}{a_1^2}
			\sum_{\substack{\nu_1\mid d_1d_2\\ (\nu_1,r_1r_2a_1k_1k_2)=1}}\frac{\mu(\nu_1)}{\nu_1}\\\nonumber
			&\cdot\sum_{\substack{ (\gamma,6d_1d_2\nu_1r_2k_2)=1}}\mu(\gamma)\sum_{(\kappa,6d_1d_2\nu_1r_2k_2)=1}\sum_{\substack{\nu_2\kappa\mid \nu_1\gamma\\ (\nu_2,r_1r_2a_1k_1k_2)=1}} \sum_{\substack{a_2\mid r_1\gamma\kappa\\ (a_2,a_1k_1)=1}}\frac{\mu(a_2)}{a_2}\cdot \frac{\mu(\nu_2\kappa)}{\nu_2}\lambda_f\left(\frac{\nu_1\gamma}{\nu_2\kappa}\right)\\ \nonumber
			&\cdot \Psi_1\bfrac{d_1\gamma\kappa r_1a_1k_1}{Q_1}\Psi_2\left(\frac{d_2r_2k_2}{Q_2} \right)\mu(d_1)\mu(d_2)\frac{1}{k_1}\frac{\phi(r_2k_2)}{r_2k_2^2}\\ \nonumber
			&\cdot\sumthree_{\substack{e, g, n \\ (g, d_1d_2r_1r_2a_1k_1k_2\gamma\kappa) = 1\\ (n,d_1d_2r_2k_2)=1}}\lambda_f(e)1*\lambda_f(\gamma \kappa n)S(e,\overline{\nu_1\nu_2g}n;k_1a_1k_2) \Phi_V\left(\frac{eE}{\nu_1\nu_2\gamma\kappa a_1^2k_1^2k_2^2}\right)  V\bfrac{g}{G} V\bfrac{\gamma\kappa n}{N}.\nonumber
		\end{align}
		To apply bounds for sums of Kloosterman sums, we aim to choose the largest $k_1$ as our smooth modulus.  Hence we need to ensure that $k_1$ is coprime only to
		$a_1 \nu_1 g$, while no coprimality condition is imposed with respect to the remaining
		parameters.
		To this end, we remove the condition
		$
		(k_1,\, 6 d_2r_2k_2a_2)=1
		$
		by Möbius inversion, introducing $\mu(w)$. Hence we get
		\begin{align}\label{KL2}
			&\mathcal{S}(d_1d_2r_1r_2\gamma\kappa\leq D) = \frac{\sqrt{E}}{\sqrt{GN}} \sumsix_{\substack{d_1,d_2,r_1,r_2,k_1,k_2\\(d_1d_2r_1r_2k_2,6)=1\\(d_1r_1,d_2r_2k_2)=1\\d_1d_2r_1r_2\gamma\kappa\leq D}}
			\sum_{(a_1,6d_2r_2k_2)=1}\frac{\mu(a_1)}{a_1^2}
			\sum_{\substack{\nu_1\mid d_1d_2\\ (\nu_1,r_1r_2a_1k_1k_2)=1}}\frac{\mu(\nu_1)}{\nu_1}\\ \nonumber
			&\cdot\sum_{\substack{ (\gamma,6d_1d_2\nu_1r_2k_2)=1}}\mu(\gamma)\sum_{(\kappa,6d_1d_2\nu_1r_2k_2)=1}\sum_{\substack{\nu_2\kappa\mid \nu_1\gamma\\ (\nu_2,r_1r_2a_1k_1k_2)=1}} \sum_{\substack{a_2\mid r_1\gamma\kappa\\ (a_2,a_1)=1}}\frac{\mu(a_2)}{a_2}\cdot \frac{\mu(\nu_2\kappa)}{\nu_2}\lambda_f\left(\frac{\nu_1\gamma}{\nu_2\kappa}\right)\\ \nonumber
			&\cdot \sum_{\substack{w\mid 6d_2r_2k_2a_2\\ (w,\nu_1\nu_2g)=1}}   \frac{\mu(w)}{w}  \Psi_1\bfrac{d_1\gamma\kappa r_1a_1wk_1}{Q_1}\Psi_2\left(\frac{d_2r_2k_2}{Q_2} \right)\mu(d_1)\mu(d_2)\frac{1}{k_1}\frac{\phi(r_2k_2)}{r_2k_2^2}
			\sumthree_{\substack{e, g, n \\ (g, d_1d_2r_1r_2a_1wk_1k_2\gamma\kappa) = 1\\ (n,d_1d_2r_2k_2)=1}}
			\\  \nonumber
			&\cdot  \lambda_f(e)1*\lambda_f(\gamma \kappa n)S(e,\overline{\nu_1\nu_2g}n;k_1wa_1k_2) \Phi_V\left(\frac{eE}{\nu_1\nu_2\gamma\kappa a_1^2w^2k_1^2k_2^2}\right)  V\bfrac{g}{G} V\bfrac{\gamma\kappa n}{N}.\nonumber
		\end{align}
		Now, we let $k_1=\widetilde{c}$, $e=\widetilde{m}$, $n=\widetilde{n}$ and group $wa_1k_2=\widetilde{s}$ and $\nu_1\nu_2g=\widetilde{r}$. We first truncate the ranges of these variables. Using the decay properties of $\Phi_V$ and the support of $\Psi_1$ and $\Psi_2$, we can get
		$$
		\widetilde{m}\ll \frac{\nu_1\nu_2\gamma\kappa a_1^2w^2k_1^2k_2^2}{E}Q^\varepsilon\ll Q^{2+\varepsilon}\frac{\nu_1\nu_2}{E(d_1d_2r_1r_2)^2\gamma\kappa}.
		$$
		Next we split variables dyadically, so that $d_i\sim D_i$, $r_i\sim R_i$, $\gamma\sim \Gamma$, $\kappa\sim \mathcal{K}$, $a_i\sim A_i$, $k_2\sim K_2$, $\nu_i\sim V_i$ and $w\sim W$. Then
		$$
		\widetilde{m}\asymp \widetilde{M}\in\left[1,\frac{Q^{2+\varepsilon}V_1V_2}{E(D_1D_2R_1R_2)^2\Gamma\mathcal{K}}\right],\ \ \widetilde{c}\asymp \widetilde{C}=\frac{Q_1}{D_1\Gamma\mathcal{K}R_1A_1W},
		$$
		$$
		\widetilde{n}\asymp \widetilde{N}\asymp \frac{N}{\Gamma\mathcal{K}},\ \ \widetilde{s}\asymp\widetilde{S}=WA_1K_2,\ \ \widetilde{r}\asymp\widetilde{R}=V_1V_2G.
		$$
		
		Returning to our computation, we can use Mellin inversion to separate the variables, and we may omit the detail. Then the contribution of $\mathcal{S}(d_1d_2r_1r_2\gamma\kappa\leq D)$ is essentially bounded by
		\begin{equation}\label{applykuz}
			Q^\varepsilon \frac{\sqrt{E}}{\sqrt{GN}}\frac{D_1D_2R_1R_2\Gamma\mathcal{K}}{\widetilde{C}A_2}\max_{\substack{d_1,d_2,r_1,r_2,\gamma,\kappa\\ d_1d_2r_1r_2\gamma\kappa\leq D\\ d_1\sim D_1,d_2\sim D_2\\ r_1\sim R_1,r_2\sim R_2\\ \gamma\sim \Gamma,\kappa\sim \mathcal{K}}}\left|\sum_{\substack{\widetilde{r} \sim \widetilde{R} \\ \widetilde{s} \sim \widetilde{S}\\ (\widetilde{r},\widetilde{s})=1}} \sum_{\substack{\widetilde{m} \sim \widetilde{M} \\ \widetilde{n} \sim \widetilde{N}}} a_{\widetilde{m}} b_{\widetilde{n}, \widetilde{r}, \widetilde{s}} \sum_{\substack{\widetilde{c} \sim \widetilde{C} \\(\widetilde{c}, \widetilde{r})=1}} g(\widetilde{c}, \widetilde{m}, \widetilde{n}, \widetilde{r}, \widetilde{s}) S( \pm \widetilde{n}, \widetilde{m} \overline{\widetilde{r}}, \widetilde{s} \widetilde{c}) \right|
		\end{equation}
		where $g$ is a smooth function satisfying the conditions of Lemma \ref{le:Klo1}, $
		a_{\widetilde{m}}=\lambda_f(\widetilde{m}),
		$ and 
		\begin{align*}
			b_{\widetilde{n}, \widetilde{r}, \widetilde{s}}&=
			\sum_{\substack{\widetilde{r}=\nu_1\nu_2g\\ \nu_1 \mid d_1d_2, \ \nu_2\kappa\mid \nu_1\gamma\\ (\nu_1\nu_2,r_1r_2k_1)=1\\ (\nu_1,\gamma\kappa)=1 \\ (g,d_1d_2r_1r_2\gamma\kappa)=1 \\ \nu_1\sim V_1, \ \nu_2\sim V_2}}
			\sum_{\substack{\widetilde{s}=wa_1k_2\\ w\mid 6d_2r_2k_2a_2\\ (w,\nu_1\nu_2g)=1\\ (k_2,6d_1r_1\nu_1\nu_2)=1\\  (a_1,6d_2r_2k_2a_2\nu_1\nu_2)=1\\ a_1\sim A_1,\ k_2\sim K_2\\ \ w\sim W}}\delta_{(\widetilde{n},d_1d_2r_2k_2)=1}1*\lambda_f(\gamma \kappa \widetilde{n})\lambda_f(\frac{\nu_1\gamma}{\nu_2\kappa})\\
			&\hskip 2in\cdot\frac{\mu(\nu_1)\mu(a_1)\mu(\nu_2\kappa)\mu(w)\phi(r_2k_2)}{\nu_1\nu_2r_2a_1^2k_2^2w}V\left(\frac{\gamma\kappa \widetilde{n}}{N}\right)V\left(\frac{g}{G}\right).
		\end{align*}
		\subsection{Applying the Kloosterman sum bounds}
		Now we are ready to apply Lemma \ref{le:Klo1} to \eqref{applykuz}. Notice first that
		$$\Vert a_{\widetilde{m}} \Vert_2^2 \ll  \widetilde{M}\ll \frac{Q^{2+\varepsilon}V_1V_2}{E(D_1D_2R_1R_2)^2\Gamma\mathcal{K}}, $$
		\begin{align*}
			\Vert b_{\widetilde{n}, \widetilde{r}, \widetilde{s}} \Vert_2^2& \leq  \frac{Q^\varepsilon}{V_1^2V_2^2A_1^4K_2^2W^2}
			\sum_{\widetilde{n}\sim \widetilde{N}}	\sum_{\widetilde{r}\sim \widetilde{R}}	\sum_{\widetilde{s}\sim \widetilde{S}}
			\left|
			\sum_{\substack{\widetilde{r}=\nu_1\nu_2g\\ \nu_1 \mid d_1d_2, \ \nu_2\mid \nu_1\gamma\\ \nu_1\sim V_1, \ \nu_2\sim V_2}}
			\sum_{\substack{\widetilde{s}=wa_1k_2\\ w\mid 6d_2r_2k_2a_2\nu_1\nu_2 \\  a_1\sim A_1,\ k_2\sim K_2,\\ w\sim W}}1
			\right|^2\\
			&\ll \frac{Q^\varepsilon}{V_1^2V_2^2A_1^4K_2^2W^2} \widetilde{N}\frac{\widetilde{R}}{V_1V_2}\frac{\widetilde{S}}{W}\ll \frac{Q^\varepsilon NG}{V_1^2V_2^2A_1^3W^2K_2\Gamma\mathcal{K}},
		\end{align*}
		and 
		$$
		\widetilde{C}\widetilde{S}\sqrt{\widetilde{R}}\asymp \widetilde{C}WA_1K_2\sqrt{V_1V_2G}.
		$$
		Hence
		$$
		\frac{\sqrt{E}}{\sqrt{GN}}\frac{D_1D_2R_1R_2\Gamma\mathcal{K}}{\widetilde{C}A_2}    \Vert a_{\widetilde{m}} \Vert_2	
		\Vert b_{\widetilde{n}, \widetilde{r}, \widetilde{s}} \Vert_2     \widetilde{C}\widetilde{S}\sqrt{\widetilde{R}}
		\ll  \frac{Q^{1+\varepsilon}\sqrt{GK_2}}{A_2\sqrt{A_1}}.
		$$
		Futhermore,
		\begin{align*}
			\widetilde{X}:=\frac{\widetilde{C} \widetilde{S} \sqrt{\widetilde{R}}}{4 \pi \sqrt{\widetilde{M} \widetilde{N}}}\gg_{\varepsilon} Q^{-\varepsilon} \frac{\sqrt{EG}}{\sqrt{N}}\frac{D_2R_2K_2}{Q_2}\gg_\varepsilon Q^{-\varepsilon} \frac{\sqrt{EG}}{\sqrt{N}}\gg 1
		\end{align*}
		since $EG \asymp M\gg NQ^{\varepsilon} $.
		By Lemma \ref{le:Klo1} $\mathcal{S}(d_1d_2r_1r_2\gamma\kappa\leq D)$ is essentially 
		\begin{align*}
			\ll \frac{Q^{1+\varepsilon}\sqrt{GK_2}}{A_2\sqrt{A_1}} \cdot \sqrt{\widetilde{R}\widetilde{S}}     \left(1+\frac{\sqrt{\widetilde{M}}}{\sqrt{\widetilde{R}\widetilde{S}}}\right)
			\left(1+\frac{\sqrt{\widetilde{N}}}{\sqrt{\widetilde{R}\widetilde{S}}}\right)\left(1+\frac{\widetilde{X}^2}{(1+\frac{\widetilde{R}\widetilde{S}}{\widetilde{M}})^2(1+\frac{\widetilde{R}\widetilde{S}}{\widetilde{N}})}\right)^{\theta}.
		\end{align*}
		Assuming Selberg's eigenvalue conjecture, $\theta=0$. Since $EG\gg Q^{2-\delta_0} $, we have 
		$$
		\frac{\sqrt{\widetilde{M}}}{\sqrt{\widetilde{R}\widetilde{S}}}\ll Q^{\frac12\delta_0+\varepsilon}K_2^{-\frac12}.
		$$
		Hence $\mathcal{S}(d_1d_2b_1b_2r_1r_2\leq D)$ is essentially 
		\begin{align}\label{S1kuzsmall}
			&\ll K_2^{-\frac12} \frac{Q^{1+\frac12\delta_0+\varepsilon}\sqrt{GK_2}}{A_2\sqrt{A_1}} \cdot \left( \sqrt{\widetilde{R}\widetilde{S}}+\sqrt{\widetilde{N}}  \right)\ll \frac{K_2^{-\frac12}Q^{1+\frac12\delta_0+\varepsilon}GK_2\sqrt{V_1V_2W}}{A_2}+\frac{K_2^{-\frac12}Q^{1+\frac12\delta_0+\varepsilon}\sqrt{GNK_2}}{A_2\sqrt{A_1\Gamma\mathcal{K}}}\nonumber\\ 
			&\ll Q^{1+\frac12\delta_0+\varepsilon}\left(GD^2Q_2+\sqrt{GN}\right)
		\end{align}
		since $V_1V_2W\ll D^4Q_2$ and $K_2\ll Q_2$. 
		\begin{rem}
			We remark that if we do not remove $(n,r_1)=1$ at the begining, then $W$ will be of size $Q_2D^2N$ while in our final case $W$ is of size $Q_2D^2$.
		\end{rem}

		\section{Bound from additive reciprocity formula}
		\label{Additiverecip}
		  In this section, we want to apply additive reciprocity formula, and we begin from \eqref{Starterror}. We apply Lemma \ref{lem:firstpoisson} on the summation over $g$, getting
		\begin{align}\label{gPoisson}
			\calS &=\frac{\sqrt{G}}{\sqrt{EN}} \sumfour_{\substack{d_1,d_2,r_1,r_2\\(d_1d_2r_1r_2,6)=1\\(d_1r_1,d_2r_2)=1}} \Psi_1\bfrac{d_1r_1}{Q_1}\Psi_2\left(\frac{d_2r_2}{Q_2} \right)\mu(d_1)\mu(d_2)\frac{ \phi(r_1)}{r_1}\frac{\phi(r_2)}{r_2}  \sum_{\substack{\nu_1\mid d_1d_2\\ (\nu_1,r_1r_2)=1}} \frac{\mu(\nu_1)}{\nu_1}\\
			&\cdot \sumthree_{\substack{e, g, n \\  (en, d_1d_2r_1r_2) = 1}} \lambda_f(e)1*\lambda_f(n)e\left(\frac{\overline{\nu_1e}gn}{r_1r_2}\right)V\bfrac{e}{E} \widehat{V}\left(\frac{gG}{\nu_1r_1r_2}\right) V\bfrac{n}{N}.\nonumber
		\end{align}
		Now we split $\calS$ into $\calS(g=0)$ and $\calS(g\neq 0)$. We first bound $\calS(g=0)$. We group $q_1=d_1r_1$ and $q_2=d_2r_2$, getting the sum over $d_1,d_2,r_1,r_2$ equals
		\begin{align*}
			&\sum_{q_1=d_1r_1}\sum_{q_2=d_2r_2}\mu(d_1)\mu(d_2)\frac{\phi(r_1)}{r_1}\frac{\phi(r_2)}{r_2}\sum_{\substack{\nu_1\mid d_1d_2\\ (\nu_1,r_1r_2)=1}}\frac{\mu(\nu_1)}{\nu_1}\\
			&=	\sum_{q_1=d_1r_1}\sum_{q_2=d_2r_2}\mu(d_1)\mu(d_2)\prod_{\substack{p\mid r_1}}\left(1-\frac{1}{p}\right)\prod_{\substack{p\mid r_1}}\left(1-\frac{1}{p}\right)\prod_{\substack{p\mid d_1d_2\\ p\nmid r_1r_2}}\left(1-\frac{1}{p}\right)\\
			&=	\sum_{q_1=d_1r_1}\sum_{q_2=d_2r_2}\mu(d_1)\mu(d_2)\prod_{\substack{p\mid r_1r_2}}\left(1-\frac{1}{p}\right)\prod_{\substack{p\mid d_1d_2\\ p\nmid r_1r_2}}\left(1-\frac{1}{p}\right)\\
			&=	\sum_{q_1=d_1r_1}\sum_{q_2=d_2r_2}\mu(d_1)\mu(d_2)\prod_{\substack{p\mid q_1q_2}}\left(1-\frac{1}{p}\right).
		\end{align*}
		Hence the sum over $d_1,d_2,r_1,r_2$ vanishes as long as $q_1q_2>1$. Next we consider $\calS(g\neq 0)$.

		We use the additive reciprocity relation 
		$$
		e\left(\frac{\bar{a}}{b}\right)e\left(\frac{\bar{b}}{a}\right)=e\left(\frac{1}{ab}\right)
		$$
		to get
		\begin{align*}
			\calS(g\neq 0) &=\frac{\sqrt{G}}{\sqrt{EN}} \sumfour_{\substack{d_1,d_2,r_1,r_2\\(d_1d_2r_1r_2,6)=1\\(d_1r_1,d_2r_2)=1}} \Psi_1\bfrac{d_1r_1}{Q_1}\Psi_2\left(\frac{d_2r_2}{Q_2} \right)\mu(d_1)\mu(d_2)\frac{ \phi(r_1)}{r_1}\frac{\phi(r_2)}{r_2}  \sum_{\substack{\nu_1\mid d_1d_2\\ (\nu_1,r_1r_2)=1}} \frac{\mu(\nu_1)}{\nu_1}\\
			&\cdot \sumthree_{\substack{e, g\neq 0, n \\  (en, d_1d_2r_1r_2) = 1}} \lambda_f(e)1*\lambda_f(n)e\left(-\frac{gn\overline{r_1r_2}}{\nu_1e}\right)e\left(\frac{gn}{\nu_1er_1r_2}\right)V\bfrac{e}{E} \widehat{V}\left(\frac{gG}{\nu_1r_1r_2}\right) V\bfrac{n}{N}.\nonumber
		\end{align*}
		We write $e\left(\frac{gn}{\nu_1er_1r_2}\right)=1+O\left(\frac{gn}{\nu_1er_1r_2}\right)=1+O\left(\frac{Q^\varepsilon\frac{\nu_1r_1r_2}{G}N}{\nu_1r_1r_2E}\right)=1+O\left(\frac{Q^\varepsilon N}{EG}\right)$, and bound the error term trivially, getting
		\begin{align}\label{ErrorinAdditiverecip}
			\calS(g\neq 0) &=\frac{\sqrt{G}}{\sqrt{EN}} \sumfour_{\substack{d_1,d_2,r_1,r_2\\(d_1d_2r_1r_2,6)=1\\(d_1r_1,d_2r_2)=1}} \Psi_1\bfrac{d_1r_1}{Q_1}\Psi_2\left(\frac{d_2r_2}{Q_2} \right)\mu(d_1)\mu(d_2)\frac{ \phi(r_1)}{r_1}\frac{\phi(r_2)}{r_2}  \sum_{\substack{\nu_1\mid d_1d_2\\ (\nu_1,r_1r_2)=1}} \frac{\mu(\nu_1)}{\nu_1}\\
			&\cdot \sumthree_{\substack{e, g\neq0, n \\  (en, d_1d_2r_1r_2) = 1}} \lambda_f(e)1*\lambda_f(n)e\left(-\frac{gn\overline{r_1r_2}}{\nu_1e}\right)V\bfrac{e}{E} \widehat{V}\left(\frac{gG}{\nu_1r_1r_2}\right) V\bfrac{n}{N}+O\left(\frac{Q^{2.5+\varepsilon}}{G}\right).\nonumber
		\end{align}
		Next we open $1*\lambda_f(n)$, and use a smooth partition of unity and remove $V\left(\frac{n}{N}\right)$ by Mellin inversion by the same manner as above. It suffices to bound 
		\begin{align*}
			\calS_1 &:=\frac{\sqrt{G}}{\sqrt{EN}} \sumfour_{\substack{d_1,d_2,r_1,r_2\\(d_1d_2r_1r_2,6)=1\\(d_1r_1,d_2r_2)=1}} \Psi_1\bfrac{d_1r_1}{Q_1}\Psi_2\left(\frac{d_2r_2}{Q_2} \right)\mu(d_1)\mu(d_2)\frac{ \phi(r_1)}{r_1}\frac{\phi(r_2)}{r_2}  \sum_{\substack{\nu_1\mid d_1d_2\\ (\nu_1,r_1r_2)=1}} \frac{\mu(\nu_1)}{\nu_1}\\
			&\cdot \sumfour_{\substack{e, g\neq 0, n_1,n_2 \\  (en_1n_2, d_1d_2r_1r_2) = 1}} \lambda_f(e)\lambda_f(n_1)e\left(-\frac{gn_1n_2\overline{r_1r_2}}{\nu_1e}\right)V\bfrac{e}{E} \widehat{V}\left(\frac{gG}{\nu_1r_1r_2}\right) V\bfrac{n_1}{N_1}V\bfrac{n_2}{N_2}.\nonumber
		\end{align*}
		Next we want to apply Voronoi on $n_1$, and to do this, we shall first guarantee $(gn_2,\nu_1e)=1$. Hence, we let $(g,e)=x_1$ firstly, getting
		\begin{align*}
			\calS_1&=\frac{\sqrt{G}}{\sqrt{EN}} \sumfive_{\substack{d_1,d_2,r_1,r_2,x_1\\(d_1d_2r_1r_2,6x_1)=1\\(d_1r_1,d_2r_2)=1}} \Psi_1\bfrac{d_1r_1}{Q_1}\Psi_2\left(\frac{d_2r_2}{Q_2} \right)\mu(d_1)\mu(d_2)\frac{ \phi(r_1)}{r_1}\frac{\phi(r_2)}{r_2}  \sum_{\substack{\nu_1\mid d_1d_2\\ (\nu_1,r_1r_2)=1}} \frac{\mu(\nu_1)}{\nu_1}\\
			&\cdot \sumfour_{\substack{e, g\neq0, n_1,n_2 \\  (en_1n_2, d_1d_2r_1r_2) = 1\\ (g,e)=1}} \lambda_f(x_1e)\lambda_f(n_1)e\left(-\frac{gn_1n_2\overline{r_1r_2}}{\nu_1e}\right)V\bfrac{x_1e}{E} \widehat{V}\left(\frac{x_1gG}{\nu_1r_1r_2}\right) V\bfrac{n_1}{N_1}V\bfrac{n_2}{N_2}.\nonumber 
		\end{align*}
		Next, we let $(g,\nu_1)=x_2$, getting
		\begin{align*}
			\calS_1&=\frac{\sqrt{G}}{\sqrt{EN}} \sumsix_{\substack{d_1,d_2,r_1,r_2,x_1,x_2\\(d_1d_2r_1r_2,6x_1)=1\\(d_1r_1,d_2r_2)=1\\ (x_2,r_1r_2)=1}} \Psi_1\bfrac{d_1r_1}{Q_1}\Psi_2\left(\frac{d_2r_2}{Q_2} \right)\mu(d_1)\mu(d_2)\frac{ \phi(r_1)}{r_1}\frac{\phi(r_2)}{r_2}  \sum_{\substack{x_2\nu_1\mid d_1d_2\\ (\nu_1,r_1r_2)=1}} \frac{\mu(x_2\nu_1)}{x_2\nu_1}\\
			&\cdot \sumfour_{\substack{e, g\neq0, n_1,n_2 \\  (e, d_1d_2r_1r_2x_2) = 1\\(n_1n_2, d_1d_2r_1r_2) = 1\\ (g,e\nu_1)=1}} \lambda_f(x_1e)\lambda_f(n_1)e\left(-\frac{gn_1n_2\overline{r_1r_2}}{\nu_1e}\right)V\bfrac{x_1e}{E} \widehat{V}\left(\frac{x_1gG}{\nu_1r_1r_2}\right) V\bfrac{n_1}{N_1}V\bfrac{n_2}{N_2}.\nonumber 
		\end{align*}
		Next, we let $(n_2,e)=x_3$, getting
		\begin{align*}
			\calS_1&=\frac{\sqrt{G}}{\sqrt{EN}} \sumseven_{\substack{d_1,d_2,r_1,r_2,x_1,x_2,x_3\\(d_1d_2r_1r_2,6x_1x_3)=1\\(d_1r_1,d_2r_2)=1\\ (x_2,r_1r_2x_3)=1}} \Psi_1\bfrac{d_1r_1}{Q_1}\Psi_2\left(\frac{d_2r_2}{Q_2} \right)\mu(d_1)\mu(d_2)\frac{ \phi(r_1)}{r_1}\frac{\phi(r_2)}{r_2}  \sum_{\substack{x_2\nu_1\mid d_1d_2\\ (\nu_1,r_1r_2)=1}} \frac{\mu(x_2\nu_1)}{x_2\nu_1}\\
			&\cdot \sumfour_{\substack{e, g\neq0, n_1,n_2 \\  (e, d_1d_2r_1r_2x_2) = 1\\(n_1n_2, d_1d_2r_1r_2) = 1\\ (g,x_3e\nu_1)=1\\ (n_2,e)=1}} \lambda_f(x_1x_3e)\lambda_f(n_1)e\left(-\frac{gn_1n_2\overline{r_1r_2}}{\nu_1e}\right)V\bfrac{x_1x_3e}{E} \widehat{V}\left(\frac{x_1gG}{\nu_1r_1r_2}\right) V\bfrac{n_1}{N_1}V\bfrac{x_3n_2}{N_2}.\nonumber 
		\end{align*}
		Finally, we let $(n_2,\nu_1)=x_4$, getting
		\begin{align*}
			&\calS_1=\frac{\sqrt{G}}{\sqrt{EN}} \sumeight_{\substack{d_1,d_2,r_1,r_2,x_1,x_2,x_3,x_4\\(d_1d_2r_1r_2,6x_1x_3x_4)=1\\(d_1r_1,d_2r_2)=1\\ (x_2,r_1r_2x_3)=1}} \Psi_1\bfrac{d_1r_1}{Q_1}\Psi_2\left(\frac{d_2r_2}{Q_2} \right)\mu(d_1)\mu(d_2)\frac{ \phi(r_1)}{r_1}\frac{\phi(r_2)}{r_2}  \sum_{\substack{x_2x_4\nu_1\mid d_1d_2\\ (\nu_1,r_1r_2)=1}} \frac{\mu(x_2x_4\nu_1)}{x_2x_4\nu_1}\\
			&\cdot \sumfour_{\substack{e, g\neq0, n_1,n_2 \\  (e, d_1d_2r_1r_2x_2x_4) = 1\\(n_1n_2, d_1d_2r_1r_2) = 1\\ (g,x_3x_4e\nu_1)=1\\ (n_2,e\nu_1)=1}} \lambda_f(x_1x_3e)\lambda_f(n_1)e\left(-\frac{gn_1n_2\overline{r_1r_2}}{\nu_1e}\right)V\bfrac{x_1x_3e}{E} \widehat{V}\left(\frac{x_1gG}{x_4\nu_1r_1r_2}\right) V\bfrac{n_1}{N_1}V\bfrac{x_3x_4n_2}{N_2}.\nonumber 
		\end{align*}
		Then we remove $(n_1,d_1d_2r_1r_2)=1$ by Möbius inversion, introducing $\mu(y_1)$. Afterward, we use the Hecke relation of $\lambda_f(n)$, introducing $\mu(y_2)$. Hence, we get
		\begin{align*}
			\calS_1&=\frac{\sqrt{G}}{\sqrt{EN}} \sumeight_{\substack{d_1,d_2,r_1,r_2,x_1,x_2,x_3,x_4\\(d_1d_2r_1r_2,6x_1x_3x_4)=1\\(d_1r_1,d_2r_2)=1\\ (x_2,r_1r_2x_3)=1}}\mu(d_1)\mu(d_2)\sum_{y_1\mid d_1d_2r_1r_2}\mu(y_1)\sum_{y_2\mid y_1}\mu(y_2)\lambda_f(\frac{y_1}{y_2}) \\
			&\cdot\frac{ \phi(r_1)}{r_1}\frac{\phi(r_2)}{r_2}  \Psi_1\bfrac{d_1r_1}{Q_1}\Psi_2\left(\frac{d_2r_2}{Q_2}\right)\sum_{\substack{x_2x_4\nu_1\mid d_1d_2\\ (\nu_1,r_1r_2)=1}} \frac{\mu(x_2x_4\nu_1)}{x_2x_4\nu_1}\sumfour_{\substack{e, g\neq0, n_1,n_2 \\  (e, d_1d_2r_1r_2x_2x_4) = 1\\ (g,x_3x_4e\nu_1)=1\\ (n_2,e\nu_1d_1d_2r_1r_2)=1}}\\
			&  \cdot \lambda_f(x_1x_3e)\lambda_f(n_1)e\left(-\frac{gy_1y_2n_1n_2\overline{r_1r_2}}{\nu_1e}\right)V\bfrac{x_1x_3e}{E} \widehat{V}\left(\frac{x_1gG}{x_4\nu_1r_1r_2}\right) V\bfrac{y_1y_2n_1}{N_1}V\bfrac{x_3x_4n_2}{N_2}.\nonumber 
		\end{align*}
		Next, we treat $y_1$ and $y_2$ in the same way as $g$ and $n_2$, writing
		\[
		(y_1,e)=h_1,\ (y_1,\nu_1)=h_2,\ (y_2,e)=h_3,\ (y_2,\nu_1)=h_4
		\]
		respectively, so that the remaining variables are coprime to $e\nu_1$. Hence we obtain
		\begin{align*}
			&\calS_1=\frac{\sqrt{G}}{\sqrt{EN}} \sumtwlve_{\substack{d_1,d_2,r_1,r_2,x_1,x_2,x_3,x_4,h_1,h_2,h_3,h_4\\(d_1d_2r_1r_2,6x_1x_3x_4h_1h_3)=1\\(d_1r_1,d_2r_2)=1\\ (x_2,r_1r_2x_3)=1\\ (h_1h_3,x_2x_4)=1\\ (h_2,r_1r_2h_3)=1\\ (h_4,r_1r_2)=1}}\mu(d_1)\mu(d_2)\sum_{h_1h_2y_1\mid d_1d_2r_1r_2}\mu(h_1h_2y_1)\\
			&\cdot\sum_{h_3h_4y_2\mid h_1h_2y_1}\mu(h_3h_4y_2)\lambda_f(\frac{h_1h_2y_1}{h_3h_4y_2}) \frac{ \phi(r_1)}{r_1}\frac{\phi(r_2)}{r_2}  \Psi_1\bfrac{d_1r_1}{Q_1}\Psi_2\left(\frac{d_2r_2}{Q_2}\right)\sum_{\substack{x_2x_4h_2h_4\nu_1\mid d_1d_2\\ (\nu_1,r_1r_2)=1}} \frac{\mu(x_2x_4h_2h_4\nu_1)}{x_2x_4h_2h_4\nu_1}\\
			&\cdot\sumfour_{\substack{e, g\neq0, n_1,n_2 \\  (e, d_1d_2r_1r_2x_2x_4h_2h_4) = 1\\ (g,x_3x_4e\nu_1h_1h_2h_3h_4)=1\\ (n_2,e\nu_1h_1h_2h_3h_4d_1d_2r_1r_2)=1\\ (y_1,e\nu_1h_4)=1\\ (y_2,e\nu_1)=1}}\lambda_f(x_1x_3h_1h_3e)\lambda_f(n_1)e\left(-\frac{gy_1y_2n_1n_2\overline{r_1r_2}}{\nu_1e}\right)V\bfrac{x_1x_3h_1h_3e}{E}V\bfrac{x_3x_4n_2}{N_2}\\
			&  \cdot  \widehat{V}\left(\frac{x_1gG}{x_4h_2h_4\nu_1r_1r_2}\right) V\bfrac{y_1y_2h_1h_2h_3h_4n_1}{N_1}.\nonumber 
		\end{align*}
		We aim to apply Poisson on $r_1$, and we shall smooth the variable $r_1$. Firstly, we use $\frac{\phi(r_1)}{r_1}=\sum_{k_1\mid r_1}\frac{\mu(k_1)}{k_1}$, and write $r_1k_1$ for $r_1$.
		Next we shall deal with $h_1h_2y_1\mid d_1d_2k_1r_1r_2$. We may restrict to the case $\mu(h_1 h_2 y_1) \neq 0$, in which case $h_1$, $h_2$, and $y_1$ are pairwise coprime. Hence $h_1h_2y_1\mid d_1d_2k_1r_1r_2$ is equivalent to $h_1\mid d_1d_2k_1r_1r_2$, $h_2\mid d_1d_2k_1r_1r_2$, and $y_1\mid d_1d_2k_1r_1r_2$. Since $(h_1h_2,r_1)=1$, we have $h_1\mid d_1d_2k_1r_2$ and $h_2\mid d_1d_2k_1r_2$. Using the fact that if $a \mid bc$, then there exists a unique factorization $a = a_1 a_2$ such that $a_1 \mid b$, $a_2 \mid c$, and $(a_2,\, b/a_1)=1$, we may decompose
		\[
		y_1 = y_{11} y_{12},
		\]
		where $y_{11}\mid r_1$, $y_{12}\mid d_1d_2k_1r_2$, and $(y_{12},\frac{r_1}{y_{11}})=1$.\\
		Hence we have
		\begin{align*}
			&\calS_1=\frac{\sqrt{G}N_1}{\sqrt{EN}} \sumonefour_{\substack{d_1,d_2,r_1,r_2,x_1,x_2,x_3,x_4,h_1,h_2,h_3,h_4,k_1,y_{11}\\(d_1d_2k_1y_{11}r_1r_2,6x_1x_3x_4h_1h_3)=1\\(d_1k_1y_{11}r_1,d_2r_2)=1\\ (x_2,k_1y_{11}r_1r_2x_3)=1\\ (h_1h_3,x_2x_4)=1\\ (h_2,k_1y_{11}r_1r_2h_3)=1\\ (h_4,k_1y_{11}r_1r_2)=1\\ h_1h_2\mid d_1d_2k_1r_2}}\mu(d_1)\mu(d_2)\sum_{\substack{y_{12}\mid d_1d_2k_1r_2\\ (y_{12},r_1)=1 }}\mu(h_1h_2y_{11}y_{12}) \\
			&\cdot\sum_{h_3h_4y_2\mid h_1h_2y_{11}y_{12}}\frac{\mu(k_1)}{k_1}\mu(h_3h_4y_2)\lambda_f(\frac{h_1h_2y_{11}y_{12}}{h_3h_4y_2})\frac{\phi(r_2)}{r_2}  \Psi_1\bfrac{d_1k_1y_{11}r_1}{Q_1}\Psi_2\left(\frac{d_2r_2}{Q_2}\right)\sum_{\substack{x_2x_4h_2h_4\nu_1\mid d_1d_2\\ (\nu_1,k_1y_{11}r_1r_2)=1}} \\
			&  \cdot\frac{\mu(x_2x_4h_2h_4\nu_1)}{x_2x_4h_2h_4\nu_1}\sumfour_{\substack{e, g\neq0, n_1,n_2 \\  (e, d_1d_2k_1y_{11}r_1r_2x_2x_4h_2h_4) = 1\\ (g,x_3x_4e\nu_1h_1h_2h_3h_4)=1\\ (n_2,e\nu_1h_1h_2h_3h_4d_1d_2k_1y_{11}r_1r_2)=1\\ (y_1,e\nu_1h_4)=1\\ (y_2,e\nu_1)=1}}\lambda_f(x_1x_3h_1h_3e)\lambda_f(n_1)e\left(-\frac{gy_{12}y_2n_1n_2\overline{k_1r_1r_2}}{\nu_1e}\right) \\
			&\cdot V\bfrac{x_1x_3h_1h_3e}{E}\widehat{V}\left(\frac{x_1gG}{x_4h_2h_4\nu_1k_1y_{11}r_1r_2}\right)V\bfrac{y_{11}y_{12}y_2h_1h_2h_3h_4n_1}{N_1}V\bfrac{x_3x_4n_2}{N_2}.\nonumber 
		\end{align*}
		We apply Wilton's bound (see Lemma \ref{Wilton}) on $n_1$, and then bound the sums trivially, getting
		\begin{align}\label{lowboundn}
			\calS_1\ll \frac{Q^{1+\varepsilon}\sqrt{ENG}\cdot\sqrt{N}}{G\sqrt{N_2}}.
		\end{align}
		This bound would give a lower bound of $N$ in further proof of the Proposition \ref{prop:unbalanced}.\par
		Now, we split $\mathcal{S}_1$ into $\mathcal{S}_1(d_1y_{11}\leq Q^{\delta})$ and $\mathcal{S}_1(d_1y_{11}> Q^{\delta})$. We first consider $\mathcal{S}_1(d_1y_{11}>Q^{\delta})$. We apply Wilton's bound on $n_1$, and then bound the sums trivially, getting
		\begin{align}\label{largeAdditive}
			\mathcal{S}_1(d_1y_{11}>Q^{\delta})\ll \frac{Q^{1+\varepsilon}\sqrt{ENG}\cdot\sqrt{N}}{G\sqrt{N_2}}\sumtwo_{\substack{d_1,y_{11}\\ d_1y_{11}>Q^\delta}}\frac{1}{d_1^2y_{11}^{1.5}}\ll \frac{Q^{1+\varepsilon-\frac{1}{2}\delta}\sqrt{ENG}\cdot\sqrt{N}}{G\sqrt{N_2}}.
		\end{align}
		Next, we consider $\mathcal{S}_1$ into $\mathcal{S}_1(d_1y_{11}\leq Q^{\delta})$. We apply Voronoi on $n_1$. Using Lemma \ref{lem:voronoi}, we obtain
		\begin{align*}
			&=\frac{\sqrt{G}N_1}{\sqrt{EN}} \sumonefour_{\substack{d_1,d_2,r_1,r_2,x_1,x_2,x_3,x_4,h_1,h_2,h_3,h_4,k_1,y_{11}\\(d_1d_2k_1y_{11}r_1r_2,6x_1x_3x_4h_1h_3)=1\\(d_1k_1y_{11}r_1,d_2r_2)=1\\ (x_2,k_1y_{11}r_1r_2x_3)=1\\ (h_1h_3,x_2x_4)=1\\ (h_2,k_1y_{11}r_1r_2h_3)=1\\ (h_4,k_1y_{11}r_1r_2)=1\\ h_1h_2\mid d_1d_2k_1r_2, d_1y_{11}\leq Q^{\delta}}}\mu(d_1)\mu(d_2)\sum_{\substack{y_{12}\mid d_1d_2k_1r_2\\ (y_{12},r_1)=1 }}\frac{\mu(h_1h_2y_{11}y_{12})}{h_1h_2y_{11}y_{12}}\sum_{h_3h_4y_2\mid h_1h_2y_{11}y_{12}} \\
			&\cdot\frac{\mu(k_1)}{k_1}\frac{\mu(h_3h_4y_2)}{h_3h_4y_2}\lambda_f(\frac{h_1h_2y_{11}y_{12}}{h_3h_4y_2})\frac{\phi(r_2)}{r_2}  \Psi_1\bfrac{d_1k_1y_{11}r_1}{Q_1}\Psi_2\left(\frac{d_2r_2}{Q_2}\right)\sum_{\substack{x_2x_4h_2h_4\nu_1\mid d_1d_2\\ (\nu_1,k_1y_{11}r_1r_2)=1}} \frac{\mu(x_2x_4h_2h_4\nu_1)}{x_2x_4h_2h_4\nu_1^2}\\
			&  \cdot \sumfour_{\substack{e, g\neq0, n_1,n_2 \\  (e, d_1d_2k_1y_{11}r_1r_2x_2x_4h_2h_4) = 1\\ (g,x_3x_4e\nu_1h_1h_2h_3h_4)=1\\ (n_2,e\nu_1h_1h_2h_3h_4)=1\\ (y_1,e\nu_1h_4)=1\\ (y_2,e\nu_1)=1}}\frac{\lambda_f(x_1x_3h_1h_3e)}{e}\lambda_f(n_1)e\left(\frac{\overline{gy_{12}y_2n_2}k_1r_1r_2n_1}{\nu_1e}\right)V\bfrac{x_1x_3h_1h_3e}{E} \\
			&\cdot\widehat{V}\left(\frac{x_1gG}{x_4h_2h_4\nu_1k_1y_{11}r_1r_2}\right)\Phi_V\left(\frac{n_1N_1}{y_{11}y_{12}y_2h_1h_2h_3h_4\nu_1^2e^2}\right)V\bfrac{x_3x_4n_2}{N_2}.\nonumber 
		\end{align*}

		Now we deal with the sum over $r_1$. We interchange the order of summation, and we need to give a bound for 
		\begin{equation}
			\sum_{(r_1,6d_2r_2x_1x_2x_3x_4h_1h_2h_3h_4y_{12}\nu_1e)}e\left(\frac{\overline{gy_{12}y_2n_2}k_1r_1r_2n_1}{\nu_1e}\right)\Psi_1\bfrac{d_1k_1y_{11}r_1}{Q_1}\widehat{V}\left(\frac{x_1gG}{x_4h_2h_4\nu_1k_1y_{11}r_1r_2}\right)
		\end{equation}
		We truncate $|g|\leq Q^\varepsilon \frac{x_4h_2h_4\nu_1k_1y_{11}r_1r_2}{x_1G}\leq 10Q^\varepsilon \frac{x_4h_2h_4\nu_1Q}{d_1d_2x_1G}$, and apply Mellin inversion to $\widehat{V}(\cdot )$, getting 
		$$
		\widehat{V}\left(\frac{x_1gG}{x_4h_2h_4\nu_1k_1y_{11}r_1r_2}\right)=\frac{1}{2\pi i}\int_{(\varepsilon)}\widetilde{\widehat{V}}(s)\left(\frac{x_1gG}{x_4h_2h_4\nu_1k_1y_{11}r_1r_2}\right)^{-s}ds.
		$$
		Since $\widetilde{\widehat{V}}(s)\ll s^{-A}$, we can truncate the integral at $\text{Im}(s)\ll Q^\varepsilon$. Then we take $\Psi_3(x):=x^s\Psi_1(x)$, and we only need to consider
		\begin{equation}\label{rPoisson}
			\sum_{(r_1,6d_2r_2x_1x_2x_3x_4h_1h_2h_3h_4y_{12}\nu_1e)}e\left(\frac{\overline{gy_{12}y_2n_2}k_1r_1r_2n_1}{\nu_1e}\right)\Psi_3\bfrac{d_1k_1y_{11}r_1}{Q_1}.
		\end{equation}
		We record the following lemma for our use.
		\begin{lem}\label{FinalPoisson}
			For $\lambda,b\in \mathbb{Z}$ and $(a,q)=1$, we have
			$$
			\sum_{\substack{r\geq 1\\(r,\lambda)=1}}e\left(\frac{\overline{a}br}{q}\right)V\left(\frac{r}{R}\right)=R\sum_{t_1\mid \lambda}\frac{\mu(t_1)}{t_1}\sum_{r\equiv -\overline{a}bt_1\bmod q}\widehat{V}\left(\frac{rR}{qt_1}\right).
			$$
		\end{lem}
		\begin{proof}
			We first remove $(r,\lambda)=1$, introducing $\mu(t_1)$. Hence we get
			$$
			\sum_{\substack{r\geq 1\\(r,\lambda)=1}}e\left(\frac{\overline{a}br}{q}\right)V\left(\frac{r}{R}\right)=\sum_{t_1\mid \lambda}\mu(t_1)\sum_{x\bmod q}e\left(\frac{\overline{a}bt_1x}{q}\right)\sum_{r\equiv x\bmod q}V\left(\frac{t_1r}{R}\right).
			$$
			Finally, we use Lemma \ref{lem:firstpoisson} on the inner sum, and we get what we want.
		\end{proof}
		We apply this lemma on \eqref{rPoisson}, getting that \eqref{rPoisson} equals
		$$
		\frac{Q_1}{d_1k_1y_{11}}\sum_{t_1\mid 6d_2r_2x_1x_2x_3x_4h_1h_2h_3h_4y_{12}\nu_1e }\frac{\mu(t_1)}{t_1}\sum_{\substack{r_1\\r_1gy_{12}y_2n_2\equiv -k_1r_2n_1\bmod \nu_1e}}\widehat{\Psi}_3\left(\frac{r_1Q_1}{\nu_1ed_1k_1y_{11}t_1}\right).
		$$
		Since the length of $n_1$ is of size $Q^\varepsilon\frac{y_{11}y_{12}y_2h_1h_2h_3h_4\nu_1^2e^2}{N_1}\geq Q^\varepsilon\nu_1e\frac{e}{N_1}\gg Q^\varepsilon \nu_1e\frac{E}{N}\gg Q^{\varepsilon}\nu_1e$, the length of $k_1r_2n_1$ is larger than the length of $\nu_1e$. Hence, we may bound trivially and the condition $r_1gy_{12}y_2n_2\equiv -k_1r_2n_1\bmod \nu_1e$ can save a $\nu_1e$, and we obtain that $\calS_1(d_1y_{11}\leq Q^{\delta})$ is bounded trivially by
		\begin{align}\label{smallAdditive}
			\calS_1(d_1y_{11}\leq Q^{\delta})\ll  \frac{Q^{1+\varepsilon}E^2N_2Q_2}{\sqrt{EGN}}\sumtwo_{\substack{d_1,y_{11}\\ d_1y_{11}\leq Q^\delta}}1\ll \frac{Q^{1+\delta+\varepsilon}E^2N_2Q_2}{\sqrt{EGN}}.
		\end{align}

		\section{The second Kuznetsov}
		\label{Secondkuz}
		We start from \eqref{gPoisson}. As usual, the contribution of the term $g=0$ vanishes, and we only need to consider $\calS(g\neq 0)$. We first let $(g,r_1)=x_1$ and $(g,r_2)=x_2$, getting
		\begin{align}
			\calS(g\neq 0) &=\frac{\sqrt{G}}{\sqrt{EN}} \sumsix_{\substack{d_1,d_2,x_1,x_2,r_1,r_2\\(d_1d_2x_1x_2r_1r_2,6)=1\\(d_1x_1r_1,d_2x_2r_2)=1}} \Psi_1\bfrac{d_1x_1r_1}{Q_1}\Psi_2\left(\frac{d_2x_2r_2}{Q_2} \right)\mu(d_1)\mu(d_2)\frac{ \phi(r_1x_1)}{r_1x_1}\frac{\phi(r_2x_2)}{r_2x_2} \\
			&\cdot  \sum_{\substack{\nu_1\mid d_1d_2\\ (\nu_1,r_1r_2)=1}} \frac{\mu(\nu_1)}{\nu_1}\sumthree_{\substack{e,g, n \\  (en, d_1d_2x_1x_2r_1r_2) = 1\\g\neq 0\\ (g,r_1r_2)=1}} \lambda_f(e)1*\lambda_f(n)e\left(\frac{\overline{\nu_1e}gn}{r_1r_2}\right)V\bfrac{e}{E} \widehat{V}\left(\frac{gG}{\nu_1r_1r_2}\right) V\bfrac{n}{N}.\nonumber
		\end{align}
		Next we open $1*\lambda_f(n)$, and use a smooth partition of unity and remove $V\left(\frac{n}{N}\right)$ by Mellin inversion in the same manner as above. It suffices to bound 
		\begin{align}
			\calS(g\neq 0) &=\frac{\sqrt{G}}{\sqrt{EN}} \sumsix_{\substack{d_1,d_2,x_1,x_2,r_1,r_2\\(d_1d_2x_1x_2r_1r_2,6)=1\\(d_1x_1r_1,d_2x_2r_2)=1}} \Psi_1\bfrac{d_1x_1r_1}{Q_1}\Psi_2\left(\frac{d_2x_2r_2}{Q_2} \right)\mu(d_1)\mu(d_2)\frac{ \phi(r_1x_1)}{r_1x_1}\frac{\phi(r_2x_2)}{r_2x_2} \\
			&\cdot  \sum_{\substack{\nu_1\mid d_1d_2\\ (\nu_1,x_1x_2r_1r_2)=1}} \frac{\mu(\nu_1)}{\nu_1}\sumfour_{\substack{e,g, n_1,n_2 \\  (en_1n_2, d_1d_2x_1x_2r_1r_2) = 1\\g\neq 0\\ (g,r_1r_2)=1}} \lambda_f(e)\lambda_f(n_1)e\left(\frac{\overline{\nu_1e}gn_1n_2}{r_1r_2}\right)\nonumber\\
			&\cdot V\bfrac{e}{E} \widehat{V}\left(\frac{gG}{\nu_1r_1r_2}\right) V\bfrac{n_1}{N_1}V\bfrac{n_2}{N_2}.\nonumber
		\end{align}
		Next we apply Lemma \ref{FinalPoisson} on $n_2$, getting 
		\begin{align}
			\calS(g\neq 0) &=\frac{\sqrt{G}N_2}{\sqrt{EN}} \sumsix_{\substack{d_1,d_2,x_1,x_2,r_1,r_2\\(d_1d_2x_1x_2r_1r_2,6)=1\\(d_1x_1r_1,d_2x_2r_2)=1}} \Psi_1\bfrac{d_1x_1r_1}{Q_1}\Psi_2\left(\frac{d_2x_2r_2}{Q_2} \right)\mu(d_1)\mu(d_2)\frac{ \phi(r_1x_1)}{r_1x_1}\frac{\phi(r_2x_2)}{r_2x_2} \\
			&\cdot  \sum_{\substack{\nu_1\mid d_1d_2\\ (\nu_1,x_1x_2r_1r_2)=1}} \frac{\mu(\nu_1)}{\nu_1}\sum_{\substack{\nu_2\mid d_1d_2x_1x_2r_1r_2}}\frac{\mu(\nu_2)}{\nu_2}
			\sumfour_{\substack{e,g, n_1,n_2 \\  (en_1, d_1d_2x_1x_2r_1r_2) = 1\\g\neq 0\\ (g,r_1r_2)=1\\ n_2\equiv- \overline{\nu_1e}\nu_2gn_1\bmod r_1r_2}} \lambda_f(e)\lambda_f(n_1)V\bfrac{e}{E} \widehat{V}\left(\frac{gG}{\nu_1r_1r_2}\right)\nonumber\\ &\cdot V\bfrac{n_1}{N_1}\widehat{V}\left(\frac{n_2N_2}{\nu_2r_1r_2}\right) .\nonumber
		\end{align}
		We split $\calS(g\neq 0)$ into $\calS(g\neq 0,n_2=0)$ and $\calS(g\neq 0,n_2\neq 0)$. For the case where $n_2=0$, we get $r_1r_2\mid \nu_2$ since $(gn_1,r_1r_2)=1$, and we bound trivially that 
		\begin{align}\label{n2=0}
			\calS(g\neq0,n_2=0)\ll Q^\varepsilon\frac{Q\sqrt{EN}}{\sqrt{G}}\ll \frac{Q^{2.5+\varepsilon}}{G}.
		\end{align}
		Next, we assume that $n_2\neq 0$. We aim to apply the Voronoi formula to the $e$-sum, and to do so we first need to ensure that $(n_2\nu_2, r_1r_2)=1$; otherwise, we would meet an obstruction. Let $(n_2,r_1)=y_1$ and $(n_2,r_2)=y_2$, and find that $\calS(g\neq 0,n_2\neq 0)$ is equal.
		\begin{align}
			&\frac{\sqrt{G}N_2}{\sqrt{EN}} \sumeight_{\substack{d_1,d_2,x_1,x_2,y_1,y_2,r_1,r_2\\(d_1d_2x_1x_2y_1y_2r_1r_2,6)=1\\(d_1x_1y_1r_1,d_2x_2y_2r_2)=1}} \Psi_1\bfrac{d_1x_1y_1r_1}{Q_1}\Psi_2\left(\frac{d_2x_2y_2r_2}{Q_2} \right)\mu(d_1)\mu(d_2)\frac{ \phi(r_1x_1y_1)}{r_1x_1y_1}\frac{\phi(r_2x_2y_2)}{r_2x_2y_2}  \\
			&\cdot \sum_{\substack{\nu_1\mid d_1d_2\\ (\nu_1,x_1x_2y_1y_2r_1r_2)=1}} \frac{\mu(\nu_1)}{\nu_1}\sum_{\substack{\nu_2\mid d_1d_2x_1x_2y_1y_2r_1r_2}}\frac{\mu(\nu_2)}{\nu_2}
			\sumfour_{\substack{e,g, n_1,n_2 \\  (en_1, d_1d_2x_1x_2y_1y_2r_1r_2) = 1\\g\neq 0,\ n_2\neq 0\\ (g,y_1y_2r_1r_2)=1,\ (n_2,r_1r_2)=1\\ y_1y_2n_2\equiv- \overline{\nu_1e}\nu_2gn_1\bmod y_1y_2r_1r_2}} \lambda_f(e)\lambda_f(n_1)V\bfrac{e}{E} \nonumber\\
			&\cdot\widehat{V}\left(\frac{gG}{\nu_1y_1y_2r_1r_2}\right) V\bfrac{n_1}{N_1}\widehat{V}\left(\frac{n_2N_2}{\nu_2r_1r_2}\right) .\nonumber
		\end{align}
		We shall notice now that we have $y_1y_2\mid \nu_2gn_1$ and we must have $y_1y_2\mid \nu_2 $. Hence, we have $$n_2\equiv -\overline{\nu_1e}\frac{\nu_2}{y_1y_2}gn_1\bmod r_1r_2,$$ and we must have $(\frac{\nu_2}{y_1y_2},r_1r_2)=1$ since $(n_2\nu_1egn_1,r_1r_2)=1$. Hence, we obtain $\calS(g\neq 0,n_2\neq 0)$ equals 
		\begin{align}
			&\frac{\sqrt{G}N_2}{\sqrt{EN}} \sumeight_{\substack{d_1,d_2,x_1,x_2,y_1,y_2,r_1,r_2\\(d_1d_2x_1x_2y_1y_2r_1r_2,6)=1\\(d_1x_1y_1r_1,d_2x_2y_2r_2)=1}} \Psi_1\bfrac{d_1x_1y_1r_1}{Q_1}\Psi_2\left(\frac{d_2x_2y_2r_2}{Q_2} \right)\mu(d_1)\mu(d_2)\frac{ \phi(r_1x_1y_1)}{r_1x_1y_1}\frac{\phi(r_2x_2y_2)}{r_2x_2y_2}  \\
			&\cdot \sum_{\substack{\nu_1\mid d_1d_2\\ (\nu_1,x_1x_2y_1y_2r_1r_2)=1}} \frac{\mu(\nu_1)}{\nu_1}\sum_{\substack{\nu_2\mid d_1d_2x_1x_2\\ (\nu_2,r_1r_2)=1}}\frac{\mu(y_1y_2\nu_2)}{y_1y_2\nu_2}
			\sumfour_{\substack{e,g, n_1,n_2 \\  (en_1, d_1d_2x_1x_2y_1y_2r_1r_2) = 1\\g\neq 0,\ n_2\neq 0\\ (g,y_1y_2r_1r_2)=1,\ (n_2,r_1r_2)=1\\ e\equiv- \overline{\nu_1n_2}\nu_2gn_1\bmod r_1r_2}} \lambda_f(e)\lambda_f(n_1)V\bfrac{e}{E}\nonumber \\
			&\cdot\widehat{V}\left(\frac{gG}{\nu_1y_1y_2r_1r_2}\right) V\bfrac{n_1}{N_1}\widehat{V}\left(\frac{n_2N_2}{\nu_2y_1y_2r_1r_2}\right) .\nonumber
		\end{align}
		Now, we apply Lemma \ref{lem:Voro} on $e$, and we may notice that we have $\gamma\mid \nu_2gn_1$ and $\kappa\mid \frac{\nu_2gn_1}{\gamma}$. Therefore, we can have a unique factorization that $\gamma=\gamma_1\gamma_2\gamma_3$ with $\gamma_1\mid \nu_2$, $\gamma_2\mid g$, $\gamma_3\mid n_1$, $(\gamma_2\gamma_3,\frac{\nu_2}{\gamma_1})=1$ and $(\gamma_3,\frac{g}{\gamma_2})=1$. Similarly, we can have a unique factorization that $\kappa=\kappa_1\kappa_2\kappa_3$ with $\kappa_1\mid \frac{\nu_2}{\gamma_1}$, $\kappa_2\mid \frac{g}{\gamma_2}$, $\kappa_3\mid \frac{n_1}{\gamma_3}$, $(\kappa_2\kappa_3,\frac{\nu_2}{\gamma_1\kappa_1})=1$ and $(\kappa_3,\frac{g}{\gamma_2\kappa_2})=1$. Hence we have $\calS(g\neq 0,n_2\neq 0)$ equals 
		\begin{align}\label{Secondlargesievebegin}
			&\frac{\sqrt{EG}N_2}{\sqrt{N}} \sumonesix_{\substack{d_1,d_2,x_1,x_2,y_1,y_2,\gamma_1,\gamma_2,\gamma_3,\kappa_1,\kappa_2,\kappa_3,r_1,r_2,k_1,k_2\\(d_1d_2x_1x_2y_1y_2\gamma_1\gamma_2\gamma_3\kappa_1\kappa_2\kappa_3r_1r_2k_1k_2,6)=1\\(d_1x_1y_1\gamma_1\gamma_2\gamma_3\kappa_1\kappa_2\kappa_3r_1k_1,d_2x_2y_2r_2k_2)=1\\ (\gamma_2\gamma_3,\kappa_1)=1, (\gamma_3,\kappa_2)=1}} \Psi_1\bfrac{d_1x_1y_1\gamma_1\gamma_2\gamma_3\kappa_1\kappa_2\kappa_3r_1k_1}{Q_1}\\
			&\cdot \Psi_2\left(\frac{d_2x_2y_2r_2k_2}{Q_2} \right)\mu(d_1)\mu(d_2)\frac{\mu(\gamma_1\gamma_2\gamma_3)}{\gamma_1\gamma_2\gamma_3}\frac{ \phi(x_1y_1\gamma_1\gamma_2\gamma_3\kappa_1\kappa_2\kappa_3r_1k_1)}{x_1y_1\gamma_1\gamma_2\gamma_3\kappa_1\kappa_2\kappa_3r_1^2k_1^3}\frac{\phi(x_2y_2r_2k_2)}{x_2y_2r_2^2k_2^3}  \nonumber \\
			&\cdot \sum_{\substack{\nu_1\mid d_1d_2\\ (\nu_1,x_1x_2y_1y_2\gamma_1\gamma_2\gamma_3\kappa_1\kappa_2\kappa_3r_1r_2k_1k_2)=1}} \frac{\mu(\nu_1)}{\nu_1} \sum_{\substack{\gamma_1\kappa_1\nu_2\mid d_1d_2x_1x_2\\ (\nu_2,\gamma_2\gamma_3\kappa_2\kappa_3r_2k_2)=1}}\frac{\mu(y_1y_2\gamma_1\kappa_1\nu_2)}{y_1y_2\gamma_1\kappa_1\nu_2}\sum_{\substack{\nu_3\mid d_1d_2x_1x_2y_1y_2\\ (\nu_3,r_1r_2k_1k_2)=1}}\frac{\mu(\nu_3)}{\nu_3}
			\nonumber \\
			&\cdot\sum_{\substack{\kappa_1\kappa_2\kappa_3\nu_4\mid \gamma_1\gamma_2\gamma_3\nu_3\\ (\nu_4,r_1r_2k_1k_2)=1}}\frac{\mu(\kappa_1\kappa_2\kappa_3\nu_4)}{\kappa_1\kappa_2\kappa_3\nu_4}\lambda_f\left(\frac{\gamma_1\gamma_2\gamma_3\nu_3}{\kappa_1\kappa_2\kappa_3\nu_4}\right)\sumfour_{\substack{e,g, n_1,n_2 \\  (\gamma_3\kappa_3n_1, d_1d_2x_1x_2y_1y_2r_2k_2) = 1\\g\neq 0,\ n_2\neq 0\\ (\gamma_2\kappa_2g,y_1y_2\gamma_3\kappa_3r_2k_2)=1\\ (n_2,\gamma_1\gamma_2\gamma_3\kappa_1\kappa_2\kappa_3r_1r_2k_1k_2)=1}} \lambda_f(e)\lambda_f(\gamma_3\kappa_3n_1)\nonumber\\
			&\cdot S\left(e,-\overline{\nu_1\nu_3\nu_4n_2}\nu_2gn_1;k_1k_2\right)\Phi_V\left(\frac{eE}{\nu_3\nu_4\gamma_1\gamma_2\gamma_3\kappa_1\kappa_2\kappa_3k_1^2k_2^2}\right)\widehat{V}\left(\frac{gG}{\nu_1y_1y_2\gamma_1\gamma_3\kappa_1\kappa_3r_1r_2k_1k_2}\right)\nonumber \\
			&\cdot V\bfrac{\gamma_3\kappa_3n_1}{N_1}\widehat{V}\left(\frac{n_2N_2}{\nu_2y_1y_2\gamma_1^2\gamma_2\gamma_3\kappa_1^2\kappa_2\kappa_3r_1r_2k_1k_2}\right) .\nonumber
		\end{align}
		Now we split $\calS(g\neq 0,n_2\neq 0)$ into $$\calS(d_1d_2x_1x_2y_1y_2\gamma_1\gamma_2\gamma_3\kappa_1\kappa_2\kappa_3r_1r_2\leq D)$$ and $$\calS(d_1d_2x_1x_2y_1y_2\gamma_1\gamma_2\gamma_3\kappa_1\kappa_2\kappa_3r_1r_2> D).$$
		We first consider $\calS(d_1d_2x_1x_2y_1y_2\gamma_1\gamma_2\gamma_3\kappa_1\kappa_2\kappa_3r_1r_2\leq D)$, and the case $>D$ we shall consider in Section \ref{sectionlargesieve}.\par 
		Next, we smooth the variable $k_1$ in preparation for applying the Kuznetsov trace formula. We write
		$$
		\frac{\phi(x_1y_1\gamma_1\gamma_2\gamma_3\kappa_1\kappa_2\kappa_3r_1k_1)}{x_1y_1\gamma_1\gamma_2\gamma_3\kappa_1\kappa_2\kappa_3r_1k_1}=\sum_{a\mid x_1y_1\gamma_1\gamma_2\gamma_3\kappa_1\kappa_2\kappa_3r_1k_1}\frac{\mu(a)}{a}=\sum_{a_2\mid x_1y_1\gamma_1\gamma_2\gamma_3\kappa_1\kappa_2\kappa_3r_1}\sum_{\substack{a_1\mid k_1\\ (a_2,\frac{k_1}{a_1})=1}}\frac{\mu(a_1)\mu(a_2)}{a_1a_2},
		$$
		and we write $k_1a_1$ for $k_1$. Then we remove $(k_1,6d_2x_2y_2a_2r_2k_2)=1$, introducing $\mu(w)$. Hence, we get $\calS(d_1d_2x_1x_2y_1y_2\gamma_1\gamma_2\gamma_3\kappa_1\kappa_2\kappa_3r_1r_2\leq D)$ equals 
		\begin{align}
			&\frac{\sqrt{EG}N_2}{\sqrt{N}} \sumonesix_{\substack{d_1,d_2,x_1,x_2,y_1,y_2,\gamma_1,\gamma_2,\gamma_3,\kappa_1,\kappa_2,\kappa_3,r_1,r_2,a_1,k_2\\(d_1d_2x_1x_2y_1y_2\gamma_1\gamma_2\gamma_3\kappa_1\kappa_2\kappa_3r_1r_2a_1k_2,6)=1\\(d_1x_1y_1\gamma_1\gamma_2\gamma_3\kappa_1\kappa_2\kappa_3r_1a_1,d_2x_2y_2r_2k_2)=1\\(\gamma_2\gamma_3,\kappa_1)=1, (\gamma_3,\kappa_2)=1\\d_1d_2x_1x_2y_1y_2\gamma_1\gamma_2\gamma_3\kappa_1\kappa_2\kappa_3r_1r_2\leq D}}\sum_{a_2\mid x_1y_1\gamma_1\gamma_2\gamma_3\kappa_1\kappa_2\kappa_3r_1}\frac{\mu(a_2)}{a_2} \Psi_2\left(\frac{d_2x_2y_2r_2k_2}{Q_2} \right)\\
			&\cdot \mu(d_1)\mu(d_2)\frac{\mu(\gamma_1\gamma_2\gamma_3)}{\gamma_1\gamma_2\gamma_3}\frac{1}{r_1a_1^2}\frac{\phi(x_2y_2r_2k_2)}{x_2y_2r_2^2k_2^3}  \sum_{\substack{\nu_1\mid d_1d_2\\ (\nu_1,x_1x_2y_1y_2\gamma_1\gamma_2\gamma_3\kappa_1\kappa_2\kappa_3r_1r_2a_1k_2)=1}} \frac{\mu(\nu_1)}{\nu_1} \sum_{\substack{\gamma_1\kappa_1\nu_2\mid d_1d_2x_1x_2\\ (\nu_2,\gamma_2\gamma_3\kappa_2\kappa_3r_2k_2)=1}}\nonumber \\
			&\cdot \frac{\mu(y_1y_2\gamma_1\kappa_1\nu_2)}{y_1y_2\gamma_1\kappa_1\nu_2}\sum_{\substack{\nu_3\mid d_1d_2x_1x_2y_1y_2\\ (\nu_3,r_1r_2a_1k_2)=1}}\frac{\mu(\nu_3)}{\nu_3}\sum_{\substack{\kappa_1\kappa_2\kappa_3\nu_4\mid \gamma_1\gamma_2\gamma_3\nu_3\\ (\nu_4,r_1r_2a_1k_2)=1}}\frac{\mu(\kappa_1\kappa_2\kappa_3\nu_4)}{\kappa_1\kappa_2\kappa_3\nu_4}\lambda_f\left(\frac{\gamma_1\gamma_2\gamma_3\nu_3}{\kappa_1\kappa_2\kappa_3\nu_4}\right)	\sum_{\substack{w\mid 6d_2x_2y_2a_2r_2k_2\\ (w,\nu_1\nu_3\nu_4)=1}}\mu(w)
			\nonumber \\
			&\cdot\sumfour_{\substack{e,g, n_1,n_2 \\  (\gamma_3\kappa_3n_1, d_1d_2x_1x_2y_1y_2r_2k_2) = 1\\g\neq 0,\ n_2\neq 0\\ (\gamma_2\kappa_2g,y_1y_2\gamma_3\kappa_3r_2k_2)=1\\ (n_2,\gamma_1\gamma_2\gamma_3\kappa_1\kappa_2\kappa_3r_1r_2a_1k_2w)=1}}
			\lambda_f(e)\lambda_f(\gamma_3\kappa_3n_1)\sum_{(k_1,\nu_1\nu_3\nu_4n_2)=1}\frac{1}{k_1^2}S\left(e,-\overline{\nu_1\nu_3\nu_4n_2}\nu_2gn_1;k_1a_1wk_2\right)\nonumber\\
			&\cdot\Psi_1\bfrac{d_1x_1y_1\gamma_1\gamma_2\gamma_3\kappa_1\kappa_2\kappa_3r_1a_1wk_1}{Q_1}\Phi_V\left(\frac{eE}{\nu_3\nu_4\gamma_1\gamma_2\gamma_3\kappa_1\kappa_2\kappa_3k_1^2a_1^2w^2k_2^2}\right)V\bfrac{\gamma_3\kappa_3n_1}{N_1}\nonumber\\
			&\cdot \widehat{V}\left(\frac{gG}{\nu_1y_1y_2\gamma_1\gamma_3\kappa_1\kappa_3r_1r_2k_1a_1wk_2}\right) \widehat{V}\left(\frac{n_2N_2}{\nu_2y_1y_2\gamma_1^2\gamma_2\gamma_3\kappa_1^2\kappa_2\kappa_3r_1r_2k_1a_1wk_2}\right) .\nonumber
		\end{align}
		Now, let $k_1=\widetilde{c}$, $e=\widetilde{m}$ and group  $\nu_2gn_1=\widetilde{n}$, $wa_1k_2=\widetilde{s}$ and $\nu_1\nu_3\nu_4n_2=\widetilde{r}$. We first truncate the ranges of these variables. Using the decay properties of $\Phi_V$ and the support of $\Psi_1$ and $\Psi_2$, up to a negligible error term, we can restrict ourselves to
		$$
		\widetilde{m}\ll \frac{\nu_3\nu_4\gamma_1\gamma_2\gamma_3\kappa_1\kappa_2\kappa_3k_1^2a_1^2w^2k_2^2}{E}Q^\varepsilon\ll Q^{2+\varepsilon}\frac{\nu_3\nu_4}{E(d_1d_2x_1x_2y_1y_2r_1r_2)^2\gamma_1\gamma_2\gamma_3\kappa_1\kappa_2\kappa_3},
		$$
		and $$
		\widetilde{n}\ll \frac{\nu_2N_1}{\gamma_3\kappa_3}\frac{\nu_1y_1y_2\gamma_1\gamma_3\kappa_1\kappa_3r_1r_2k_1a_1wk_2}{G}Q^\varepsilon\ll \frac{\nu_1\nu_2N_1Q^{1+\varepsilon}}{Gd_1d_2x_1x_2\gamma_2\gamma_3\kappa_2\kappa_3},
		$$
		and $$
		\widetilde{r}\ll \nu_1\nu_3\nu_4\frac{\nu_2y_1y_2\gamma_1^2\gamma_2\gamma_3\kappa_1^2\kappa_2\kappa_3r_1r_2k_1a_1wk_2}{N_2}Q^\varepsilon\ll \frac{\nu_1\nu_2\nu_3\nu_4\gamma_1\kappa_1Q^{1+\varepsilon}}{N_2d_1d_2x_1x_2}.
		$$
		Next we split variables dyadically, so that $d_i\sim D_i$, $x_i\sim X_i$, $y_i\sim Y_i$ ,$r_i\sim R_i$, $\gamma_i\sim \Gamma_i$, $\kappa_i\sim \mathcal{K}_i$, $a_i\sim A_i$, $k_2\sim K_2$, $\nu_i\sim V_i$, and $w\sim W$. Then
		$$
		\widetilde{m}\asymp \widetilde{M}\in\left[1,\frac{Q^{2+\varepsilon}V_3V_4}{E(D_1D_2X_1X_2Y_1Y_2R_1R_2)^2\Gamma_1\mathcal{K}_1\Gamma_2\mathcal{K}_2\Gamma_3\mathcal{K}_3}\right],\ \ \widetilde{c}\asymp \widetilde{C}=\frac{Q_1}{D_1 X_1Y_1\Gamma_1\Gamma_2\Gamma_3\mathcal{K}_1\mathcal{K}_2\mathcal{K}_3R_1A_1W},
		$$
		$$
		\widetilde{n}\asymp \widetilde{N}\in  [\frac14 V_2\frac{N_1}{\Gamma_3\mathcal{K}_3},\frac{V_1V_2N_1Q^{1+\varepsilon}}{GD_1D_2X_1X_2\Gamma_2\Gamma_3\mathcal{K}_2\mathcal{K}_3} ] ,\ \ \widetilde{s}\asymp\widetilde{S}=WA_1K_2, 
		$$
		and$$ 
		\widetilde{r}\asymp\widetilde{R}\in [\frac 18V_1V_3V_4,\frac{V_1V_2V_3V_4\Gamma_1\mathcal{K}_1Q^{1+\varepsilon}}{N_2D_1D_2X_1X_2}].
		$$
		Returning to our computation, we can use Mellin inversion to separate the variables, and we may omit the detail. Then the contribution of $\mathcal{S}(d_1d_2x_1x_2y_1y_2r_1r_2\gamma_1\gamma_2\gamma_3\kappa_1\kappa_2\kappa_3\leq D)$ is essentially bounded by
		\begin{align}\label{applykuz2}
			Q^\varepsilon &\frac{\sqrt{EGN_2}}{\sqrt{N_1}}\frac{D_1D_2X_1X_2Y_1Y_2R_1R_2\Gamma_1\Gamma_2\Gamma_3\mathcal{K}_1\mathcal{K}_2\mathcal{K}_3}{\widetilde{C}^2A_2}\\
			&\max_{\substack{d_1,d_2,x_1,x_2,y_1,y_2,\gamma_1,\gamma_2,\gamma_3,\kappa_1,\kappa_2,\kappa_3,r_1,r_2\\ d_1d_2x_1x_2y_1y_2r_1r_2\gamma_1\gamma_2\gamma_3\kappa_1\kappa_2\kappa_3\leq D\\ d_1\sim D_1,d_2\sim D_2\\ r_1\sim R_1,r_2\sim R_2\\ x_1\sim X_1,x_2\sim X_2\\ y_1\sim Y_1,y_2\sim Y_2\\ \gamma_i\sim \Gamma_i,\kappa_i\sim \mathcal{K}_i}}\left|\sum_{\substack{\widetilde{r} \sim \widetilde{R} \\ \widetilde{s} \sim \widetilde{S}\\ (\widetilde{r},\widetilde{s})=1}} \sum_{\substack{\widetilde{m} \sim \widetilde{M} \\ \widetilde{n} \sim \widetilde{N}}} a_{\widetilde{m}} b_{\widetilde{n}, \widetilde{r}, \widetilde{s}} \sum_{\substack{\widetilde{c} \sim \widetilde{C} \\(\widetilde{c}, \widetilde{r})=1}} g(\widetilde{c}, \widetilde{m}, \widetilde{n}, \widetilde{r}, \widetilde{s}) S( \pm \widetilde{n}, \widetilde{m} \overline{\widetilde{r}}, \widetilde{s} \widetilde{c}) \right|\nonumber
		\end{align}
		where $g$ is a smooth function satisfying the conditions of Lemma \ref{le:Klo1}, $
		a_{\widetilde{m}}=\lambda_f(\widetilde{m}),
		$
		and
		\begin{align*}
			b_{\widetilde{n}, \widetilde{r}, \widetilde{s}}=\sum_{\substack{\widetilde{n}=\nu_2gn_1\\ \gamma_1\kappa_1\nu_2\mid d_1d_2x_1x_2\\ (\nu_2,r_2\gamma_2\gamma_3\kappa_2\kappa_3)=1\\ (\gamma_2\kappa_2g,r_2\gamma_3\kappa_3y_1y_2)=1\\ (\gamma_3\kappa_3n_1,d_1d_2r_2x_1x_2y_1y_2)=1\\ \nu_2\sim V_2}}&
			\sum_{\substack{\widetilde{r}=\nu_1\nu_3\nu_4n_2\\ \nu_1 \mid d_1d_2\\ \nu_3\mid d_1d_2x_1x_2y_1y_2\\ \kappa_1\kappa_2\kappa_3\nu_4\mid \gamma_1\gamma_2\gamma_3\nu_3\\ (\nu_1,x_1x_2y_1y_2\gamma_1\gamma_2\gamma_3\kappa_1\kappa_2\kappa_3)=1 \\(\nu_1\nu_3\nu_4,r_1r_2)=1 \\ (n_2,r_1r_2\gamma_1\gamma_2\gamma_3\kappa_1\kappa_2\kappa_3)=1 \\ \nu_1\sim V_1, \ \nu_3\sim V_3\\ \nu_4\sim V_4}}
			\sum_{\substack{\widetilde{s}=wa_1k_2\\ w\mid 6d_2x_2y_2a_2r_2k_2\\ (w,\nu_1\nu_3\nu_4n_2)=1\\ (k_2,6d_1x_1y_1r_1a_1\gamma_1\gamma_2\gamma_3\kappa_1\kappa_2\kappa_3\nu_1\nu_2\nu_3\nu_4gn_1n_2)=1\\  (a_1,6d_2x_2y_2r_2\nu_1\nu_3\nu_4n_2)=1\\ a_1\sim A_1,\ k_2\sim K_2\\ \ w\sim W}}\\
			&\cdot\lambda_f(\gamma_3\kappa_3n_1)\frac{\mu(y_1y_2\gamma_1\kappa_1\nu_2)}{y_1y_2\gamma_1^2\gamma_2\gamma_3\kappa_1\nu_2}\frac{\mu(\nu_1)}{\nu_1}\frac{\mu(\nu_3)}{\nu_3}\frac{\mu(\kappa_1\kappa_2\kappa_3\nu_4)}{\kappa_1\kappa_2\kappa_3\nu_4}\\
			&\hskip0.5in\cdot\lambda_f\left(\frac{\gamma_1\gamma_2\gamma_3\nu_3}{\kappa_1\kappa_2\kappa_3\nu_4}\right)\frac{\mu(a_1)}{a_1^3}\frac{\phi(x_2y_2r_2k_2)}{x_2y_2r_2^2k_2^3r_1}\frac{\mu(w)}{w^2}V\left(\frac{\gamma_3\kappa_3n_1}{N_1}\right).
		\end{align*}
		\subsection{Applying the Kloosterman sum bounds}
		Now we are ready to apply Lemma \ref{le:Klo1} to \eqref{applykuz2}. Notice first that
		$$\Vert a_{\widetilde{m}} \Vert_2^2 \ll  \widetilde{M}\ll \frac{Q^{2+\varepsilon}V_3V_4}{E(D_1D_2X_1X_2Y_1Y_2R_1R_2)^2\Gamma_1\mathcal{K}_1\Gamma_2\mathcal{K}_2\Gamma_3\mathcal{K}_3}, $$
		\begin{align*}
			\Vert b_{\widetilde{n}, \widetilde{r}, \widetilde{s}} \Vert_2^2& \ll \frac{Q^\varepsilon}{Y_1^2Y_2^2\Gamma_1^4\mathcal{K}_1^4\Gamma_2^2\mathcal{K}_2^2\Gamma_3^2\mathcal{K}_3^2V_1^2V_2^2V_3^2V_4^2A_1^6R_1^2R_2^2K_2^4W^4}
			\\
			&\cdot \sum_{\widetilde{n}\sim \widetilde{N}}	\sum_{\widetilde{r}\sim \widetilde{R}}	\sum_{\widetilde{s}\sim \widetilde{S}}
			\Bigg|
			\sum_{\substack{\widetilde{n}=\nu_2gn_1\\ \gamma_1\kappa_1\nu_2\mid d_1d_2x_1x_2\\\nu_2\sim V_2}}
			\sum_{\substack{\widetilde{r}=\nu_1\nu_3\nu_4n_2\\ \nu_1 \mid d_1d_2\\  \nu_3\mid d_1d_2x_1x_2y_1y_2\\ \kappa_1\kappa_2\kappa_3\nu_4\mid \gamma_1\gamma_2\gamma_3\nu_3\\ \nu_1\sim V_1, \ \nu_3\sim V_3\\ \nu_4\sim V_4}}
			\sum_{\substack{\widetilde{s}=wa_1k_2\\ w\mid 6d_2x_2y_2a_2r_2k_2\\  a_1\sim A_1,\ k_2\sim K_2,\\ w\sim W}}1
			\Bigg|^2\\
			&\ll \frac{Q^\varepsilon}{Y_1^2Y_2^2\Gamma_1^4\mathcal{K}_1^4\Gamma_2^2\mathcal{K}_2^2\Gamma_3^2\mathcal{K}_3^2V_1^2V_2^2V_3^2V_4^2A_1^6R_1^2R_2^2K_2^4W^4} \frac{\widetilde{N}}{V_2}\frac{\widetilde{R}}{V_1V_3V_4}\frac{\widetilde{S}}{W}\\
			&\ll \frac{Q^{2+\varepsilon}N_1}{GN_2K_2^3(D_1D_2X_1X_2Y_1Y_2R_1R_2)^2(\Gamma_1\Gamma_2\Gamma_3\mathcal{K}_1\mathcal{K}_2\mathcal{K}_3)^3V_1V_2V_3^2V_4^2A_1^5W^4}
		\end{align*}
		and 
		$$
		\widetilde{C}\asymp \frac{Q_1}{D_1X_1Y_1R_1\Gamma_1\Gamma_2\Gamma_3\mathcal{K}_1\mathcal{K}_2\mathcal{K}_3A_1W}, K_2\asymp \frac{Q_2}{D_2X_2Y_2R_2}.
		$$
		Hence
		$$
		Q^\varepsilon \frac{\sqrt{EGN_2}}{\sqrt{N_1}}\frac{D_1D_2X_1X_2Y_1Y_2R_1R_2\Gamma_1\Gamma_2\Gamma_3\mathcal{K}_1\mathcal{K}_2\mathcal{K}_3}{\widetilde{C}^2A_2}  \Vert a_{\widetilde{m}} \Vert_2	
		\Vert b_{\widetilde{n}, \widetilde{r}, \widetilde{s}} \Vert_2     \widetilde{C}\widetilde{S}\sqrt{\widetilde{R}}
		\ll  \frac{Q^{1+\varepsilon}\sqrt{K_2}}{A_2\sqrt{A_1V_1V_2V_3V_4}}\sqrt{\widetilde{R}}.
		$$
		Futhermore,
		\begin{align*}
			\widetilde{X}:=\frac{\widetilde{C} \widetilde{S} \sqrt{\widetilde{R}}}{4 \pi \sqrt{\widetilde{M} \widetilde{N}}}\gg_{\varepsilon} Q^{-\varepsilon} \frac{\sqrt{EG}}{\sqrt{N_1N_2}}\frac{D_2X_2Y_2R_2K_2}{Q_2}\gg_\varepsilon Q^{-\varepsilon} \frac{\sqrt{EG}}{\sqrt{N_!N_2}}\gg 1
		\end{align*}
		since $EG \asymp M\gg NQ^{\varepsilon} \asymp Q^{\varepsilon} N_1N_2$.
		By Lemma \ref{le:Klo1} $\mathcal{S}(d_1d_2x_1x_2y_1y_2r_1r_2\gamma_1\gamma_2\gamma_3\kappa_1\kappa_2\kappa_3\leq D)$ is essentially 
		\begin{align*}
			&	\ll \frac{Q^{1+\varepsilon}\sqrt{K_2}}{A_2\sqrt{A_1V_1V_2V_3V_4}}\sqrt{\widetilde{R}} \cdot \sqrt{\widetilde{R}\widetilde{S}}     \left(1+\frac{\sqrt{\widetilde{M}}}{\sqrt{\widetilde{R}\widetilde{S}}}\right)
			\left(1+\frac{\sqrt{\widetilde{N}}}{\sqrt{\widetilde{R}\widetilde{S}}}\right)\left(1+\frac{\widetilde{X}^2}{(1+\frac{\widetilde{R}\widetilde{S}}{\widetilde{M}})^2(1+\frac{\widetilde{R}\widetilde{S}}{\widetilde{N}})}\right)^{\theta}\\
			&\ll\frac{Q^{1+\varepsilon}\sqrt{K_2}}{A_2\sqrt{A_1V_1V_2V_3V_4}} \cdot     \left(\sqrt{\widetilde{R}\widetilde{S}} +\sqrt{\widetilde{M}}\right)
			\left(\sqrt{\widetilde{R}}+\frac{\sqrt{\widetilde{N}}}{\sqrt{\widetilde{S}}}\right)\left(1+\frac{\widetilde{X}^2}{(1+\frac{\widetilde{R}\widetilde{S}}{\widetilde{M}})^2(1+\frac{\widetilde{R}\widetilde{S}}{\widetilde{N}})}\right)^{\theta}.
		\end{align*}
		Under Selberg's eigenvalue conjecture, $\theta=0$. We have 
		$\mathcal{S}(d_1d_2x_1x_2y_1y_2r_1r_2\gamma_1\gamma_2\gamma_3\kappa_1\kappa_2\kappa_3\leq D)$ is essentially 
		\begin{align}\label{S2kuzDsmall}
			&\ll \frac{Q^{1+\varepsilon}\sqrt{K_2}}{A_2\sqrt{A_1V_1V_2V_3V_4}} \cdot    \left(\widetilde{R}\sqrt{\widetilde{S}} +\sqrt{\widetilde{M}\widetilde{R}}+
			\sqrt{\widetilde{R}\widetilde{N}}+\frac{\sqrt{\widetilde{M}\widetilde{N}}}{\sqrt{\widetilde{S}}}\right)\nonumber\\
			&\ll \frac{Q^{2+\varepsilon}Q_2\Gamma_1\mathcal{K}_1\sqrt{V_1V_2V_3V_4W}}{N_2A_2D_1D_2^2X_1X_2^2Y_2R_2}+\frac{Q^{2.5+\varepsilon}\sqrt{Q_2V_3V_4}}{\sqrt{EN_2}A_2D_1D_2^2X_1X_2^2Y_1Y_2R_1R_2\sqrt{\Gamma_2\Gamma_3\mathcal{K}_2\mathcal{K}_3Y_2R_2D_1X_1}}\nonumber\\
			&+\frac{Q^{2+\varepsilon}\sqrt{Q_2N_1\Gamma_1\mathcal{K}_1V_1V_2}}{\sqrt{N_2G}D_1D_2X_1X_2}+\frac{Q^{2.5+\varepsilon}\sqrt{N_1}}{\sqrt{EG}}\nonumber\\
			&\ll \frac{Q^{2+\varepsilon}Q_2^{\frac32}D}{N_2}+\frac{Q^{2.5+\varepsilon}\sqrt{Q_2D}}{\sqrt{EN_2}}+\frac{Q^{2+\varepsilon}\sqrt{Q_2N_1}}{\sqrt{N_2G}}+\frac{Q^{2.5+\varepsilon}\sqrt{N_1}}{\sqrt{EG}}.
		\end{align}
		\section{\texorpdfstring{When $d_1d_2r_1r_2\gamma\kappa$ and $d_1d_2x_1x_2y_1y_2r_1r_2\gamma_1\gamma_2\gamma_3\kappa_1\kappa_2\kappa_3$ large}{When the modulus is small}}
		\label{sectionlargesieve}
		In this section, we apply the large sieve inequality, in a manner comparable to \cite[Section~8]{CLMR2}. Owing to the presence of $\lambda_f(e)$, we are able to open the standard Kloosterman sum and apply the large sieve inequality directly, without following the argument in \cite[Section~8]{CLMR2}. This leads to a simplification of the proof.\par
		We first give the proof of Theorem \ref{Largesieve}.
		\begin{proof}
			We only give the proof of the second bound, and the proof of the first one is similar. We have
			\begin{align*}
				&\sum_{q\leq Q}\ \ \sumstar_{b\bmod q}\left|\sum_{m}\sum_{(n,\lambda q)=1}\lambda_f(m)e\left(\frac{m\overline{n}b}{q}\right)V_1\left(\frac{mM}{q^2}\right)V_2\left(\frac{nN}{q}\right) \right|^2\\
				&\leq \sum_{q\leq Q} \sum_{b\bmod q}\left|\sum_{m}\sum_{(n,\lambda q)=1}\lambda_f(m)e\left(\frac{m\overline{n}b}{q}\right)V_1\left(\frac{mM}{q^2}\right)V_2\left(\frac{nN}{q}\right) \right|^2\\
				&=\sum_{q\leq Q}q\sumfour_{\substack{m_1,m_2,n_1,n_2\\ (n_1n_2,\lambda q)=1\\ m_1n_2\equiv m_2n_1\bmod q}}\lambda_f(m_1)\lambda_f(m_2)V_1\left(\frac{m_1M}{q^2}\right)\overline{V_1\left(\frac{m_2M}{q^2}\right)}V_2\left(\frac{n_1N}{q}\right) \overline{V_2\left(\frac{n_2N}{q}\right) }.
			\end{align*}
			Now we let $(m_1,q)=d$, and write $m_1d$ for $m_1$ and $qd$ for $q$, getting the above equals
			\begin{align*}
				=\sum_{d\leq D}\sum_{q\leq\frac{Q}{d}}qd\sumfour_{\substack{m_1,m_2,n_1,n_2\\ (n_1n_2,\lambda qd)=1\\ (m_1,q)=1\\ dm_1n_2\equiv m_2n_1\bmod dq}}\lambda_f(m_1d)\lambda_f(m_2)V_1\left(\frac{m_1M}{q^2d}\right)\overline{V_1\left(\frac{m_2M}{q^2d^2}\right)}V_2\left(\frac{n_1N}{qd}\right) \overline{V_2\left(\frac{n_2N}{qd}\right) },
			\end{align*}
			where $D=\min(Q,\frac{Q^2}{M})$.\par 
			Then we must have $d\mid m_2$, and we have $(\frac{m_2}{d},q)=1$ since $(m_1n_2,q)=1$, and we find that the above equals
			\begin{align*}
				=\sum_{d\leq D}\sum_{q\leq\frac{Q}{d}}qd\sumfour_{\substack{m_1,m_2,n_1,n_2\\ (n_1n_2,\lambda qd)=1\\ (m_1m_2,q)=1\\ m_1n_2\equiv m_2n_1\bmod q}}\lambda_f(m_1d)\lambda_f(m_2d)V_1\left(\frac{m_1M}{q^2d}\right)\overline{V_1\left(\frac{m_2M}{q^2d}\right)}V_2\left(\frac{n_1N}{qd}\right) \overline{V_2\left(\frac{n_2N}{qd}\right) }.
			\end{align*}
			Since now $(m_1m_2n_1n_2,q)=1$, we write $m_1n_2\equiv m_2n_1\bmod q$ as a sum over character $\chi\bmod q$, getting
			\begin{align}\label{Sumoverchi}
				&=\sum_{d\leq D}\sum_{q\leq\frac{Q}{d}}\frac{qd}{\phi(q)}\sum_{\chi \bmod q}\sumfour_{\substack{m_1,m_2,n_1,n_2\\ (n_1n_2,\lambda d )=1}}\lambda_f(m_1d)\lambda_f(m_2d)\chi(m_1)\chi(n_2)\overline{\chi}(m_2)\overline{\chi}(n_1)\\
				&\hskip2in\cdot V_1\left(\frac{m_1M}{q^2d}\right)\overline{V_1\left(\frac{m_2M}{q^2d}\right)}V_2\left(\frac{n_1N}{qd}\right) \overline{V_2\left(\frac{n_2N}{qd}\right) }.\nonumber
			\end{align}
			When $\chi=\chi_0$, the above equals
			\begin{align*}
				&=\sum_{d\leq D}\sum_{q\leq\frac{Q}{d}}\frac{qd}{\phi(q)}\sumfour_{\substack{m_1,m_2,n_1,n_2\\ (n_1n_2,\lambda dq )=1\\ (m_1m_2,q)=1}}\lambda_f(m_1d)\lambda_f(m_2d) V_1\left(\frac{m_1M}{q^2d}\right)\overline{V_1\left(\frac{m_2M}{q^2d}\right)}V_2\left(\frac{n_1N}{qd}\right) \overline{V_2\left(\frac{n_2N}{qd}\right) }.
			\end{align*}
			Next we remove $(m_1,q)=1$ and $(m_2,q)=1$, introducing $\mu(\nu_1)$ and $\mu(\nu_2)$, getting 
			\begin{align*}
				&	=\sum_{\nu_1\leq Q}\sum_{\nu_2\leq Q}\mu(\nu_1)\mu(\nu_2)\sum_{d\leq D}\sum_{q\leq\frac{Q}{d[\nu_1,\nu_2]}}\frac{qd[\nu_1,\nu_2]}{\phi(q[\nu_1,\nu_2])}\sumfour_{\substack{m_1,m_2,n_1,n_2\\ (n_1n_2,\lambda dq[\nu_1,\nu_2] )=1}}\lambda_f(m_1\nu_1d)\lambda_f(m_2\nu_2d)\\
				&\hskip1.5in\cdot V_1\left(\frac{m_1\nu_1M}{q^2[\nu_1,\nu_2]^2d}\right)\overline{V_1\left(\frac{m_2\nu_2M}{q^2[\nu_1,\nu_2]^2d}\right)}V_2\left(\frac{n_1N}{q[\nu_1,\nu_2]d}\right) \overline{V_2\left(\frac{n_2N}{q[\nu_1,\nu_2]d}\right) }.
			\end{align*}
			Next we use the Hecke relation, getting
			$$
			\lambda_f(m_i\nu_id)=\sum_{\substack{u_i\mid \nu_id\\ u_i\mid m_i}}\mu(u_i)\lambda_f(\frac{\nu_id}{u_i})\lambda_f(\frac{m_i}{u_i}).
			$$
			Hence, we get that the above equals
			\begin{align*}
				&	=\sum_{\nu_1\leq Q}\mu(\nu_1)\sum_{\nu_2\leq Q}\mu(\nu_2)\sum_{d\leq D}\sum_{u_1\mid \nu_1d}\mu(u_1)\lambda_f\left(\frac{\nu_1d}{u_1}\right)\sum_{u_2\mid \nu_2d}\mu(u_2)\lambda_f\left(\frac{\nu_2d}{u_2}\right)\sum_{q\leq\frac{Q}{d[\nu_1,\nu_2]}}\frac{qd[\nu_1,\nu_2]}{\phi(q[\nu_1,\nu_2])}\\
				&\hskip1in\cdot\sum_{m_1}\lambda_f(m_1)V_1\left(\frac{m_1u_1\nu_1M}{q^2[\nu_1,\nu_2]^2d}\right)\sum_{m_2}\lambda_f(m_2) \overline{V_1\left(\frac{m_2u_2\nu_2M}{q^2[\nu_1,\nu_2]^2d}\right)}\\
				&\hskip1in\cdot\sum_{(n_1,\lambda dq[\nu_1,\nu_2] )=1}V_2\left(\frac{n_1N}{q[\nu_1,\nu_2]d}\right) \sum_{(n_2,\lambda dq[\nu_1,\nu_2] )=1}\overline{V_2\left(\frac{n_2N}{q[\nu_1,\nu_2]d}\right) }.
			\end{align*}
			We apply Mellin inversion on $V_1$ and $\overline{V_1}$, and then shift the contour to $\text{Re}(s)=\varepsilon$. Since $\widetilde{V}_1(s)\ll |s|^{-A}$, we can truncate the imaginary part $\ll Q^\varepsilon$ with a negligible term. Hence we find that the $m_1$-sum and $m_2$-sum have a bound of size $\ll Q^{\varepsilon}$, and we trivially bound the sum to see that the above is dominated by
			\begin{align*}
				&	\ll \frac{Q^{2+\varepsilon}}{N^2}\sum_{\nu_1\leq Q}\sum_{\nu_2\leq Q}\sum_{d\leq D}\sum_{u_1\mid \nu_1d}\sum_{u_2\mid \nu_2d}\sum_{q\leq \frac{Q}{d[\nu_1,\nu_2]}}d\\
				&\ll 	\frac{Q^{2+\varepsilon}}{N^2}\sum_{\nu_1\leq Q}\sum_{\nu_2\leq Q}\sum_{d\leq D}\sum_{u_1\mid \nu_1d}\sum_{u_2\mid \nu_2d}\frac{Q}{[\nu_1,\nu_2]}\ll \frac{Q^{3+\varepsilon}}{N^2}\min(Q,\frac{Q^2}{M}).
			\end{align*}
			Next we consider the contribution from the non-trivial characters, and we start from \eqref{Sumoverchi}. Now we use the Hecke relation, getting
			$$
			\lambda_f(m_id)=\sum_{\substack{t_i\mid d\\ t_i\mid m_i}}\mu(t_i)\lambda_f(\frac{d}{t_i})\lambda_f(\frac{m_i}{t_i}).
			$$
			Hence, we have $[t_1,t_2]\mid d$, and we obtain the contribution from those non-trivial characters equals
			\begin{align*}
				&=\sum_{t_1}\sum_{t_2}\mu(t_1)\mu(t_2)\sum_{d\leq \frac{Q}{[t_1,t_2]}}\lambda_f\left(\frac{d[t_1,t_2]}{t_1}\right)\lambda_f\left(\frac{d[t_1,t_2]}{t_2}\right)\sum_{q\leq\frac{Q}{d[t_1,t_2]}}\frac{qd[t_1,t_2]}{\phi(q)}\sum_{\substack{\chi \bmod q\\ \chi\neq\chi_0}}\sumfour_{\substack{m_1,m_2,n_1,n_2\\ (n_1n_2,\lambda d[t_1,t_2])=1}}\\
				&\hskip 0.5in\cdot\lambda_f(m_1)\lambda_f(m_2)\chi(t_1)\overline{\chi}(t_2)\chi(m_1)\chi(n_2)\overline{\chi}(m_2)\overline{\chi}(n_1)V_1\left(\frac{m_1t_1M}{q^2d[t_1,t_2]}\right)\overline{V_1\left(\frac{m_2t_2M}{q^2d[t_1,t_2]}\right)}\\
				&\hskip0.5in\cdot V_2\left(\frac{n_1N}{qd[t_1,t_2]}\right) \overline{V_2\left(\frac{n_2N}{qd[t_1,t_2]}\right) }.
			\end{align*}
			Now we aim to apply Mellin inversion and transfrom the $m,n$-sum into $L$ functions and apply the large sieve inequality. However, the $d$-sum would diverge if we proceed in this way. Therefore, we first insert a smooth partition of unity: 
			$$
			1=\sumd_{M_i}V\left(\frac{m_i}{M_i}\right) \text{and}\ 1=\sumd_{N_i}V\left(\frac{n_i}{N_i}\right).
			$$
			Since $V_i(x)\ll |x|^{-A}$, we can bound $M_i\asymp m_i\ll \frac{Q^{2+\varepsilon}}{M}$ and $N_i\asymp n_i\ll \frac{Q^{1+\varepsilon}}{N}$, and the dyadic sum over $M_i$ and $N_i$ has at most $O(\log^4 Q)$ terms.
			
			Then we apply Mellin inversion on all the smooth function, and noting that for suitable $c$
			$$
			\overline{V(m)}=\frac{1}{2\pi i}\int_{(c)}\overline{\widetilde{V}(\overline{s})}m^{-s}ds,
			$$
			we have that
			\begin{align*}
				&=\frac{1}{(2\pi i)^8}\int_{(\varepsilon)}\int_{(\varepsilon)}\int_{(\varepsilon)}\int_{(\varepsilon)}\int_{(3)}\int_{(3)}\int_{(3)}\int_{(3)}\widetilde{V}_1(s_1)\widetilde{V}_2(s_3)\overline{\widetilde{V}_1(\overline{s_2})}\overline{\widetilde{V}_2(\overline{s_4})}\widetilde{V}(z_1)\widetilde{V}(z_2)\widetilde{V}(z_3)\widetilde{V}(z_4)\\
				&\cdot M^{-s_1-s_2}N^{-s_3-s_4}M_1^{z_1}M_2^{z_2}N_1^{z_3}N_2^{z_4}\sum_{t_1}\sum_{t_2}\mu(t_1)\mu(t_2)\sum_{d\leq \frac{Q}{[t_1,t_2]}}\lambda_f\left(\frac{d[t_1,t_2]}{t_1}\right)\lambda_f\left(\frac{d[t_1,t_2]}{t_2}\right)\\
				&\cdot\frac{(d[t_1,t_2])^{1+s_1+s_2+s_3+s_4}}{t_1^{s_1}t_2^{s_2}}\sum_{q\leq\frac{Q}{d[t_1,t_2]}}\frac{q^{1+2s_1+2s_2+s_3+s_4}}{\phi(q)}\sum_{\substack{\chi \bmod q\\ \chi\neq\chi_0}}L(s_1+z_1,f\otimes\chi)L(s_2+z_2,f\otimes\overline{\chi})\\
				&\cdot\frac{L(s_3+z_3,\overline{\chi})}{\prod_{p\mid \lambda d[t_1,t_2]}\left(1-\frac{\overline{\chi}(p)}{p^{s_3+z_3}}\right)^{-1}}\cdot\frac{L(s_4+z_4,\chi)}{\prod_{p\mid \lambda d[t_1,t_2]}\left(1-\frac{\chi(p)}{p^{s_4+z_4}}\right)^{-1}}\chi(t_1)\overline{\chi}(t_2)dz_1dz_2dz_3dz_4ds_1ds_2ds_3ds_4.
			\end{align*}
			We then shift the lines of integration for $z_1, z_2, z_3,$ and $z_4$ to $\text{Re}(z_i)=\tfrac{1}{2}$. Since $\chi \ne \chi_0$, no residues are encountered in the process.
			
			Since $$\prod_{p\mid \lambda d[t_1,t_2]}\left(1-\frac{\chi(p)}{p^{s_i+z_i}}\right)\ll Q^\varepsilon,$$
			we bound trivially and apply Holder's inequality to see that the above is bounded by
			\begin{align*}
				&\ll \int_{(\varepsilon)}\int_{(\varepsilon)}\int_{(\varepsilon)}\int_{(\varepsilon)}\int_{(\frac 12)}\int_{(\frac 12)}\int_{(\frac 12)}\int_{(\frac 12)}\left|\widetilde{V}_1(s_1)\widetilde{V}_2(s_3)\overline{\widetilde{V}_1(\overline{s_2})}\overline{\widetilde{V}_2(\overline{s_4})}\widetilde{V}(z_1)\widetilde{V}(z_2)\widetilde{V}(z_3)\widetilde{V}(z_4)\right|\\
				&\cdot Q^{14\varepsilon}\sqrt{M_1M_2N_1N_2}\sum_{t_1}\sum_{t_2}\sum_{d\leq \frac{Q}{[t_1,t_2]}}\frac{d[t_1,t_2]}{t_1^{\varepsilon}t_2^{\varepsilon}}\sum_{q\leq\frac{Q}{d[t_1,t_2]}}\sum_{\substack{\chi \bmod q\\ \chi\neq\chi_0}}\Bigg(\left|L(s_1+z_1,f\otimes\chi)\right|^4+\left|L(s_2+z_2,f\otimes\overline{\chi})\right|^4\\
				&\cdot \left|L(s_3+z_3,\overline{\chi})\right|^4+\left|L(s_4+z_4,\chi)\right|^4\Bigg)dz_1dz_2dz_3dz_4ds_1ds_2ds_3ds_4.
			\end{align*}
			Since $\widetilde{V_i}(s)\ll |s|^{-A}$, we can truncate the imaginary parts of all variables at $\ll Q^{\varepsilon}$ with a negligible term. After transfroming the sum over $\chi$ into a sum over primitive character (which can be seen in ~\cite[Proposition~3.2]{CLMR}), we can apply the large sieve inequality (analogously to~\cite[Theorem 7.34]{IK}), and we get the above is bounded by
			$$
			Q^{\varepsilon}\sqrt{M_1M_2N_1N_2}\sum_{t_1}\sum_{t_2}\sum_{d\leq \frac{Q}{[t_1,t_2]}}\frac{Q^2}{d[t_1,t_2]}\ll \frac{Q^{5+\varepsilon}}{MN}.
			$$
		\end{proof}
		Now we are ready to get our bound. We start from \eqref{Start} and \eqref{Secondlargesievebegin}. Firstly, we note that
		$$
		S(e,\overline{\nu_1\nu_2g}n;k_1k_2)=S(e\overline{g},\overline{\nu_1\nu_2}n;k_1k_2)
		$$
		and $$S(e,-\overline{\nu_1\nu_3\nu_4n_2}\nu_2gn_1;k_1k_2)=S(e\overline{n_2},-\overline{\nu_1\nu_3\nu_4}\nu_2gn_1;k_1k_2),$$
		then we open up the Kloosterman sum in \eqref{Start} and \eqref{Secondlargesievebegin}, getting 
		\begin{align}
			&	\mathcal{S}(Q^{1-\delta}>d_1d_2r_1r_2\gamma\kappa>D)\ll Q^\varepsilon\frac{\sqrt{E}}{\sqrt{GN}}\sumsix_{\substack{d_1,d_2,r_1,r_2,\gamma,\kappa\\ Q^{1-\delta}>d_1d_2r_1r_2\gamma\kappa>D}}\sumtwo_{k_1,k_2}\frac{1}{k_1k_2}\Psi_1\bfrac{d_1\gamma \kappa r_1k_1}{Q_1}\Psi_2\left(\frac{d_2r_2k_2}{Q_2} \right)\\
			&\cdot \sum_{\substack{\nu_1\mid d_1d_2\\ (\nu_1,k_1k_2)=1}}\frac{1}{\nu_1}\sum_{\substack{\kappa\nu_2\mid \gamma\nu_1\\ (\nu_2,k_1k_2)=1}}\frac{1}{\nu_2}\  \cdot\sumstar_{b\bmod k_1k_2}\Bigg| \sumtwo_{\substack{e,g\\ (g,d_1d_2r_1r_2\gamma\kappa k_1k_2)=1}}\lambda_f(e)e\left(\frac{e\overline{g}b}{k_1k_2}\right)\Phi_V\left(\frac{eE}{\nu_1\nu_2\gamma\kappa k_1^2k_2^2}\right)\nonumber\\ 
			&\cdot V\bfrac{g}{G}\Bigg|\cdot \left|\sum_{(n,d_1d_2r_2k_2)=1}1*\lambda_f(\gamma\kappa n)e\left(\frac{\overline{b\nu_1\nu_2}n}{k_1k_2}\right)V\bfrac{\gamma\kappa n}{N}\right|
			\nonumber
		\end{align}
		and
		\begin{align}
			&	\mathcal{S}(d_1d_2x_1x_2y_1y_2r_1r_2\gamma_1\gamma_2\gamma_3\kappa_1\kappa_2\kappa_3> D)\ll Q^\varepsilon\frac{\sqrt{EG}N_2}{\sqrt{N}}\int_{\varepsilon-iQ^\varepsilon}^{\varepsilon+iQ^\varepsilon} \sumonefour_{\substack{d_1,d_2,x_1,x_2,y_1,y_2,\gamma_1,\gamma_2,\gamma_3,\kappa_1,\kappa_2,\kappa_3,r_1,r_2\\ d_1d_2x_1x_2y_1y_2r_1r_2\gamma_1\gamma_2\gamma_3\kappa_1\kappa_2\kappa_3> D}}\\ \nonumber
			&\cdot\frac{\left|\widetilde{\widehat{V}}(s) \right|}{y_1y_2r_1r_2\gamma_1^2\gamma_2\gamma_3\kappa_1^2\kappa_2\kappa_3}\sumtwo_{k_1,k_2}\frac{1}{k_1^2k_2^2} \Psi_1\bfrac{d_1x_1y_1\gamma_1\gamma_2\gamma_3\kappa_1\kappa_2\kappa_3r_1k_1}{Q_1}\Psi_2\left(\frac{d_2x_2y_2r_2k_2}{Q_2} \right)\sum_{\substack{\nu_1\mid d_1d_2\\ (\nu_1,k_1k_2)=1}} \frac{1}{\nu_1}\\ \nonumber
			&\cdot \sum_{\substack{\gamma_1\kappa_1\nu_2\mid d_1d_2x_1x_2\\ (\nu_2,k_2)=1}}\frac{1}{\nu_2}\sum_{\substack{\nu_3\mid d_1d_2x_1x_2y_1y_2\\ (\nu_3,k_1k_2)=1}}\frac{1}{\nu_3}\sum_{\substack{\kappa_1\kappa_2\kappa_3\nu_4\mid \gamma_1\gamma_2\gamma_3\nu_3\\ (\nu_4,k_1k_2)=1}}\frac{1}{\nu_4}\cdot\ \ \sumstar_{b\bmod k_1k_2}\Bigg|\sumtwo_{\substack{e,n_2\\n_2\neq 0\\ (n_2,\gamma_1\gamma_2\gamma_3\kappa_1\kappa_2\kappa_3r_1r_2k_1k_2)=1}}\lambda_f(e)\\\nonumber
			&\cdot e\left(\frac{e\overline{n_2}b}{k_1k_2}\right)\Phi_V\left(\frac{eE}{\nu_3\nu_4\gamma_1\gamma_2\gamma_3\kappa_1\kappa_2\kappa_3k_1^2k_2^2}\right)\widehat{V}\left(\frac{n_2N_2}{\nu_2y_1y_2\gamma_1^2\gamma_2\gamma_3\kappa_1^2\kappa_2\kappa_3r_1r_2k_1k_2}\right) \Bigg| 
			\nonumber \\
			&\cdot\Bigg|\sumtwo_{\substack{g,n_1\\ g\neq0 \\ (g,y_1y_2r_2\gamma_3\kappa_3k_2)=1\\ (n_1,d_1d_2x_1x_2y_1y_2r_2k_2)=1\\ |g|\leq G^\prime}}\frac{1}{g^s} \lambda_f(\gamma_3\kappa_3n_1)e\left(\frac{-\overline{b\nu_1\nu_3\nu_4}\nu_2gn_1}{k_1k_2}\right) V\bfrac{\gamma_3\kappa_3n_1}{N_1}\Bigg| ds,\nonumber
		\end{align}
		where $$G^\prime=\frac{Q^{1+\varepsilon}\nu_1y_1y_2\gamma_1\gamma_3\kappa_1\kappa_3r_1r_2}{Gd_1d_2x_1x_2y_1y_2r_1r_2\gamma_1\gamma_2\gamma_3\kappa_1\kappa_2\kappa_3}.$$
		In $\mathcal{S}(d_1d_2x_1x_2y_1y_2r_1r_2\gamma_1\gamma_2\gamma_3\kappa_1\kappa_2\kappa_3> D)$, we truncate the $g$-sum at the above bound and apply Mellin inversion to $\widehat{V}(\cdot)$, which enables a subsequent application of the large sieve inequality.
		\par 
		Then we apply the Cauchy--Schwarz inequality, getting
		\begin{align}
			&	\mathcal{S}(Q^{1-\delta}>d_1d_2r_1r_2\gamma\kappa>D)\ll Q^\varepsilon\frac{\sqrt{E}}{\sqrt{GN}}\sumsix_{\substack{d_1,d_2,r_1,r_2,\gamma,\kappa\\ Q^{1-\delta}>d_1d_2r_1r_2\gamma\kappa>D}}\sumtwo_{k_1,k_2}\frac{1}{k_1k_2}\Psi_1\bfrac{d_1\gamma \kappa r_1k_1}{Q_1}\Psi_2\left(\frac{d_2r_2k_2}{Q_2} \right)\\
			&\cdot \sum_{\substack{\nu_1\mid d_1d_2\\ (\nu_1,k_1k_2)=1}}\frac{1}{\nu_1}\sum_{\substack{\kappa\nu_2\mid \gamma\nu_1\\ (\nu_2,k_1k_2)=1}}\frac{1}{\nu_2}\  \cdot\Bigg( \ \ \sumstar_{b\bmod k_1k_2}\Bigg| \sumtwo_{\substack{e,g\\ (g,d_1d_2r_1r_2\gamma\kappa k_1k_2)=1}}\lambda_f(e)e\left(\frac{e\overline{g}b}{k_1k_2}\right)\Phi_V\left(\frac{eE}{\nu_1\nu_2\gamma\kappa k_1^2k_2^2}\right)\nonumber\\ 
			&\cdot V\bfrac{g}{G}\Bigg|^2\Bigg)^{\frac 12}\cdot \left(\ \sumstar_{b\bmod k_1k_2}\left|\sum_{(n,d_1d_2r_2k_2)=1}1*\lambda_f(\gamma\kappa n)e\left(\frac{bn}{k_1k_2}\right)V\bfrac{\gamma\kappa n}{N}\right|^2\right)^{\frac12}.
			\nonumber
		\end{align}
		Next we remove $(\nu_1\nu_2,k_1k_2)=1$, and apply the Cauchy--Schwarz inequality again, getting
		\begin{align}\label{S1large}
			\mathcal{S}(Q^{1-\delta}>d_1d_2r_1r_2\gamma\kappa>D)\ll Q^\varepsilon\frac{\sqrt{E}}{\sqrt{GN}}\sumsix_{\substack{d_1,d_2,r_1,r_2,\gamma,\kappa\\ Q^{1-\delta}>d_1d_2r_1r_2\gamma\kappa>D}}\frac{d_1d_2r_1r_2\gamma\kappa}{Q}\sum_{\substack{\nu_1\mid d_1d_2}}\frac{1}{\nu_1}\sum_{\substack{\kappa\nu_2\mid \gamma\nu_1}}\frac{1}{\nu_2}\cdot\sqrt{S_1S_2},
		\end{align}
		where
		$$
		S_1:=\sum_{q\leq \frac{Q}{d_1d_2r_1r_2\gamma\kappa}}\ \sumstar_{b\bmod q}\left| \sumtwo_{\substack{e,g\\ (g,d_1d_2r_1r_2\gamma\kappa q)=1}}\lambda_f(e)e\left(\frac{e\overline{g}b}{q}\right)\Phi_V\left(\frac{eE}{\nu_1\nu_2\gamma\kappa q^2}\right)V\bfrac{g}{G}\right|^2
		$$
		and
		$$ 
		S_2:=\sumtwo_{\substack{k_1,k_2\\ k_1k_2\leq \frac{Q}{d_1d_2r_1r_2\gamma\kappa}}} \ \sumstar_{b\bmod k_1k_2}\left|\sum_{(n,d_1d_2r_2k_2)=1}1*\lambda_f(\gamma\kappa n)e\left(\frac{bn}{k_1k_2}\right)V\bfrac{\gamma\kappa n}{N}\right|^2.
		$$
		We treat $	\mathcal{S}(d_1d_2x_1x_2y_1y_2r_1r_2\gamma_1\gamma_2\gamma_3\kappa_1\kappa_2\kappa_3> D)$ with the same argument, getting
		\begin{align}\label{S2large}
			&	\mathcal{S}(d_1d_2x_1x_2y_1y_2r_1r_2\gamma_1\gamma_2\gamma_3\kappa_1\kappa_2\kappa_3> D)\ll
			Q^\varepsilon\frac{\sqrt{EGN_2}}{\sqrt{N_1}}\sumonefour_{\substack{d_1,d_2,x_1,x_2,y_1,y_2,\gamma_1,\gamma_2,\gamma_3,\kappa_1,\kappa_2,\kappa_3,r_1,r_2\\ d_1d_2x_1x_2y_1y_2r_1r_2\gamma_1\gamma_2\gamma_3\kappa_1\kappa_2\kappa_3> D}}\\\nonumber
			&	\cdot\int_{\varepsilon-iQ^\varepsilon}^{\varepsilon+iQ^\varepsilon} \frac{\left|\widetilde{\widehat{V}}(s)\right|}{y_1y_2r_1r_2\gamma_1^2\gamma_2\gamma_3\kappa_1^2\kappa_2\kappa_3}\frac{\left(d_1d_2x_1x_2y_1y_2r_1r_2\gamma_1\gamma_2\gamma_3\kappa_1\kappa_2\kappa_3\right)^2}{Q^2}\sum_{\substack{\nu_1\mid d_1d_2}} \frac{1}{\nu_1}\sum_{\nu_2\mid \frac{d_1d_2x_1x_2}{\gamma_1\kappa_1}}\frac{1}{\nu_2}\\\nonumber
			&\cdot\sum_{\nu_3\mid d_1d_2x_1x_2y_1y_2}\frac{1}{\nu_3}\sum_{\nu_4\mid \frac{\gamma_1\gamma_2\gamma_3\nu_3}{\kappa_1\kappa_2\kappa_3}}\frac{1}{\nu_4}\sqrt{S_3S_4}ds,\nonumber
		\end{align}
		where 
		
		\begin{align}
			S_3&:= \sum_{q\leq \frac{Q}{d_1d_2x_1x_2y_1y_2r_1r_2\gamma_1\gamma_2\gamma_3\kappa_1\kappa_2\kappa_3}}\ \sumstar_{b\bmod q}\Bigg|\sumtwo_{\substack{e,n_2\\n_2\neq 0\\ (n_2,\gamma_1\gamma_2\gamma_3\kappa_1\kappa_2\kappa_3r_1r_2q)=1}}\lambda_f(e)e\left(\frac{e\overline{n_2}b}{q}\right)\\\nonumber
			&\cdot \Phi_V\left(\frac{eE}{\nu_3\nu_4\gamma_1\gamma_2\gamma_3\kappa_1\kappa_2\kappa_3q^2}\right)\widehat{V}\left(\frac{n_2N_2}{\nu_2y_1y_2\gamma_1^2\gamma_2\gamma_3\kappa_1^2\kappa_2\kappa_3r_1r_2q}\right) \Bigg| ^2
			\nonumber
		\end{align}
		and \begin{align}
			S_4&:=\sumtwo_{\substack{k_1,k_2\\ k_1k_2\leq  \frac{Q}{d_1d_2x_1x_2y_1y_2r_1r_2\gamma_1\gamma_2\gamma_3\kappa_1\kappa_2\kappa_3}}}
			\ \sumstar_{b\bmod k_1k_2}\Bigg| \sumtwo_{\substack{g,n_1\\ g\neq0 \\ (g,y_1y_2r_2\gamma_3\kappa_3k_2)=1\\ (n_1,d_1d_2x_1x_2y_1y_2r_2k_2)=1\\ |g|\leq G^\prime}}\frac{1}{g^s}\lambda_f(\gamma_3\kappa_3n_1)\\\nonumber
			&\cdot  e\left(\frac{b\nu_2gn_1}{k_1k_2}\right) V\bfrac{\gamma_3\kappa_3n_1}{N_1}\Bigg|^2.
		\end{align}
		We apply Lemma \ref{Largesieve} to handle $S_1$ and $S_3$. It therefore suffices to bound $S_2$ and $S_4$. We begin with $S_2$. Removing $(n,k_2)=1$, we introduce $\mu(x)$ and then apply the Cauchy--Schwarz inequality and the large sieve inequality to obtain
		\begin{align*}
			S_2&\leq \sumtwo_{\substack{k_1,k_2\\ k_1k_2\leq \frac{Q}{d_1d_2r_1r_2\gamma\kappa}}} \ \sumstar_{b\bmod k_1k_2}\left|\sum_{x\mid k_2}1\cdot\left|\sum_{(n,d_1d_2r_2)=1}1*\lambda_f(\gamma\kappa xn)e\left(\frac{bn}{k_1\frac{k_2}{x}}\right)V\bfrac{\gamma\kappa xn}{N}\right|\right|^2\\
			&\ll Q^\varepsilon\sumtwo_{\substack{k_1,k_2\\ k_1k_2\leq \frac{Q}{d_1d_2r_1r_2\gamma\kappa}}} \ \sumstar_{b\bmod k_1k_2}\sum_{x\mid k_2}\left|\sum_{(n,d_1d_2r_2)=1}1*\lambda_f(\gamma\kappa xn)e\left(\frac{bn}{k_1\frac{k_2}{x}}\right)V\bfrac{\gamma\kappa xn}{N}\right|^2\\
			&=Q^\varepsilon\sum_{x\leq Q}\sumtwo_{\substack{k_1,k_2\\ k_1k_2\leq \frac{Q}{d_1d_2r_1r_2\gamma\kappa x}}}\frac{\phi(k_1k_2x)}{\phi(k_1k_2)} \ \sumstar_{b\bmod k_1k_2}\left|\sum_{(n,d_1d_2r_2)=1}1*\lambda_f(\gamma\kappa xn)e\left(\frac{bn}{k_1k_2}\right)V\bfrac{\gamma\kappa xn}{N}\right|^2\\
			&\ll Q^\varepsilon\sum_{x\leq Q}x\sum_{\substack{q\leq \frac{Q}{d_1d_2r_1r_2\gamma\kappa x}}} \ \sumstar_{b\bmod q}\left|\sum_{(n,d_1d_2r_2)=1}1*\lambda_f(\gamma\kappa xn)e\left(\frac{bn}{q}\right)V\bfrac{\gamma\kappa xn}{N}\right|^2\\
			&\ll Q^\varepsilon\sum_{x\leq Q}x\left(\frac{Q^2}{\left(d_1d_2r_1r_2\gamma\kappa\right)^2x^2}+\frac{N}{\gamma\kappa x}\right)\cdot\frac{N}{\gamma\kappa x}\ll \frac{Q^{2+\varepsilon}N}{\left(d_1d_2r_1r_2\gamma\kappa\right)^2\gamma\kappa}+\frac{Q^\varepsilon N^2}{\left(\gamma\kappa\right)^2}.
		\end{align*}
		Next we deal with $S_4$. We remove $(gn_1,k_2)=1$, introducing $\mu(w)$. Since $w\mid gn_1$, we split $w=w_1w_2$ with $w_1\mid g$, $w_2\mid n_1$ and $(w_2,\frac{g}{w_1})=1$, getting
		\begin{align*}
			S_4&=\sumtwo_{\substack{k_1,k_2\\ k_1k_2\leq  \frac{Q}{d_1d_2x_1x_2y_1y_2r_1r_2\gamma_1\gamma_2\gamma_3\kappa_1\kappa_2\kappa_3}}}
			\ \sumstar_{b\bmod k_1k_2}\Bigg| \sumtwo_{\substack{w_1,w_2\\ w_1w_2\mid k_2\\ (w_1,y_1y_2r_2\gamma_3\kappa_3)=1\\ (w_2,d_1d_2x_1x_2y_1y_2r_2)=1}}\frac{\mu(w_1w_2)}{w_1^s}\sumtwo_{\substack{g,n_1\\ g\neq0 \\ (g,y_1y_2r_2\gamma_3\kappa_3w_2)=1\\ (n_1,d_1d_2x_1x_2y_1y_2r_2)=1\\ |g|\leq \frac{G^\prime}{w_1}}}\\\nonumber
			&\cdot  \frac{1}{g^s}\lambda_f(\gamma_3\kappa_3w_2n_1)e\left(\frac{b\nu_2gn_1}{k_1\frac{k_2}{w_1w_2}}\right) V\bfrac{\gamma_3\kappa_3w_2n_1}{N_1}\Bigg|^2.
		\end{align*}
		Next we apply the Cauchy--Schwarz inequality and the large sieve inequality, getting
		\begin{align*}
			S_4&\ll Q^\varepsilon\sumtwo_{\substack{k_1,k_2\\ k_1k_2\leq  \frac{Q}{d_1d_2x_1x_2y_1y_2r_1r_2\gamma_1\gamma_2\gamma_3\kappa_1\kappa_2\kappa_3}}}
			\ \sumstar_{b\bmod k_1k_2}\sumtwo_{\substack{w_1,w_2\\ w_1w_2\mid k_2}}\Bigg| \sumtwo_{\substack{g,n_1\\ g\neq0 \\ (g,y_1y_2r_2\gamma_3\kappa_3w_2)=1\\ (n_1,d_1d_2x_1x_2y_1y_2r_2)=1\\ |g|\leq \frac{G^\prime}{w_1}}}\frac{1}{g^s}\lambda_f(\gamma_3\kappa_3w_2n_1)\\\nonumber
			&\hskip4in\cdot  e\left(\frac{b\nu_2gn_1}{k_1\frac{k_2}{w_1w_2}}\right) V\bfrac{\gamma_3\kappa_3w_2n_1}{N_1}\Bigg|^2\\
			&\leq Q^\varepsilon\sum_{w_1\leq Q}\sum_{w_2\leq Q}\sumtwo_{\substack{k_1,k_2\\ k_1k_2\leq  \frac{Q}{d_1d_2x_1x_2y_1y_2r_1r_2\gamma_1\gamma_2\gamma_3\kappa_1\kappa_2\kappa_3w_1w_2}}}\frac{\phi(k_1k_2w_1w_2)}{\phi(k_1k_2)}
			\ \sumstar_{b\bmod k_1k_2}\Bigg| \sumtwo_{\substack{g,n_1\\ g\neq0 \\ (g,y_1y_2r_2\gamma_3\kappa_3w_2)=1\\ (n_1,d_1d_2x_1x_2y_1y_2r_2)=1\\ |g|\leq \frac{G^\prime}{w_1}}}\frac{1}{g^s}\\\nonumber
			&\hskip3.25in\cdot  \lambda_f(\gamma_3\kappa_3w_2n_1)e\left(\frac{b\nu_2gn_1}{k_1k_2}\right) V\bfrac{\gamma_3\kappa_3w_2n_1}{N_1}\Bigg|^2\\
			&\ll Q^\varepsilon\sum_{w_1\leq Q}w_1\sum_{w_2\leq Q}w_2\sum_{\substack{q\leq  \frac{Q}{d_1d_2x_1x_2y_1y_2r_1r_2\gamma_1\gamma_2\gamma_3\kappa_1\kappa_2\kappa_3w_1w_2}}}
			\ \sumstar_{b\bmod q}\Bigg| \sumtwo_{\substack{g,n_1\\ g\neq0 \\ (g,y_1y_2r_2\gamma_3\kappa_3w_2)=1\\ (n_1,d_1d_2x_1x_2y_1y_2r_2)=1\\ |g|\leq \frac{G^\prime}{w_1}}}\frac{1}{g^s}\lambda_f(\gamma_3\kappa_3w_2n_1)\\\nonumber
			&\hskip4in\cdot  e\left(\frac{b\nu_2gn_1}{q}\right) V\bfrac{\gamma_3\kappa_3w_2n_1}{N_1}\Bigg|^2\\
			&\ll Q^\varepsilon\sum_{w_1\leq Q}w_1\sum_{w_2\leq Q}w_2\left(\frac{Q^2}{\left(d_1d_2x_1x_2y_1y_2r_1r_2\gamma_1\gamma_2\gamma_3\kappa_1\kappa_2\kappa_3\right)^2\left(w_1w_2\right)^2}+\frac{\nu_2G^\prime N_1}{w_1w_2\gamma_3\kappa_3}\right)\cdot\frac{\nu_2G^\prime N_1}{w_1w_2\gamma_3\kappa_3}\\
			&\ll Q^\varepsilon\left(\frac{Q^2\nu_2G^\prime N_1}{\left(d_1d_2x_1x_2y_1y_2r_1r_2\gamma_1\gamma_2\gamma_3\kappa_1\kappa_2\kappa_3\right)^2\gamma_3\kappa_3}+\frac{\left(\nu_2G^\prime N_1\right)^2}{\gamma^2_3\kappa^2_3}\right)\\
			&\ll  \frac{Q^{3+\varepsilon}\nu_1\nu_2y_1y_2r_1r_2\gamma_1\kappa_1 N_1}{G\left(d_1d_2x_1x_2y_1y_2r_1r_2\gamma_1\gamma_2\gamma_3\kappa_1\kappa_2\kappa_3\right)^3}+\frac{Q^{2+\varepsilon}\left(\nu_1\nu_2y_1y_2r_1r_2\gamma_1\kappa_1\right)^2N_1^2}{G^2\left(d_1d_2x_1x_2y_1y_2r_1r_2\gamma_1\gamma_2\gamma_3\kappa_1\kappa_2\kappa_3\right)^2}.
		\end{align*}
		Hence we have obtained a bound of $S_2$ and $S_4$. Combining with the bound of $S_1$ and $S_3$ from Lemma \ref{Largesieve}, we get \eqref{S1large} and \eqref{S2large} are bounded by
		\begin{align}\label{S1kuzmiddle}
			\mathcal{S}(Q^{1-\delta}>d_1d_2r_1r_2\gamma\kappa>D)&\ll Q^\varepsilon\frac{\sqrt{E}}{\sqrt{GN}}\sumsix_{\substack{d_1,d_2,r_1,r_2,\gamma,\kappa\\ Q^{1-\delta}>d_1d_2r_1r_2\gamma\kappa>D}}\frac{d_1d_2r_1r_2\gamma\kappa}{Q}\sum_{\substack{\nu_1\mid d_1d_2}}\frac{1}{\nu_1}\sum_{\substack{\kappa\nu_2\mid \gamma\nu_1}}\frac{1}{\nu_2}\\\nonumber
			&\cdot \left(\frac{Q^2 \sqrt{G\nu_1\nu_2\gamma\kappa}}{\left(d_1d_2r_1r_2\gamma\kappa\right)^2\sqrt{E}}+\frac{Q^{1.5}G\sqrt{\nu_1\nu_2\gamma\kappa}}{\left(d_1d_2r_1r_2\gamma\kappa\right)^{1.5}\sqrt{E}}\right)\cdot\left(\frac{Q\sqrt{N}}{d_1d_2r_1r_2\gamma\kappa\sqrt{\gamma\kappa}}+\frac{N}{\gamma\kappa}\right)\\\nonumber
			&\ll \frac{Q^{2+\varepsilon}}{D}+Q^{1+\varepsilon}\sqrt{N}+\frac{Q^{1.5+\varepsilon}\sqrt{G}}{\sqrt{D}}+Q^{1+\varepsilon-\frac{\delta}{2}}\sqrt{GN},\nonumber
		\end{align}
		and
		\begin{align}\label{S2kuzDlarge}
			&	\mathcal{S}(d_1d_2x_1x_2y_1y_2r_1r_2\gamma_1\gamma_2\gamma_3\kappa_1\kappa_2\kappa_3> D)\ll
			Q^\varepsilon\frac{\sqrt{EGN_2}}{\sqrt{N_1}}\sumonefour_{\substack{d_1,d_2,x_1,x_2,y_1,y_2,\gamma_1,\gamma_2,\gamma_3,\kappa_1,\kappa_2,\kappa_3,r_1,r_2\\ d_1d_2x_1x_2y_1y_2r_1r_2\gamma_1\gamma_2\gamma_3\kappa_1\kappa_2\kappa_3> D}}\\\nonumber
			&	\cdot \frac{1}{y_1y_2r_1r_2\gamma_1^2\gamma_2\gamma_3\kappa_1^2\kappa_2\kappa_3}\frac{\left(d_1d_2x_1x_2y_1y_2r_1r_2\gamma_1\gamma_2\gamma_3\kappa_1\kappa_2\kappa_3\right)^2}{Q^2}\sum_{\substack{\nu_1\mid d_1d_2}} \frac{1}{\nu_1}\sum_{\nu_2\mid \frac{d_1d_2x_1x_2}{\gamma_1\kappa_1}}\frac{1}{\nu_2}\sum_{\nu_3\mid d_1d_2x_1x_2y_1y_2}\frac{1}{\nu_3}\\\nonumber
			&\cdot\sum_{\nu_4\mid \frac{\gamma_1\gamma_2\gamma_3\nu_3}{\kappa_1\kappa_2\kappa_3}}\frac{1}{\nu_4}\cdot \left(\frac{Q^5}{\left(d_1d_2x_1x_2y_1y_2r_1r_2\gamma_1\gamma_2\gamma_3\kappa_1\kappa_2\kappa_3\right)^5}\cdot\frac{\nu_2\nu_3\nu_4y_1y_2r_1r_2\gamma_1^3\gamma_2^2\gamma_3^2\kappa_1^3\kappa_2^2\kappa_3^2}{EN_2}\right)^{\frac12}\\\nonumber
			&\cdot \left( \frac{Q^{1.5}\sqrt{\nu_1\nu_2y_1y_2r_1r_2\gamma_1\kappa_1 N_1}}{\sqrt{G}\left(d_1d_2x_1x_2y_1y_2r_1r_2\gamma_1\gamma_2\gamma_3\kappa_1\kappa_2\kappa_3\right)^{1.5}}+\frac{Q\nu_1\nu_2y_1y_2r_1r_2\gamma_1\kappa_1N_1}{Gd_1d_2x_1x_2y_1y_2r_1r_2\gamma_1\gamma_2\gamma_3\kappa_1\kappa_2\kappa_3}\right)\\\nonumber
			&\ll \frac{Q^{2+\varepsilon}}{D}+\frac{Q^{1.5+\varepsilon}\sqrt{N_1}}{\sqrt{G}}.\nonumber
		\end{align}
		\section{Square root cancellation }
		\label{Squarerootcancellation}
		
		We now start from \eqref{Starterror} since in this section we need $(n,r_1)=1$. We apply Voronoi on $e$, and we record the following lemma.
		\begin{lem}\label{lem:Vorosimple}
			Let $r, n, \lambda \in \mathbb{N}$ with $(n, r) = 1$. Let $V(\cdot)\in \mathcal{C}_{\mathcal{C}}(0, \infty)$. Then
			\begin{align*}
				\sum_{\substack{e\equiv n \bmod r \\ (e, \lambda) = 1}} \lambda_f(e)V\bfrac{e}{E} = \frac{E}{r} &\sum_{\substack{\nu_1\mid \lambda\\(\nu_1,r)=1}}\frac{\mu(\nu_1)}{\nu_1}\sum_{\substack{a_1\mid \nu_1\\ (a_1,r)=1}}\frac{\mu(a_1)}{a_1}\lambda_f\left(\frac{\nu_1}{a_1}\right)\\
				&\hskip1in\cdot\sum_{k\mid r}\frac{1}{k}\sum_{e=1}^{\infty}\lambda_f(e)S(e,\overline{a_1\nu_1}n;k)\Phi_V\left(\frac{eE}{\nu_1a_1k^2}\right),
			\end{align*}
			where $\Phi_V(\cdot)$ is given in \eqref{Besseltrans}.
		\end{lem}
		\begin{proof}
			The proof follows the same strategy of Lemma \ref{lem:Voro}. We include the proof details for completeness.\par 
			We first remove $(e,\lambda)=1$ by the Möbius inversion, introducing $\mu(\nu_1)$. We obtain
			\begin{align*}
				\sum_{\substack{e\equiv n \bmod r \\ (e, \lambda) = 1}} \lambda_f(e)V\bfrac{e}{E} =\sum_{\substack{\nu_1\mid \lambda\\(\nu_1,r)=1}}\mu(\nu_1)\sum_{e\equiv \overline{\nu_1}n\bmod r}\lambda_f(\nu_1e)V\bfrac{\nu_1e}{E}.
			\end{align*}
			Next we use the multiplicativity relations 
			$$
			\lambda_f(m n)=\sum_{d \mid(m, n)} \mu(d) \lambda_f\left(\frac{m}{d}\right) \lambda_f\left(\frac{n}{d}\right) ,
			$$
			getting 
			\begin{align*}
				\sum_{\substack{e\equiv n \bmod r \\ (e, \lambda) = 1}} \lambda_f(e)V\bfrac{e}{E} =\sum_{\substack{\nu_1\mid \lambda\\(\nu_1,r)=1}}\mu(\nu_1)\sum_{\substack{a_1\mid \nu_1\\ (a_1,r)=1}}&\mu(a_1)\lambda_f\left(\frac{\nu_1}{a_1}\right)\\
				&\cdot\sum_{e\equiv \overline{a_1\nu_1}n\bmod r}\lambda_f(e)V\bfrac{\nu_1a_1e}{E}.
			\end{align*}
			Finnally, the condition $e \equiv  \overline{a_1\nu_1} n\bmod r$ can be enforced by inserting the double sum 
			$$\frac{1}{r}\sum_{k\mid r}\ \ \sumstar_{b_1\bmod k}\ex\left(\frac{b_1(\overline{a_1\nu_1}n-e)}{k}\right),$$
			and apply Lemma \ref{lem:voronoi}.
		\end{proof}
		Applying Lemma \ref{lem:Vorosimple} in \eqref{Starterror}, we get
		\begin{align*}
			\calS = \frac{\sqrt{E}}{\sqrt{GN}} &\sumsix_{\substack{d_1,d_2,r_1,r_2,k_1,k_2\\(d_1d_2r_1r_2k_1k_2,6)=1\\(d_1r_1k_1,d_2r_2k_2)=1}}\sum_{\substack{\nu_1\mid d_1d_2\\ (\nu_1,r_1r_2k_1k_2)=1}}\sum_{\substack{a_1\mid \nu_1}}\lambda_f\left(\frac{\nu_1}{a_1}\right) \Psi_1\bfrac{d_1r_1k_1}{Q_1}\Psi_2\left(\frac{d_2r_2k_2}{Q_2} \right)\mu(d_1)\mu(d_2)\\
			&\cdot\frac{\phi(r_1k_1)}{r_1k_1^2}\frac{\phi(r_2k_2)}{r_2k_2^2} \frac{\mu(\nu_1)}{\nu_1}\frac{\mu(a_1)}{a_1}
			\sumthree_{\substack{e, g, n \\ (gn, d_1d_2r_1r_2k_1k_2) = 1}} \lambda_f(e)1*\lambda_f(n)S(e,\overline{a_1\nu_1g}n;k_1k_2)\\ &\cdot\Phi_V\left(\frac{eE}{\nu_1a_1k_1^2k_2^2}\right)  V\bfrac{g}{G} V\bfrac{n}{N}.
		\end{align*}
		Next we aim to apply Poisson summation on $g$ to transform the standard Kloosterman sum into
		$$
		\mathcal{KS}(e,g,n;k_1k_2).
		$$
		Whenever $(en,k_1k_2)=1$, we may rewrite
		$$
		\mathcal{KS}(e,g,n;k_1k_2)
		=
		\mathcal{KS}(egn,1,1;k_1k_2).
		$$
		And then we can group $egn$ for further computation. Since $(n,k_1k_2)=1$, we want to let $(e,k_1k_2)=1$. In order to do this, we let $(e,k_1k_2)=x$. Hence $x\mid k_1k_2$, and we get $x=x_1x_2$ with $x_1\mid k_1$ and $x_2\mid k_2$ since $(k_1,k_2)=1$. We obtain
		\begin{align}\label{KL2ek}
			\calS &= \frac{\sqrt{E}}{\sqrt{GN}} \sumeight_{\substack{d_1,d_2,r_1,r_2,x_1,x_2,k_1,k_2\\(d_1d_2r_1r_2x_1x_2k_1k_2,6)=1\\(d_1r_1x_1k_1,d_2r_2x_2k_2)=1}}\mu(d_1)\mu(d_2)\sum_{\substack{\nu_1\mid d_1d_2\\  (\nu_1,r_1r_2x_1x_2k_1k_2)=1}}\frac{\mu(\nu_1)}{\nu_1}\sum_{\substack{a_1\mid \nu_1}}\frac{\mu(a_1)}{a_1}\lambda_f\left(\frac{\nu_1}{a_1}\right) \\ \nonumber
			&\cdot\Psi_1\bfrac{d_1r_1x_1k_1}{Q_1}\Psi_2\left(\frac{d_2r_2x_2k_2}{Q_2} \right)\frac{\phi(r_1x_1k_1)}{r_1x_1^2k_1^2}\frac{\phi(r_2x_2k_2)}{r_2x_2^2k_2^2}
			\sumthree_{\substack{e, g, n \\ (gn, d_1d_2r_1r_2x_1x_2k_1k_2) = 1\\ (e,k_1k_2)=1}} \lambda_f(x_1x_2e) \\ \nonumber
			&\cdot 1*\lambda_f(n)S(x_1x_2e,\overline{a_1\nu_1g}n;x_1x_2k_1k_2) \Phi_V\left(\frac{eE}{\nu_1a_1x_1x_2k_1^2k_2^2}\right)  V\bfrac{g}{G} V\bfrac{n}{N}.\nonumber
		\end{align}
		Now we aim to simplify $S(x_1x_2e,\overline{a_1\nu_1g}n;x_1x_2k_1k_2)$. 
		We record the following lemma for our use.
		\begin{lem}\label{smalllem}
			Let $(q,2)=1$ and $(b,q)=1$, then $S(xa,b;xq)$ vanishes unless $(x,q)=1$, in which case 
			$$
			S(xa,b;xq)=R_x(b)S(a,b\overline{x};q),
			$$
			where $R_x(b)$ is the Ramanujan sum $$R_x(b)=\sumstar_{y\bmod x}e\left(\frac{by}{x}\right).$$
			If we also have $(b,x)=1$, then $S(xa,b;xq)=\mu(x)S(a,b\overline{x};q)$.
		\end{lem}
		\begin{proof}
			If $(x,q)>1$, then there exists a prime $p$ such that $p\mid x$ and $p\mid q$, and we must have $p>2$ since $(q,2)=1$. Hence $p^2\mid xq$, and we write $xq=p^{t}c$, where $t\geq 2$ and $(p,c)=1$. Since $S(a,b;xy)=S(a\bar{x},b\bar{x};y)S(a\bar{y},b\bar{y};x)$ for $(x,y)=1$, we have 
			$$S(xa,b;xq)=S(xa,b;p^tc)=S(xa\overline{p^t},b\overline{p^t};c)S(xa\overline{c},b\overline{c};p^t).
			$$
			Since $p>2$, $(b\overline{c},p)=1$ and $t\geq 2$, we use Lemma \ref{KL2square} to imply that $S(xa\overline{c},b\overline{c};p^t)$ vanishes unless $(xab\overline{c^2},p)=1$, while we must have $(x,p)=p$, hence $S(xa,b;xq)$ must vanish. Therefore, $(x,q)=1$, and we get 
			$$
			S(xa,b;xq)=S(xa\overline{q},b\overline{q};x)S(a,b\overline{x};q)=S(0,b\overline{q};x)S(a,b\overline{x};q)=R_x(b)S(a,b\overline{x};q).
			$$
		\end{proof}
		We apply Lemma \ref{smalllem} to \eqref{KL2ek}, getting 
		\begin{align}\label{KL2ekfinal}
			\calS &= \frac{\sqrt{E}}{\sqrt{GN}} \sumeight_{\substack{d_1,d_2,r_1,r_2,x_1,x_2,k_1,k_2\\(d_1d_2r_1r_2x_1x_2k_1k_2,6)=1\\(d_1r_1x_1k_1,d_2r_2x_2k_2)=1}}\mu(d_1)\mu(d_2)\mu(x_1)\mu(x_2)\sum_{\substack{\nu_1\mid d_1d_2\\  (\nu_1,r_1r_2x_1x_2k_1k_2)=1}}\frac{\mu(\nu_1)}{\nu_1}\sum_{\substack{a_1\mid \nu_1}}\frac{\mu(a_1)}{a_1}\lambda_f\left(\frac{\nu_1}{a_1}\right) \\ \nonumber
			&\cdot\Psi_1\bfrac{d_1r_1x_1k_1}{Q_1}\Psi_2\left(\frac{d_2r_2x_2k_2}{Q_2} \right)\frac{\phi(r_1x_1k_1)}{r_1x_1^2k_1^2}\frac{\phi(r_2x_2k_2)}{r_2x_2^2k_2^2}
			\sumthree_{\substack{e, g, n \\ (gn, d_1d_2r_1r_2x_1x_2k_1k_2) = 1\\ (e,k_1k_2)=1}} \lambda_f(x_1x_2e) \\ \nonumber
			&\cdot 1*\lambda_f(n)S(e,\overline{a_1\nu_1x_1x_2g}n;k_1k_2) \Phi_V\left(\frac{eE}{\nu_1a_1x_1x_2k_1^2k_2^2}\right)  V\bfrac{g}{G} V\bfrac{n}{N}.\nonumber
		\end{align}
		Next we shall apply Poisson on $g$, and we record a lemma for our use, which is \cite[Lemma 8.1]{CLMR2}.
		\begin{lem}\label{lem:3timespoisson}
			Let $\nu_1, \nu_2, f, n, e, r \in \mathbb{N}$ with $(\nu_1\nu_2, r) = 1$. Let $V(\cdot)\in \mathcal{C}_{\mathcal{C}}(0, \infty)$. Then
			\begin{align*}
				\sum_{\substack{g\\(g, \alpha r) = 1}}S(\overline{\nu_2} f, ne\overline{\nu_1 g}; r) V\bfrac{g}{G}= \sum_{\substack{\nu_3|\alpha \\ (\nu_3, r) = 1}}\frac{\mu(\nu_3)}{\nu_3} \frac{G}{r} \sum_{g} \mathcal {KS}(\overline{\nu_2}f, \overline{\nu_3} g, \overline{\nu_1 } ne;r) \widehat{V}\bfrac{gG}{\nu_3 r}.
			\end{align*}
		\end{lem}

		We apply Lemma \ref{lem:3timespoisson} in \eqref{KL2ekfinal} on $g$, getting that \eqref{KL2ekfinal} equals
		\begin{align}\label{KL3egn}
			\calS &= \frac{\sqrt{EG}}{\sqrt{N}} \sumeight_{\substack{d_1,d_2,r_1,r_2,x_1,x_2,k_1,k_2\\(d_1d_2r_1r_2x_1x_2k_1k_2,6)=1\\(d_1r_1x_1k_1,d_2r_2x_2k_2)=1}}\mu(d_1)\mu(d_2)\mu(x_1)\mu(x_2)\sum_{\substack{\nu_1\mid d_1d_2\\  (\nu_1,r_1r_2x_1x_2k_1k_2)=1}}\frac{\mu(\nu_1)}{\nu_1}\sum_{\substack{a_1\mid \nu_1}}\frac{\mu(a_1)}{a_1}\lambda_f\left(\frac{\nu_1}{a_1}\right) \\ \nonumber
			&\cdot\sum_{\substack{\nu_2\mid d_1d_2r_1r_2x_1x_2\\ (\nu_2,k_1k_2)=1}}\frac{\mu(\nu_2)}{\nu_2}\Psi_1\bfrac{d_1r_1x_1k_1}{Q_1}\Psi_2\left(\frac{d_2r_2x_2k_2}{Q_2} \right)\frac{\phi(r_1x_1k_1)}{r_1x_1^2k_1^3}\frac{\phi(r_2x_2k_2)}{r_2x_2^2k_2^3}
			\sumthree_{\substack{e\geq1, g\in \mathbb{Z}, n \\ (n, d_1d_2r_1r_2x_1x_2k_1k_2) = 1\\ (e,k_1k_2)=1}}  \\ \nonumber
			&\cdot \lambda_f(x_1x_2e)1*\lambda_f(n)\mathcal{KS}(e,\overline{\nu_2}g,\overline{a_1\nu_1x_1x_2}n;k_1k_2) \Phi_V\left(\frac{eE}{\nu_1a_1x_1x_2k_1^2k_2^2}\right) \widehat{V}\bfrac{gG}{\nu_2 k_1k_2}  V\bfrac{n}{N}.\nonumber
		\end{align}
		We split $\calS$ into $\calS(g=0)$ and $\calS(g\neq 0)$. We consider $\calS(g=0)$ first.
		Since $(en,k_1k_2)=1$, we have $$\mathcal{KS}(e,0,\overline{a_1\nu_1x_1x_2}n;k_1k_2)=\mathcal{KS}(1,0,1;k_1k_2)=R_{k_1k_2}^2(1)=\mu^2(k_1k_2).$$
		Then we bound trivially, getting 
		\begin{align}\label{kl3g=0}
			\mathcal{S}(g= 0)\ll Q^\varepsilon\frac{\sqrt{EG}}{Q^2\sqrt{N}}\cdot\frac{Q^2}{E}\cdot N\ll Q^{1+\varepsilon}\frac{\sqrt{EGN}}{E}\ll\frac{Q^{2.5+\varepsilon}}{E}.
		\end{align}

		\par 
		Next we consider $\calS(g\neq 0)$.\\
		Since now $(en,k_1k_2)=1$, we have $$\mathcal{KS}(e,\overline{\nu_2}g,\overline{a_1\nu_1x_1x_2}n;k_1k_2)=\mathcal{KS}(egn\overline{\nu_1a_1\nu_2x_1x_2},1,1;k_1k_2)=k_1k_2\text{Kl}_3(egn\overline{\nu_1a_1\nu_2x_1x_2};k_1k_2).$$
		We now want to group $egn=l$. We may explain that we will interchange the order of summation over $k_1$ and $l$ 
		to apply the Cauchy--Schwarz inequality. 
		However, the presence of $(en,k_1)=1$ 
		creates an obstruction to this step.
		Hence, we shall first remove $(en,k_1)=1$. To do so, we record the following lemma, which can be compared with Lemma  \ref{smalllem}.
		\begin{lem}\label{KL3smalllem}
			Let $(q,6)=1$ and $(bc,q)=1$, then $\mathcal{KS}(xa,b,c;xq)$ vanishes unless $(x,q)=1$, in which case 
			$$
			\mathcal{KS}(xa,b,c;xq)=R_x(b)R_x(c)\mathcal{KS}(a,b\bar{x},c\bar{x};q).
			$$
			If we also have $(b,x)=1$, then $\mathcal{KS}(xa,b,c;xq)=\mu^2(x)\mathcal{KS}(a,b\bar{x},c\bar{x};q)$.
		\end{lem}
		\begin{proof}
			Note that in this case we need to apply Lemma \ref{Kloo}, whose proof will be given in Section \ref{padiccomputation}. The proof is essentially identical to that of Lemma~\ref{smalllem}, 
			and is therefore omitted.
		\end{proof}
		In \eqref{KL3egn}, we remove $(e,k_1)=1$ and $(n,k_1)=1$ respectively, introducing $\mu(y_1)$ and $\mu(y_2)$. We then apply Lemma \ref{KL3smalllem}, getting $	\calS(g\neq 0)$
		\begin{align}\label{KL3egn2}
			&= \frac{\sqrt{EG}}{\sqrt{N}} \sumten_{\substack{d_1,d_2,r_1,r_2,x_1,x_2,y_1,y_2,k_1,k_2\\(d_1d_2r_1r_2x_1x_2y_1y_2k_1k_2,6)=1\\(d_1r_1x_1y_1y_2k_1,d_2r_2x_2k_2)=1\\ (y_2,d_1r_1x_1y_1)=1}}\mu(d_1)\mu(d_2)\mu(x_1)\mu(x_2)\mu(y_1)\mu(y_2)\sum_{\substack{\nu_1\mid d_1d_2\\  (\nu_1,r_1r_2x_1x_2y_1y_2k_1k_2)=1}}\frac{\mu(\nu_1)}{\nu_1}\\ \nonumber
			&\cdot\sum_{\substack{a_1\mid \nu_1}}\frac{\mu(a_1)}{a_1}\lambda_f\left(\frac{\nu_1}{a_1}\right) 
			\sum_{\substack{\nu_2\mid d_1d_2r_1r_2x_1x_2\\ (\nu_2,y_1y_2k_1k_2)=1}}\frac{\mu(\nu_2)}{\nu_2}\Psi_1\bfrac{d_1r_1x_1y_1y_2k_1}{Q_1}\Psi_2\left(\frac{d_2r_2x_2k_2}{Q_2} \right)\frac{\phi(r_1x_1y_1y_2k_1)}{r_1x_1^2y_1^3y_2^3k_1^2}\frac{\phi(r_2x_2k_2)}{r_2x_2^2k_2^2}
			\\ \nonumber
			&\cdot	\sumthree_{\substack{e\geq1, g\in \mathbb{Z}, n \\ (n, d_1d_2r_1r_2x_1x_2y_1y_2k_2) = 1\\ (e,k_2)=1\\ g\neq 0}} \lambda_f(x_1x_2y_1e)1*\lambda_f(y_2n)\text{Kl}_3(egn\overline{\nu_1a_1\nu_2x_1x_2y_1^2y_2^2};k_1k_2) \Phi_V\left(\frac{eE}{\nu_1a_1x_1x_2y_1y_2^2k_1^2k_2^2}\right) \\\nonumber
			&\cdot	\widehat{V}\bfrac{gG}{\nu_2 y_1y_2k_1k_2}  V\bfrac{y_2n}{N}.\nonumber
		\end{align}
		\par
		Now, we can continue our treatment. We first insert a smooth partition of unity that 
		$$
		1=\sumd_{L}V\left(\frac{egn}{L}\right),
		$$
		where the dyadic sum over $L$ has at most $O(\log Q) $ terms.\par 
		Using the fact that $\Phi_V(x),\widehat{V}(x)\ll x^{-A}$, we can truncate 
		\begin{align}\label{RangeL}
			L&\leq 2egn\ll Q^\varepsilon \frac{\nu_1a_1x_1x_2y_1y_2^2k_1^2k_2^2}{E}\cdot\frac{\nu_2y_1y_2k_1k_2}{G}\cdot\frac{N}{y_2}\\
			&\ll Q^\varepsilon\frac{\nu_1a_1\nu_2y_1^2y_2^2x_1x_2N}{EG}\frac{Q^3}{(d_1d_2r_1r_2x_1x_2y_1y_2)^3}.\nonumber
		\end{align}
		Then we apply Mellin inversion to express $\Phi_V$, $\widehat{V}$ and $V$, writing
		\[
		\Phi_V(x)
		= \frac{1}{2\pi i}\int_{(\varepsilon)} \widetilde{\Phi}_V(s_1)\, x^{-s_1}\, ds_1,
		\quad 
		\widehat{V}(x)
		= \frac{1}{2\pi i}\int_{(\varepsilon)} \widetilde{\widehat{V}}(s_2)\, x^{-s_2}\, ds_2,
		\]
		and
		\[
		V(x)
		= \frac{1}{2\pi i}\int_{(\varepsilon)} \widetilde{V}(s_3)\, x^{-s_3}\, ds_3.
		\]
		By repeated integration by parts, we obtain the rapid decay
		\[
		\widetilde{\Phi}_V(s), \ \widetilde{\widehat{V}}(s), \ \widetilde{V}(s)
		\ll_A (1+|s|)^{-A}
		\]
		for any $A>0$. Consequently, the Mellin integrals may be truncated at
		\[
		|\text{Im}(s_1)|, \ |\text{Im}(s_2)|, \ |\text{Im}(s_3)| \ll Q^{\varepsilon},
		\]
		with a negligible error.
		\par 
		Hence for $\calS(g\neq 0)$, we have
		\begin{align}\label{Sgneq0}
			&	\calS(g\neq 0)\ll Q^\varepsilon \frac{\sqrt{EG}}{\sqrt{N}}\int_{\varepsilon-iQ^\varepsilon}^{\varepsilon+iQ^\varepsilon}\int_{\varepsilon-iQ^\varepsilon}^{\varepsilon+iQ^\varepsilon}\int_{\varepsilon-iQ^\varepsilon}^{\varepsilon+iQ^\varepsilon}
			|\widetilde{\Phi}_V(s_1)| |\widetilde{\widehat{V}}(s_2)| |\widetilde{V}(s_3)|\\
			&\cdot	\sumeight_{\substack{d_1,d_2,r_1,r_2,x_1,x_2,y_1,y_2\\(d_1d_2r_1r_2x_1x_2y_1y_2,6)=1\\(d_1r_1x_1y_1y_2,d_2r_2x_2)=1\\ (y_2,d_1r_1x_1y_1)=1}}
			\sum_{\nu_1\mid d_1d_2}\sum_{a_1\mid \nu_1}\sum_{\nu_2\mid d_1d_2r_1r_2x_1x_2}\frac{1}{(x_1x_2\nu_1a_1\nu_2)^{1-\varepsilon}(y_1y_2)^{2-2\varepsilon}}\nonumber\\
			&\cdot \max_{\substack{d_1,d_2,r_1,r_2,x_1,x_2,y_1,y_2,\nu_1,a_1,\nu_2,s_1,s_2,s_3\\ \nu_1\mid d_1d_2,a_1\mid \nu_1,\nu_2\mid d_1d_2r_1r_2x_1x_2\\ \text{Re}s_i=\varepsilon, -Q^\varepsilon\leq\text{Im}s_i\leq Q^\varepsilon }}\left| f(d_1,d_2,r_1,r_2,x_1,x_2,y_1,y_2,\nu_1,a_1,\nu_2,s_1,s_2,s_3)\right|d s_1d s_2d s_3,\nonumber
		\end{align}
		where 
		\begin{align*}
			&	f(d_1,d_2,r_1,r_2,x_1,x_2,y_1,y_2,\nu_1,a_1,\nu_2,s_1,s_2,s_3)=\sum_{(k_2,6d_1r_1x_1y_1y_2\nu_1a_1\nu_2)=1}\frac{\phi(r_2x_2k_2)}{r_2x_2k_2^{2-2s_1-s_2}}\Psi_2\left(\frac{d_2r_2x_2k_2}{Q_2}\right)\\
			&\cdot \sum_{\substack{e=1\\ (e,k_2)=1}}^{\infty}\sum_{\substack{g=-\infty\\g\neq 0}}^\infty\sum_{(n,d_1d_2r_1r_2x_1x_2y_1k_2)=1}\lambda_f(x_1x_2y_1e)e^{-s_1}g^{-s_2}n^{-s_3}1*\lambda_f(y_2n)\sum_{(k_1,6d_2r_2x_2\nu_1a_1\nu_2k_2)=1}\\
			&\cdot \frac{\phi(r_1x_1y_1y_2k_1)}{r_1x_1y_1y_2k_1^{2-2s_1-s_2}}\Psi_1\left(\frac{d_1r_1x_1y_1y_2k_1}{Q_1}\right)\cdot \text{Kl}_3(egn\overline{\nu_1a_1\nu_2x_1x_2y_1^2y_2^2};k_1k_2) V\left(\frac{egn}{L}\right).
		\end{align*}
		Next we group $egn=l$, getting $f$ equals
		\begin{align}\label{Cauchybegin}
			&\sum_{(k_2,6d_1r_1x_1y_1y_2\nu_1a_1\nu_2)=1}\frac{\phi(r_2x_2k_2)}{r_2x_2k_2^{2-2s_1-s_2}}\Psi_2\left(\frac{d_2r_2x_2k_2}{Q_2}\right)\sum_{l}\alpha(l,d_1,d_2,r_1,r_2,x_1,x_2,y_1,k_2)\\ \nonumber
			&\cdot \sum_{(k_1,6d_2r_2x_2\nu_1a_1\nu_2k_2)=1}\frac{\phi(r_1x_1y_1y_2k_1)}{r_1x_1y_1y_2k_1^{2-2s_1-s_2}}\Psi_1\left(\frac{d_1r_1x_1y_1y_2k_1}{Q_1}\right)\cdot\text{Kl}_3(l\overline{\nu_1a_1\nu_2x_1x_2y_1^2y_2^2};k_1k_2) V\left(\frac{l}{L}\right),
		\end{align}
		where 
		$$
		\alpha(l,d_1,d_2,r_1,r_2,x_1,x_2,y_1,k_2)=\sumthree_{\substack{e,g,n\\ l=egn\\ (e,k_2)=1,g\neq 0\\ (n,d_1d_2r_1r_2x_1x_2y_1k_2)=1}}\lambda_f(x_1x_2y_1e)e^{-s_1}g^{-s_2}n^{-s_3}1*\lambda_f(y_2n)
		$$
		satisfyng $|\alpha|\ll Q^\varepsilon$.
		The next lemma we shall apply is \cite[Lemma 7.1]{CLMR2}.
		\begin{lem}\label{lem:firstpoisson}
			Let $r, f, g, n, \lambda \in \mathbb{N}$ with $(fgn, r) = 1$. Then
			\begin{align*}
				\sum_{\substack{e\equiv \overline{fg} n \Mod r \\ (e, \lambda) = 1}} V\bfrac{e}{E} = \frac{E}{r} \sum_{\substack{\nu_1 | \lambda \\ (\nu_1, r) = 1}} \frac{\mu(\nu_1)}{\nu_1} \sum_e \mathrm{e}\bfrac{ne\overline{\nu_1 fg}}{r} \widehat V\bfrac{Ee}{\nu_1 r},
			\end{align*}
		\end{lem}
		where $\widehat{V}(\cdot)$ is defined in \eqref{Fourier}.

		Next in \eqref{Cauchybegin} we apply the Cauchy--Schwarz inequality on $l$, getting $f(\cdots)$ is bounded by
		\begin{align}\label{boundoff}
			&\ll \sqrt{L}\Biggl(\sum_{(k_2,6d_1r_1x_1y_1y_2\nu_1a_1\nu_2)=1}\left|\frac{\phi(r_2x_2k_2)}{r_2x_2k_2^{2-2s_1-s_2}}\Psi_2\left(\frac{d_2r_2x_2k_2}{Q_2}\right)\right|\sum_{l}\V\left(\frac{l}{L}\right)\\
			&\cdot \left|\sum_{(k_1,6d_2r_2x_2\nu_1a_1\nu_2k_2)=1} \frac{\phi(r_1x_1y_1y_2k_1)}{r_1x_1y_1y_2k_1^{2-2s_1-s_2}}\Psi_1\left(\frac{d_1r_1x_1y_1y_2k_1}{Q_1}\right)\cdot\text{Kl}_3(l\overline{\nu_1a_1\nu_2x_1x_2y_1^2y_2^2};k_1k_2) \right|^2\Biggl)^{\frac 12}\nonumber\\
			&\ll Q^\varepsilon \sqrt{L}\Bigg( \sum_{(k_2,6d_1r_1x_1y_1y_2\nu_1a_1\nu_2)=1}\frac{1}{k_2}\Psi_2\left(\frac{d_2r_2x_2k_2}{Q_2}\right)
			\sum_{(k_1,6d_2r_2x_2\nu_1a_1\nu_2k_2)=1}\sum_{(k_1^\prime,6d_2r_2x_2\nu_1a_1\nu_2k_2)=1}\nonumber\\
			&	\cdot \frac{\phi(r_1x_1y_1y_2k_1)}{r_1x_1y_1y_2k_1^{2-2s_1-s_2}}\Psi_1\left(\frac{d_1r_1x_1y_1y_2k_1}{Q_1}\right)\frac{\phi(r_1x_1y_1y_2k_1^\prime)}{r_1x_1y_1y_2k_1^{\prime(2-\overline{2s_1-s_2})}}\Psi_1\left(\frac{d_1r_1x_1y_1y_2k_1^\prime}{Q_1}\right) \nonumber\\
			&\cdot \sum_{l}\text{Kl}_3(l\overline{\nu_1a_1\nu_2x_1x_2y_1^2y_2^2};k_1k_2)\overline{\text{Kl}_3(l\overline{\nu_1a_1\nu_2x_1x_2y_1^2y_2^2};k_1^\prime k_2)}V\left(\frac{l}{L}\right)\Bigg)^{\frac 12},\nonumber
		\end{align}
		where $\text{Kl}_3(u;q)$ is given in \eqref{KL3}.\par 
		We now want to get a bound for
		\begin{equation}\label{Kl3poisoon}
			\sum_{l}\text{Kl}_3(l\overline{\nu_1a_1\nu_2x_1x_2y_1^2y_2^2};k_1k_2)\overline{\text{Kl}_3(l\overline{\nu_1a_1\nu_2x_1x_2y_1^2y_2^2};k_1^\prime k_2)}V\left(\frac{l}{L}\right).
		\end{equation}
		Let $\nu_1a_1\nu_2x_1x_2y_1^2y_2^2=\nu$, then we apply Lemma \ref{lem:firstpoisson} to \eqref{Kl3poisoon}, getting 
		\begin{align}
			&	\sum_{l}\text{Kl}_3(l\overline{\nu};k_1k_2)\overline{\text{Kl}_3(l\overline{\nu};k_1^\prime k_2)}V\left(\frac{l}{L}\right)\\
			&=\sum_{b\bmod k_2[k_1,k_1^\prime]}\text{Kl}_3(b\overline{\nu};k_1k_2)\overline{\text{Kl}_3(b\overline{\nu};k_1^\prime k_2)}\sum_{l\equiv b\bmod(k_2[k_1,k_1^\prime])}V\left(\frac{l}{L}\right)\nonumber \\ 
			&=\frac{L}{k_2[k_1,k_1^\prime]}\sum_h\widehat{V}\left(\frac{hL}{k_2[k_1,k_1^\prime]}\right)
			\sum_{b\bmod k_2[k_1,k_1^\prime]}\text{Kl}_3(b\overline{\nu};k_1k_2)\overline{\text{Kl}_3(b\overline{\nu};k_1^\prime k_2)}\,e\left(\frac{hb}{k_2[k_1,k_1^\prime]}\right)\nonumber.
		\end{align}
		Now, we aim to compute 
		\begin{equation}\label{3}
			\sum_{b\bmod k_2[k_1,k_1^\prime]}\text{Kl}_3(b\overline{\nu};k_2k_1)\overline{\text{Kl}_3(b\overline{\nu};k_2k_1^\prime )}\,\ex\bfrac{hb}{k_2[k_1,k_1^\prime]}.
		\end{equation}
		Take $K_1=\frac{Q_1}{d_1r_1x_1y_1y_2}$ and $K_2=\frac{Q_2}{d_2r_2x_2}$. Then we apply Theorem \ref{Squarerootcancel} on \eqref{boundoff}, and note that $(a,b)^{\#}\leq (a,b)$, getting
		\begin{align*}
			&	f(d_1,d_2,r_1,r_2,x_1,x_2,y_1,y_2,\nu_1,a_1,\nu_2,s_1,s_2,s_3)\\
			&\ll Q^\varepsilon\sqrt{L}\left(\sumthree_{\substack{k_1\sim K_1\\ k_1^\prime\sim K_1\\ k_2\sim K_2}}\frac{1}{k_1k_1^\prime k_2}\frac{L}{k_2[k_1,k_1^\prime]}\sum_h\widehat{V}\left(\frac{hL}{k_2[k_1,k_1^\prime]}\right)
			(k_2[k_1,k_1^\prime])^{1/2}(h,k_2(k_1,k_1^\prime))^{1/2}\right)^{\frac 12}.
		\end{align*}
		For $h=0$, we make a change of variable that $(k_1,k_1^\prime)=t$, getting
		\begin{align}\label{diag}
			f\ll Q^\varepsilon\sqrt{L}\left(\sumthree_{\substack{k_1\sim K_1 \\k_1^\prime \sim K_1\\k_2\sim K_2}}\frac{L}{(k_1k_1^\prime)^{\frac 32}k_2}(k_1,k_1^\prime)\right)^{1/2}\ll 
			Q^\varepsilon\sqrt{L}\left(\sum_{t\leq 2K_1}\frac{1}{t^2}\sumthree_{\substack{k_1\sim K_1/t \\k_1^\prime \sim K_1/t\\k_2\sim K_2}}\frac{L}{(k_1k_1^\prime)^{\frac 32}k_2}\right)^{1/2}\ll Q^\varepsilon \frac{L}{\sqrt{K_1}}.
		\end{align}
		For $h\neq 0$, using the fact that $\widehat{V}(x)\ll x^{-A}$, we make the change of variables $(k_1,k_1')=t$ and set $r=k_2t$. Finally writing $(h,r)=m$, we obtain
		\begin{align}\label{off-dia}
			f&\ll Q^\varepsilon\sqrt{L}\left(\sumthree_{\substack{k_1\sim K_1\\ k_1^\prime\sim K_1\\ k_2\sim K_2}}\frac{1}{k_1k_1^\prime k_2}\sum_{h\neq 0} \frac{1}{h^{1+\varepsilon}}
			(k_2[k_1,k_1^\prime])^{1/2}(h,k_2(k_1,k_1^\prime))^{1/2}\right)^{\frac 12}\\
			&\ll  Q^\varepsilon\sqrt{L}\left(\sum_{t\leq 2K_1}\frac{1}{t}\sumthree_{\substack{k_1\sim K_1/t\\ k_1^\prime\sim K_1/t\\ k_2\sim K_2}}\frac{1}{k_1k_1^\prime k_2t}\sum_{h\neq 0} \frac{1}{h^{1+\varepsilon}}
			(k_2t[k_1,k_1^\prime])^{1/2}(h,k_2t)^{1/2}\right)^{\frac 12}\nonumber\\
			&\ll  Q^\varepsilon\sqrt{L}\left(\sum_{r\leq 4Q}\sumtwo_{\substack{t,k_2\\ t\leq 2K_1\\ k_2\sim K_2\\ r=k_2t}}\frac{1}{t}\sumtwo_{\substack{k_1\sim K_1/t\\ k_1^\prime\sim K_1/t}}\frac{1}{k_1k_1^\prime r}\sum_{h\neq 0} \frac{1}{h^{1+\varepsilon}}
			(r[k_1,k_1^\prime])^{1/2}(h,r)^{1/2}\right)^{\frac 12}\nonumber\\
			&\ll  Q^\varepsilon\sqrt{L}\left(\sum_{m\leq 4Q}\sum_{r\leq 4Q/m}\sumtwo_{\substack{t,k_2\\ t\leq 2K_1\\ k_2\sim K_2\\ rm=k_2t}}\frac{1}{t}\sumtwo_{\substack{k_1\sim K_1/t\\ k_1^\prime\sim K_1/t}}\frac{1}{k_1k_1^\prime rm}\sum_{h\neq 0}\frac{1}{(hm)^{1+\varepsilon}}
			(rm[k_1,k_1^\prime])^{1/2}m^{1/2}\right)^{\frac 12}\nonumber\\
			&\ll  Q^\varepsilon\sqrt{L}\left(\sum_{m\leq 4Q}\sum_{r\leq 4Q/m}\sumtwo_{\substack{t,k_2\\ t\leq 2K_1\\ k_2\sim K_2\\ rm=k_2t}}\frac{K_1}{t^2m\sqrt{r}}\right)^{\frac 12}=Q^\varepsilon\sqrt{L}\left(\sum_{k_2\sim K_2}\sum_{t\leq 2K_1}\sumtwo_{\substack{m,r\\ m\leq 4Q\\ r\leq 4Q/m\\ k_2t=rm}}\frac{K_1}{t^{2.5}\sqrt{mk_2}}\right)^{\frac 12}\nonumber\\
			&\ll Q^\varepsilon\sqrt{L}\left(K_1\sum_{k_2\sim K_2}\frac{1}{\sqrt{k_2}}\right)^{\frac 12}\ll Q^\varepsilon \sqrt{LK_1}(K_2)^{\frac 14}.\nonumber
		\end{align}
		Combining the diagonal term \eqref{diag}, the off-diagonal term \eqref{off-dia} and the range of $L$ in \eqref{RangeL}, we obtain that the $\calS(g\neq 0)$ in \eqref{Sgneq0} is bounded by
		\begin{align}\label{BoundfromSquarerootcancel}
			\calS(g\neq 0)\ll \frac{Q^{2+\varepsilon}\sqrt{NQQ_2}}{\sqrt{EG}}+Q^{2+\varepsilon}Q_2^{-\frac 14}.
		\end{align}
		Now it remains to give the proof of Theorem \ref{Squarerootcancel}.
		\subsection{Set up}
		\par 
		Let $k_1k_2=p_1^{\alpha_1}  p_2^{\alpha_2} \cdots p_s^{\alpha_s}$ and $k_1^\prime k_2=u_1^{\beta_1} u_2^{\beta_2}\cdots u_t^{\beta_t}$  be the canonical factorization of $k_1k_2$ and $k_1^\prime k_2$ with $\alpha_i, \ \beta_i \geq 1$. We assume $(p_1 p_2 \cdots p_s,u_1 u_2\cdots u_t)=p_{i_1} p_{i_2}\cdots p_{i_n}=u_{j_1} u_{j_2}\cdots u_{j_n}$ and without loss of generality we can further assume that $(p_1 p_2 \cdots p_s,u_1 u_2\cdots u_t)=p_{1} p_{2}\cdots p_{n}=u_{1} u_{2}\cdots u_{n}$ and $p_i=u_i$ for $1\leq i\leq n$. Hence, we can rewrite $k_1k_2$ and $k_1^\prime k_2$ as follows:
		$$
		k_1k_2=p_1^{\alpha_1}  p_2^{\alpha_2} \cdots p_n^{\alpha_n}p_{n+1}^{\alpha_{n+1}}\cdots  p_s^{\alpha_s}
		$$
		$$
		k_1^\prime k_2=p_1^{\beta_1} p_2^{\beta_2}\cdots p_n^{\beta_n}u_{n+1}^{\beta _{n+1}} \cdots u_t^{\beta_t},
		$$
		where $(p_i,u_j)=1$ for $i,j\geq n+1$. Therefore, 
		$$k_2[k_1,k_1^\prime]=p_1^{\max\{\alpha_1,\beta_1\}}  p_2^{\max\{\alpha_2,\beta_2\}} \cdots p_n^{\max\{\alpha_n,\beta_n\}}p_{n+1}^{\alpha_{n+1}}\cdots  p_s^{\alpha_s}u_{n+1}^{\beta _{n+1}} \cdots u_t^{\beta_t}.
		$$
		We can also assume without loss of generality that
		$$k_2[k_1,k_1^\prime]=p_1^{\alpha_1}  p_2^{\alpha_2} \cdots p_n^{\alpha_n}p_{n+1}^{\alpha_{n+1}}\cdots  p_s^{\alpha_s}u_{n+1}^{\beta _{n+1}} \cdots u_t^{\beta_t}, \ \alpha_i\geq \beta_i \ \text{for} \ 1\leq i\leq n .
		$$
		We assume $w_i=p_i^{\alpha_i}$ for $1\leq i\leq s$ and $v_i=u_i^{\beta_i}$ for $1\leq i\leq t$,  $k_2[k_1,k_1^\prime]=w_1w_2\cdots w_s v_{n+1}v_{n+2}\cdots v_t:=A$. \par 
		Now by the Chinese Remainder Theorem, we may write $b=b_1 \frac{A}{w_1}\overline{\frac{A}{w_1}}+b_2 \frac{A}{w_2}\overline{\frac{A}{w_2}}+\cdots +b_s \frac{A}{w_s}\overline{\frac{A}{w_s}}+c_{n+1}\frac{A}{v_{n+1}}\overline{\frac{A}{v_{n+1}}}+\cdots +c_{t}\frac{A}{v_{t}}\overline{\frac{A}{v_{t}}}$ in \eqref{3}, and applying the multiplicativity relation $$\text{Kl}_3(a;pq)=\text{Kl}_3(\overline{p}^3a;q)\text{Kl}_3(\overline{q}^3a;p)$$ for $(p,q)=1$, we get
		\begin{align}\label{Notation}
			&\sum_{b\bmod k_2[k_1,k_1^\prime]}\text{Kl}_3(b\overline{\nu};k_2k_1)\overline{\text{Kl}_3(b\overline{\nu};k_2k_1^\prime )}\ex\bfrac{hb}{k_2[k_1,k_1^\prime]}\\
			&=\left(\prod_{i=1}^n \sum_{b_i\bmod w_i}\text{Kl}_3(\overline{\frac{k_1k_2}{w_i}}^3b_i\overline{\nu};w_i) \overline{\text{Kl}_3(\overline{\frac{k_1^\prime k_2}{v_i}}^3b_i\overline{\nu};v_i)}\ex\bfrac{hb_i\overline{\frac{A}{w_i}}}{w_i}\right)\nonumber\\
			&\times \left(\prod_{j=n+1}^{s}\sum_{b_j\bmod w_j}\text{Kl}_3(\overline{\frac{k_1k_2}{w_j}}^3b_j\overline{\nu};w_j)\ex\bfrac{hb_j\overline{\frac{A}{w_j}}}{w_j}\right)\times 
			\left(\prod_{k=n+1}^{t}\sum_{c_k\bmod v_k}\overline{\text{Kl}_3(\overline{\frac{k_1^\prime k_2}{v_k}}^3c_k\overline{\nu};v_k)}\ex\bfrac{hc_k\overline{\frac{A}{v_k}}}{v_k}\right).\nonumber
		\end{align}
		Therefore, we only need to compute the bound for the following sums:
		\begin{equation}\label{ntb}
			\left\{
			\begin{array}{l}
				\sum_{b\bmod p^\alpha}\text{Kl}_3(A^3\overline{\nu}b;p^\alpha)\overline{\text{Kl}_3(B^3\overline{\nu}b;p^\beta)}\ex \bfrac{Cb}{p^\alpha}\\
				\sum_{b\bmod p^\alpha}\text{Kl}_3(A^3\overline{\nu}b;p^\alpha)\ex \bfrac{Cb}{p^\alpha}
			\end{array}
			\right.
		\end{equation}
		where $\alpha\geq \beta\geq 1$ and $(AB\nu,p)=1$. 
		\begin{lem}\label{lem:primecase}
			For $\alpha=1$, $(AB\nu,p)=1$, we have
			$$
			\sum_{b\bmod p}\text{Kl}_3(A^3\overline{\nu}b;p)\overline{\text{Kl}_3(B^3\overline{\nu}b;p)}\,\ex \bfrac{Cb}{p}\ll p^{1/2}(A^3-B^3,C,p)^{1/2}
			$$
			and
			$$
			\sum_{b\bmod p}\text{Kl}_3(A^3\overline{\nu}b;p)\ex \bfrac{Cb}{p}\ll p^{1/2}(A^3,C,p)^{1/2},
			$$
			where the implied constant does not depend on $A$, $B$, $C$ and $\nu$.
		\end{lem}
		\begin{proof}  See \cite[Corollary 3.3]{FKM} for $(C,p)=1$ or $A^3\nequiv B^3\bmod p$.
		\end{proof}
		Next, we will compute the case of $p$-power. We will firstly quote a lemma, see \cite[Lemma 22]{BM}.
		
		\begin{lem}\label{KL2square}
			Let $p>2$ be a prime, let $s \geqslant 2$, and let $S\left(m, n, p^s\right)$ be the usual Kloosterman sum. Assume that $(m, p)=1$. Then $S\left(m, n, p^s\right)=0$ unless
			
			$$
			m n \in(\mathbb{Z} / p \mathbb{Z})^{\times 2},
			$$
			in which case it equals
			
			$$
			S\left(m, n, p^s\right)=p^{s / 2} \sum_{ \pm} \tau\left( \pm(m n)_{1 / 2}, p^s\right) e\left( \pm \frac{2(m n)_{1 / 2}}{p^s}\right).
			$$
		\end{lem}
		We may explain the notation of the lemma. We define $A^{\times 2}:=\{x\in A^\times : x=y^2,\ y\in A^\times\}$. Let $\tau$ denote the Gauss sum, and let $(\cdot)_{1/2}$ denote the $p$-adic square root, which we will discuss in next section. \par 
		We will end this section by giving a similar lemma, whose notation and proof will be given in the next section. One can view the lemma as an analogue of \cite[Eq. (1.16)]{DF}.
		\begin{lem}\label{Kloo}
			Let $p>3$ be a prime, let $s \geqslant 2$. Assume that $(mn,p)=1$. Then $\mathcal {KS}(m,n,u; p^s)$=0 unless
			$$
			unm \in(\mathbb{Z} / p \mathbb{Z})^{\times 3},
			$$
			in which case when $p\equiv 2\bmod3$ it equals 
			$$
			\mathcal {KS}(m,n,u; p^s)=p^s\tau((unm)_{1/3},p^s)\tau((1-\overline{4})unm,p^s)\ex \bfrac{3(unm)_{1/3}}{p^s},
			$$
			and when $p\equiv 1\bmod 3$ it equals
			$$
			\mathcal {KS}(m,n,u; p^s)=\sum_{i=1}^3 p^s\tau((unm)_{1/3},p^s)\tau((1-\overline{4})unm,p^s)\ex \bfrac{3w^i(unm)_{1/3}}{p^s},
			$$
			where $w\in \mathbb{Z}_p^\times$ is one solution of $w^2+w+1=0$.
		\end{lem}
		\section{\texorpdfstring{$p$-adic stationary phase computation}{p-adic stationary phase computation}}
		\label{padiccomputation}
		\subsection{\texorpdfstring{$p$-adic preliminaries}{p-adic preliminaries}}
		In this section, we mostly use the notation from \cite[Section 9.2]{BM}, and give the proof of our result.
		Let $p>2$ be a prime. For $s \geqslant 1$ and $(A, p)=1$, let
		
		\begin{equation}\label{tau}
			\tau\left(A, p^s\right):=p^{-s / 2} \sum_{x \bmod p^s} e\left(\frac{A x^2}{p^s}\right)= \begin{cases}1, & 2 \mid s, p \text { odd }, \\ \left(\frac{A}{p}\right), & 2 \nmid s, p \equiv 1 \quad(\bmod 4), \\ \left(\frac{A}{p}\right) i, & 2 \nmid s, p \equiv 3 \quad(\bmod 4),\end{cases}
		\end{equation}
		be the sign of the Gauss sum. From the definition we can see that $A$ depends only on the residue class modulo $p$. \par 
		Next we discuss the notation of $(\cdot)_{1/2}$ and $(\cdot)_{1/3}$ in the last section to make it sense. 
		We first ``define" $(\cdot)_{1/2}: \mathbb{Z}_p^{\times 2} \to \mathbb{Z}_p^\times \ x\mapsto \text{solution of }u^2=x  $ and $(\cdot)_{1/3}: \mathbb{Z}_p^{\times 3} \to \mathbb{Z}_p^\times \ x\mapsto \text{solution of }u^3=x  $, this is not well-defined since $u^2=x$ and $u^3=x$ may not have only one solution. We will further discuss in detail to make this ``definition" really well-defined.

		To start, we consider the solution of $u^2=x$ for given $x\in \mathbb{Z}_p^{\times}$. To solve this equation, we first work modulo $p$ on both sides, therefore we need to solve $u^2\equiv x\bmod p$ for $x\in(\mathbb{Z}/p\mathbb{Z})^\times$; this equation has solutions if and only if $\left(\frac{x}{p}\right)=1$, in which case $u\in (\mathbb{Z}/p\mathbb{Z})^\times $ has exactly two solutions. Fix one solution $u$, we can use Hensel’s lemma (for $p>2$) to lift the solution $u\in(\mathbb{Z}/p\mathbb{Z})^\times$ to all solutions of $u^2\equiv x\bmod p^k$ with $k\geq 1$ and $u\in(\mathbb{Z}/p^k\mathbb{Z})^\times$ uniquely, and we naturally get a solution $u=u_0\in \mathbb{Z}_p^\times$ satisfying $u_0^2=x$ with $u_0$ depending on $x$ . Hence, the function $(\cdot)_{1/2}:\mathbb{Z}_p^{\times 2} \to \mathbb{Z}_p^\times \ x\mapsto u_0$ is really well-defined after we choose one solution of $u^2\equiv x\bmod p$ for every $x\in \mathbb{Z}_p^{\times 2}$.
		\par 
		For $(\cdot)_{1/3}$, we have a similar argument. For $x\in \mathbb{Z}_p^\times$, to solve $u^3=x$, we only need to solve $u^3\equiv x\bmod p$ for $x\in (\mathbb{Z}/p\mathbb{Z})^\times$. Let $g$ be a primitive root of all $p^k$ for $k\geq 1$. Let $u\equiv g^r\bmod p$ and $x\equiv g^a\bmod p$ since $(x,p)=1$, hence, it suffices to solve $3r\equiv a(\bmod \ p-1)$; for $p\equiv 2  \bmod 3$, it has exactly one solution for any fixed $x\in (\mathbb{Z}/p\mathbb{Z})^\times$, and we denote this solution by $u^\prime $; for $p\equiv 1  \bmod 3$, it has exactly three solutions for any fixed $x\in (\mathbb{Z}/p\mathbb{Z})^\times$ satisfying $x\equiv g^{3a^\prime} \bmod p$, i.e. $x\in (\mathbb{Z}/p\mathbb{Z})^{\times 3}$, we denote these solutions $u_1$, $u_2$ and $u_3$ respectively. Similar with the above argument, by using Hensel’s lemma (for $p>3$), we can lift $u^\prime $, $u_1$, $u_2$, and $u_3$ into a solution in $\mathbb{Z}_p^\times$ uniquely. Hence, when $p\equiv 2  \bmod 3$, the map $(\cdot)_{1/3}:\mathbb{Z}_p^{\times 3} =\mathbb{Z}_p^{\times }\to \mathbb{Z}_p^\times \ x\mapsto u^\prime $ is well-defined, and when $p\equiv 1  \bmod 3$, the map $(\cdot)_{1/3,i}:\mathbb{Z}_p^{\times 3} \to \mathbb{Z}_p^\times \ x\mapsto u_i $ is also well-defined. We shall mention that when $p\equiv 1\bmod 3$, let $(\cdot)_{1/3}$ be one of the $(\cdot)_{1/3,i}$, so that $(\cdot)_{1/3}$ will cause no confusion, hence when $p\equiv 1  \bmod 3$, $x_{1/3}$, $wx_{1/3}$ and $w^2x_{1/3}$ are all the three solutions of $u^3=x$, where $w$ is a primitive cube root of unity, i.e. $w^3=1$ and $w\neq 1$. Moreover, we can extend the domain by the natural embedding $
		\mathbb{Z}\hookrightarrow\mathbb{Z}/p\mathbb{Z}$ and $
		\mathbb{Z}\hookrightarrow\mathbb{Z}_p$ that: for those $x\in \mathbb{Z}$ satisfying $x\in (\mathbb{Z}/p\mathbb{Z})^{\times 3}$, we view those $x$ as elements of $\mathbb{Z}_p^{\times 3}$. Therefore, we can view $(\cdot)_{1/2}$ and $(\cdot)_{1/3}$ as functions on subset of $\mathbb{Z}$ in our application.
		\par    
		For $u\in \mathbb{Z}_p^\times $, we denote by $\overline{x}=x^{-1}$ the inverse of $x$, which agrees with the multiplicative inverse of $x$ to any prime power $p^k$. By the definition of $(\cdot)_{1/2}$ and $(\cdot)_{1/3}$, we see immediately that $\overline{x}_{1/2}^2\equiv \overline{x}\bmod p^k$ and $\overline{x}_{1/3}^3\equiv \overline{x}\bmod p^k$ for all $k\geq 1$ simply because $x_{1/2}^2=x$ and $x_{1/3}^3=x$.
		
		After a simple calculation, we have 
		\begin{equation}\label{1/2}
			\left(u+p^\kappa t\right)_{1 / 2} \equiv u_{1 / 2}+\overline{2} \cdot \overline{u_{1 / 2}} p^\kappa t-\overline{8} \cdot{\overline{u_{1 / 2}}}^3 p^{2 \kappa} t^2 \left(\bmod p^{3 \kappa}\right)
		\end{equation}
		and
		\begin{equation}\label{1/3}
			\left(u+p^\kappa t\right)_{1 / 3} \equiv u_{1 / 3}+\overline{3} \cdot \overline{u_{1 / 3}}^2 p^\kappa t-\overline{9} \cdot{\overline{u_{1 / 3}}}^5 p^{2 \kappa} t^2 \left(\bmod p^{3 \kappa}\right)
		\end{equation}
		for all $\kappa \geq 1$, $t\in\mathbb{Z}$ and suitable $u$ such that $u_{1/2}$ or $u_{1/3}$ is defined. 
		\par 
		We shall quote a lemma to simplify our computation, see \cite[Lemma 6]{BM2}.
		\begin{lem}\label{lem:Useful}
			Let $p$ be an odd prime, $(A,p)=1, p^nB \in  \mathbb{Z}, s \in\{0,1\}, n \geqslant s$, and let
			
			$$
			S=\sum_{x \bmod p^n} \ex \left(A p^{-s} x^2+B x\right).
			$$
			Then $S=0$ unless $p^sB \in \mathbb{Z}$, in which case
			
			$$
			S=p^{n-s / 2} \tau\left(A, p^s\right) \ex\left(-\overline{4A}B^2 p^s\right).
			$$
		\end{lem}
		Now, we are ready to  give the proof of Lemma \ref{Kloo}.
		\subsection{Proof of Lemma \ref{Kloo}} 
		\begin{proof}
			By definition, we have
			$$
			\mathcal {KS}(m,n,u; p^s)=\sumstar_{a\bmod p^s}\ex \bfrac{\overline{a}m}{p^s}S(n,au;p^s).
			$$
			Since $(mn,p)=1$, it suffices to prove the situation of $	\mathcal {KS}(1,1,u; p^s)$.
			Applying Lemma \ref{KL2square}, we have
			\begin{equation}\label{KSstart}
				\mathcal {KS}(1,1,u; p^s)=p^{s/2}\sum_{ \pm}\ \sumstar_{a\bmod p^s}\ex \bfrac{\overline{a}}{p^s}\tau(\pm (au)_{1/2},p^s)\ex \left(\pm \frac{2(au)_{1/2}}{p^s}\right),
			\end{equation}
			where $au$ must belong to $(\mathbb{Z} / p \mathbb{Z})^{\times 2}$.\par 
			We make a change of variable that $a=a_1+p^ka_2$, where $k= \frac{s}{2}$ if $s$ is even and $k=\frac{s-1}{2}$ if $s$ is odd. Since $s\geq 2$, we have $k\geq 1$. Note that 
			$$\overline{u+p^\kappa t}\equiv \overline{u}-\overline{u}^2p^\kappa t+\overline{u}^3p^{2\kappa}t^2\bmod p^{3\kappa},$$
			therefore, we further apply \eqref{1/2} to get
			\begin{equation}\label{KSjinxing}
				\begin{aligned}
					\mathcal {KS}(1,1,u; p^s)=p^{s/2}&\sum_{ \pm}\ \sumstar_{a_1\bmod p^k}\tau(\pm (a_1u)_{1/2},p^s)\ex \bfrac{\overline{a}_1\pm 2(a_1u)_{1/2}}{p^s}\\
					&\cdot\sum_{a_2\bmod p^{s-k}}\ex \left(\frac{a_2(\pm \overline{(a_1u)}_{1/2}u-\overline{a}_1^2)}{p^{s-k}}+\frac{a_2^2(\overline{a_1}^3\mp \overline{4}\overline{(a_1u)}_{1/2}^3u^2)}{p^{s-2k} }\right).
				\end{aligned}
			\end{equation}
			When $p\mid \overline{a_1}^3\mp \overline{4}\overline{(a_1u)}_{1/2}^3u^2$, we get $a_1^3\equiv 16\overline{u}\bmod p$, in which case the inner sum vanishes unless $\pm \overline{(a_1u)}_{1/2}u\equiv \overline{a}_1^2\bmod p^{s-k}$. Therefore, we must have $a_1^3\equiv \overline{u}\bmod p$. Hence $p$ must satisfy $16\equiv 1\bmod p$, and we get $p=5$ since $p>3$. 
			
			We assume $p>5$ firstly, and we can get $(\overline{a_1}^3\mp \overline{4}\overline{(a_1u)}_{1/2}^3u^2,p)=1$. Hence, we apply Lemma \ref{lem:Useful} on the inner sum to infer that $\mathcal {KS}(1,1,u; p^s)=0$ unless $\frac{\pm \overline{(a_1u)}_{1/2}u-\overline{a}_1^2}{p^{s-k}}\cdot p^{s-2k}\in \mathbb{Z}$, i.e. $\pm \overline{(a_1u)}_{1/2}u-\overline{a}_1^2\equiv 0\bmod p^k$, in which case we have 
			\begin{equation}\label{KS}
				\begin{aligned}
					\mathcal {KS}(1,1,u; p^s)=p^s &\sum_{\pm }\sumstar_{\substack{a_1\bmod p^k\\ \pm \overline{(a_1u)}_{1/2}u-\overline{a}_1^2\equiv 0\bmod p^k}}\tau(\pm (a_1u)_{1/2},p^s)\tau(\overline{a_1}^3\mp\overline{4}\overline{(a_1u)}_{1/2}^3u^2,p^s)\\
					&\cdot \ex \bfrac{\overline{a_1}\pm 2(a_1u)_{1/2}}{p^s}\ex \bfrac{-\overline{4\left(\overline{a_1}^3\mp\overline{4}\overline{(a_1u)}_{1/2}^3u^2\right)}\cdot \left(\pm \overline{(a_1u)}_{1/2}u-\overline{a}_1^2\right)^2}{p^s}
				\end{aligned}
			\end{equation}
			since $\tau(A,p^s)=\tau(A,p^{s-2k})$ by \eqref{tau}. \par 
			We need to solve $\pm \overline{(a_1u)}_{1/2}u\equiv \overline{a}_1^2\bmod p^k$. Taking square on both sides we have $a_1^3\equiv \overline{u}\bmod p^k$, hence $u$ must belong to $(\mathbb{Z} / p \mathbb{Z})^{\times 3}$. 
			
			We first assume $p\equiv 2\bmod 3$, in which case $a_1^3\equiv \overline{u}\bmod p^k$ has only one solution $a_1\equiv \overline{u}_{1/3}\bmod p^k$. We should substitute this solution into $\pm \overline{(a_1u)}_{1/2}u\equiv \overline{a}_1^2\bmod p^k$ to check whether this solution is indeed a solution since the square operation is not invertible. 
			
			We have $(u^2_{1/3})_{1/2}=u_{1/3}$ or $-u_{1/3}$. Firstly, let $(u^2_{1/3})_{1/2}=u_{1/3}$. Substituting $a_1\equiv \overline{u}_{1/3}\bmod p^k$ into $\pm \overline{(a_1u)}_{1/2}u\equiv \overline{a}_1^2\bmod p^k$, we find that it is the solution of the positive part only. Similarly, when $(u^2_{1/3})_{1/2}=-u_{1/3}$, $a_1\equiv \overline{u}_{1/3}\bmod p^k$ is the solution of the negative part only.
			
			We shall make an important remark that just as Blomer and Mili\'cevi\'c's observation in \cite[Lemma 7]{BM2}, after a simple calculation, the summand on the right hand side is $p^k$-periodic, and hence we may substitute $a_1=\overline{u}_{1/3}$ by its $p$-adic value into \eqref{KS}, and after a simple calculation we find that the result does not depend on the sign of $(u^2_{1/3})_{1/2}$. Hence we get when $p\equiv 2\bmod 3$,
			$$
			\mathcal {KS}(1,1,u; p^s)=p^s\tau((u)_{1/3},p^s)\tau((1-\overline{4})u,p^s)\ex \bfrac{3(u)_{1/3}}{p^s}.
			$$
			Now we consider the case $p\equiv 1\bmod 3$, in which case $a_1^3\equiv \overline{u}\bmod p^k$ has three solutions $$a_1\equiv \overline{u_{1/3}},w\overline{u_{1/3}},w^2\overline{u_{1/3}} \bmod p^k.$$
			We have a similar argument as above. We substitute all three solutions and verify that they are indeed valid. We then observe that the sign of $(w^iu^2_{1/3})_{1/2}$ must be taken into account. Moreover, each case contributes only to either the positive or the negative part in \eqref{KS}. By the important remark above, we may substitute $a_1=w^i\overline{u_{1/3}}$ by its $p$-adic value into \eqref{KS}, and we finally get when $p\equiv 1\bmod 3$,
			\begin{equation}
				\begin{aligned}
					\mathcal {KS}(1,1,u; p^s)&=\sum_{i=1}^3 p^s\tau(w^i(u)_{1/3},p^s)\tau((1-\overline{4})u,p^s)\ex \bfrac{3w^i(u)_{1/3}}{p^s}\\
					&=\sum_{i=1}^3 p^s\tau((u)_{1/3},p^s)\tau((1-\overline{4})u,p^s)\ex \bfrac{3w^i(u)_{1/3}}{p^s}
				\end{aligned}
			\end{equation}
			since $\tau(x^2y,p^s)=\tau(y,p^s)$ and $w^3=1$.
			
			Now it remains to consider $p=5$. From \eqref{KSjinxing} we can see if $s$ is even, we can compute the inner sum immediately, and  the result follows the same strategy as above.\par
			If $s=2k+1$ is odd,  \eqref{KSjinxing} does not work. We shall use a different stationary phase method to evaluate the sum. We start from \eqref{KSstart}. We make a change of variable that $a=a_1+5^{k+1}a_2$, getting
			\begin{equation}\label{p5}
				\begin{aligned}
					\mathcal {KS}(1,1,u; 5^s)&=5^{s/2}\sum_{ \pm}\ \sumstar_{a_1\bmod 5^{k+1}}\tau(\pm (a_1u)_{1/2},5^s)\ex \bfrac{\overline{a}_1\pm 2(a_1u)_{1/2}}{5^s}\\
					&\hskip2in	\cdot	\sum_{a_2\bmod 5^{k}}\ex \left(\frac{a_2(\pm \overline{(a_1u)}_{1/2}u-\overline{a}_1^2)}{5^{k}}\right).
				\end{aligned}
			\end{equation}
			The inner sum vanishes unless $\pm \overline{(a_1u)}_{1/2}u-\overline{a}_1^2\equiv 0\bmod 5^k$, in which case we have
			\begin{equation}\label{p52}
				\begin{aligned}
					\mathcal {KS}(1,1,u; 5^s)=5^{s-1/2}\sum_{ \pm}\ \sumstar_{\substack{a_1\bmod 5^{k+1}\\ \pm \overline{(a_1u)}_{1/2}u-\overline{a}_1^2\equiv 0\bmod 5^k}}\tau(\pm (a_1u)_{1/2},5^s)\ex \bfrac{\overline{a}_1\pm 2(a_1u)_{1/2}}{5^s}.
				\end{aligned}
			\end{equation}
			Next we let $a_1=a_1^\prime+5^ka_2^\prime$, getting
			
			\begin{align}
				\mathcal {KS}(1,1,u; 5^s)&=5^{s-1/2}\sum_{ \pm}\ \sumstar_{\substack{a_1^\prime\bmod 5^{k}\\ \pm \overline{(a_1^\prime u)}_{1/2}u-\overline{a_1^\prime}^2\equiv 0\bmod 5^k}}\tau(\pm (a_1^\prime u)_{1/2},5^s)\ex \bfrac{\overline{a_1^\prime}\pm 2(a_1^\prime u)_{1/2}}{5^s}\\
				&\qquad\qquad\qquad\qquad\cdot \sum_{a_2^\prime \bmod 5}\ex \left(\frac{a_2^\prime(\pm \overline{(a_1^\prime u)}_{1/2}u-\overline{a^\prime_1}^2)}{5^{k+1}}+\frac{(a_2^\prime)^2(\overline{a_1^\prime}^3\mp \overline{4}\overline{(a_1^\prime u)}_{1/2}^3u^2)}{5 }\right).\nonumber
			\end{align}
			This sum is analogous to \eqref{KSjinxing}, but with an extra condition $\pm \overline{(a_1^\prime u)}_{1/2}u-\overline{a_1^\prime}^2\equiv 0\bmod 5^k$. We can get $(a_1^\prime)^3\equiv \overline{u}\bmod 5^k$, and $a_1^\prime\equiv \overline{u_{1/3}}\bmod 5^k$ since $5\equiv 2\bmod 3$. 
			
			When $(u_{1/3}^2)_{1/2}=u_{1/3}$, we see that $a_1^\prime\equiv \overline{u_{1/3}}\bmod 5^k$ is a solution that occurs only in the positive part
			$
			\overline{(a_1^\prime u)}_{1/2}u-\overline{a_1^\prime}^2\equiv 0\bmod 5^k.
			$
			In this case,
			$
			\overline{a_1^\prime}^3- \overline{4}\,\overline{(a_1^\prime u)}_{1/2}^3u^2\equiv (1-\overline{4})u\not\equiv 0\bmod 5,
			$
			and hence we may apply Lemma \ref{lem:Useful}, yielding
			$$
			\mathcal {KS}(1,1,u; 5^s)=5^s\tau((u)_{1/3},5^s)\tau((1-\overline{4})u,5^s)\ex \bfrac{3(u)_{1/3}}{5^s}.
			$$
			When $(u_{1/3}^2)_{1/2}=-u_{1/3}$, the solution lies only in the negative part, and we follow the same strategy, also getting
			$$
			\mathcal {KS}(1,1,u; 5^s)=5^s\tau((u)_{1/3},5^s)\tau((1-\overline{4})u,5^s)\ex \bfrac{3(u)_{1/3}}{5^s}.
			$$
		\end{proof}
		Next we shall give the bound for the sums \eqref{ntb}.
		\subsection{Proof of Theorem \ref{cruciallemma}}
		\begin{proof}
			We only give the proof on the hardest sum $\sum_{b\bmod p^\alpha}\text{Kl}_3(A\overline{\nu}b;p^\alpha)\overline{\text{Kl}_3(B\overline{\nu}b;p^\beta)}\, \ex \bfrac{Cb}{p^\alpha}$ and the hardest situation where $p\equiv 1\bmod 3$. The other situation uses the same method. By Lemma \ref{lem:primecase}, we can assume $\alpha\geq 2$. 
			
			We firstly assume $\beta \geq 2$, and use Lemma \ref{Kloo}, getting
			\begin{equation}\label{Begin}
				\begin{aligned}
					&\sum_{b\bmod p^\alpha}\text{Kl}_3(A\overline{\nu}b;p^\alpha)\overline{\text{Kl}_3(B\overline{\nu}b;p^\beta)} \,\ex \bfrac{Cb}{p^\alpha}=\sum_{i=1}^{3}\sum_{j=1}^{3}\ \sumstar_{b\bmod p^\alpha}
					\tau((A\overline{\nu}b)_{1/3},p^\alpha)\cdot \tau((1-\overline{4})A\overline{\nu}b,p^\alpha)\\
					&\qquad\qquad\qquad\cdot \overline{\tau((B\overline{\nu}b)_{1/3},p^\beta)}\cdot \overline{\tau((1-\overline{4})B\overline{\nu}b,p^\beta)}\,\ex\left(\frac{3w^i(A\overline{\nu}b)_{1/3}}{p^\alpha}-\frac{3w^j(B\overline{\nu}b)_{1/3}}{p^\beta }\right) \ex \bfrac{Cb}{p^\alpha}.
				\end{aligned}
			\end{equation}
			\par 
			Let us assume $\alpha=2t+1$ and $\beta=2k+1$ be odd firstly, with $k\leq t$.
			\par 
			We make a change of variable that $b=b_1+p^tb_2$, and apply \eqref{1/3}, getting
			\begin{align}\label{Bound}
				&\sum_{b\bmod p^\alpha}\text{Kl}_3(A\overline{\nu}b;p^\alpha)\overline{\text{Kl}_3(B\overline{\nu}b;p^\beta)}\,\ex \bfrac{Cb}{p^\alpha}=\sum_{i=1}^{3}\sum_{j=1}^{3}\ \sumstar_{b_1\bmod p^t}
				\tau((A\overline{\nu}b_1)_{1/3},p)\cdot \tau((1-\overline{4})A\overline{\nu}b_1,p)\\
				&\hskip1in\cdot	\overline{\tau((B\overline{\nu}b_1)_{1/3},p)}\cdot \overline{\tau((1-\overline{4})B\overline{\nu}b_1,p)} \,\ex \left(\frac{3w^i(A\overline{\nu}b_1)_{1/3}+Cb_1}{p^{2t+1}}-\frac{3w^j(B\overline{\nu}b_1)_{1/3}}{p^{2k+1}}\right)\nonumber\\
				&\hskip1in\cdot \sum_{b_2\bmod p^{t+1}}\ex \Bigg(b_2\left(\frac{w^i\overline{(A\overline{\nu}b_1)_{1/3}}^2A\overline{\nu}+C}{p^{t+1}}-\frac{w^j\overline{(B\overline{\nu}b_1)_{1/3}}^2B\overline{\nu}}{p^{2k+1-t}}\right)\nonumber\\
				&\hskip1in+b_2^2\left(\frac{-\overline{3}w^i\overline{(A\overline{\nu}b_1)_{1/3}}^5A^2\overline{\nu}^2}{p}+\frac{\overline{3}w^j\overline{(B\overline{\nu}b_1)_{1/3}}^5B^2\overline{\nu}^2}{p^{2k-2t+1}}\right)\Bigg).\nonumber
			\end{align}
			When $\beta <\alpha$, we have $2k-2t+1\leq -1$ and $(-\overline{3}w^i\overline{(Ab_1)_{1/3}}^5A^2,p)=1$, hence we apply Lemma \ref{lem:Useful}, the $b_2$-sum vanishes unless $p\left(\frac{w^i\overline{(A\overline{\nu}b_1)_{1/3}}^2A\overline{\nu}+C}{p^{t+1}}-\frac{w^j\overline{(B\overline{\nu}b_1)_{1/3}}^2B\overline{\nu}}{p^{2k+1-t}}\right)\in \mathbb{Z}$, i.e., $$w^i\overline{(A\overline{\nu}b_1)_{1/3}}^2A\overline{\nu}+C\equiv w^j\overline{(B\overline{\nu}b_1)_{1/3}}^2B\overline{\nu}p^{2t-2k}\bmod p^k.$$
			Though we may not have $(A\overline{\nu}b_1)_{1/3}=A_{1/3}(\overline{\nu}b_1)_{1/3}$, there must exist $a\in\mathbb{Z}$ such that 
			$(A\overline{\nu}b_1)_{1/3}=w^aA_{1/3}(\overline{\nu}b_1)_{1/3}$. Hence the above expression can be written as
			$$
			\overline{(\overline{\nu}b_1)_{1/3}}^2(w^{i_1}A_{1/3}-w^{j_1}B_{1/3}p^{2t-2k})\equiv -C\nu\bmod p^t.
			$$
			Cubing both sides, we obtain
			$$
			\overline{b_1}^2(w^{i_1}A_{1/3}-w^{j_1}B_{1/3}p^{2t-2k})^3\equiv -C^3\nu\bmod p^t.
			$$
			Since $(A,p)=1$, we must have $(w^{i_1}A_{1/3}-w^{j_1}B_{1/3}p^{2t-2k},p)=1$ unless $t=k$, i.e. $\alpha=\beta $. Hence for $\alpha>\beta $, $b_1$ has at most 2 solutions modulo $p^t$, in which case \eqref{Bound} can be bounded by 
			$$
			\sum_{b\bmod p^\alpha}\text{Kl}_3(A\overline{\nu}b;p^\alpha)\overline{\text{Kl}_3(B\overline{\nu}b;p^\beta)}\,\ex \bfrac{Cb}{p^\alpha}\ll p^{t+1/2}=p^{s/2}.
			$$
			
			Next we consider the case where $\alpha=\beta $, in which case $k=t$. We have 
			\begin{equation}\label{Bound2}
				\begin{aligned}
					&\sum_{b\bmod p^\alpha}\text{Kl}_3(A\overline{\nu}b;p^\alpha)\overline{\text{Kl}_3(B\overline{\nu}b;p^\beta)}\,\ex \bfrac{Cb}{p^\alpha}=\sum_{i=1}^{3}\sum_{j=1}^{3}\ \sumstar_{b_1\bmod p^t}
					\tau((A\overline{\nu}b_1)_{1/3},p)\cdot \tau((1-\overline{4})A\overline{\nu}b_1,p)\\
					&\hskip1in\cdot	\overline{\tau((B\overline{\nu}b_1)_{1/3},p)}\cdot \overline{\tau((1-\overline{4})B\overline{\nu}b_1,p)}\,\ex \left(\frac{3w^i(A\overline{\nu}b_1)_{1/3}+Cb_1}{p^{2t+1}}-\frac{3w^j(B\overline{\nu}b_1)_{1/3}}{p^{2t+1}}\right)\\
					&\hskip1in\cdot \sum_{b_2\bmod p^{t+1}}\ex \Bigg(b_2\left(\frac{w^i\overline{(A\overline{\nu}b_1)_{1/3}}^2A\overline{\nu}+C}{p^{t+1}}-\frac{w^j\overline{(B\overline{\nu}b_1)_{1/3}}^2B\overline{\nu}}{p^{t+1}}\right)\\
					&\hskip1in+b_2^2\left(\frac{-\overline{3}w^i\overline{(A\overline{\nu}b_1)_{1/3}}^5A^2\overline{\nu}^2}{p}+\frac{\overline{3}w^j\overline{(B\overline{\nu}b_1)_{1/3}}^5B^2\overline{\nu}^2}{p}\right)\Bigg).
				\end{aligned}
			\end{equation}
			For $p\mid (\overline{3}w^j\overline{(B\overline{\nu}b_1)_{1/3}}^5B^2\overline{\nu}^2-\overline{3}w^i\overline{(A\overline{\nu}b_1)_{1/3}}^5A^2\overline{\nu}^2)$, we must have $A\equiv B\bmod p$. Hence for $A\nequiv B\bmod p$, we have $(\overline{3}w^j\overline{(B\overline{\nu}b_1)_{1/3}}^5B^2\overline{\nu}^2-\overline{3}w^i\overline{(A\overline{\nu}b_1)_{1/3}}^5A^2\overline{\nu}^2,p)=1$, and we can apply Lemma \ref{lem:Useful},  the $b_2$-sum vanishes unless $p\left(\frac{w^i\overline{(A\overline{\nu}b_1)_{1/3}}^2A\overline{\nu}+C}{p^{t+1}}-\frac{w^j\overline{(B\overline{\nu}b_1)_{1/3}}^2B\overline{\nu}}{p^{t+1}}\right)\in \mathbb{Z}$, i.e., $$w^i\overline{(A\overline{\nu}b_1)_{1/3}}^2A\overline{\nu}+C\equiv w^j\overline{(B\overline{\nu}b_1)_{1/3}}^2B\overline{\nu}\bmod p^t.$$
			Following the same strategy as above, we get
			$$	\overline{(\overline{\nu}b_1)_{1/3}}^2(w^{i_1}A_{1/3}-w^{j_1}B_{1/3})\equiv -C\nu\bmod p^t.$$
			Since $A\nequiv B\bmod p$, $b_1$ also has at most 2 solutions, and we can bound 
			$$
			\sum_{b\bmod p^\alpha}\text{Kl}_3(A\overline{\nu}b;p^\alpha)\overline{\text{Kl}_3(B\overline{\nu}b;p^\beta)}\,\ex \bfrac{Cb}{p^\alpha}\ll p^{t+1/2}=p^{s/2}.
			$$
			It remains to consider $A\equiv B \bmod p$, in which case in \eqref{Begin} the Gauss sum $	\tau((A\overline{\nu}b_1)_{1/3},p)\cdot \tau((1-\overline{4})A\overline{\nu}b_1,p)\overline{\tau((B\overline{\nu}b_1)_{1/3},p)}\cdot \overline{\tau((1-\overline{4})B\overline{\nu}b_1,p)}$ is trivial. We get 
			\begin{equation}\label{Begin2}
				\begin{aligned}
					\sum_{b\bmod p^\alpha}\text{Kl}_3(A\overline{\nu}b;p^\alpha)\overline{\text{Kl}_3(B\overline{\nu}b;p^\beta)}\,\ex \bfrac{Cb}{p^\alpha}=\sum_{i=1}^{3}\sum_{j=1}^{3}\ \sumstar_{b\bmod p^\alpha}
					\ex\left(\frac{3w^i(A\overline{\nu}b)_{1/3}}{p^\alpha}-\frac{3w^j(B\overline{\nu}b)_{1/3}}{p^\alpha }\right) \ex \bfrac{Cb}{p^\alpha}.
				\end{aligned}
			\end{equation}
			We set $A=B+p^rm$, where $(m,p)=1$. We have $(A\overline{\nu}b)_{1/3}\equiv (B\overline{\nu}b)_{1/3}\bmod p^r$. 
			
			For $i\neq j$, we set $b=b_1+p^tb_2$, getting
			\begin{equation}
				\begin{aligned}
					\sum_{b\bmod p^\alpha}\text{Kl}_3(A\overline{\nu}b;p^\alpha)\overline{\text{Kl}_3(B\overline{\nu}b;p^\alpha)}\,\ex \bfrac{Cb}{p^\alpha}&=\sum_{i=1}^{3}\sum_{\substack{j=1\\i\neq j}}^{3}\ \sumstar_{b_1\bmod p^t}
					\ex \left(\frac{3w^i(A\overline{\nu}b_1)_{1/3}+Cb_1}{p^{2t+1}}-\frac{3w^j(B\overline{\nu}b_1)_{1/3}}{p^{2t+1}}\right)\\
					&\cdot \sum_{b_2\bmod p^{t+1}}\ex \Bigg(b_2\left(\frac{w^i\overline{(A\overline{\nu}b_1)_{1/3}}^2A\overline{\nu}+C}{p^{t+1}}-\frac{w^j\overline{(B\overline{\nu}b_1)_{1/3}}^2B\overline{\nu}}{p^{t+1}}\right)\\
					&+b_2^2\left(\frac{-\overline{3}w^i\overline{(A\overline{\nu}b_1)_{1/3}}^5A^2\overline{\nu}^2}{p}+\frac{\overline{3}w^j\overline{(B\overline{\nu}b_1)_{1/3}}^5B^2\overline{\nu}^2}{p}\right)\Bigg)\\
					&+\sum_{i=1}^3 \ \sumstar_{b\bmod p^\alpha}
					\ex\left(\frac{3w^i(A\overline{\nu}b)_{1/3}}{p^\alpha}-\frac{3w^i(B\overline{\nu}b)_{1/3}}{p^\alpha }\right) \ex \bfrac{Cb}{p^\alpha}.
				\end{aligned}
			\end{equation}
			For $i\neq j$, $(\overline{3}w^j\overline{(B\overline{\nu}b_1)_{1/3}}^5B^2\overline{\nu}^2-\overline{3}w^i\overline{(A\overline{\nu}b_1)_{1/3}}^5A^2\overline{\nu}^2,p)=1$, hence we apply Lemma \ref{lem:Useful}, the $b_2$-sum vanishes unless $w^i\overline{(A\overline{\nu}b_1)_{1/3}}^2A\overline{\nu}-w^j\overline{(B\overline{\nu}b_1)_{1/3}}^2B\overline{\nu}\equiv-C\bmod p^t$. We assume $(A\overline{\nu}b_1)_{1/3}=w^aA_{1/3}(\overline{\nu}b_1)_{1/3}$ and $(B\overline{\nu}b_1)_{1/3}=w^bB_{1/3}(\overline{\nu}b_1)_{1/3}$. We have $w^a\equiv w^b \bmod p$ since $(A\overline{\nu}b_1)_{1/3}\equiv (B\overline{\nu}b_1)_{1/3}\bmod p^r$ and we must have $a=b$, hence $b_1$ satisfies $(w^{i+a}A_{1/3}-w^{j+a}B_{1/3})\overline{(\overline{\nu}b_1)_{1/3}}^2\equiv -C\nu\bmod p^t$. We claim that $(w^{i+a}A_{1/3}-w^{j+a}B_{1/3},p)=1$. Otherwise, we have $w^{i+a}A_{1/3}\equiv w^{j+a}B_{1/3}\bmod p$, hence we get $w^i\equiv w^j\bmod p$ since $A_{1/3}\equiv B_{1/3}\bmod p^r$, which leads to a contradiction. Therefore, $b_1$ has at most 2 solutions modula $p^t$, and we can get a bound $p^{t+1/2}=p^{\alpha/2}$.
			\par 
			It remains to consider
			$$
			\sum_{i=1}^3 \ \sumstar_{b\bmod p^\alpha}
			\ex\left(\frac{3w^i(A\overline{\nu}b)_{1/3}}{p^\alpha}-\frac{3w^i(B\overline{\nu}b)_{1/3}}{p^\alpha }\right) \ex \bfrac{Cb}{p^\alpha}.
			$$
			For $r\geq \alpha$, we have $(A\overline{\nu}b)_{1/3}\equiv (B\overline{\nu}b)_{1/3}\bmod p^\alpha$. We can easily get that the above expression equals to
			$$
			3\ \sumstar_{b\bmod p^\alpha}\ex \bfrac{Cb}{p^\alpha}=3\ \sumstar_{b_1\bmod p}\ex\bfrac{Cb_1}{p^\alpha}\sum_{b_2\bmod p^{\alpha-1}}\ex\bfrac{Cb_2}{p^{\alpha-1}}\ll p^{\alpha/2}(C,p^\alpha)^{1/2}=p^{\alpha/2}(A-B,C,p^\alpha)^{1/2}.
			$$
			For $r<\alpha$, we make a change of variable that $b=b_1+p^{\alpha-r}b_2$, getting
			\begin{align*}
				&	\sum_{i=1}^3 \ \sumstar_{b\bmod p^\alpha}
				\ex\left(\frac{3w^i(A\overline{\nu}b)_{1/3}}{p^\alpha}-\frac{3w^i(B\overline{\nu}b)_{1/3}}{p^\alpha }\right) \ex \bfrac{Cb}{p^\alpha}\\
				&=\sum_{i=1}^3\ \sumstar_{b_1\bmod p^{\alpha-r}}
				\ex\bfrac{3w^i(A\overline{\nu}b_1)_{1/3}-3w^i(B\overline{\nu}b_1)_{1/3}+Cb_1}{p^\alpha}\sum_{b_2\bmod p^r}\ex\bfrac{Cb_2}{p^r}\\
				&= p^r\sum_{i=1}^3\ \sumstar_{b_1\bmod p^{\alpha-r}}
				\ex\bfrac{3w^i(A\overline{\nu}b_1)_{1/3}-3w^i(B\overline{\nu}b_1)_{1/3}+Cb_1}{p^\alpha}\delta_{p^r\mid C}\\
				&=p^r\sum_{i=1}^3\ \sumstar_{b_1\bmod p^{\alpha-r}}
				\ex\bfrac{3w^ip^{-r}\left((A\overline{\nu}b_1)_{1/3}-(B\overline{\nu}b_1)_{1/3}\right)+C^\prime b_1}{p^{\alpha-r}}
			\end{align*}
			since $(Ab_1)_{1/3}\equiv (Bb_1)_{1/3}\bmod p^r$. When $\alpha-r=2u+1$ is odd, we make a change of variable that $b_1=b_1^\prime+p^{u}b_2^\prime$; following the same strategy as above, the $b_2^\prime$-sum contributes $p^{u+1/2}$, and $b_1^\prime$ has at most 2 solutions modulo $p^u$, hence 
			\begin{align*}
				\sum_{i=1}^3 \ \sumstar_{b\bmod p^\alpha}
				&\ex\left(\frac{3w^i(A\overline{\nu}b)_{1/3}}{p^\alpha}-\frac{3w^i(B\overline{\nu}b)_{1/3}}{p^\alpha }\right) \ex \bfrac{Cb}{p^\alpha}\\&\ll p^{r+u+1/2}=p^{\alpha/2+r/2}=p^{\alpha/2}(A-B,p^\alpha)^{1/2}=p^{\alpha/2}(A-B,C,p^\alpha)^{1/2}.
			\end{align*}
			When $\alpha-r$ is even, the same method can be applied and we may omit it.\par 
			We are left with the case $\beta=1$. It remains to consider
			\begin{align*}
				\sum_{b\bmod p^\alpha}\text{Kl}_3(A\overline{\nu}b;p^\alpha)&\overline{\text{Kl}_3(B\overline{\nu}b;p)}\,\ex \bfrac{Cb}{p^\alpha} \\
				&=\sum_{i=1}^{3} \ \sumstar_{b\bmod p^\alpha}\tau((A\overline{\nu}b)_{1/3},p^\alpha)\tau((1-\overline{4})A\overline{\nu}b,p^\alpha)\ex\bfrac{3w^i(A\overline{\nu}b)_{1/3}+Cb}{p^\alpha}\overline{\text{Kl}_3(B\overline{\nu}b;p)}.
			\end{align*}
			We make a change of variable that $b=b_1+p^tb_2$; following the same strategy as above, we can have a bound $p^{\alpha/2}$.
			
			For $\alpha$ or $\beta$ is even, the Gauss sum $\tau$ is trivial. Hence the proof is simpler, and we omit the details.
		\end{proof}
		\subsection{Proof of Theorem \ref{Squarerootcancel}}
		\begin{proof}
			We use notation setting up in \eqref{Notation}, and apply Theorem \ref{cruciallemma}, getting
			\begin{equation}
				\begin{aligned}
					\hskip 0.5in		&\sum_{b\bmod k_2[k_1,k_1^\prime]}\text{Kl}_3(b\overline{\nu};k_2k_1)\overline{\text{Kl}_3(b\overline{\nu};k_2k_1^\prime )}\,\ex\bfrac{hb}{k_2[k_1,k_1^\prime]}\\
					&\ll \prod_{\substack{i=1\\ \alpha_i=\beta_i}}^{n}p_i^{\alpha_i/2}\left(\overline{\frac{k_1k_2}{p_i^{\alpha_i}}}^3-\overline{\frac{k_1^\prime k_2}{p_i^{\beta_i}}}^3,h\overline{\frac{A}{p_i^{\alpha_i}}},p_i^{\alpha_i}\right)^{1/2}\cdot  \prod_{\substack{i=1 \\ \alpha_i\neq \beta_i}}p_i^{\alpha_i/2}\prod_{j=n+1}^sp_j^{\alpha_j/2}\prod_{k=n+1}^t u_k^{\beta_k/2}\\
					& =(k_2[k_1,k_1^\prime])^{1/2}\cdot \prod_{\substack{i=1\\ \alpha_i=\beta_i}}^{n} \left(\overline{\frac{k_1k_2}{p_i^{\alpha_i}}}^3-\overline{\frac{k_1^\prime k_2}{p_i^{\beta_i}}}^3,h\overline{\frac{A}{p_i^{\alpha_i}}},p_i^{\alpha_i}\right)^{1/2}\\
					&=(k_2[k_1,k_1^\prime])^{1/2}\cdot \prod_{\substack{i=1\\ \alpha_i=\beta_i}}^{n} \left(\overline{\frac{k_1k_2}{p_i^{\alpha_i}}}^3-\overline{\frac{k_1^\prime k_2}{p_i^{\beta_i}}}^3,h,p_i^{\alpha_i}\right)^{1/2}\\
					&\ll (k_2[k_1,k_1^\prime])^{1/2}\cdot \prod_{\substack{i=1\\ \alpha_i=\beta_i}}^{n} \left(h,p_i^{\alpha_i}\right)^{1/2}\\
					&=(k_2[k_1,k_1^\prime])^{1/2}\cdot (h,\prod_{\substack{i=1\\ \alpha_i=\beta_i}}^np_i^{\alpha_i})^{1/2}\\
					&=(k_2[k_1,k_1^\prime])^{1/2}\cdot(h,(k_1k_2,k_1^\prime k_2)^{\#})^{1/2}.\nonumber
				\end{aligned}
			\end{equation}
		\end{proof}
		\appendix
		
		\section{Proof of Theorem \ref{MTFinal}}\label{Appendixa}
		We let $\beta_1,\beta_2\to 0$ in Theorem \ref{Main theorem}. Throughout this section, we write $\sigma(m;\alpha_1,\alpha_2)$ for $\sigma(m;\al)$ and $\B_p(s;\alpha_1,\alpha_2,\beta_1,\beta_2)$ for $\B_p(s;\al,\be)$. For notational convenience, we further set $\alpha_1=x$ and $\alpha_2=y$. It remains to estimate
		\begin{align}\label{FinalMTproof}
			\begin{aligned}
				&	\left(\frac{q}{\pi}\right)^{\frac{x}{2}}\left(\frac{q}{2\pi}\right)^{y}
				\Gamma\left(\frac{1}{4}+\frac{x}{2}\right)\Gamma\left(\frac14\right)
				\Gamma\left(\frac{\kappa}{2}+y\right)\Gamma\left(\frac{\kappa}{2}\right)\prod_{p\nmid q}\B_p\left(\frac12;x,y,0,0\right)\\
				&+	\left(\frac{q}{\pi}\right)^{-\frac{x}{2}}\left(\frac{q}{2\pi}\right)^{y}
				\Gamma\left(\frac{1}{4}-\frac{x}{2}\right)\Gamma\left(\frac14\right)
				\Gamma\left(\frac{\kappa}{2}+y\right)\Gamma\left(\frac{\kappa}{2}\right)\prod_{p\nmid q}\B_p\left(\frac12;-x,0,0,-y\right)\\
				&+	\left(\frac{q}{\pi}\right)^{\frac{x}{2}}\left(\frac{q}{2\pi}\right)^{-y}
				\Gamma\left(\frac{1}{4}+\frac{x}{2}\right)\Gamma\left(\frac14\right)
				\Gamma\left(\frac{\kappa}{2}-y\right)\Gamma\left(\frac{\kappa}{2}\right)\prod_{p\nmid q}\B_p\left(\frac12;x,0,0,y\right)\\
				&+	\left(\frac{q}{\pi}\right)^{-\frac{x}{2}}\left(\frac{q}{2\pi}\right)^{-y}
				\Gamma\left(\frac{1}{4}-\frac{x}{2}\right)\Gamma\left(\frac14\right)
				\Gamma\left(\frac{\kappa}{2}-y\right)\Gamma\left(\frac{\kappa}{2}\right)\prod_{p\nmid q}\B_p\left(\frac12;-x,-y,0,0\right).
			\end{aligned}
		\end{align}
		Next, we derive the explicit form of $\B_p\left(\frac12;x_1,x_2,y_1,y_2\right)$ using the Rankin--Selberg method; see \cite[Section 1.6]{Bump} for a detailed exposition. By Dirichlet convolution, for $|x|<1$ (since finally we shall let $x_i,y_j\to 0$) we obtain
		\begin{align}
			\sum_{r=0}^\infty \sigma\left(p^r;x_1,x_2\right)x^r&=\sum_{r=0}^\infty\left(\frac{x}{p^{x_1}}\right)^r\sum_{r=0}^\infty\lambda_f(p^r)\left(\frac{x}{p^{x_2}}\right)^r\\
			&=\left(1-p^{-x_1}x\right)^{-1}\left(1-p^{-x_2}\alpha_px\right)^{-1}\left(1-p^{-x_2}\beta_px\right)^{-1}\nonumber
		\end{align}
		and
		\begin{align}
			\sum_{r=0}^\infty \sigma\left(p^r;-y_1,-y_2\right)x^r
			=\left(1-p^{y_1}x\right)^{-1}\left(1-p^{y_2}\alpha_px\right)^{-1}\left(1-p^{y_2}\beta_px\right)^{-1},\nonumber
		\end{align}
		where $\alpha_p+\beta_p=\lambda_f(p)$ and $\alpha_p\beta_p=1$.\\
		Hence  
		\begin{align*}
			&	\B_p\left(\frac12;x_1,x_2,y_1,y_2\right)=\sum_{r=0}^\infty\frac{\sigma\left(p^r;x_1,x_2\right)\sigma\left(p^r;-y_1,-y_2\right)}{p^r}\\
			&=\sumtwo_{\substack{r,r^\prime}}\frac{\sigma\left(p^r;x_1,x_2\right)\sigma\left(p^{r^\prime};-y_1,-y_2\right)}{p^r}\cdot\frac{1}{2\pi i}\int_{C}s^{r-r^\prime}\frac{ds}{s}\\
			&=\frac{1}{2\pi i}\int_{C}\sum_r \sigma\left(p^r;x_1,x_2\right)\left(\frac{s}{p}\right)^r\sum_{r^\prime}\sigma\left(p^{r^\prime};-y_1,-y_2\right)\left(\frac{1}{s}\right)^{r^\prime}\frac{ds}{s}\\
			&= \frac{1}{2\pi i}\int_{C} \left(1-\frac{s}{p^{1+x_1}}\right)^{-1}\left(1-\frac{\alpha_ps}{p^{1+x_2}}\right)^{-1}\left(1-\frac{\beta_ps}{p^{1+x_2}}\right)^{-1}\\
			&\hskip2in\cdot\left(1-\frac{p^{y_1}}{s}\right)^{-1}\left(1-\frac{p^{y_2}\alpha_p}{s}\right)^{-1}\left(1-\frac{p^{y_2}\beta_p}{s}\right)^{-1}\frac{ds}{s},
		\end{align*}
		where $C$ is the positively oriented circle centered at the origin with
		$
		|s|=\frac{2p}{3}
		$. Hence, when $\alpha_p\neq \beta_p$, by the Cauchy residue theorem, the integrand has three simple poles at
		\[
		s=p^{y_1},\quad p^{y_2}\alpha_p,\quad p^{y_2}\beta_p.
		\]
		Therefore, $	\B_p\left(\frac12;x_1,x_2,y_1,y_2\right)$ equals
		\begin{align*}
			&=\left(1-\frac{1}{p^{1+x_1-y_1}}\right)^{-1}\left(1-\frac{\alpha_p}{p^{1+x_2-y_1}}\right)^{-1}\left(1-\frac{\beta_p}{p^{1+x_2-y_1}}\right)^{-1}\cdot\frac{p^{2y_1}}{(p^{y_1}-p^{y_2}\alpha_p)(p^{y_1}-p^{y_2}\beta_p)}\\
			&+\left(1-\frac{\alpha_p}{p^{1+x_1-y_2}}\right)^{-1}\left(1-\frac{\alpha^2_p}{p^{1+x_2-y_2}}\right)^{-1}\left(1-\frac{1}{p^{1+x_2-y_2}}\right)^{-1}\cdot\frac{p^{y_2}\alpha_p^2}{(p^{y_2}\alpha_p-p^{y_1})(\alpha_p-\beta_p)}\\
			&+\left(1-\frac{\beta_p}{p^{1+x_1-y_2}}\right)^{-1}\left(1-\frac{\beta^2_p}{p^{1+x_2-y_2}}\right)^{-1}\left(1-\frac{1}{p^{1+x_2-y_2}}\right)^{-1}\cdot\frac{p^{y_2}\beta_p^2}{(p^{y_2}\beta_p-p^{y_1})(\beta_p-\alpha_p)}.
		\end{align*}
		With a little calculation, we have 
		\begin{equation}\label{Bp}
			\B_p\left(\frac12;x,y,0,0\right)=\zeta_p(1+x)\zeta_p(1+y)L_p(1+x,f)L_p(1+y,f)L_p(1+y,\text{sym}^2f)G_p(x,y)
		\end{equation}
		and 
		\begin{equation}\label{Bpp}
			\B_p\left(\frac12;x,0,0,y\right)=\zeta_p(1+x)\zeta_p(1-y)L_p(1,f)L_p(1+x-y,f)L_p(1-y,\text{sym}^2f)H_p(x,y),
		\end{equation}
		where
		\begin{align*}
			G_p(x,y)&=1-\frac{\lambda_f^2(p)+\lambda_f(p)}{p^{2+x+y}}-\frac{\lambda_f(p)+1}{p^{2+2y}}+\frac{2\lambda_f^2(p)+2\lambda_f(p)}{p^{3+x+2y}}+\frac{\lambda_f(p)}{p^{3+2x+y}}+\frac{\lambda_f(p)}{p^{3+3y}}\\
			&\hskip2in-\frac{\lambda_f(p)+1}{p^{4+2x+2y}}-\frac{\lambda_f^2(p)+\lambda_f(p)}{p^{4+x+3y}}+\frac{1}{p^{6+2x+4y}}
		\end{align*}
		and
		\begin{align*}
			H_p(x,y)&=1-\frac{1}{p^{2-2y}}-\frac{\lambda_f(p)}{p^{2-y}}-\frac{\lambda_f(p)}{p^{2+x-2y}}-\frac{\lambda_f^2(p)}{p^{2+x-y}}+\frac{\lambda_f(p)}{p^{3+x-y}}+\frac{\lambda_f(p)}{p^{3-2y}}+\frac{\lambda_f(p)}{p^{3+2x-2y}}+\frac{\lambda_f(p)}{p^{3+x-3y}}\\
			&\hskip1.2in+\frac{2\lambda_f^2(p)}{p^{3+x-2y}}-\frac{1}{p^{4+2x-2y}}-\frac{\lambda_f(p)}{p^{4+x-2y}}-\frac{\lambda_f(p)}{p^{4+2x-3y}}-\frac{\lambda_f^2(p)}{p^{4+x-3y}}+\frac{1}{p^{6+2x-4y}}.
		\end{align*}
		Note that we have  $G_p(0,0)=H_p(0,0)=\left(1-\frac1p\right)^2\left(1-\frac{\lambda_f(p)}{p}+\frac{1}{p^2}\right)\left(1+\frac{\lambda_f(p)+2}{p}+\frac{1}{p^2}\right)$.\\
		It remains to consider $\alpha_p=\beta_p$ in which case $\alpha_p=1$ or $\alpha_p=-1$. At this stage we view \eqref{Bp} and \eqref{Bpp} as equations of $\alpha_p$ with $\beta_p=\alpha_p^{-1}$ and $\lambda_f(p)=\alpha_p+\alpha_p^{-1}$, and let $\alpha_p\to 1$ and $\alpha_p\to -1$ respectively. By continuity, the identities \eqref{Bp} and \eqref{Bpp} remain valid when $\alpha_p=\beta_p$.

		Substituting $\B_p\left(\frac12;x,y,0,0\right)$ and
		$\B_p\left(\frac12;x,0,0,y\right)$ into \eqref{FinalMTproof}, and using the Laurent expansion
		\[
		\zeta(1+x)=\frac{1}{x}+\gamma+g(x),
		\qquad g(x)=o(1),
		\]
		we then apply L'Hôpital's rule by letting $x,y\to 0$. After a straightforward calculation, we obtain the terms appearing in Theorem \ref{MTFinal},
		where 
		$$
		c_1=\prod_p\left(1-\frac1p\right)^2\left(1+\frac{\lambda_f(p)+2}{p}+\frac{1}{p^2}\right),
		$$
		$$
		c_2=-\log\pi+2\gamma +\frac{\Gamma^\prime(\frac14)}{\Gamma(\frac14)}+2\frac{L^\prime(1,f)}{L(1,f)}+2\sum_p\frac{\left.\frac{\partial}{\partial x}G_p(x,0)\right|_{x=0}}{G_p(0,0)},
		$$
		$$
		c_3=-\log(2\pi)+\gamma+\frac{\Gamma^\prime(\frac{\kappa}{2})}{\Gamma(\frac{\kappa}{2})}+\frac{L^\prime(1,f)}{L(1,f)}+\frac{L^\prime(1,\text{sym}^2f)}{L(1,\text{sym}^2f)}+\sum_p\frac{\left.\frac{\partial}{\partial y}G_p(0,y)\right|_{y=0}}{G_p(0,0)},
		$$
		\begin{align*}
			c_4&=\frac{L^{\prime\prime}(1,f)L(1,f)-\left(L^{\prime}(1,f)\right)^2}{L^2(1,f)}\\
			&\hskip0.8in+\sum_p\frac{\left.\frac{\partial}{\partial x}\frac{\partial}{\partial y}\left(G_p(x,y)-H_p(x,y)\right)\right|_{x=y=0}G_p(0,0)-2\left.\frac{\partial}{\partial y}G_p(0,y)\right|_{y=0}\left.\frac{\partial}{\partial x}G_p(x,0)\right|_{x=0}}{G_p^2(0,0)},
		\end{align*}
		$$
		h_q=\prod_{p\mid q}\frac{\left(1-\frac{\lambda_f(p)}{p}+\frac{1}{p^2}\right)\left(1-\frac{\lambda_f(p^2)}{p}+\frac{\lambda_f(p^2)}{p^2}-\frac{1}{p^3}\right)}{1+\frac{\lambda_f(p)+2}{p}+\frac{1}{p^2}},
		$$
		$$
		i_q=2\sum_{p\mid q}\frac{\log p}{p-1}-2\frac{L^\prime_q(1,f)}{L_q(1,f)}-2\sum_{p\mid q}\frac{\left.\frac{\partial}{\partial x}G_p(x,0)\right|_{x=0}}{G_p(0,0)},
		$$
		$$
		j_q=\sum_{p\mid q}\frac{\log p}{p-1}-\frac{L^\prime_q(1,f)}{L_q(1,f)}-\frac{L^\prime_q(1,\text{sym}^2f)}{L_q(1,\text{sym}^2f)}-\sum_{p\mid q}\frac{\left.\frac{\partial}{\partial y}G_p(0,y)\right|_{y=0}}{G_p(0,0)},
		$$
		\begin{align*}
			k_q&=-\frac{L^{\prime\prime}_q(1,f)L_q(1,f)-\left(L^{\prime}_q(1,f)\right)^2}{L^2_q(1,f)}\\
			&\hskip0.8in-\sum_{p\mid q}\frac{\left.\frac{\partial}{\partial x}\frac{\partial}{\partial y}\left(G_p(x,y)-H_p(x,y)\right)\right|_{x=y=0}G_p(0,0)-2\left.\frac{\partial}{\partial y}G_p(0,y)\right|_{y=0}\left.\frac{\partial}{\partial x}G_p(x,0)\right|_{x=0}}{G_p^2(0,0)},
		\end{align*}
		with $$
		L_q(s,f)=\prod_{p\mid q}\left(1-\frac{\lambda_f(p)}{p^s}+\frac{1}{p^{2s}}\right)^{-1}
		$$
		and
		$$
		L_q(s,\text{sym}^2f)=\prod_{p\mid q}\left(1-\frac{\lambda_f(p^2)}{p^s}+\frac{\lambda_f(p^2)}{p^{2s}}-\frac{1}{p^{3s}}\right)^{-1}.
		$$
	
 \subsection*{Acknowledgments}  The authors would like to thank Prof. Jianya Liu for his encouragement. 

This work was partially supported by the Key R\&D Program of Shandong Province, China (No. 2026CXPT038) and the National Key R\&D Program of China (No. 2021YFA1000700).

	\end{document}